\documentclass[twoside,11pt]{article}

\usepackage{jmlr2e}

\usepackage{mathrsfs}
\usepackage{color}
\usepackage{bm}
\usepackage{algorithm}
\usepackage{algpseudocode}
\usepackage{algpascal}
\usepackage{url}
\usepackage{afterpage}
\usepackage{enumitem} 
\usepackage{multirow}
\usepackage{booktabs}
\usepackage{subfigure}

\usepackage{mathtools}
\usepackage{dsfont}
\usepackage{alphabeta}

\usepackage{xpatch}
\xpatchcmd{\citet}{;}{,}{}{}
\usepackage{tikz}
\usetikzlibrary{fit, positioning}
\usetikzlibrary{arrows.meta,positioning}

\def\cC{\mathcal C}

\def\cF{\mathcal F}
\def\cG{\mathcal G}
\def\cH{\mathcal H}
\def\cI{\mathcal I}

\def\cN{\mathcal N}
\def\cP{\mathcal P}
\def\cQ{\mathcal Q}
\def\cR{\mathcal R}

\def\cV{\mathcal V}
\def\cX{\mathcal X}
\def\cY{\mathcal Y}

\newcommand{\bb}{{\bf b}}
\newcommand{\bB}{{\bf B}}

\newcommand{\bd}{{\bf d}}

\newcommand{\bff}{{\bf f}}
\newcommand{\bg}{{\bf g}}
\newcommand{\bF}{{\bf F}}

\newcommand{\bj}{{\bf j}}
\newcommand{\bk}{{\bf k}}

\newcommand{\bL}{{\bf L}}
\newcommand{\bfm}{{\bf m}}

\newcommand{\bS}{{\bf S}}

\newcommand{\bw}{{\bf w}}

\newcommand{\bx}{{\bf x}}
\newcommand{\bX}{{\bf X}}
\newcommand{\by}{{\bf y}}
\newcommand{\bY}{{\bf Y}}
\newcommand{\bz}{{\bf z}}
\newcommand{\bZ}{{\bf Z}}

\newcommand{\bbE}{{\mathbb E}}
\newcommand{\bbI}{{\mathbb I}}
\newcommand{\bbR}{{\mathbb R}}
\newcommand{\bbZ}{{\mathbb Z}}
\newcommand{\bbN}{{\mathbb N}}

\newcommand{\bbP}{{\mathbb P}}

\newcommand{\E}{\mathbb{E}}

\newcommand{\balpha}{\bm{\alpha}}

\newcommand{\bgamma}{\bm{\gamma}}

\newcommand{\iid}{{\rm i.i.d.}}

\newcommand{\Holder}{H\"{o}lder}

\newcommand{\bc}{\begin{center}}
\newcommand{\ec}{\end{center}}
\newcommand{\be}{\begin{equation}}
\newcommand{\ee}{\end{equation}}
\newcommand{\ba}{\begin{array}}
\newcommand{\ea}{\end{array}}
\newcommand{\bean}{\setlength\arraycolsep{1pt}\begin{eqnarray*}}
\newcommand{\eean}{\end{eqnarray*}}
\newcommand{\bea}{\setlength\arraycolsep{1pt}\begin{eqnarray}}
\newcommand{\eea}{\end{eqnarray}}
\newcommand{\ben}{\begin{enumerate}}
\newcommand{\een}{\end{enumerate}}
\newcommand{\bed}{\begin{itemize}}
\newcommand{\eed}{\end{itemize}}

\DeclareMathOperator*{\argmin}{argmin}

\DeclareMathOperator*{\minimize}{minimize}

\def\d{{\rm d}}
\def\D{{\rm D}}

\def\defeq{ \stackrel{\rm def}{=} }

\newcommand{\vertiii}[1]{{\left\vert\kern-0.25ex\left\vert\kern-0.25ex\left\vert #1 
    \right\vert\kern-0.25ex\right\vert\kern-0.25ex\right\vert}}

\usepackage{lastpage}
\jmlrheading{27}{2026}{1-\pageref{LastPage}}{1/25; Revised
10/25}{1/26}{25-0121}{Hyeok Kyu Kwon, Dongha Kim, Ilsang Ohn and Minwoo Chae}

\ShortHeadings{Nonparametric Estimation of a Factorizable Density using Diffusion Models}{Kwon, Kim, Ohn and Chae}
\firstpageno{1}

\begin{document}

\title{Nonparametric Estimation of a Factorizable Density using Diffusion Models}

\author{\name Hyeok Kyu Kwon \email khkunok@postech.ac.kr \\
       \addr Department of Industrial and Management Engineering \\
       Pohang University of Science and Technology (POSTECH) \\
       Pohang, Gyeongbuk 37673, South Korea
       \AND
       \name Dongha Kim \email dongha0718@sungshin.ac.kr \\
       \addr Department of Statistics; Center for Data Science \\
       Sungshin Women’s University \\
       Seoul, 02844, South Korea
       \AND
       \name Ilsang Ohn \email ilsang.ohn@inha.ac.kr \\
       \addr Department of Statistics \\
       Inha University \\
       Incheon, 22212, South Korea
       \AND
       \name Minwoo Chae\thanks{Corresponding author.} \email mchae@postech.ac.kr \\
       \addr Department of Industrial and Management Engineering \\
       Pohang University of Science and Technology (POSTECH) \\
       Pohang, Gyeongbuk 37673, South Korea       
       }

\editor{Bryon Aragam}

\maketitle

\begin{abstract}
In recent years, diffusion models, and more generally score-based deep generative models, have achieved remarkable success in various applications, including image and audio generation.
In this paper, we view diffusion models as an implicit approach to nonparametric density estimation and study them within a statistical framework to analyze their surprising performance.
A key challenge in high-dimensional statistical inference is leveraging low-dimensional structures inherent in the data to mitigate the curse of dimensionality.
We assume that the underlying density exhibits a low-dimensional structure by factorizing into low-dimensional components, a property common in examples such as Bayesian networks and Markov random fields.
Under suitable assumptions, we demonstrate that an implicit density estimator constructed from diffusion models adapts to the factorization structure and achieves the minimax optimal rate with respect to the total variation distance.
In constructing the estimator, we design a sparse weight-sharing neural network architecture, where sparsity and weight-sharing are key features of practical architectures such as convolutional neural networks and recurrent neural networks.
\end{abstract}

\begin{keywords}
Bayesian network, diffusion model, factorizable density, Markov random field, minimax optimality, score-based generative model, weight-sharing neural network
\end{keywords}

\addtocontents{toc}{\protect\setcounter{tocdepth}{-1}}

\section{Introduction}
\label{sec:intro}

Suppose we have observations $\bX^{1}, \ldots, \bX^{n}$, which are independent and identically distributed $D$-dimensional random variables following an unknown distribution $P_{0}$ with density $p_{0}$. 
Inference of the unknown $P_0$ or its density $p_{0}$ is a fundamental task in unsupervised learning, and various methodologies and related theories have been developed over the past few decades (e.g., \citealp{hastie2009elements, tsybakov2008introduction, gine2016mathematical}).

Among these approaches, \emph{diffusion models} have recently demonstrated remarkable success across a wide range of applications. Even when compared to other modern deep generative models such as variational autoencoders (VAEs) \citep{kingma2013auto, rezende2014stochastic}, generative adversarial networks (GANs) \citep{goodfellow2014generative, arjovsky2017wasserstein, mroueh2018sobolev}, and normalizing flows \citep{dinh2015nice, rezende2015variational}, diffusion models have achieved state-of-the-art performance in domains including image \citep{rombach2022high, dhariwal2021diffusion}, video \citep{ho2022video}, and audio generation \citep{kong2021diffwave}.

Rather than directly estimating $p_0$, diffusion models, more generally referred to as score-based generative models, aim to estimate the map $\bx \mapsto \nabla \log p_0(\bx)$, commonly known as the score function. They can be regarded as an implicit approach to density estimation because, although no direct estimator of $p_0$ is constructed, the estimated score function enables sampling from the learned distribution, for example, via score-based Markov chain Monte Carlo algorithms such as Hamiltonian or Langevin Monte Carlo \citep{neal2011mcmc}.
The idea of score-function estimation was first proposed by \citet{hyvarinen2005estimation} and subsequently extended by \citet{vincent2011connection} and \citet{song2020sliced}.

A diffusion model consists of two diffusion processes, namely the forward and the backward/reverse processes. The forward process is typically a simple and well-known diffusion, such as the Ornstein--Uhlenbeck (OU) process, while the drift term in the reverse process involves the score functions corresponding to the marginal densities of the forward process. Rather than estimating a single score function $\nabla \log p_{0}(\cdot)$, diffusion models \citep{sohl2015deep, ho2020denoising, song2019generative, song2021scorebased} jointly estimate the family of score functions; see Section~\ref{sec:model} for further details.

For large $D$, inferring high-dimensional distributions becomes prohibitively difficult due to the well-known curse of dimensionality.
Why, then, do diffusion models perform so well in practice, especially in high-dimensional domains such as images and videos? A natural explanation is that although the ambient dimension $D$ is large, real-world data often lie on, or near, low-dimensional structures that effectively mitigate this curse. This perspective aligns with a long line of statistical research showing that structural assumptions, such as sparsity \citep{hastie2015statistical}, additivity \citep{hastie1990generalized}, and manifold structures \citep{genovese2012minimax}, can substantially improve statistical efficiency. In the presence of such low-dimensional structures, estimators can achieve significantly faster convergence rates even in high-dimensional settings. In practice, however, the underlying structure is typically unknown, and procedures that can adapt to such unknown structures are therefore preferred.
Deep neural networks (DNNs), for instance, have been shown to possess remarkable adaptivity in a variety of function estimation problems \citep{imaizumi2022adaptive, schmidt2020nonparametric, tang2024adaptivity, chae2023likelihood}.

Motivated by this viewpoint, we focus on a specific low-dimensional structure that captures a broad family of distributions: \emph{factorizable densities}.
We assume that the density function $p_{0}$ admits the factorization
\be \label{eq:factorization}
p_{0}(\bx) = \prod_{I \in \cI} g_{I}(\bx_{I}),
\ee
where $\cI \subseteq 2^{[D]}$ is a collection of index sets, $\bx_{I} = (x_{i})_{i \in I}$, and each $g_{I}$ is a $\lvert I \rvert$-variate function.
Here, $[D] = \{1, \ldots, D\}$ and $\lvert I \rvert$ denotes the cardinality of $I$. Such factorizable densities naturally arise in graphical model contexts \citep{liu2019nonparametric}, including Bayesian networks and Markov random fields.
In particular, the conditional-independence structures induced by undirected graphical models (Markov random fields) are well suited for modeling images, where spatially adjacent pixels tend to be strongly correlated, whereas distant pixels exhibit weak correlations \citep{ji2019probabilistic, vandermeulen2024dimension, vandermeulen2024breaking}; see Section~\ref{sec:factassm} for further details.

Although this structure is well known in the statistical community, nonparametric adaptive procedures for such models have rarely been investigated in the literature. For $\beta$-\Holder\ densities (as defined in Section \ref{sec:not}), classical nonparametric theory suggests that incorporating a factorization structure can improve the convergence rate with respect to the total variation distance, from $n^{-\beta/(D+2\beta)}$ \citep{tsybakov2008introduction, gine2016mathematical} to $n^{-\beta/(d+2\beta)}$, where $d = \max_{I\in\cI}|I|$ denotes the effective dimension corresponding to the largest component function. Here, we assume that $D$ is fixed and that all component functions have the same smoothness level $\beta$.

Once the factorization form in \eqref{eq:factorization} is known, it is not difficult to construct an estimator for $p_{0}$ that achieves the rate $n^{-\beta/(d+2\beta)}$ under suitable technical conditions.
It remains challenging, however, to construct an estimator that adapts to the unknown factorization structure.
To the best of our knowledge, theoretically adaptive estimators (though not necessarily achieving the optimal rate) have been considered only in a few recent works \citep{bos2024supervised, vandermeulen2024dimension, vandermeulen2024breaking}.

In this paper, we show that the implicit density estimator derived from diffusion models is adaptive to the underlying factorization structure and achieves the minimax-optimal convergence rate for estimating $\beta$-\Holder\ factorizable densities, up to logarithmic factors (Theorem~\ref{secthm:2}). The main theoretical challenge lies in approximating the joint score functions, which are the score functions associated with the marginal densities of the forward diffusion process, using neural networks, since these functions are defined through $D$-dimensional integrals. While prior theoretical works \citep{oko2023diffusion, tang2024adaptivity} are based on vanilla sparse neural networks, we employ a novel architecture: \emph{sparse weight-sharing neural networks}.
Here, parameters are sparse and shared within each layer to reduce model complexity.

Although sparse weight-sharing neural networks are relatively new in statistical theory, widely used architectures such as convolutional neural networks \citep[CNNs;][]{lecun1989handwritten, krizhevsky2012imagenet} and recurrent neural networks \citep{rumelhart1986learning, sutskever2014sequence} can be regarded as representative examples. From a theoretical perspective, to the best of our knowledge, only a few works have established that CNNs perform as well as vanilla feedforward networks in terms of approximation or estimation rates \citep{petersen2020equivalence, oono2019approximation, yang2024rates, fang2023optimal}.

While finalizing the revision of this article, we became aware of a very recent preprint by \citet{fan2025optimal}, which investigates essentially the same problem, namely, the same estimator, the same structural assumption, and the same convergence result. The key difference lies in the network architecture: they employ fully connected neural networks. Given this finding, our results do not demonstrate a distinct theoretical advantage of sparse weight-sharing neural networks. Nevertheless, based on our proof techniques, it is plausible to conjecture that sparse weight-sharing architectures may approximate certain general classes of functions more efficiently than fully connected ones. We believe that exploring this theoretically intriguing problem would provide valuable insights into the distinct benefits of sparse weight-sharing networks.

Before concluding the introduction, it is worthwhile to review recent advances in the statistical theory of diffusion models.
\citet{oko2023diffusion} proved that the implicit density estimator from the diffusion model is minimax optimal within the nonparametric smooth density estimation framework, using total variation and Wasserstein distances as evaluation metrics.
Subsequently, \citet{zhang2024minimax} and \citet{wibisono2024} relaxed certain technical assumptions in \citet{oko2023diffusion}.
Although these papers introduced several interesting mathematical techniques for handling diffusion models, they did not address the issue of the curse of dimensionality.
To tackle this issue, \citet{tang2024adaptivity} demonstrated that the estimator from the diffusion model is minimax optimal with respect to the Wasserstein metric under the smooth manifold assumption.
Under a similar regime, \citet{azangulov2024convergence} established tighter upper bounds for the convergence rate in terms of the ambient dimension $D$.
While the manifold structure is an interesting low-dimensional structure, an optimal estimator adaptive to this structure can also be constructed using methods other than diffusion models \citep{tang2023minimax, stephanovitch2024wasserstein}.
Diffusion models have also been analyzed under other low-dimensional structures beyond manifolds \citep{chen2023score, wang2024diffusion, yakovlev2025generalization}. 
Although minimax optimality is not guaranteed in these settings, they commonly assume that $P_{0}$ is singular with respect to the $D$-dimensional Lebesgue measure, with additional structural constraints.

In particular, while various interesting statistical theories have been developed for VAEs \citep{kwon2024minimax, chae2023likelihood} and GANs \citep{liang2021well, uppal2019nonparametric, chae2022convergence, stephanovitch2024wasserstein, tang2023minimax, puchkin2024rates}, the factorization structure \eqref{eq:factorization}, which is closely related to the conditional independence structure of directed and undirected graphs, has not been explored in the literature on deep generative models.
While the estimators proposed in \citet{bos2024supervised}, \citet{vandermeulen2024breaking}, and \cite{vandermeulen2024dimension} are adaptive to the factorization structure, diffusion models can adapt not only to this structure but also to other structures discussed above, making them significantly more practical alternatives.

The remainder of this paper is organized as follows. In Section~\ref{sec:model}, we introduce diffusion models and define our implicit density estimator. Section~\ref{sec:wsnn} presents the class of sparse weight-sharing networks, while Section~\ref{sec:factassm} details the main assumption: the factorization assumption. Our main theoretical results are provided in Section~\ref{sec:main}.
In Section~\ref{sec:naive}, we discuss the benefits of diffusion models compared to the vanilla score matching estimator. We present small experimental results in Section~\ref{sec:exp} and conclude with discussions in Section~\ref{sec:disc}.
All proofs are provided in the Appendix.

\subsection{Notations and Definitions}
\label{sec:not}

Vectors are denoted using boldface notation.
For a multi-index $\bgamma = (\gamma_1,\ldots,\gamma_D)^{\top} 
\in (\bbZ_{\geq 0})^{D}$, denote $\D^{\bgamma}$ the mixed partial derivative operator $\partial^{\gamma.} / \partial x_1^{\gamma_1} \cdots \partial x_{D}^{\gamma_D}$, where $\gamma. = \sum_{i=1}^{D} \gamma_i$.
For any $\beta, K > 0$, let $\cH^{\beta, K}_{D}(A)$ be the class of $\beta$-\Holder\ functions, consisting of every real-valued function $g$ on $A \subseteq \bbR^{D}$ such that 
\bean
\sum_{\gamma. \leq \lfloor \beta \rfloor} \sup_{\bx \in A} \vert (\D^{\bgamma} g) (\bx) \vert + \sum_{\gamma. = \lfloor \beta \rfloor} \sup_{\substack{\bx,\by \in A \\ \bx \neq \by}} \frac{\vert (\D^{\bgamma} g)(\bx) - (\D^{\bgamma} g)(\by) \vert}{\| \bx-\by\|_{\infty}^{\beta - \lfloor \beta \rfloor}} \leq K,
\eean
where $\lfloor \beta \rfloor$ denotes the largest integer strictly smaller than $\beta$.
We often denote $\cH^{\beta,K}_{D}(A)$ as $\cH^{\beta,K}(A)$ when the dimension is obvious from the context.
For a vector $\bx$, we denote the $\ell^{p}$-norm, $1 \leq p \leq \infty$, and the number of nonzero elements as $\|\bx\|_{p}$ and $\|\bx\|_{0}$, respectively.
Let $\phi_{\sigma, D}$ be the density function of the multivariate normal distribution $\cN(\bm{0}_D, \sigma^2 \bbI_D )$, where $\bm{0}_D$ and $\bbI_D$ are $D$-dimensional zero vector and identity matrix, respectively.
 For simplicity, we often denote $\phi_{\sigma, D}$ as $\phi_\sigma$ when the dimension is obvious from the context.
For real numbers $a$ and $b$, let $a \vee b$ and $a \wedge b $ be the maximum and minimum of $a$ and $b$, respectively.
The notation $a \lesssim b$ means that $a \leq C b$, where $C > 0$ is a constant not relevant to the main argument.
Similarly, $a \asymp b$ implies that $a \lesssim b$ and $b \lesssim a$. 
Finally, the notation $C=C(A_1, \ldots, A_n)$ means that the constant $C$ depends only on $A_1, \ldots, A_n$.

\section{Diffusion Models}
\label{sec:model}

In this section, we provide a brief introduction to the diffusion model proposed in \citet{song2021scorebased} and define the estimator studied in our main results.
Let $( \bX_{t} )_{t \geq 0}$ be the process satisfying the stochastic differential equation (SDE)
\be \label{eq:OU}
\d {\bX}_t = -  \alpha_{t} {\bX}_t \d t + \sqrt{2\alpha_{t}} \d \bB_t,\quad \bX_0 \sim P_0,
\ee
where $( \bB_{t} )_{t \geq 0} $ is a standard $D$-dimensional Brownian motion and $t \mapsto \alpha_{t}: [0, \infty) \to [0, \infty)$ is a (known) Borel measurable function.
The stochastic process $( \bX_{t} ) $ is often referred to as a time-inhomogeneous OU process, and has been studied in \citet{song2021scorebased} and \citet{chen2023sampling}.
For the OU process \eqref{eq:OU}, the transition kernel is explicitly given as Gaussian.
Specifically, the conditional distribution of $\bX_t$ given $\bX_0 = \bx_0$ is $\cN(\mu_t \bx_0, \sigma_t^2 \bbI_D)$, where $\mu_t = \exp(-\int_{0}^{t} \alpha_s \d s)$ and $\sigma_t^2 = 1-\mu_t^2$.
We denote this conditional distribution and the corresponding density as $P_{t}(\cdot \mid  \bx_0)$ and $p_{t}(\cdot \mid \bx_0)$, respectively.
We also denote $P_t$ and $p_t$ as the marginal distribution and density of $\bX_t$, respectively.
Hence, we have
\bean
  p_t(\bx) = \int \phi_{\sigma_t}(\bx - \mu_t \by) \d P_0(\by)
  = \int \phi_{\sigma_t}(\bx - \mu_t \by) p_0(\by) \d \by.
\eean
Note that $p_{t}$ converges very quickly to the standard Gaussian density as $t \to \infty $; see \citet{bakry2014analysis} for a rigorous statement.

Note that the map $\bx \mapsto \bff_0(\bx, t)$ is the score function corresponding to the marginal density $p_t$.
As a convention, we also call $\bff_0$ a score function.

For a given non-random $\overline{T} > 0$, let $(\bY_{t})_{t \in [0, \overline{T}) }$ be the reverse-time process defined as $\bY_{t} = \bX_{\overline{T}-t}$.
Then, it is well-known \citep{anderson1982reverse} that $(\bY_{t})_{t \in [0, \overline{T}) }$ is also a diffusion process under mild assumptions. 
More specifically, once
\bean
\int_{s}^{\overline{T}} \bbE \left[  p_t(\bX_t) + 2 \alpha_{t} \left\| \nabla \log p_t(\bX_t) \right\|_2^2 \right] \d t < \infty
\quad \forall s >0
\eean
and the map $t \mapsto \alpha_{t}$ is bounded from above,
we have 
\be \begin{split} \label{eqsec2:backward}
\d \bY_t &= \left[\alpha_{\overline{T}-t} \bY_t + 2 \alpha_{\overline{T}-t} \nabla \log p_{\overline{T}-t}(\bY_t) \right] \d t + \sqrt{2\alpha_{\overline{T}-t}} \d \bB_t
\\
&= \left[\alpha_{\overline{T}-t} \bY_t + 2 \alpha_{\overline{T}-t} \bff_0(\bY_t, \overline T - t) \right] \d t + \sqrt{2\alpha_{\overline{T}-t}} \d \bB_t,\quad \bY_0 \sim P_{\overline{T}},
\end{split} \ee
see Theorem 2.1 of \citet{haussmann1986time}.
Note that the Brownian motions in \eqref{eq:OU} and \eqref{eqsec2:backward} are not identical. However, we use the same notation $\bB_t$ to denote a standard Brownian motion as a convention throughout the paper.

Once we have an estimator $\widehat \bff$ for the score function $\bff_0$, one can simulate the reverse process starting from a standard Gaussian to obtain samples from the estimated distribution.
The score function can be estimated via the score matching \citep{hyvarinen2005estimation} or its scalable variations \citep{vincent2011connection, song2020sliced, yu2022generalized}.

Let $\cF$ be a class of functions $(\bx, t) \mapsto \bff(\bx, t)$ used to model the score function $\bff_0$.
A detailed description of the class $\cF$ in our theory is provided in Section~\ref{sec:wsnn}.
At the population level, the best approximator to $\bff_0$ in $\cF$ can be defined as the solution to the following optimization problem
\be\begin{split} \label{eq:sm}
 &\minimize_{\bff \in \cF} \int_{0}^{\overline{T}}  \lambda_t \bbE \left[ \left\| \bff(\bX_{t}, t) - \bff_0(\bX_t, t) \right\|_2^2 \right] \d t
 \\
 \Longleftrightarrow \quad & \minimize_{\bff \in \cF} \int_{0}^{\overline{T}}  \lambda_t \bbE \left[ \left\| \bff(\bX_{t}, t) - \nabla \log p_t(\bX_{t}) \right\|_2^2 \right] \d t,
\end{split}\ee
where $\lambda_t \geq 0$ is a weight.
Based on the well-known fact \citep{vincent2011connection} that
\be
\bbE \left[ \left\| \bff(\bX_{t}, t) - \nabla \log p_t(\bX_{t}) \right\|_2^2 \right]
= \bbE \left[ \left\| \bff(\bX_{t}, t) - \nabla \log p_t(\bX_{t} \mid \bX_{0}) \right\|_2^2 \right] + C_{t}, \label{eq:smvincent}
\ee
where $C_{t}$ is a constant depending only on $t$ and $\bff_0$ and 
\bean
 \nabla \log p_t(\bx_t \mid \bx_0) = \frac{\partial \log p_t(\bx_t \mid \bx_0)}{\partial \bx_t}
 = -\frac{\bx_{t} - \mu_t \bx_{0}}{\sigma_t^2},
\eean
the minimization problem \eqref{eq:sm} can be equivalently written as
\bean
 && \minimize_{\bff \in \cF} \int_{0}^{\overline{T}}  \lambda_t \bbE \left[ \left\| \bff(\bX_{t}, t) - \nabla \log p_t(\bX_{t} \mid \bX_{0}) \right\|_2^2 \right] \d t
 \\
 \Longleftrightarrow \quad && \minimize_{\bff \in \cF} \int_{0}^{\overline{T}}  \lambda_t \bbE \Big[ \bbE \Big( \Big\| \bff(\bX_{t}, t) + \frac{\bX_t - \mu_t \bX_0}{\sigma_t^2} \Big\|_2^2 ~ \big| ~ \bX_0 \Big) \Big] \d t.
\eean
In practice, $\lambda_t$ is set to zero for sufficiently small $t$ to avoid potential singularity issues.
This leads to the following ERM (empirical risk minimization) estimator
\be
\widehat \bff \in \argmin_{\bff \in \cF}
\frac{1}{n} \sum_{i=1}^n \ell_{\bff}(\bX^{i}),
\label{eqthm2:erm}
\ee
where
\be
\ell_{\bff}(\bx) = \int_{\underline{T}}^{\overline{T}}   \lambda_t \bbE \Big[ \Big\| \bff(\bX_{t}, t) + \frac{\bX_t - \mu_t \bX_0}{\sigma_t^2} \Big\|_2^2 ~ \big| ~ \bX_0 = \bx \Big]\d t
\label{eq:loss}
\ee
is the loss function and $\underline{T} > 0$ is a sufficiently small number.

Let $(\widehat \bY_{t})_{t \in [0, \overline{T}-\underline{T}]}$ be the solution to the SDE
\be
\d \widehat \bY_t = \left[\alpha_{\overline{T}-t} \widehat \bY_t + 2 \alpha_{\overline{T}-t} \widehat \bff \left(\widehat \bY_t,\overline{T}-t \right) \right] \d t + \sqrt{2\alpha_{\overline{T}-t}} \d \bB_t,\quad \widehat \bY_0 \sim \cN(\mathbf{0}_{D}, \bbI_{D})
\label{eq:diffmodel}
\ee
and $\widehat\bX_t = \widehat\bY_{\overline{T} - t}$.
Let $\widehat P_{t}$ be the marginal distribution of $\widehat \bX_{t}$ and $\widehat p_t$ be the corresponding Lebesgue density.
Also, let $\widehat P = \widehat P_{\underline{T}}$ and $\widehat p = \widehat p_{\underline{T}}$.
The existence of $\widehat p_t$ is guaranteed under mild assumptions; see \citet{bogachev2011uniqueness} for details.
Although $\widehat p_t$ is only defined implicitly through the SDE \eqref{eq:diffmodel}, it is a function of data, hence an estimator for the unknown density $p_t$.
Since we expect that $p_t \approx p_0$ for sufficiently small $t$, $\widehat p$ can serve as an estimator for $p_0$.

\begin{remark}
Note that the loss function \eqref{eq:loss} involves integrals (with respect to $\bX_t$ and $t$) that are not directly tractable.
In practice, a slightly different loss function with augmented variables is considered for computational tractability \citep{sohl2015deep, song2019generative}. Specifically, with a slight abuse of notation, define the loss function
\bean
 \ell_{\bff}(\bx_0, \bx_t, t) = \Big\| \bff(\bx_{t}, t) + \frac{\bx_t - \mu_t \bx_0}{\sigma_t^2} \Big\|_2^2.
\eean
By regarding $T$ as a random variable independent of the stochastic process $(\bX_t)$ and supported on $[\underline{T}, \overline{T}]$ with the density proportional to $\lambda_t$, we have
\bean
C_{T}
 \E  \left[ \ell_{\bff}(\bX_0, \bX_T, T) \right]
 = \E \left[ \ell_{\bff} (\bX_0) \right],
\eean
where $ C_{T} =  \int_{\underline{T}}^{\overline{T}} \lambda_t \d t$ is a normalizing constant.
Therefore, although the loss function \eqref{eq:loss} itself is not directly tractable, one can approximate the solution to the minimization problem \eqref{eqthm2:erm} using stochastic gradient methods.
\end{remark}

\begin{remark}
The target estimator in our theoretical study in Section~\ref{sec:main} is $\widehat p$ as defined above, which is the density of $\widehat \bX_{\underline{T}} = \widehat \bY_{\overline{T}-\underline{T}}$. In practice, samples from the estimated distribution are generated by numerically solving the SDE (\ref{eq:diffmodel}) using methods such as Euler-Maruyama discretization \citep{kloeden2011numerical, song2021scorebased}, starting from an initial sample drawn from the standard normal distribution.
Hence, more delicate statistical theory should incorporate these discretization errors. As an independent line of work, there are various articles studying discretization errors in diffusion models \citep{oko2023diffusion, chen2023sampling, nakano2024convergence, debortoli2022convergence, benton2024nearly, chen2023probability, li2024towards, liang2025low}.
Combining our statistical theory given in Section~\ref{sec:main} with these works, the main results remain valid if the SDE is discretized with a sufficiently fine time partition. For additional details, we refer to Section 5.3 of \citet{oko2023diffusion}.
\end{remark}

\section{Sparse Weight-Sharing Neural Networks}
\label{sec:wsnn}

In this section, we define neural networks that are used as a function class $\cF$ to model the score function $\bff_0$ described in Section~\ref{sec:model}. Instead of vanilla feedforward neural networks, we consider sparse weight-sharing architectures, which are widely used in practical applications. By incorporating such sparsity and weight-sharing structures into the network architecture that models the score function, one can substantially reduce the model complexity (often expressed in terms of metric entropy; see Lemma~\ref{sec:covering}), ultimately leading to a reduction in estimation error.

For a positive integer $m$, let $\rho_{m} : \bbR^{m} \to \bbR^{m}$ be the coordinatewise ReLU activation function defined as $\rho_{m}(\bz) = \left(\max\{z_1,0\}, \ldots, \max\{z_m, 0\} \right)^{{\top}}$ for $\bz = (z_1,\ldots,z_m)^{\top}$.
For simplicity, we often denote $\rho_{m}$ as $\rho$.
For $L \in \bbN_{\geq 2}$, $\bd = (d_1,\ldots,d_{L+1}) \in  \bbN^{L+1}, s \in \bbN, M > 0, \bfm = (m_1,\ldots,m_{L-1}) \in \bbN^{L-1}$ and 
\bean
\cP_{\bfm} = ( ( \cQ_{i}, \cR_{i} ) )_{i \in [L-1]},
\eean
where $\cQ_i = ( Q_i^{(j)})_{j \in [m_i]} $ is a collection of $d_{i} \times d_{i}$ permutation matrices  and $\cR_i = ( R_i^{(j)} )_{j \in [m_i]} $ is a collection of $d_{i+1} \times d_{i+1}$ permutation matrices, let $\cF_{\rm WSNN} = \cF_{\rm WSNN} (L,\bd,s,M,\cP_{\bfm})$ be the class of functions $\bff: \bbR^{d_1} \to \bbR^{d_{L+1}}$ of the form
\be \begin{split}
& \bff(\bz) = W_{L} \left( \bff_{L-1} \circ \cdots \circ \bff_{1} \right) (\bz) + \bb_{L},
\\ 
& \bff_i(\cdot) = \rho \left( \sum_{j=1}^{m_i} R_{i}^{(j)} \left(W_i  Q_{i}^{(j)} \cdot + \bb_{i} \right) \right) \label{eqwsnn}
\end{split} \ee
with $W_i \in \bbR^{d_{i+1} \times d_{i}}$ and $\bb_{i} \in \bbR^{d_{i+1}}$ satisfying
\bean
\max_{1 \leq i \leq L}\left\{ \max \left(\|W_i\|_{\infty} , \| \bb_i\|_{\infty} \right) \right\} \leq M, \quad \sum_{i=1}^{L} \left( \|W_i\|_0 + \| \bb_i\|_0 \right) \leq s.
\eean
Here, $\|W_i\|_\infty$ and $\| W_i\|_0$ denote the entrywise maximum norm and the number of nonzero elements of the matrix $W_i$, respectively.

The network \eqref{eqwsnn} includes vanilla feedforward neural networks as special cases. For example, if $\bfm = (1, \ldots, 1)$ and all matrices in $\cP_{\bfm}$ are identity matrices, the class $\cF_{\rm WSNN}$ reduces to the usual class of sparse networks considered in the literature. In this case, we often denote $\cF_{\rm WSNN}(L, \bd, s, M, \cP_{\bfm})$ as $\cF_{\rm NN}(L, \bd, s, M)$.

The network \eqref{eqwsnn} is designed to incorporate sparsity and weight-sharing into the architecture.
As a simple example, consider a weight matrix $W \in \bbR^{3d \times 4d}$ with the following block structure, where each block has the same size of $d \times d$:
\bean
W = 
\begin{bmatrix}
W_0 \ & 0_{d,d} \ & 0_{d,d} \ & W_0 \
\\
0_{d,d} \ & 0_{d,d} \ & W_0 \ & 0_{d,d} \
\\
0_{d,d} \ & W_0 \ & 0_{d,d} \ & W_0 \
\end{bmatrix}.
\eean
Here, $0_{d,d}$ denotes the $d \times d$ zero-matrix.
Two important features of $W$ are that it is sparse, in the sense that many elements of $W$ are exactly zero, and that the sub-matrix $W_0$ is shared across different rows and columns.
The formula \eqref{eqwsnn} is one way to effectively represent neural networks with sparse weight-sharing matrices like $W$. For example, it is straightforward to construct $3d \times 3d$ permutation matrices $R^{(j)}$ and $4d \times 4d$ permutation matrices $Q^{(j)}$, for $j \leq 5$, that satisfy
\bean
W = \sum_{j=1}^{5} 
R^{(j)}
\begin{bmatrix}
W_0 \ & 0_{d,3d}
\\
0_{2d,d} \ & 0_{2d,3d} 
\end{bmatrix}
Q^{(j)}.
\eean
The number of nonzero elements of $W$ is $5d^2$, but we can express it with a smaller sparsity of $d^2$ by weight-sharing architecture.

Sparse weight-sharing matrices are used in many practically important architectures, such as convolutional neural networks \citep{lecun1989handwritten, krizhevsky2012imagenet} and recurrent neural networks \citep{rumelhart1986learning, sutskever2014sequence}; see also \citet{zhang2021beyond} and \citet{jagtap2022deep} for additional examples. Note that CNN architectures are frequently adopted in diffusion models \citep{sohl2015deep, ho2020denoising, ronneberger2015unet}.

As an illustrative example, consider a convolution operation with an input image $\bx$ of size $d_1 = s_1^2$, a filter of size $s = s_0^2$, and an output image $\by$ of size $d_2 = (s_1 - s_0 + 1)^2$, as illustrated in Figure \ref{fig:cnn}.
Let $\bw = (w_1, \ldots, w_s)^\top$ be the vectorized version of the filter, concatenating all columns of the filter sequentially, and let $\widetilde W \in \bbR^{d_2 \times d_1}$ be the weight matrix corresponding to the convolution operation, such that $\by = \widetilde W \bx$.
One can observe that $\widetilde W$ is a sparse weight-sharing matrix and can be represented in the form of \eqref{eqwsnn}.
To see this, note that each row of $\widetilde W$ can be obtained by permuting the vector $(\bw^{\top},{0}_{1, d_1-s}) \in \bbR^{d_1}$.
Thus, for $j \in [d_2]$, there exists a $d_1 \times d_1$ permutation matrix $Q^{(j)}$ such that $y_j =  (\bw^{\top}, {0}_{1, d_1-s}) Q^{(j)} \bx$.
Also, for $j \in [d_2]$, let $R^{(j)}$ be the $d_2 \times d_2$ permutation matrix that swaps the first and the $j$th row when it is left-multiplied by a matrix.
Then, we can express the matrix $\widetilde W$ in the form of (\ref{eqwsnn}) by
\bean
\widetilde W = \sum_{j=1}^{ d_2} R^{(j)} W Q^{(j)},
\quad 
W = \begin{pmatrix}
\bw^{\top} \ & {0}_{1, d_1-s}
\\
{0}_{d_2 - 1, s} \ & {0}_{d_2-1, d_1-s}
\end{pmatrix} \in \bbR^{d_2 \times d_1}.
\eean
One may additionally incorporate padding and stride operations \citep{paszke2019pytorch} into the convolution operation described in Figure \ref{fig:cnn}. The corresponding weight matrix, with these additional operations, can also be expressed in the form of (\ref{eqwsnn}) by carefully selecting the permutation matrices.

\begin{figure*}[!t]
    \centering
        \includegraphics[width=0.85\linewidth]{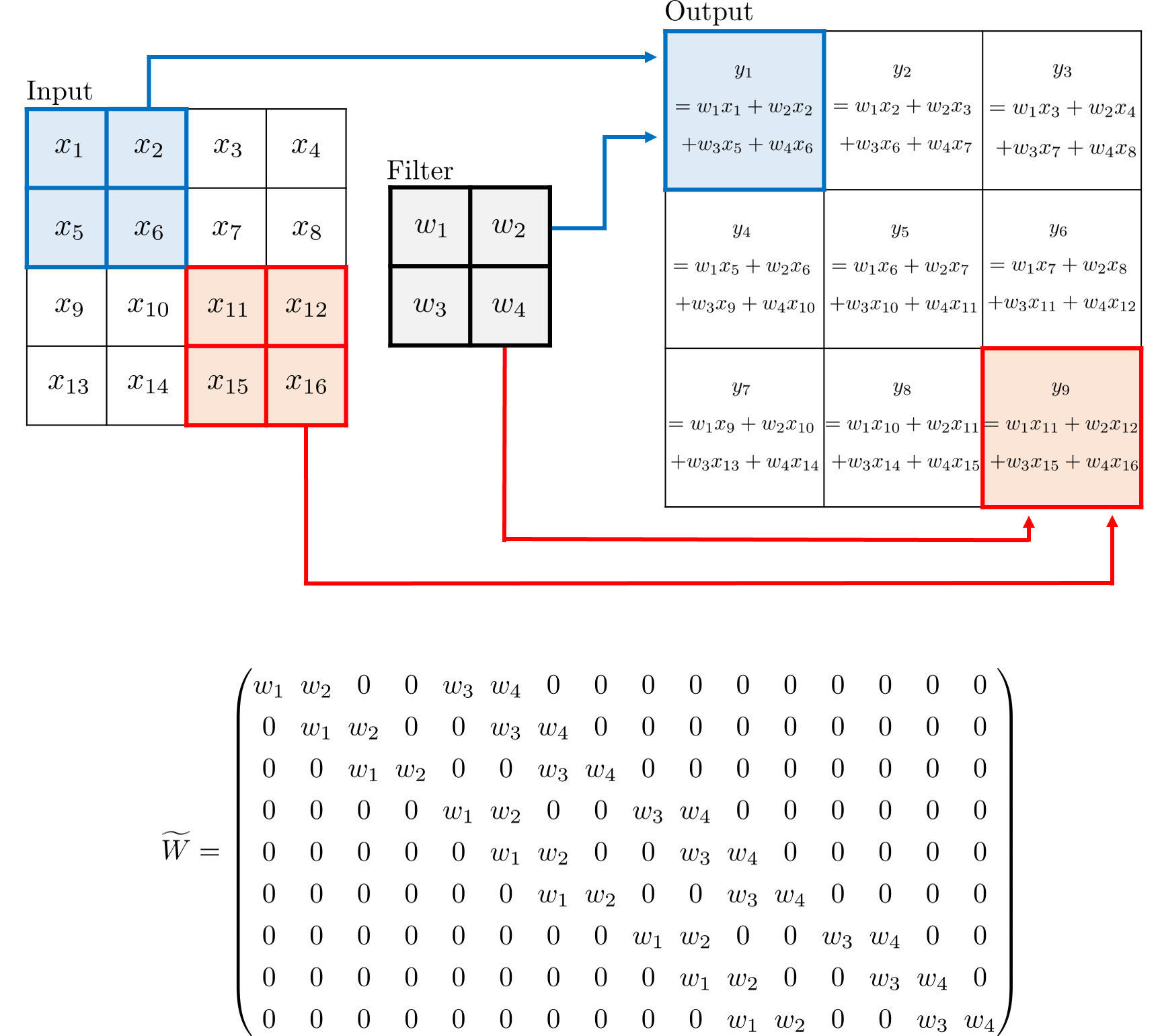}
    \caption{
    Example of a 2-dimensional convolution operation with an input $\bx \in \bbR^{16}$, a filter vector $\bw \in \bbR^{4}$ and output $\by \in \bbR^9$.
    The operation can be represented as a matrix multiplication \citep{goodfellow2016deep}, given by $\by = \widetilde W \bx$.
    } 
    \label{fig:cnn}
\end{figure*}

\section{Factorizable Densities}
\label{sec:factassm}

\subsection{Factorization Assumption}

In this section, we introduce the low-dimensional assumption considered in our main results and provide some well-known examples. Formally, we consider the following factorization assumption.

\bed
\item[] (\bF) There exists a set $\cI \subseteq 2^{[D]}$ and functions $g_{I} : \bbR^{\vert I \vert} \to \bbR$ for each $I \in \cI$ such that
\bean
p_0(\bx) = \prod_{I \in \cI} g_{I}(\bx_{I}), \quad \forall \bx \in \bbR^D.
\eean
\eed

For a density $p_0$ satisfying (\bF), let $d = \max_{I \in \cI} \vert I \vert$ denote the largest number of variables that any $g_I$ depends on.
We refer to $d$ as the effective dimension corresponding to the factorizable density $p_0$.
As a simple example, if $p_0$ is the density of a random vector $\bX = (X_1, \ldots, X_D)$ and each component of $\bX$ is mutually independent, then $d = 1$.
In the following subsections, we present examples based on conditional independence structures, which are often represented using graphical models.

\subsection{Example: Bayesian Networks}

A Bayesian network is a random vector $\bX = (X_1,\ldots,X_{D})$ whose conditional independence structure can be represented by a directed acyclic graph (DAG) with the vertex set $\{1,\ldots,D\}$.
For a Bayesian network, each variable $X_{i} $ is conditionally independent of all other variables given its parent variables $\bX_{\text{pa}(i)} =  ( X_{j})_{ j \in \text{pa}(i) }$, where $\text{pa}(i)$ denotes the set of parent indices of vertex $i$.
Accordingly, the density $p_0(\cdot)$ of a Bayesian network $\bX$ factorizes as
\bean
p_0 ( \bx) = \prod_{i=1}^{D}  p_{i } ( x_i \mid \bx_{ \text{pa} (i) } ) 
= \prod_{i=1}^{D}  g_{i } ( x_i, \bx_{ \text{pa} (i) } ),
\eean
where $ p_{i } ( \cdot \mid \bx_{ \text{pa} (i) } )$ is the conditional density of $X_i$ given $\bX_{\text{pa}(i)} = \bx_{\text{pa}(i)} $ and $g_{i } ( x_i, \bx_{ \text{pa} (i) } ) = p_{i } ( x_i \mid \bx_{ \text{pa} (i) } )$. 
Hence, $p_0$ satisfies the assumption (\bF) with $\cI = \{ \{i\} \cup \text{pa}(i) : i \in [D]  \}$ and $ d = 1 + \max_{i \in [D]} \vert \text{pa}(i) \vert $; see Figure \ref{Fig:bn} for an illustrative example.

\begin{figure*}[!t]
    \centering
    \subfigure[ A Bayesian network with $d = 4$ ]{
        \includegraphics[width=0.45\linewidth]{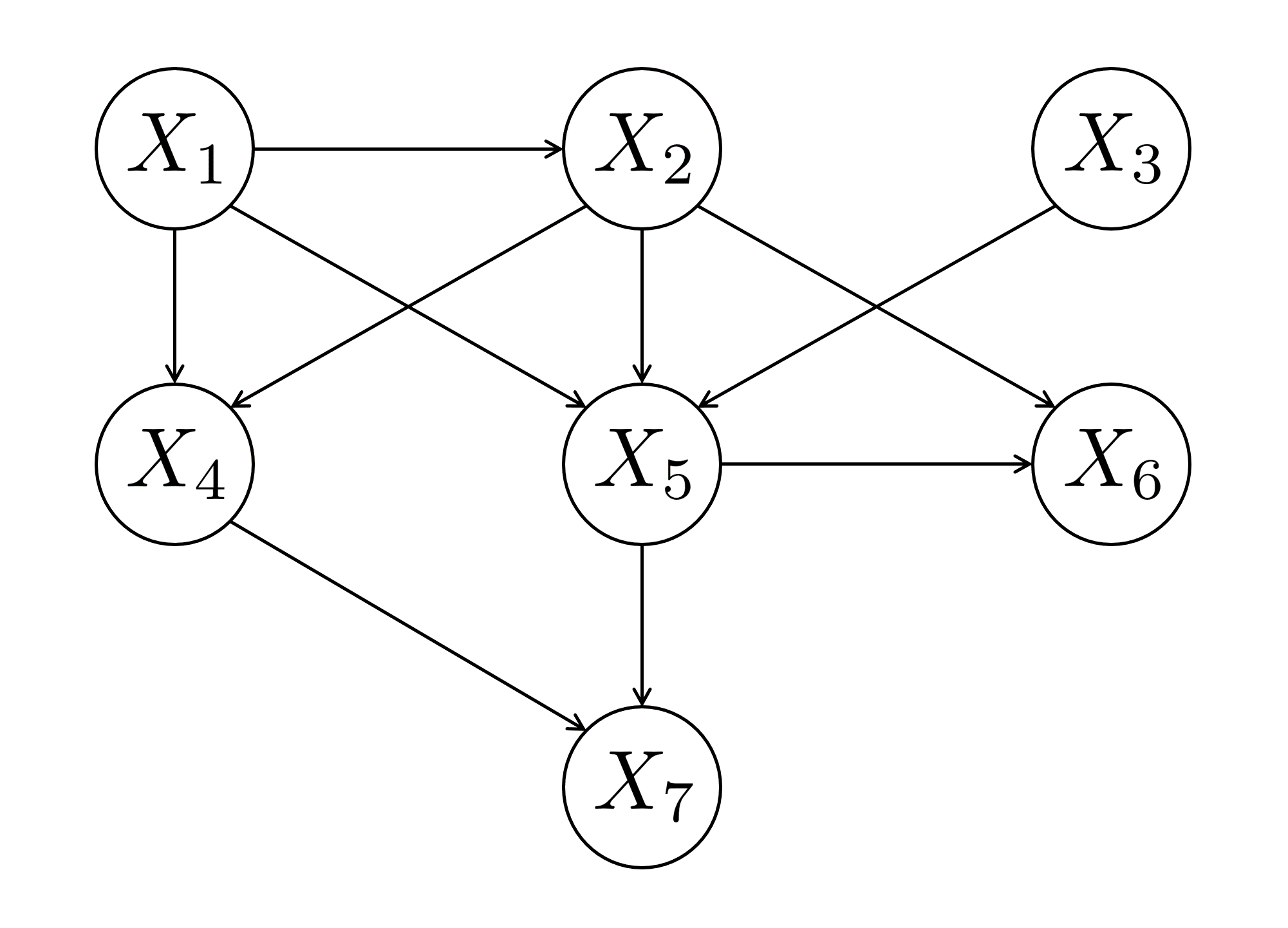}
        \label{Fig:bn}
    }
    \subfigure[ A Markov random field with $d = 3$ ]{
        \includegraphics[width=0.45\linewidth]{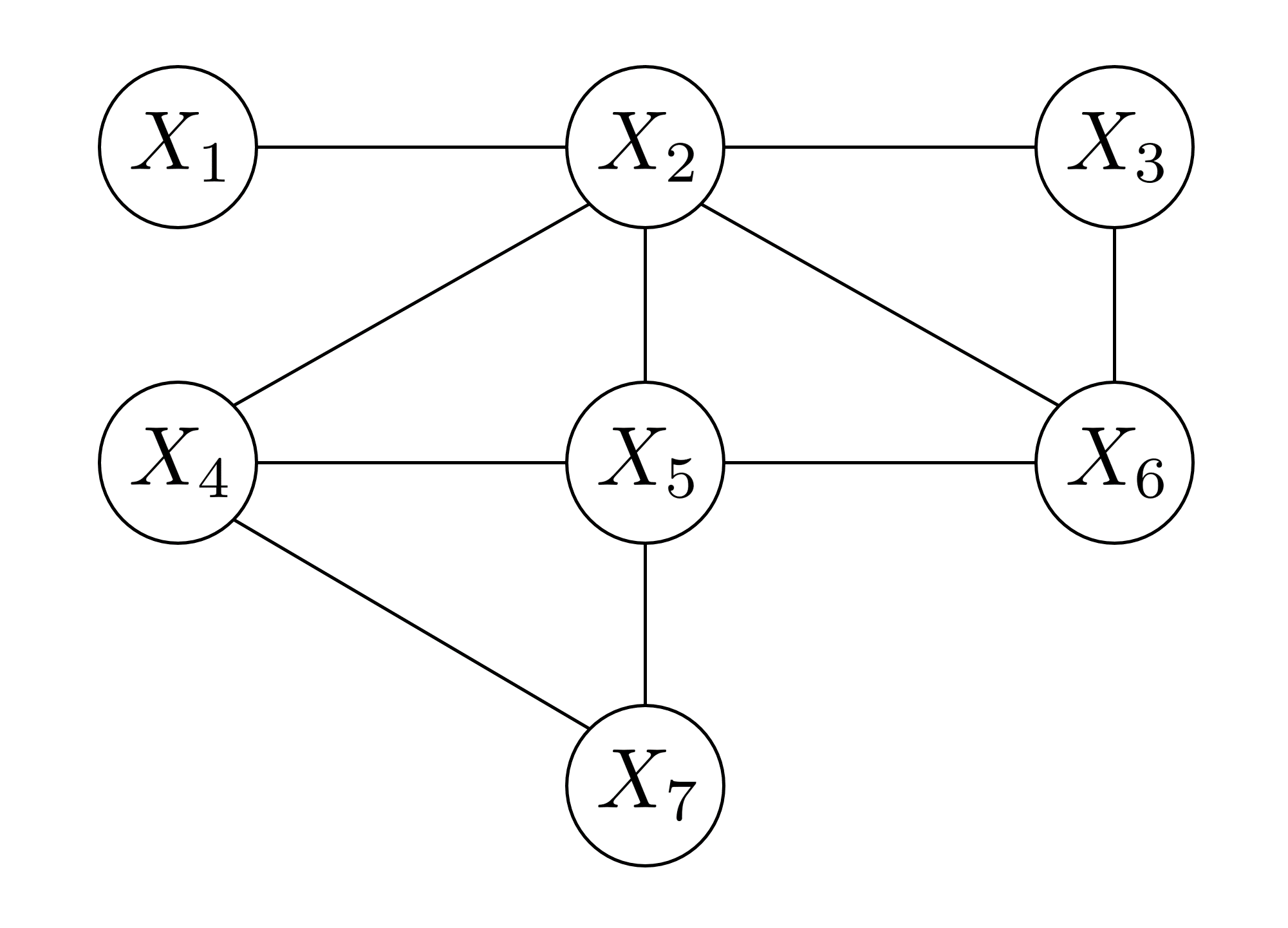}
        \label{Fig:mrf}
    }
    
    \caption{Examples of directed and undirected graphical model structures for a 7-dimensional random vector. In both cases, the effective dimension $d$ is strictly less than $D=7$.
    } \label{fig:examples}
\end{figure*}

\subsection{Example: Markov Random Fields}

A Markov random field over an undirected graph $G$ with vertex set $\{1, \ldots, D\}$ is a random vector $\bX = (X_1, \ldots, X_D)$ with the property that each variable $X_i$ is conditionally independent of all other variables given its neighbors, often referred to as the local Markov property \citep{lauritzen1996graphical}.
If the density $p_0(\cdot)$ of $\bX$ is strictly positive, the local Markov property holds if and only if
\be
p_0(\bx) = \prod_{C \in \cC} g_{C} (\bx_{C})
\label{mrf}
\ee
for some functions $g_{C}$, where $\cC$ denotes the set of all (maximal) cliques in the graph, as stated by the celebrated Hammersley-Clifford theorem \citep{hammersley1971markov, lauritzen1996graphical}.
Here, a clique is a fully connected subset of the vertex set in a graph, and the factors $g_{C}$ are referred to as potential functions in Markov random fields.
Therefore, the assumption (\bF) holds with $\cI = \cC$ and $d = \max_{C \in \cC} \vert C \vert$, where $d$ represents the maximum number of vertices in the (maximal) cliques; see Figure \ref{Fig:mrf} for an illustrative example.

Note that images consist of pixels with strong spatial correlations. It is, therefore, natural to assume that each pixel is conditionally independent of all other pixels given the pixels in its neighborhood. This makes the local Markov property particularly suitable for image data. For example, one might consider a graphical model structure, such as in Figure \ref{Fig:d2}, which has a very small $d$ (e.g., $d = 2$ in this example).

\begin{figure*}[!t]
    \centering
    \subfigure[ MNIST image of the digit `0' ]{
        \includegraphics[width=0.52 \linewidth]{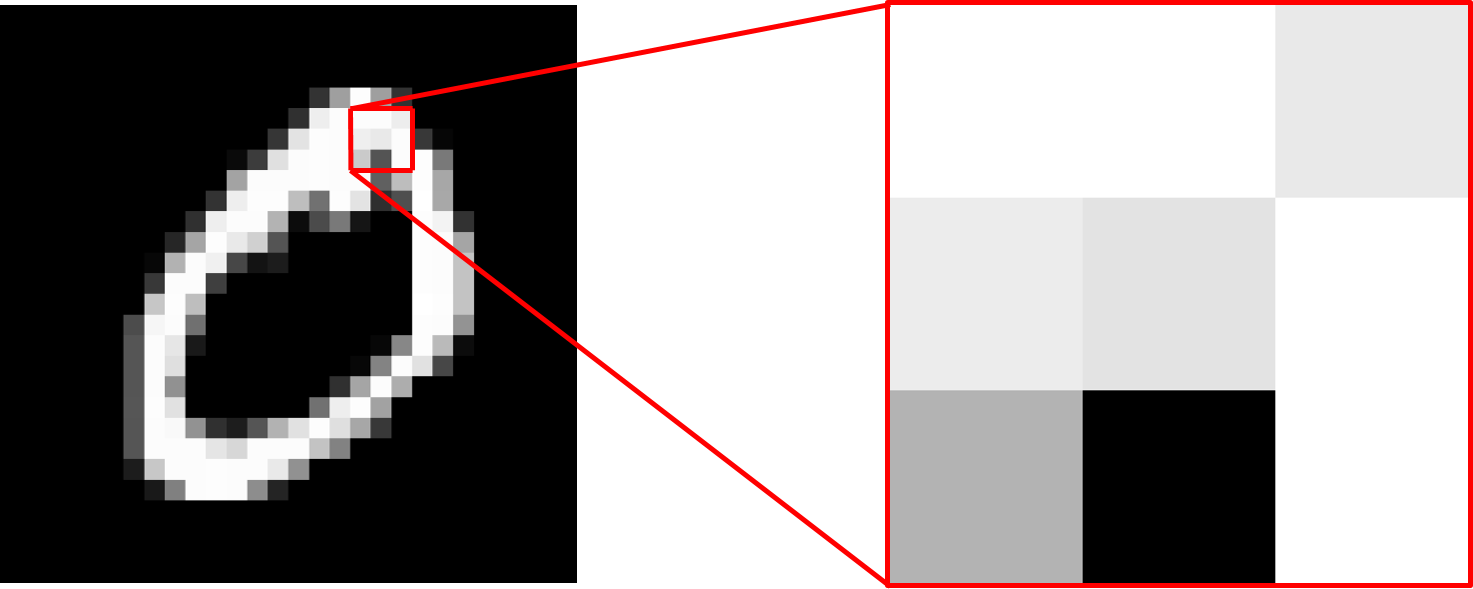}
    }
    \subfigure[ $d = 2$ ]{
        \includegraphics[width=0.2\linewidth]{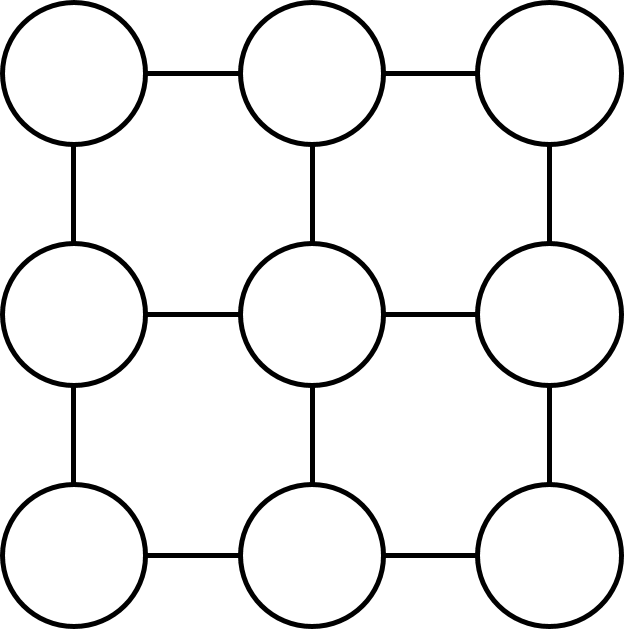}
        \label{Fig:d2}
    }
    \subfigure[ $d = 4$  ]{
        \includegraphics[width=0.2\linewidth]{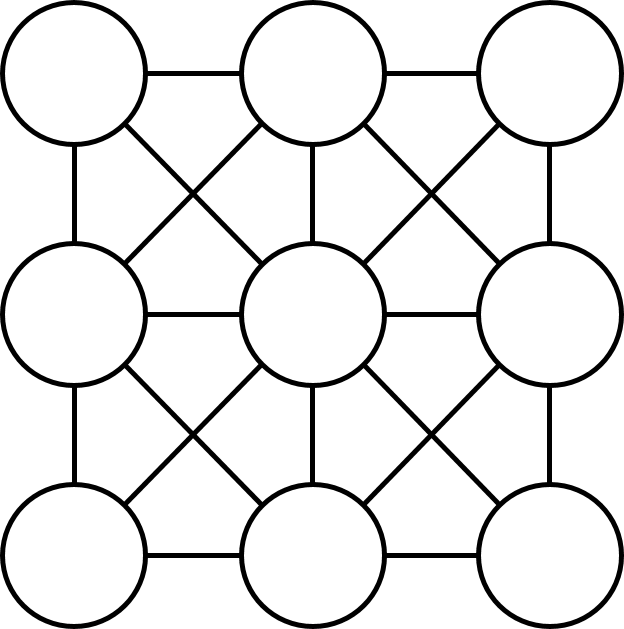}
        \label{Fig:d4}
    }
    
    \caption{An image of the digit `0' from the MNIST data set \citep{lecun1998gradient}, along with two possible undirected graph structures for MNIST. For each pixel, a larger neighborhood may be considered depending on the degree of spatial correlations.}
     \label{fig:mnist}
\end{figure*}

\section{Main Results}
\label{sec:main}


In this section, we present the main results of the paper. Section~\ref{sec:assumption} introduces the assumptions on $p_0$ required for the main results. Section~\ref{sec:conv} presents our main theorem (Theorem~\ref{secthm:2}), which establishes the convergence rate of the diffusion estimator $\widehat{p}$ introduced in Section~\ref{sec:model}. Finally, Section~\ref{sec:approx} describes the key technical components used to prove Theorem~\ref{secthm:2}, namely the approximation of the score function $\bff_0$ using sparse weight-sharing neural networks.

\subsection{Assumptions}
\label{sec:assumption}

For given data $\bX^1, \ldots, \bX^n$, let $\widehat p$ and $(\widehat p_t)$ be defined as in Section~\ref{sec:model}.
Note that the estimators $\widehat p$ and $(\widehat p_t)$ depend only on the non-random quantities $(\alpha_{t})$, $(\lambda_t)$, $\overline{T}$, $\underline{T}$ (which may depend on the sample size $n$), and the architecture $\cF = \cF_{\rm WSNN}(L, \bd, s, M, \cP_{\bfm})$.
Recall that $\alpha_{t}$ is the negative drift coefficient of the forward diffusion \eqref{eq:OU}, $\lambda_t$ is the weight for the loss function \eqref{eq:loss}, and $[\underline{T}, \overline{T}]$ is the interval defining both the loss function and the estimator $\widehat p$.
Throughout the paper, we assume the following without explicit restatement:
\bed
 \item[]1. $\bX^1, \ldots, \bX^n$ are \iid\ from $p_0$, supported on $[-1,1]^{D}$.
 \item[]2. $\sup_{\gamma \in \bbN} \vert  \partial^{\gamma } \alpha_{t} / \partial {t}^{\gamma}  \vert \leq 1 $  for all $t \geq 0$ and $\underline{\tau} \leq \alpha_{t} \leq \overline{\tau}$ for constants $\underline{\tau}, \overline{\tau} > 0$.
 \item[]3. $\lambda_t = 1$ for all $t \geq \underline{T}$.
\eed

Note that the standard OU process corresponds to $\alpha_{t} = 1$, and widely used diffusion models in practice (DDPM; \citealp{ho2020denoising}) often set $t \mapsto \alpha_t$ as a linear function; both choices satisfy the above requirements. 
For simplicity of the proofs, we set the weight function $\lambda_t$ to be constant. 
We also note that, with additional technical details, our main results can be extended to more general choices of forward diffusion processes and weight functions. 
For various designs of such choices in practice, we refer to \citet{karras2022elucidating}.
In addition to these basic assumptions, we will require the following additional assumptions:

\bed
\item[](\bS) The factorization assumption (\bF) is satisfied and there exist constants $\beta, K > 0$ such that $p_0 \in \cH^{\beta,K}([-1,1]^{D})$ and $g_I\in \cH^{\beta,K}([-1,1]^{|I|})$ for every $I \in \cI$.
\item[](\bL) There exists a constant $\tau_{1} > 0$ such that $p_0(\bx) \geq \tau_{1}$ for every $\bx \in [-1,1]^{D}$.
\item[](\bB) In addition to (\bS), there exists a constant $\tau_{2} \in (0,1)$ such that
\bean
\sup_{\bgamma \in \bbN^{D}} \sup_{\bx: 1-\tau_{2} \leq \| \bx\|_{\infty} \leq 1  }  \left\vert \D^{\bgamma} p_0(\bx)  \right\vert \leq K.
\eean
\eed

As in \citet{oko2023diffusion}, assumptions (\bB) and (\bL) are imposed purely for technical convenience. Although assuming that the density vanishes near the boundary of its support is more natural and plausible, assumption (\bL) greatly simplifies many proofs related to score-function estimation, since the score function $\nabla \log p_t(\bx)$ is defined as the ratio $\nabla p_t(\bx) / p_t(\bx)$. For this reason, following \citet{oko2023diffusion}, we adopted this assumption for tractability.

It should be noted that assumption (\bL) conflicts with the regularity or smoothness condition: once the density is bounded from below by a positive constant on its support, it automatically becomes non-smooth near the boundary when viewed as a function on $\bbR^D$.
This irregularity necessitates a careful approximation analysis near the boundary. To facilitate this analysis, we introduced the technical assumption (\bB), again following \citet{oko2023diffusion}. We also note that our main result, Theorem~\ref{secthm:2}, remains valid if the constant $\tau_2$ in assumption (\bB) is replaced by $(\log n)^{-\tau_{\rm bd}}$, where $\tau_{\rm bd}$ is an arbitrarily large constant.

In a recent preprint, \citet{fan2025optimal} obtained similar results without assumption (\bB), and their approach could, in principle, be adapted to our framework to remove (\bB) as well. We also note that several recent works \citep{zhang2024minimax, wibisono2024, fu2025approximation} have established results comparable to or weaker than those of \citet{oko2023diffusion} without relying on assumptions (\bB) and (\bL). However, their techniques do not appear to extend easily to the setting of factorizable densities considered in our paper.

The factorization assumption (\bF), combined with the smooth components assumption (\bS), forms the key structural assumption for our main results. Note that our results can be easily extended to the case where each factor function possesses a different level of smoothness.
The factorization assumption with smooth nonparametric components has been investigated in the statistical literature under the framework of nonparametric graphical models, specifically in the Markov random fields.
\citet{liu2011forest}, \citet{liu2012exponential}, \citet{gyorfi2023tree} focused on undirected acyclic graphs (forests), where $d$ is at most 2, and employed kernel methods. With this simple graph structure, \citet{liu2011forest} developed a consistent graph selection method, while \citet{liu2012exponential} constructed a minimax optimal density estimator for the special case of $\beta = 2$. Further advancements for the case $\beta = 1$ were studied in \citet{gyorfi2023tree}.

Nonparametric statistical theory for general undirected graph structures has been studied in some recent articles. In the case of $\beta = 1$, \citet{vandermeulen2024breaking} and \citet{vandermeulen2024dimension} proposed estimators whose convergence rates do not depend on the data dimension $D$. More specifically, \citet{vandermeulen2024breaking} introduced a novel quantity called the graph resilience $r$ and derived a convergence rate of $n^{-1/(r+2)}$ (up to a logarithmic factor) with respect to the total variation distance. They showed that this quantity satisfies $d \leq r \leq D$, meaning their rate is optimal only in special cases where $d = r$. Notably, $r$ can be much larger than $d$, for example, when the graph is a tree. \citet{vandermeulen2024dimension} studied a more tractable DNN-based estimator with a convergence rate of $n^{-1/(d+4)}$, which is sub-optimal.
More recently, \citet{gottwald2025localized} introduced a quantitative condition that characterizes the locality structure of the graph. They further showed that diffusion models yield consistent estimators when $D \gtrsim \log n$, with convergence rates that depend on this additional locality structure.

\citet{bos2024supervised} considered a slightly more general structure than the factorization assumption (\bF), using a different type of estimator. Specifically, they assumed that $p_0$ has a composite structure with smooth component functions. It is well known that deep neural networks can adapt to composite structures in nonparametric function estimation; see \citet{schmidt2020nonparametric}, \citet{bauer2019deep}, and \citet{kohler2021rate}. \citet{bos2024supervised} transformed the density estimation problem into a nonparametric regression problem and then constructed a density estimator. With this approach, they achieved a convergence rate of $n^{-\beta/(d+2\beta)} \vee n^{-\beta/D}$ (up to a logarithmic factor). While this rate improves upon existing results, it is optimal only when $D \leq 2\beta + d$. Note that our main results can also be extended to the composite structure considered in \citet{bos2024supervised} without significant difficulty.

\subsection{Convergence Rate}
\label{sec:conv}

The total variation distance between two Borel probability measures $P$ and $Q$ on $\bbR^{D}$ is defined as
\bean
d_{\rm TV}(P,Q) = \sup_{A} \vert P(A) - Q(A) \vert,
\eean
where the supremum is taken over every Borel subset $A$ of $\bbR^{D}$.
We often denote $d_{\rm TV}(P,Q)$ as $d_{\rm TV}(p,q)$, where $p$ and $q$ are Lebesgue densities of $P$ and $Q$, respectively.
The following theorem provides the convergence rate of $\widehat p$ with respect to the total variation, which is our main result.
Recall the definitions of the estimators $\widehat{\bff}$ and $\widehat{p}$ from Section~\ref{sec:model}, and note that the score function is denoted by $\bff_0(\bx, t) = \nabla \log p_t(\bx)$.

\begin{theorem} \label{secthm:2}
Suppose that $p_0$ satisfies (\bS), (\bL), and (\bB).
Let $\tau_{\min}$ and $\tau_{\max}$ be constants with

\bean
\tau_{\rm min} \geq \frac{d}{2\beta + d} \left( \frac{4\beta}{d (\beta \wedge 1)} \vee \frac{1}{3D} \right)
\quad \text{and}
\quad \tau_{\rm max} \geq \frac{\beta}{\underline{\tau}(2\beta+d)} \vee \frac{2dD}{(2\beta + d ) \overline{\tau}}.
\eean

Let $\underline{T} = n^{- \tau_{\rm min}}$ and $\overline{T} = \tau_{\rm max} \log n$.
Then, for every $n \geq C_{2}$, there exist a collection of permutation matrices $\cP_{\bfm} = ( ( \cQ_{i}, \cR_{i} ) )_{i \in [L-1] }$ and a class of weight-sharing neural networks $\cF_{\rm WSNN} = \cF_{\rm WSNN}(L,\bd,s,M,\cP_{\bfm})$ with
\bean
&& L \leq  C_{1} ( \log n )^{6} \log \log n , \quad \| \bd\|_{\infty} \leq  C_{1}  n^{\frac{d (D+1) }{2\beta + d}},
\\
&& s \leq  C_{1}  n^{\frac{d  }{2\beta + d}}  (\log n)^{5} \log \log n, \quad M \leq \exp(  C_{1}  \{ \log n \}^6 ),
\\ 
&& \| \bfm \|_{\infty} \leq  C_{1}  n^{\frac{dD}{2\beta + d}} 
\eean
satisfying 

\bean
&& \bbE \left[ \left( \int_{\underline{T}}^{\overline{T}} \int_{\bbR^{D}} \left\| \widehat \bff(\bx,t) - \bff_0(\bx, t)  \right\|_2^2 p_t(\bx) \d \bx \d t \right)^{1/2} \right]
 \leq \epsilon_n,
\eean
where $\widehat \bff$ is the ERM estimator over the class $\cF = \cF_{\rm WSNN} \cap \cF_{\infty}$ with
\bean
 \cF_\infty = \bigg\{ \bff: \| \bff(\bx,t) \|_{\infty} \leq \frac{C_{1} \sqrt{\log n}}{\sigma_t} \quad \forall \bx \in \bbR^{D}, t > 0 \bigg\}
\eean
and
\bean
\epsilon_n = C_{1} n^{-\frac{\beta }{2\beta + d}} \left\{ (\log n )^{ 2D + 2\beta + 1 } +  (\log n)^{10}  \right\}.
\eean
Moreover, 
\bean
 && \bbE \left[ d_{\rm TV} \left( p_{0}, \widehat p \right) \right]
 \leq \epsilon_n,
\eean
where $\widehat p = \widehat p_{\underline{T}}$ is the corresponding estimator defined through the SDE (\ref{eq:diffmodel}).
Here, $C_{1}, C_{2}$ are constants depending only on $(\beta,d,D,K, \tau_{\rm min}, \tau_{\rm max}, \tau_{1}, \tau_{2}, \overline{\tau}, \underline{\tau})$.

\end{theorem}

The proof of Theorem~\ref{secthm:2} is provided in the Appendix.
Note that $\overline{\tau}$ and $\underline{\tau}$ can be treated as known constants. For example, if $\alpha_t = 1$ for all $t$, both constants can be set to 1. The constants $\tau_{\min}$ and $\tau_{\max}$ can also be chosen to depend solely on the single quantity $\beta/d$, ignoring their dependence on the known quantities $(D, \overline{\tau}, \underline{\tau})$.

It can be observed that the class of permutation matrices $\cP_{\bfm}$ can likewise be chosen to depend only on $\beta/d$; see Section~\ref{sec:approx} for details.
Similarly, if $\tau_1$, $\tau_2$, and $K$ are treated as known constants, the hyperparameters $(L, \bd, s, M, \cP_{\bfm})$ defining the neural network $\cF_{\rm WSNN}$ can also be selected to depend solely on $\beta/d$.
Therefore, the estimators $\widehat \bff$ and $\widehat p$ ultimately depend only on $\beta/d$.

In this sense, $\widehat \bff$ and $\widehat p$ are adaptive to the factorization structure because their construction does not rely on the structural information. 
We do not aim in this paper to construct a fully adaptive estimator, in the sense of estimators that do not depend on $\beta/d$.
Although the architectures in Theorem~\ref{secthm:2}, including the class $\cP_{\bfm}$ of permutation matrices and the hyperparameters $(L, \bd, s, M)$, can be chosen to depend solely on $\beta/d$, in practice, much more complex architectures are often used, and hyperparameters are carefully tuned based on extensive experimental work.


Note that the convergence rate in Theorem~\ref{secthm:2} is minimax-optimal up to a logarithm factor over the class of factorizable densities.
Specifically, for $\beta, K > 0$ and $ D, d \in \bbN$ with $d \leq D$, let
\be
\begin{split}
 & \cG ( \beta, D, d, K )
 =
 \bigg\{ g_0 \in \cH^{\beta, K}([-1,1]^{D}):
 g_0(\bx) =  \prod_{I \in \cI} g_{I}(\bx_I),
 \\
 & \qquad\qquad\qquad\qquad\qquad\qquad
 g_{I} \in \cH^{\beta, K} ( [-1,1]^{ \vert I \vert } ),
 ~ \max_{I \in \cI} \vert I \vert = d, 
 ~ \cI \subseteq 2^{[D]}  
 \bigg\} 
 \label{eq:lowbd}
 \end{split}
\ee
be the class of factorizable densities with smooth component functions.
Then, we have
\bean
 \inf_{\widehat p} \sup_{p_0 \in \cG(\beta, D, d, K)} \E [d_{\rm TV}(p_0, \widehat p) ] \gtrsim  n^{-\beta / (2\beta + d)} ,
\eean
where the infimum is taken over all estimators.
The proof of this lower bound is straightforward, given the well-known result that the minimax rate for estimating a $d$-dimensional density in $\cH^{\beta, K}([-1,1]^d)$ is $n^{-\beta / (2\beta + d)}$ \citep{gine2016mathematical}.

Here, we present an overview of the key ideas behind the proof.
Recall that $\widehat p = \widehat p_{\underline{T}}$.
By the triangle inequality, we have
\be
\bbE \left[ d_{\rm TV} (p_{0}, \widehat p ) \right]
\leq d_{\rm TV} (p_{0}, p_{\underline{T}})  
+ \bbE \left[ d_{\rm TV} (p_{\underline{T}},  \widehat p_{\underline{T}}) \right]. \label{eqthm2:traingle}
\ee
The first term in the right-hand side of \eqref{eqthm2:traingle} scales as a polynomial order in $\underline{T} = n^{- \tau_{\rm min}}$, 
thus we can control the error by choosing a large constant $\tau_{\rm min}$; see Lemma~\ref{secsc:p0pt} for details.
The second term is the total variation distance between the distributions of $\bX_{\underline{T}} = \bY_{\overline{T} - \underline{T}}$ and $\widehat \bX_{\underline{T}} = \widehat \bY_{ \overline{T}- \underline{T}}$.
Note that the two processes $(\bY_t)_{t \in [0, \overline{T}]}$ and $(\widehat \bY_t)_{t \in [0, \overline{T}]}$ differ only in their initial distributions and drift functions.
Hence, based on well-known results, we can bound the total variation distance by controlling each difference separately as follows:

\be \begin{split} \label{eq:tv}
& \bbE\left[   d_{\rm TV} (p_{\underline{T}},  \widehat p_{\underline{T}}) \right]
\\
& \leq d_{\rm TV} \left( P_{\overline{T}}, \cN( \mathbf{0}_{D}, \bbI_{D} ) \right) + \bbE \left[
\left( \int_{\underline{T}}^{\overline{T}} \int_{\bbR^{D}} 4 \alpha_{t} \left\| \widehat \bff (\bx,t) - \bff_0(\bx, t) \right\|_2^2 p_t(\bx) \d \bx \d t \right)^{1/2} \right];
\end{split}\ee
see Remark 2.3 of \citet{bogachev2016distances}.
Both terms on the right-hand side of \eqref{eq:tv} correspond to the differences between the initial distributions and the drift functions, respectively.
The first term can be easily controlled because $P_{\overline{T}}$ converges exponentially fast to $\cN ( \mathbf{0}_{D}, \bbI_{D} )$ as $\overline{T} = \tau_{\rm max} \log n$ increases.
Thus, we can control the error by choosing a large constant $\tau_{\rm max}$.

The second term on the right-hand side of \eqref{eq:tv} represents the risk of the empirical risk minimizer $\widehat\bff$; see Proposition~\ref{sec:oracle} for the detailed statement regarding this term.
There is a substantial body of literature introducing techniques to bound the risk of empirical risk minimizers (e.g., \citealp{van1996weak, geer2000empirical, wainwright2019high}).
Technically, the risk can be decomposed into two terms: the approximation error and the estimation error, often referred to as the bias-variance decomposition.
To bound the estimation error, the key is to control the metric entropy of the weight-sharing networks $\cF_{\rm WSNN}$; see Lemma~\ref{sec:covering} for details.
The key technical contribution of our work lies in developing a sharp approximation error bound
for score functions under the factorization assumption (\bF), which is introduced in the following subsection.

\begin{remark}
In (\bS), we assume that all factors $g_{I}$ have the same smoothness level. This assumption can be relaxed, allowing each $g_I$ to have a different level of smoothness.
Specifically, suppose that for each $I \in \cI$, we have $g_{I} \in \cH^{\beta_I,K} ( [-1,1]^{\vert I \vert} ) $ for some $\beta_I > 0$.

One can directly extend the proof of Theorem~\ref{secthm:2} to show that
\bean
\bbE \left[ d_{\rm TV} \left( p_{0}, \widehat p \right) \right]
\lesssim_{\log} n^{-\frac{  \beta_* }{2  \beta_* +  d_* }}
\eean
with a carefully chosen network architecture, where 
\bean
I_* = \argmin_{ I \in \cI } \frac{\beta_I}{ \vert I \vert },
\quad \beta_{*} = \beta_{I_*},
\quad d_* = \vert I_* \vert.
\eean

The set-up of different smoothness levels includes the case of \citet{liu2007sparse}, who considered a density of the form $p_0(\bx) = g_{I} (\bx_{I}) g_{0}(\bx)$, where $g_0$ is very smooth and $I \subseteq [D]$.
Under the assumption that $g_I$ has continuous second-order derivatives, they proposed a density estimator that achieves a convergence rate of $O(n^{-2/(4+\vert I \vert) + \epsilon} )$ for any $\epsilon > 0$.
\end{remark}

\subsection{Approximation Theory}
\label{sec:approx}

Theorem~\ref{secthm:1} below presents the approximation results for the map $(\bx, t) \mapsto \bff_0(\bx, t) = \nabla \log p_t(\bx)$ using weight-sharing neural networks, which serves as the key technical component of our main results.

\begin{theorem}
\label{secthm:1}
Suppose the density function $p_0$ satisfies the assumptions (\bS), (\bL), and (\bB). Let $\tau_{\min}$ and $\tau_{\max}$ be constants with
\bean
\tau_{\rm min} \geq \frac{4\beta}{d (\beta \wedge 1)} \vee \frac{1}{3D}
\quad \text{and} \quad
\tau_{\rm max} \geq \frac{2D}{\overline \tau}.
\eean
Then, for every $m \geq  C_{4}$, there exist a collection of permutation matrices 
\bean
\cP_{\bfm} = ( ( \cQ_i, \cR_i ) )_{i \in [L-1]}
\eean
and a class of weight-sharing neural networks $\cF_{\rm WSNN} = \cF_{\rm WSNN}(L,\bd,s,M,\cP_{\bfm})$ with
\bean
&& L \leq C_{3} ( \log m )^{6} \log \log m , \quad \| \bd\|_{\infty} \leq C_{3}  m^{D+1},
\\
&& s \leq  C_{3}  m (\log m)^{5} \log \log m, \quad M \leq \exp( C_{3}  \{ \log m \}^6 ),
\\
&& \| \bfm \|_{\infty} \leq C_{3} m^{D}
\eean
such that
\bean
 \inf_{\bff \in \cF_{\rm WSNN} \cap \cF_\infty }
 \int_{\underline{T}}^{\overline{T}} \int_{\bbR^{D}} \left\| \bff_0(\bx, t) - \bff(\bx,t) \right\|_2^2 p_t(\bx) \d \bx \d t
 \leq  C_{3}  m^{-\frac{2\beta}{d}} (\log m)^{4D+ 4\beta + 1},
\eean
where $\underline{T} = m^{-\tau_{\min}}$, $\overline{T} = \tau_{\rm max} \log m$ and
\bean
 \cF_\infty = \bigg\{ \bff: \| \bff(\bx,t) \|_{\infty} \leq \frac{C_{3} \sqrt{\log m}}{\sigma_t} \quad \forall \bx \in \bbR^{D}, t > 0 \bigg\}.
\eean
Here, $C_3$ and $C_{4}$ are constants depending only on $(\beta, d, D, K, \tau_{\rm min}, \tau_{\rm max}, \tau_{1}, \tau_{2}, \overline{\tau}, \underline{\tau})$.
\end{theorem}

The proof of Theorem~\ref{secthm:1} is provided in Appendix.
From this proof, it can be deduced that the class of permutation matrices $\cP_{\bfm}$ can be chosen to depend only on $m$ and $\beta/d$. In the proof of Theorem~\ref{secthm:2}, we select $m$ based solely on $\beta/d$, implying that the choice of $\cP_{\bfm}$ ultimately depends only on $\beta/d$.


Here, we present an overview of the key ideas behind the proof. Note that $\nabla \log p_{t}(\bx) = \nabla p_{t}(\bx) / p_{t}(\bx)$, and the division operation can be approximated by DNNs very efficiently, provided that the denominator is not too small;
see Lemma~F.7 of \citet{oko2023diffusion}. Since the ideas behind approximating the maps $(\bx,t) \mapsto p_t(\bx)$ and $(\bx, t) \mapsto \nabla p_t(\bx)$ are similar, we only present the key idea for approximating $(\bx, t) \mapsto p_t(\bx)$. For convenience, we use the informal notation $a \lesssim_{\log} b$ to indicate that $a$ is less than $b$ up to a poly-logarithmic factor, such as $\log n$, $(\log m)^2$, or $\log ( 1 / \sigma_t )$. Similarly, we use the notation $\asymp_{\log}$ to correspond to $\asymp$.

For a given (sufficiently large) positive integer $m$, which roughly corresponds to the order of the number of nonzero network parameters, we will construct DNN approximators for the map $(\bx, t) \mapsto p_t(\bx)$ in four regions and combine them. These four regions for $(\bx, t)$ can be roughly defined as follows:
\begin{enumerate}[label = {} ]
 \item[](R1) (Outside of near-support) $\|\bx\|_\infty - \mu_t \gtrsim \sigma_t \sqrt{\log m}$
 \item[](R2) (large $t$) $\|\bx\|_\infty - \mu_t \lesssim \sigma_t \sqrt{\log m}$ and $t \gtrsim m^{-(2-\delta)/D}$ for some $\delta > 0$
 \item[](R3) (Boundary of near-support) $t \lesssim m^{-(2-\delta)/D}$ and $ -\{ \log (1/\sigma_{t})\}^{-3/2} \lesssim \|\bx\|_\infty - \mu_t \lesssim \sigma_t \sqrt{\log m}$
 \item[](R4) (Interior of near-support) $t \lesssim m^{-(2-\delta)/D}$ and $\|\bx\|_\infty - \mu_t \lesssim -\{ \log (1/\sigma_{t})\}^{-3/2}$
\end{enumerate}
Recall that $\mu_t = \exp ( - \int_{0}^{t} \alpha_s \d s )$ and $ \sigma_{t} = \sqrt{1 - \mu_{t}^2}$, so the maps $t \mapsto \mu_t$ and $t \mapsto \sigma_t$ can be approximated by DNNs very efficiently.
Likewise, the map $\bx \mapsto \|\bx\|_\infty$ can also be efficiently approximated.
Therefore, once we can approximate the map $(\bx, t) \mapsto p_t(\bx)$ in each of the four regions, it is not difficult to combine them into a single function over the entire region.

In region (R1), $p_t$ is nearly zero due to the sub-Gaussianity of $p_t$, making it easy to approximate. In region (R2), $t$ is sufficiently large, and thus the map $\bx \mapsto p_t(\bx)$ is much smoother than $\bx \mapsto p_0(\bx)$. This smoothness property enables the construction of a DNN with a moderate number of nonzero parameters, as in Lemma B.7 of \citet{oko2023diffusion}; see Proposition~\ref{secpt:3}  for details. Similarly, in region (R3), the map $\bx \mapsto p_t(\bx)$ is very smooth due to the assumption (\bB), allowing us to construct a DNN with the desired approximation properties, similar to Lemmas B.2--B.5 of \citet{oko2023diffusion}; see Proposition~\ref{secpt:2} for details.

The main challenge in the proof of Theorem~\ref{secthm:1} lies in the approximation in region (R4). Note that $p_{t}(\bx) \gtrsim 1$ in region (R4), and that
\be \begin{split}\label{eqapprox:ptint}
  p_t(\bx) &= \int_{\|\bz\|_{\infty} \leq 1} p_0(\bz) \phi_{\sigma_{t}}(\bx-\mu_t \bz) \d \bz
  \\
  &= \int_{\left\| \frac{ \bx + \sigma_{t} \by }{\mu_{t}}\right\|_{\infty} \leq 1} \mu_{t}^{-D} p_0 \left( \frac{\bx + \sigma_{t} \by}{\mu_{t}} \right) \phi_{1}(\by) \d \by,
\end{split}\ee
where $\phi_{\sigma}$ denotes the density of $\cN(\mathbf{0}_{D}, \sigma^2 \bbI_{D})$ for $\sigma > 0$.
To approximate the right-hand side of \eqref{eqapprox:ptint}, we first approximate it by a finite sum via a quadrature method, and then approximate the sum using a weight-sharing neural network.

To grasp the idea of approximation, it suffices to consider the approximation of a general function
\be \label{eqapprox:gint}
  \bx \mapsto \int_{[-1,1]^D} g(\bx,\by) \d \by,
\ee
defined through a $D$-dimensional integral. Here, $g$ is a function such that for each $\bx$, the map $\by \mapsto g(\bx, \by)$ belongs to $\cH^{\beta, K}([-1,1]^D)$, where $K > 0$ is a constant independent of $\bx$. We provide an idea for constructing a weight-sharing neural network to approximate the map \eqref{eqapprox:gint} with an error of $\epsilon \asymp_{\log} m^{-\beta/d}$. It is well-known from numerical analysis \citep{novak1988deterministic} that, to achieve an approximation error of $\epsilon$ for every function in $\cH^{\beta, K}([-1,1]^D)$ using the Gauss–Legendre quadrature method, at least $O(\epsilon^{-D/\beta})$ quadrature points are necessary. Hence, \eqref{eqapprox:gint} can be approximated by a finite sum with $O(\epsilon^{-D/\beta})$ summands. However, to approximate this $O(\epsilon^{-D/\beta})$-term summation using DNNs, we would need at least $O(\epsilon^{-D/\beta})$ network parameters (up to a poly-logarithmic factor), which results in a very large estimation error. To overcome this difficulty, instead of applying a single $D$-dimensional quadrature method, we apply a 1-dimensional $m$-point quadrature method $D$ times to approximate the $D$-dimensional integral \eqref{eqapprox:gint}. Specifically, let $(w_j)_{j \in [m]}$ and $(y_j)_{j \in [m]}$ be the $m$-point quadrature weights and nodes for 1-dimensional integrals over the interval $[-1, 1]$, that is, for any $h \in \cH^{\beta,K}([-1,1])$,
\bean
  \left\vert \int_{-1}^1 h(y) \d y
  - \sum_{j =1}^m w_j h \left( y_j \right) \right\vert
  \lesssim m^{-\beta}; 
\eean
see Lemma~\ref{sec:qde} for details. Then, we can easily see that 
\bean
  \left\vert \int_{[-1,1]^D} g (\bx, \by) \d \by
  - \sum_{\bj \in [m]^D} w_\bj g \left( \bx, \by_\bj \right) \right\vert
  \lesssim m^{-\beta}
\eean
where $w_\bj = \prod_{k=1}^D w_{j_k}$ and $\by_\bj = (y_{j_1}, \ldots, y_{j_D})$. Here, we slightly abuse the notation for weights.

We next approximate the map
\be \label{eq:sum}
 \bx \mapsto \sum_{\bj \in [m]^D} w_\bj g \left( \bx, \by_\bj \right)
\ee
using weight-sharing neural networks. Although the summation in \eqref{eq:sum} consists of $m^D$ terms and resembles a $m^D$-point, $D$-dimensional quadrature, it can be approximated by weight-sharing DNNs much more efficiently than a standard $m^D$-point, $D$-dimensional quadrature approximation. The key ingredients are the approximations of the following two maps:
\be\begin{split} \label{eq:maps}
 (w_1, \ldots, w_{m}) &\mapsto (w_\bj)_{\bj \in [m]^D}: \bbR^{m} \to \bbR^{m^D}
 \\
 (\bx, y_1, \ldots, y_{m}) &\mapsto (g(\bx, \by_\bj))_{\bj \in [m]^D}: \bbR^{D + m} \to \bbR^{m^D}.
\end{split}\ee

Although the $w_\bj$'s are distinct, each $w_{\bj}$ is represented as a product of $D$ terms from the $m$ distinct values $w_1, \ldots, w_{m}$. Therefore, to approximate the first map of \eqref{eq:maps}, we only need to approximate the multiplication operation $(x, y) \mapsto xy$ and apply it multiple times. Note that multiplication can be approximated by DNNs very efficiently \citep{schmidt2020nonparametric}. With an additional trick, the repeated application of multiplication can be represented as a DNN of the form \eqref{eqwsnn}, with a suitable choice of permutation matrices. Roughly speaking, to achieve an approximation error of $\epsilon \asymp_{\log} m^{-\beta/d}$ for this map, we only need $O(\log m)$ distinct network parameters, which is the same as for approximating a single multiplication operation between two real numbers up to a constant multiple.

Similarly, weight sharing is crucial for approximating the second map in \eqref{eq:maps}. Since the function $g$ are approximated for $m^D$ instances, weight-sharing networks help reduce the number of distinct network parameters, see Figure \ref{fig:wsnn} for an illustration. The number of parameters required for a single evaluation of $g$ with an approximation error of $m^{-\beta/d}$ depends on the structure of $g$. In our case, $O(m)$ parameters (up to a poly-logarithmic factor) are sufficient, due to the factorization property of $p_0$.

By combining the results above, we can construct weight-sharing neural networks with $O(m)$ distinct parameters (up to a poly-logarithmic factor) to approximate \eqref{eqapprox:gint} with an error of $m^{-\beta/d}$.

Returning to the problem of approximating \eqref{eqapprox:ptint}, a key difference between \eqref{eqapprox:ptint} and \eqref{eqapprox:gint} lies in the range of the integral, which depends on $(\bx, t)$. In particular, the diameter of the range also varies with $t$. As a result, we must use different quadrature weights and nodes for each pair $(\bx, t)$. However, these quadrature weights and nodes can be expressed as (very) smooth functions of $(\bx, t)$, making them easily approximated by deep neural networks. Full proofs, including additional technical details, are provided in the Appendix.

\begin{figure*}[!t]
    \centering
    \includegraphics[width=0.4\linewidth]{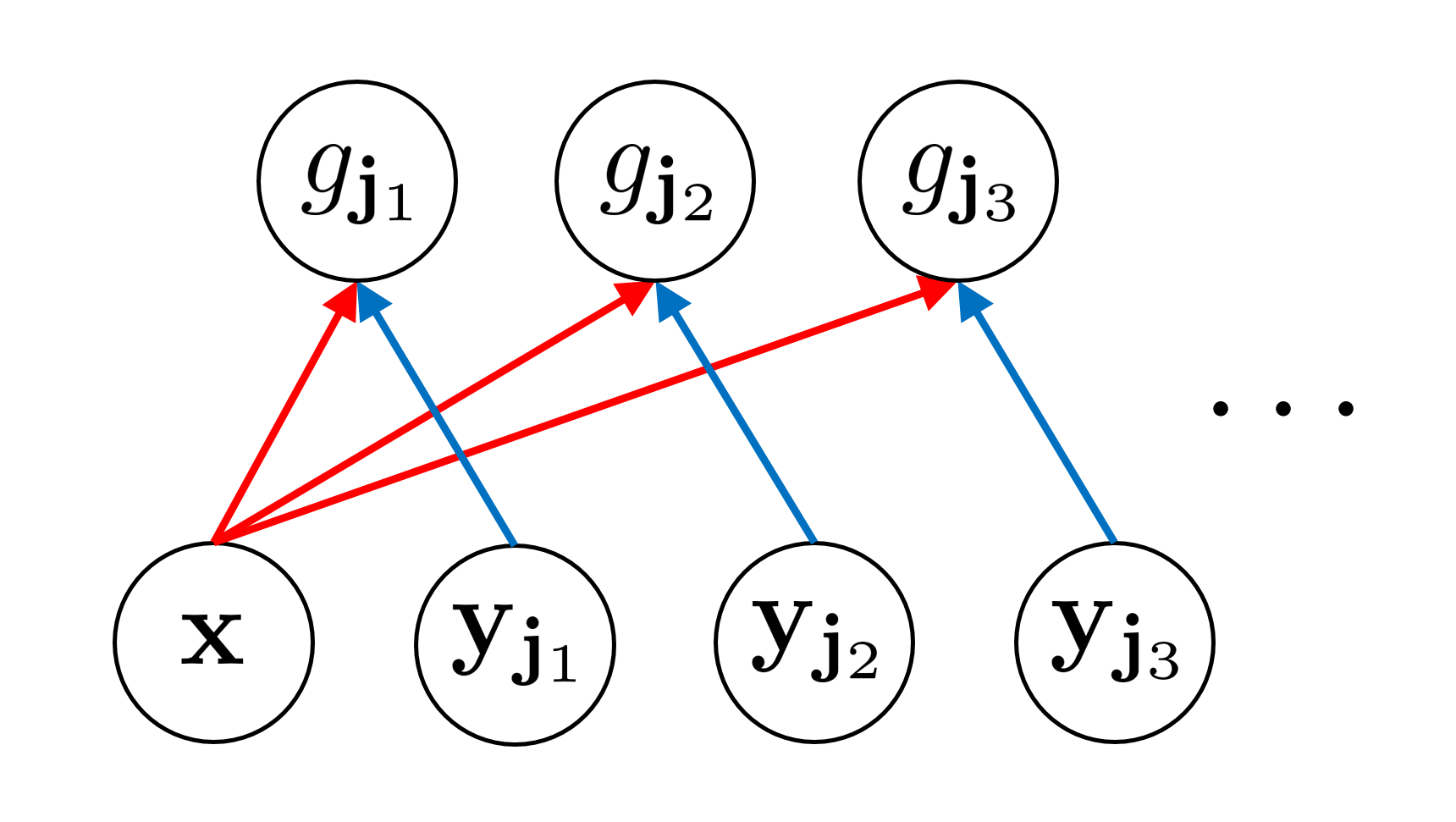}
        \caption{An illustration of why weight-sharing helps reduce model complexity: At some middle layers of the network, we need to approximate a map with inputs $\bx$ and $\by_{\bj}, \bj \in [m]^D$, and outputs $g_{\bj} = g(\bx, \by_{\bj}), \bj \in [m]^D$. Since a single function $g$ is approximated for $m^D$ instances, leaving all network parameters as free parameters is inefficient. Weight-sharing can significantly reduce the number of distinct parameters. In this illustration, all edges with the same color share the same weight parameters.}
    \label{fig:wsnn}
\end{figure*}

\section{Sub-Optimality of a Vanilla Score Matching Estimator}
\label{sec:naive}

One of the main technical difficulties in our results in Section~\ref{sec:main} arises from the fact that $p_t$ is no longer factorizable for $t > 0$.
In practice, a key component of the success of score-based generative models and diffusion models lies in jointly modeling infinitely many score functions using deep neural networks via the map $(\bx, t) \mapsto \bff(\bx, t)$ \citep{song2019generative, song2021scorebased}.
Note that early works on the score estimation have focused on estimating the single score function $\bx \mapsto \nabla \log p_0(\bx)$ via the score matching loss
\bean
 \widetilde \ell_\bff(\bx) = \text{tr} \left( \nabla \bff(\bx) \right) + \frac{1}{2} \left\| \bff(\bx) \right\|_2^2,
\eean
which is based on the fact that
\bean
 \frac{1}{2} \bbE \left[ \left\| \bff(\bX_0) - \nabla \log p_0 (\bX_0) \right\|_2^2 \right] 
 = \bbE \left[ \text{tr} \left( \nabla \bff(\bX_0) \right) + \frac{1}{2} \left\| \bff(\bX_0) \right\|_2^2 \right] + C
\eean
under mild assumptions, where $C$ is a constant depending only on $p_0$ \citep{hyvarinen2005estimation}.

For a given class $\cF$ of score functions, let $\widehat \bff_{\rm VS}$ be the corresponding empirical risk minimizer, that is,
\bean
 \widehat \bff_{\rm VS} \in \argmin_{\bff \in \cF}\frac{1}{n} \sum_{i=1}^n \left[ \text{tr} \left( \nabla \bff(\bX^i) \right) + \frac{1}{2} \left\| \bff(\bX^i) \right\|_2^2 \right].
\eean
The corresponding density estimator can be defined via the Langevin diffusion.
Specifically, let $( \bZ_{t} )_{t > 0} $ be the solution to the following Langevin equation:
\bean
\d \bZ_t = - \nabla \log p_{0} ( \bZ_{t} ) \d t + \sqrt{2} \d \bB_t, \quad \bZ_0 \sim \cN ( \mathbf{0}_{D}, \bbI_{D} ).
\eean
Then, under mild assumptions, the distribution of $\bZ_t$ converges to $p_0$ as $t \to \infty$.
The convergence speed can be exponentially fast under certain conditions on $p_0$, such as when $p_0$ satisfies a Poincar{\'e}  inequality or a log-Sobolev inequality \citep{bakry2014analysis}.
Hence, one can define a density estimator $\widehat p_{\rm VS}$ as the limit distribution of the Langevin equation, with the true score function replaced by $\widehat \bff_{\rm VS}$.
We refer to $\widehat \bff_{\rm VS}$ and $\widehat p_{\rm VS}$ as vanilla score matching estimators for the score and density functions.

Note that vanilla score matching estimators are rarely used in modern large-scale generative problems.
One reason is that the trace map $\text{tr} \left( \nabla \bff(\bx) \right)$ is computationally challenging to handle in high-dimensional (large $D$) problems, such as image generation tasks.
Therefore, one may raise an important question: if the computation of $\widehat \bff_{\rm VS}$ is tractable, would it perform well?
More theoretically, one may ask whether $\widehat \bff_{\rm VS}$, and consequently $\widehat p_{\rm VS}$, can achieve the optimal convergence rate.

Our conjecture is that, regardless of the potential minimax optimality of $\widehat{\bff}_{\rm VS}$, the estimator $\widehat{p}_{\rm VS}$ cannot achieve the optimal convergence rate.
The rationale behind this conjecture lies in the nature of estimating the score function.
From the well-established results on the minimax optimal convergence rate for estimating the density derivative \citep{stone1982optimal, singh1977improvement, shen2017posterior, yoo2016supremum}, under suitable assumptions, one can derive the following lower bound for estimating the score function:
\be
\inf_{\widehat \bff} \sup_{p_0 \in \cG(\beta,D,d,K)}  \bbE \left[ \left\| \widehat \bff(\bX_{0}) - \nabla \log p_0(\bX_0) \right\|_2 \right]
\gtrsim n^{-\frac{\beta-1}{2\beta+d}}, \label{eq:rate-ns} 
\ee
where the infimum is taken over all estimators and $\cG(\beta,D,d,K)$ denotes the class of $\beta$-smooth factorizable densities, defined in (\ref{eq:lowbd}).
Although we do not provide a specific proof in this paper, such a result can be obtained by applying the standard Fano’s method. We refer readers to \citet{wibisono2024} for the case $\beta = 2$ without factorization structures, and to \citet{tsybakov2008introduction} and \citet{wainwright2019high} for general lower-bound techniques.

The rate in (\ref{eq:rate-ns}) is a lower bound, which implies that the convergence rate of $\widehat \bff_{\rm VS}$ cannot be faster. The slower rate, compared to Theorem~\ref{secthm:2}, arises from the fact that the score function is only $(\beta-1)$-smooth, which is strictly less smooth than the density.
From the convergence rate of the score function estimator $\widehat\bff_{\rm VS}$, one can derive the same convergence rate with respect to the total variation for the corresponding density estimator $\widehat p_{\rm VS}$ via Girsanov's theorem \citep{girsanov1960transforming, le2016brownian}; see also Remark 2.3 of \citet{bogachev2016distances}.
While the rate $n^{-\frac{ \beta-1}{ 2 \beta + d }}$ is optimal for estimating the score function, the optimal rate for density estimation with respect to the total variation is strictly faster.
Hence, the optimality of $\widehat{p}_{\rm VS}$ is not guaranteed.
Although we do not have a formal proof that $d_{\rm TV}(\widehat{p}_{\rm VS}, p_0) \gtrsim n^{-(\beta-1)/(2\beta + d)}$, we conjecture that this is indeed the case and do not pursue a formal proof in the current paper.

\section{Numerical Experiments}
\label{sec:exp}

In this section, we conduct a small-scale simulation study to empirically examine our theoretical findings. The numerical experiments are designed to explore three main questions. First, we aim to assess whether diffusion models can perform reasonably well at small to moderate data scales. Although this question is not directly related to our theoretical analysis, it remains a natural and practically relevant question because diffusion models are typically applied to very high-dimensional settings with a large number of parameters. Second, we investigate whether, when a low-dimensional structure such as factorization is present, diffusion models tend to outperform classical methods such as kernel density estimation (KDE). Conversely, in the absence of such structural assumptions apart from smoothness, we conjecture that diffusion models perform comparably to classical estimators, at least in low-to-moderate dimensions (e.g., $D \le 10$). Finally, we explore the potential benefits of sparse weight-sharing neural networks compared to fully connected architectures through these experiments.

Before proceeding further, we note that although our theoretical results provide valuable insights into when diffusion models perform well compared to other methods, there remains a substantial gap between the theoretical estimator and its practical implementations. In practice, the number of nonzero network parameters in diffusion models is typically much larger than the available sample size, and various forms of algorithmic regularization (either explicit or implicit) are employed. Moreover, the U-Net architecture widely used in applications differs significantly from the theoretical weight-sharing architecture considered in this work. Given these discrepancies, it is inherently challenging to design simulation studies whose outcomes align precisely with the theoretical predictions.

\subsection{Data Set Descriptions}

Hereafter, we denote $\bX_0 = (X_1,\ldots,X_{D})$ as the random vector following the true distribution.
Throughout our experiments, we fix the data dimension at $D = 5$. We analyze three types of true data distributions characterized by different effective dimensions $d$: (1) $d = 1$, (2) $d = 2$, and (3) $d = D$. Each true distribution is specified by its marginal distributions and a dependency structure imposed through a copula.
Although not reported here, in addition to the marginal distributions and dependency structures described below, we also experimented with various other marginal and copula densities. We found that the results were qualitatively similar and consistent across these settings.

For the marginal distributions, we consider those used in the experiments of \citet{bos2024supervised}, rescaled to lie in the support $[-1,1]^{D}$.
Specifically, each variable $X_i $ follows a density $\widetilde p$, which belongs to a \Holder \ class with smoothness $\beta = 1/2$.
The explicit form of $\widetilde p$ is described in Figure~\ref{fig:bos_marginal}.


For the dependency structure, we employ three types of copulas: the independence copula ($d = 1$), the Gaussian copula ($d = 2$), and the Clayton copula ($d = D$). For the Gaussian copula, we adopt an AR(1)-type correlation structure with a correlation factor of $0.8$. Recall that the Clayton copula is defined as
\begin{equation*}
C_{\theta}(u_1, \ldots, u_D)
= \left( \sum_{i=1}^D \bigl(u_i^{-\theta} - 1 \bigr) + 1 \right)^{-1/\theta},    
\end{equation*}
where $\theta$ is a parameter controlling the strength of dependence.
Its corresponding copula density can be derived in closed form as
\bean
c_{\theta} (u_1,\ldots,u_{D}) = \left( \prod_{k=0}^{D-1} \left(1 + k \theta \right) \right) \left( \prod_{i=1}^{D} u_{i}^{-\theta-1} \right)  \left( \sum_{i=1}^D \bigl(u_i^{-\theta} - 1 \bigr) + 1 \right)^{-\frac{1+D\theta}{\theta}}.
\eean
In our experiments, we set $\theta = 5$.

For the training data, we consider sample sizes ranging from $n = 200$ to $n = 100{,}000$, while the test data consists of $5{,}000$ samples. For each setting, we repeat the experiments five times independently and report the average performance below.

\begin{figure}[!t]
    \centering
    \begin{minipage}[t]{0.4\linewidth}
        \centering
        \vspace{0pt}
        \small
        \[
        \widetilde p(x) =
        \begin{cases}
        \frac{1}{2} + \frac{1}{4D} - \frac{1}{2D}\sqrt{-\frac{1}{4} - \frac{x}{2}}, 
        & -1 \le x < -\frac{1}{2}, \\[6pt]
        \frac{1}{2} + \frac{1}{4D} - \frac{1}{2D}\sqrt{\frac{1}{4} + \frac{x}{2}}, 
        & -\frac{1}{2} \le x < 0, \\[6pt]
        \frac{1}{2} - \frac{1}{4D} + \frac{1}{2D}\sqrt{\frac{1}{4} - \frac{x}{2} }, 
        & 0 \le x < \frac{1}{2}, \\[6pt]
        \frac{1}{2} - \frac{1}{4D} + \frac{1}{2D}\sqrt{- \frac{1}{4} + \frac{x}{2} }, 
        & \frac{1}{2} \le x \le 1.
        \end{cases}
        \]
    \end{minipage}
    \hfill
    \begin{minipage}[t]{0.4\linewidth}
        \centering
        \vspace{0pt}
        \includegraphics[width=\linewidth]{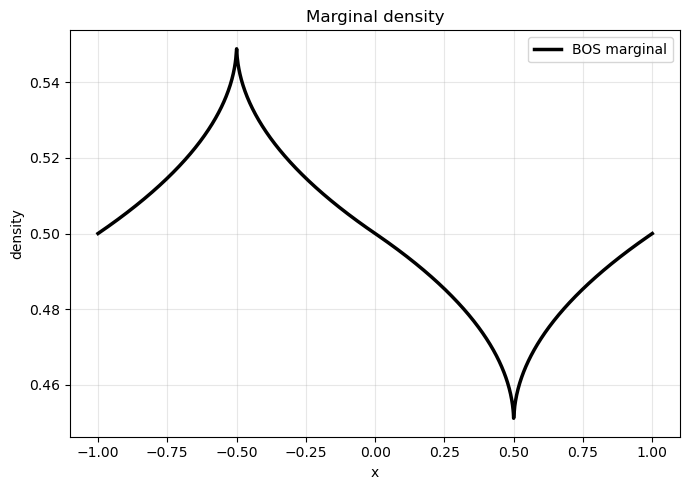}
    \end{minipage}

    \caption{Marginal density used in the experiments.}
    \label{fig:bos_marginal}
\end{figure}

\subsection{Learning Algorithms and Implementation Details} 

\subsubsection{Diffusion models}

We employ the denoising diffusion probabilistic model (DDPM; \citealp{ho2020denoising}), which corresponds to the standard OU process and is one of the most widely used diffusion frameworks. The drift coefficient is set as $\alpha_t = \alpha_{\min} + t(\alpha_{\max} - \alpha_{\min}) / (\overline{T}-1)$ and we fix $(\alpha_{\min}, \alpha_{\max}, \overline{T}) = (0.0005, 0.01, 500)$ throughout all experiments.

The original DDPM architecture, which contains approximately 37 million parameters, is excessively large for the simulated datasets considered here. To make the comparison computationally feasible, we reduce its size to about 260,000 parameters. In addition, we modify the architecture, which is originally designed for two-dimensional image inputs, to handle one-dimensional vectors of dimension five.
Based on this adjustment, we evaluate two different architectures in our experiments, both constructed to maintain comparable parameter counts (around 260K) to ensure a fair comparison.

\begin{itemize}
\item \textbf{DDPM with weight-sharing neural networks (DDPM with WSNN):} 
This architecture transforms the 5-dimensional input into a 160-dimensional spatial feature via a fully connected (FC) layer ($5 \times 160$) and embeds the scalar time input into a 160-dimensional time feature, which is further transformed by a FC layer ($160 \times 160$).
The spatial feature is then processed by stacking nine weight-sharing blocks while preserving the dimensionality throughout the network. Each block consists of two one-dimensional convolution layers (Conv1D; kernel size 11, padding 5) operating on the spatial feature, with layer normalization. The time feature is transformed by a FC layer ($160 \times 160$) and added to the spatial feature after the first convolution in each block.
A residual connection is applied to the spatial feature after the second convolution.
The final output is linearly projected back to the 5-dimensional space.
With 265,483 parameters, the model follows the pipeline:

\bc
\begin{tikzpicture}[
        >=Stealth,
        node distance = 1.5cm,
        every node/.style = {font=\large},
    ]
    \node (A0) {$5$};
    \node (B0) [right=of A0] {$160$};
    \node (C0) [right=of B0] {$160$};
    \node (D0) [right=of C0] {$160$};
    \node (H0) [right=of D0] {$5.$};

    \node (At0) [below=1cm of A0] {$1$};
    \node (At1) [below=1cm of B0] {$160$};
    \node (At) [below=1cm of C0] {$160$};

    \draw[->] (A0) -- node[midway, above, font=\scriptsize, align=center] {FC} (B0);
    \draw[->] (B0) -- node[midway, above, font=\scriptsize] {Conv1D} (C0);
    \draw[->] (C0) -- node[midway, above, font=\scriptsize] {Conv1D} (D0);
    \draw[->] (D0) -- node[midway, above, font=\scriptsize] {Linear} (H0);
    \draw[->] (At0) -- node[midway, above, font=\scriptsize, align=center] {Time\\embed} (At1);
    \draw[->] (At1) -- node[midway, above, font=\scriptsize, align=center] {FC} (At);

    \draw[->] (At) -- node[pos = 0.3, above, font=\scriptsize] {FC} (C0);

    \draw[dashed, rounded corners]
      ([xshift=-6pt,yshift=10pt]B0.north west) --
      ([xshift= 6pt,yshift=10pt]D0.north east)
     node[midway, above=2pt, font=\footnotesize] {Weight-sharing block $\times\ 9$}
      --
      ([xshift= 6pt,yshift=-10pt]D0.south east) --
      ([xshift= 6pt,yshift=-10pt]C0.south east) --
      ([xshift= 6pt,yshift=-10pt]At.south east) --
      ([xshift=-6pt,yshift=-10pt]At.south west) --
      ([xshift=-6pt,yshift=-10pt]C0.south west) --
      ([xshift=-6pt,yshift=-10pt]B0.south west) -- cycle;

\end{tikzpicture}
\ec

\item \textbf{DDPM with fully connected neural networks (DDPM with FCNN):} 
Unlike the WSNN version, this architecture replaces the Conv1D layers with $160 \times 160$ FC layers throughout the network.
To maintain comparable parameter counts, this architecture repeats the block only three times (instead of nine), resulting in 260{,}485 parameters.

\end{itemize}


Regardless of the sample size, both DDPM models are trained using the Adam optimization algorithm \citep{kingma2014adam} for 1,000 epochs, with a mini-batch size of 100 and learning rates of $5 \times 10^{-3}$ for WSNN and $10^{-3}$ for FCNN. All experiments are implemented in the \texttt{PyTorch} framework and executed on four \texttt{NVIDIA RTX 3090} GPUs.

\subsubsection{Other Baselines}

As baseline approaches, we consider one classical nonparametric method and one deep learning based method. For the classical method, we employ the KDE with a Gaussian kernel. The optimal bandwidth is selected according to Silverman’s rule of thumb \citep{silverman2018density}. The KDE implementation, including the sampling procedure, is carried out using the \texttt{Scikit-learn} module in \texttt{Python}. 

For the deep learning based method, we adopt the approach proposed by \citet{bos2024supervised}, hereafter referred to as BOS. This method follows a two-stage procedure: in the first stage, a KDE is constructed using half of the dataset, and in the second stage, a deep neural network is trained to approximate the resulting kernel density function using the remaining data. It is known that the resulting deep learning model achieves a fast convergence rate under a reasonably well-specified true structural density function.
In the KDE stage, the Epanechnikov kernel is employed, and the bandwidth is chosen of the form $C \times (\log n / n)^{1/D}$, where the constant $C$ is selected via 5-fold cross-validation.
In the supervised learning stage, a FCNN with $\lceil \log_2 (2n) \rceil$ hidden layers is employed, where each layer consists of $\lceil (2n)^{1/2} \rceil$ nodes and uses ReLU activation functions.
For further experimental details, we refer the reader to the official \texttt{GitHub} repository of \citet{bos2024supervised}.
Using the trained neural network as an estimated density function, we then generate samples via the Metropolis-Hastings algorithm (MH; \citealp{metropolis1953equation, hastings1970monte}).

\subsection{Performance Measure}

Since computing the total variation distance requires an explicit density estimator, which is not straightforward for diffusion models, we instead assess performance using the Wasserstein-1 distance based on test samples. Specifically, for each method, we independently generate $m = 5{,}000$ samples. For KDE, generating samples is straightforward because its density estimator corresponds to a Gaussian mixture. As mentioned earlier, we use the MH algorithm for BOS.

During MH sampling, we reflect proposed samples at the boundaries to ensure they remain in $[-1,1]^{D}$, since the neural network is only trained on this support.
It is well known that computing the Wasserstein-1 distance between two empirical distributions can be formulated as a linear programming problem. To stabilize the computation, we adopt the Sinkhorn algorithm proposed by \citet{cuturi2013sinkhorn}.

\begin{figure*}[!t]
    \centering
    \includegraphics[width=0.48\linewidth]{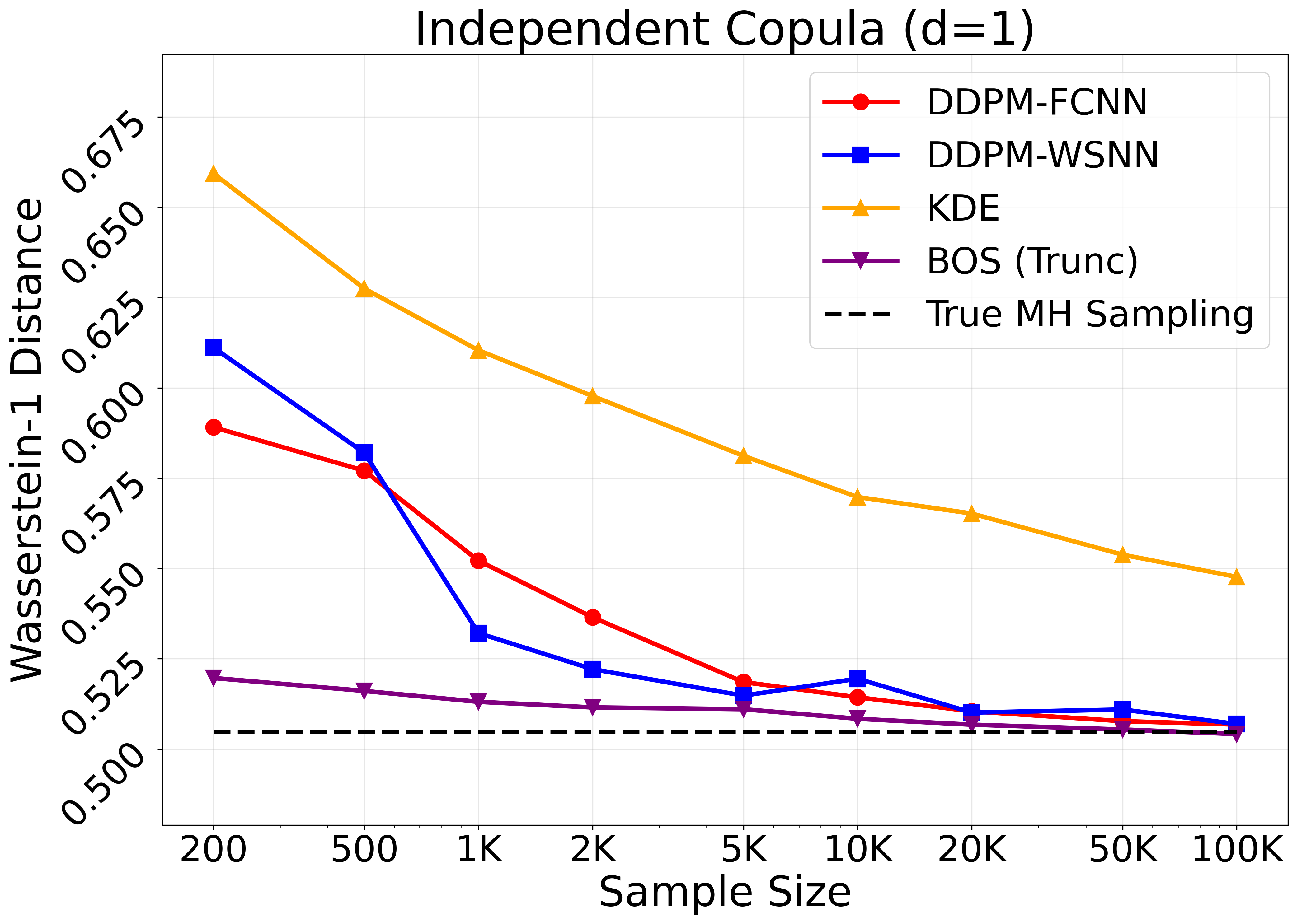}
    \hfill
    \includegraphics[width=0.48\linewidth]{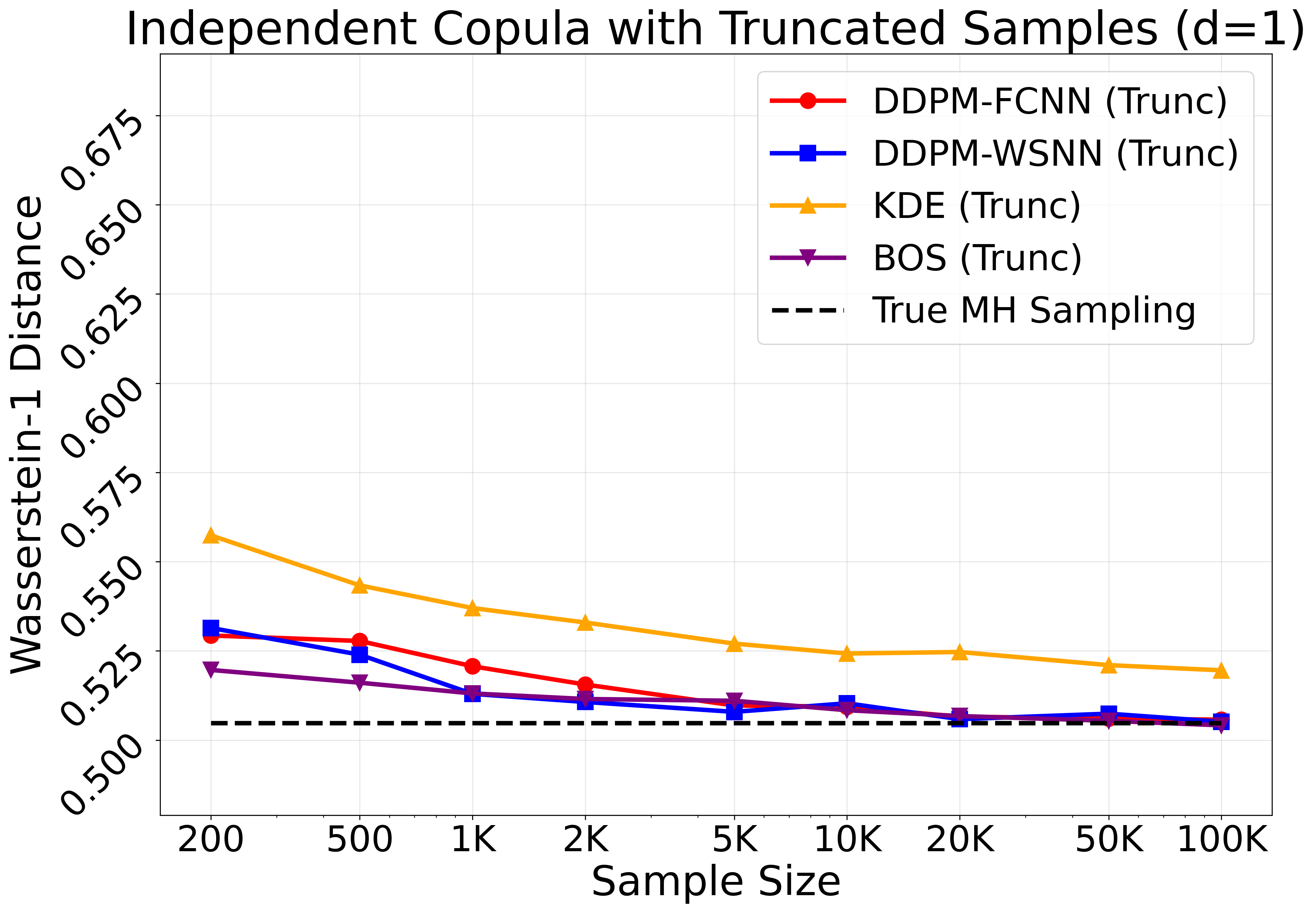}

    \vspace{0.5em}

    \includegraphics[width=0.48\linewidth]{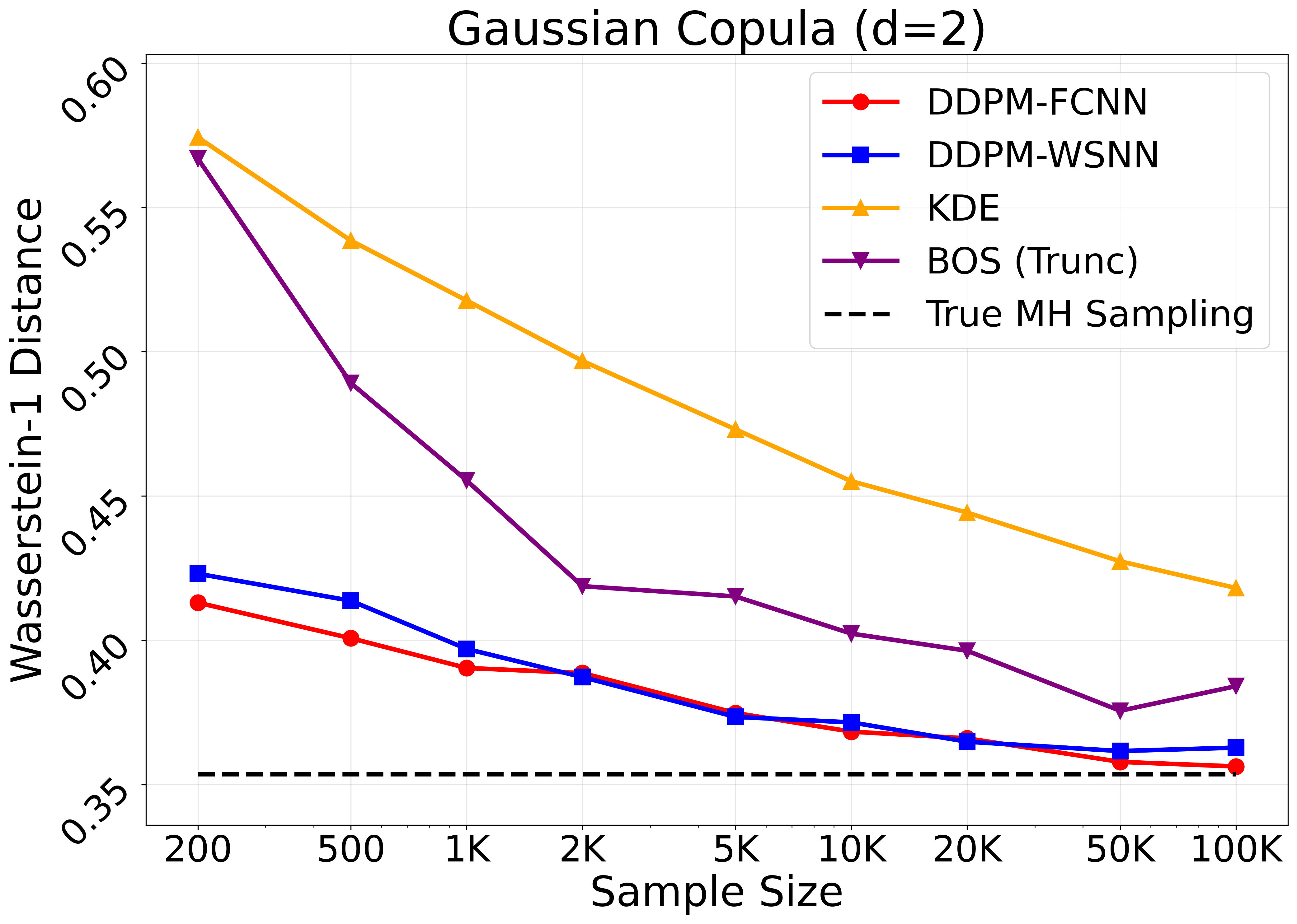}
    \hfill
    \includegraphics[width=0.48\linewidth]{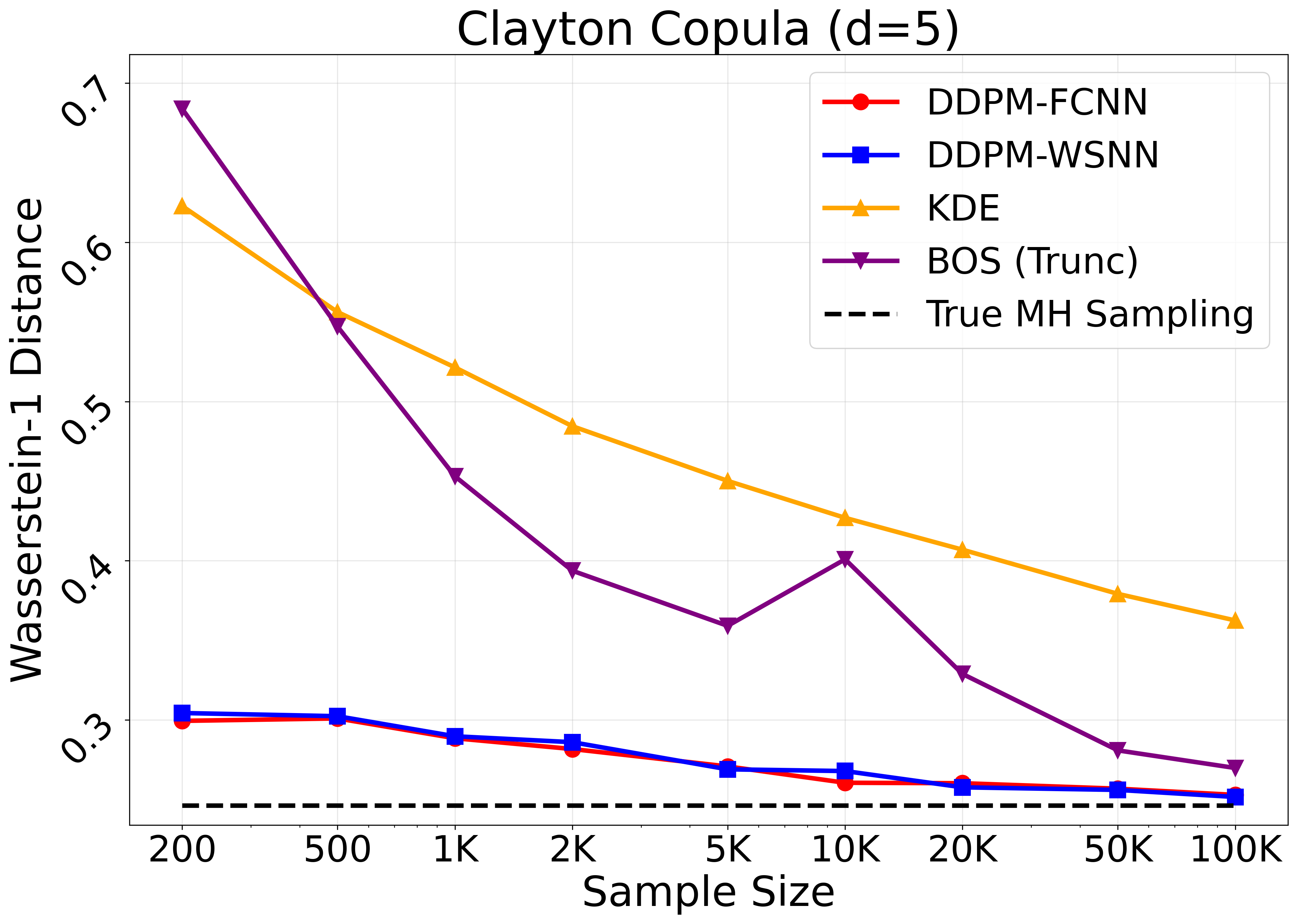}

    \caption{Wasserstein-1 distance values across different copula settings.
    Top row (left to right): independence copula results without truncation and with truncation for the generated samples.
    Bottom row (left to right): Gaussian copula and Clayton copula.}
    \label{fig:exp_n01}
\end{figure*}

\subsection{Performance Results}

For each setting, we vary the training sample size $n$ from $200$ to $100{,}000$ and compare the Wasserstein-1 distances of the two DDPM variants, one using WSNN and the other using FCNN, with those of KDE and BOS. The results are summarized in Figure~\ref{fig:exp_n01}.

Overall, DDPMs outperform KDE for all cases $d \in \{1,2,5\}$.
DDPMs also outperform BOS for $d = 2$ and $d = 5$, while their performance is comparable to that of BOS for $d = 1$, and slightly inferior when $n \leq 2{,}000$.
This behavior arises because sampling from the BOS model explicitly exploits knowledge of the support of the true density, which leads to a substantial performance advantage over DDPMs and KDE at small sample sizes.
The top-right panel of Figure~\ref{fig:exp_n01} shows that, when support information is used at the sampling stage, the performance gap between BOS and the other methods narrows considerably for small sample sizes.

The remaining panels demonstrate that DDPMs, even without using support information, still exhibit strong performance across all distributions. Remarkably, for all cases, DDPMs with large sample sizes $(n \geq 20{,}000)$ achieve performance comparable to the oracle benchmark,  where the MH algorithm is applied using the true density.

Moreover, DDPMs outperform KDE even with a small number of samples, highlighting the empirical effectiveness of diffusion models in limited-data regimes.
For $d = 1$, DDPMs with $n = 1{,}000$ samples already outperform KDE with $n = 100{,}000$ samples.
Even more strikingly, for $d = 2$ and $d = D$, DDPMs trained with only $n =  500$ samples outperform KDE trained with $n = 100{,}000$ samples.

These findings appear to deviate from our theoretical results, especially when $d = D$. After further reflection, our conjecture is as follows: diffusion models may not only adapt to the factorization structure but also to other forms of low-dimensional structure that have not yet been theoretically studied. Even when a density does not exhibit an explicit factorization structure, it may still possess hidden low-dimensional dependencies that diffusion models can effectively capture. Moreover, constructing a distribution completely devoid of any low-dimensional structure is inherently difficult, even in relatively low dimensions (e.g., $D = 5$).

We emphasize that these observations do not contradict classical nonparametric theory. In practice, estimating a five-dimensional function with $100{,}000$ samples can be more difficult than estimating a two- or three-dimensional function with only $500$ samples, due to the curse of dimensionality.

Furthermore, we conjecture that diffusion models also adapt to non-regular densities, for example, when the \Holder \ smoothness assumption is violated due to boundary singularities.
In our experiments, for $d = 2$ and $d = 5$, the copula density diverges near the boundary, where the variables $X_1,\ldots,X_{D}$ are highly correlated.
Such violations of regularity assumptions may lead to performance degradation for both BOS and KDE, since these methods rely on explicit pointwise estimation of the target density.

Although not shown in the figures, we also observe that DDPMs substantially outperform BOS and KDE when using alternative marginal distributions, including two-component Beta mixture marginals that vanish at the boundary. In this setting, the training of BOS becomes highly unstable and fails to yield performance improvements even for large sample sizes. In contrast, DDPMs consistently outperform BOS and KDE across these settings, providing empirical evidence of their robustness to a wider class of target distributions.

We also note that BOS becomes highly unstable even in moderately high dimensions (e.g., $D = 30$). This instability arises because BOS directly evaluates the density rather than its logarithm during training, and in high-dimensional settings, the absolute density values become extremely small, leading to numerical underflow. As a result, BOS is not applicable to higher-dimensional cases ($D > 30$).

Finally, across all experiments, the performance of DDPMs with fully connected networks is comparable to that of weight-sharing networks. This observation is consistent with the theoretical results reported in \citet{fan2025optimal}. Nevertheless, it is well known that sparse weight-sharing architectures such as U-Net play a crucial role in achieving state-of-the-art performance in practical applications. We therefore regard a deeper investigation into the theoretical advantages of sparse weight-sharing networks as an important direction for future work.

\section{Discussion}
\label{sec:disc}

We have demonstrated that an estimator constructed from the diffusion model is adaptive to the factorization structure and achieves the minimax optimal convergence rates.
In this section, we discuss some future directions related to our work.

Firstly, we believe that our analysis can be extended to other performance measures, such as the Wasserstein distance, following the approach of \citet{oko2023diffusion}. Also, our analysis can be extended to high-dimensional settings where the data dimension $D$ diverges as the sample size tends to infinity.
In this case, the convergence rate would depend on additional quantities such as $D$ and $|\cI|$.
An important future task would be to characterize upper bounds for $D$, which might depend on the structure of $p_0$, to guarantee statistical consistency.
Although this generalization is a natural extension for statisticians, the techniques required, such as sharp approximation theory, present significant challenges.

Secondly, while we assumed that $P_0$ possesses a Lebesgue density, this assumption might be eliminated.
For general probability distributions supported on the cube $[0,1]^D$, the minimax optimal rate with respect to the Wasserstein distance is $n^{-1/D}$ for $D > 2$.
If we restrict the class to distributions supported on a $d$-dimensional space (not necessarily a smooth manifold), the optimal rate improves to $n^{-1/d}$ \citep{weed2019sharp}.
It would be interesting to investigate whether an estimator constructed from the diffusion model achieves this optimal rate.
Further structural assumptions could be considered through conditional independence in directed and undirected graphs, which might replace the factorization assumption (\bF) for densities.

Finally, recall that a key component in constructing an optimal estimator is the use of weight-sharing neural networks to approximate functions defined through high-dimensional integrals of the form \eqref{eqapprox:gint}.
Recently, physics-informed neural networks (PINNs) have demonstrated remarkable success in modeling solutions to partial differential equations (PDEs) \citep{karniadakis2021physics, raissi2019physics}.
Notably, many solutions to PDEs, such as the heat equation and the Poisson equation, can be expressed as integrals of the form above \citep{courant2008methods}.
This suggests that weight-sharing networks could serve as a promising architecture for theoretical analysis of PINNs.


\acks{The authors are grateful to the Editor-in-Chief, Action Editor and three anonymous reviewers for their valuable comments which have led to substantial improvement in the paper. HK and MC were supported by Samsung Science and Technology Foundation under Project Number SSTF-BA2101-03.
DK was supported by the National Research Foundation of Korea (NRF) funded by the Korea government (MSIT) (RS-2023-00218231 and RS-2025-24683613). IO was supported by the National Research Foundation of Korea (NRF) funded by the Korea government (MSIT) (NRF-2022R1F1A1069695 and RS-2024-00411853).}



\bibliography{bib-short}       
\addtocontents{toc}{\protect\setcounter{tocdepth}{3}}

\newpage
\pagebreak

\appendix

\addtocontents{toc}{\protect\setcounter{tocdepth}{3}}
\renewcommand{\contentsname}{Appendix}
\tableofcontents

\section{Auxililary Lemmas}
This section provides auxiliary lemmas for proving the main theorems.

\subsection{Several Bounds Regarding $p_t(\bx)$}

In this subsection, we present several lemmas that bound $p_t(\bx), \nabla \log p_t(\bx),$ and the derivatives of $p_t(\bx)$.

\begin{lemma}[Upper and lower bounds for $p_t(\bx)$]
\label{secsc:ptbound}
Let $K, \tau_1 > 0$ be given and suppose the true density $p_0$ satisfies that $\tau_1 \leq p_0(\bx) \leq K$ for any $\bx \in [-1,1]^{D}$.  
Then, there exists a constant $C_{S,1} = C_{S,1}(D,K,\tau_1)$ such that
\bean
C_{S,1}^{-1 } \exp\left(-\frac{D   \{ (\|\bx \|_{\infty}-\mu_t) \vee 0 \}^2 }{\sigma_t^2}\right) \leq p_{t}(\bx) \leq C_{S,1} \exp\left(-\frac{   \{ (\|\bx \|_{\infty}-\mu_t) \vee 0 \}^2 }{2\sigma_t^2}\right)
\eean
for every $\bx \in \bbR^{D}$ and $t \geq 0$.
\end{lemma}

\begin{proof}
This is a re-statement of Lemma A.2 in \citet{oko2023diffusion}.
\end{proof}

\begin{lemma}[Boundedness of score function]
\label{secsc:scorebound}
Let $K, \tau_1 > 0$ be given and suppose the true density $p_0$ satisfies that $\tau_1 \leq p_0(\bx) \leq K$ for any $\bx \in [-1,1]^{D}$.  
Then, there exists a positive constant $C_{S,2} = C_{S,2}(D,K,\tau_1, \overline{\tau}, \underline{\tau} )$ such that
\bean
\left\| \nabla \log p_{t}(\bx) \right\|_{2} \leq \frac{C_{S,2}}{\sigma_t}  \left( \frac{\|\bx\|_{\infty}- \mu_t  }{\sigma_t} \vee 1 \right)
\eean
for every $\bx \in \bbR^{D}$ and $t \geq 0$.
\end{lemma}

\begin{proof}
This is a re-statement of Lemma A.3 in \citet{oko2023diffusion}.
\end{proof}

\begin{lemma}[Boundedness of derivatives]
\label{secsc:ptderivativebound}
Let $K > 0$ be given and suppose the true density $p_0$ satisfies that $ p_0(\bx) \leq K$ for any $\bx \in [-1,1]^{D}$.
For any $\bk \in  \bbZ_{\geq 0}^{D}$, there exists a positive constant $C_{S,3} = C_{S,3}(D,K,\bk, \overline{\tau}, \underline{\tau} )$ such that
\bean
\left\vert (\D^{\bk} p_t)(\bx) \right\vert \leq \frac{C_{S,3}}{\sigma_t^{k.}}
\eean
for every $\bx \in \bbR^{D}$ and $t \geq 0$.
\end{lemma}
\begin{proof}
This is a re-statement of Lemma A.3 in \citet{oko2023diffusion}, where $k. = \| \bk \|_{1}$.
\end{proof}

\subsection{Basic Approximation Results for Neural Networks}

In this subsection, we present fundamental approximation results for using ReLU networks to approximate elementary functions.
We also define the concatenation and parallelization in weight-sharing networks; see Lemma~\ref{secnn:sharing}.

\begin{lemma}[Concatenation]
\label{secnn:comp}
Let $K \in \bbN, \{d_1,\ldots,d_{K+1}\} \subseteq \bbN$ be given. Consider $L^{(k)} \in \bbN_{\geq 2}$, $ s^{(k)}, M^{(k)} > 0$ and $\bd^{(k)} = (d_1^{(k)},\ldots,d_{L^{(k)}}^{(k)})^{\top} \in \bbN^{L^{(k)}}$ with $d_1^{(k)} = d_{k}$ and $ d_{L^{(k)}}^{(k)} = d_{k+1}$ for $k \in [K]$. For any neural networks $f_1,\ldots,f_{K}$ with $f_k \in \cF_{\rm NN}(L^{(k)},\bd^{(k)},s^{(k)}, M^{(k)}), k \in [K]$, there exists a neural network $f \in \cF_{\rm NN}(L,\bd,s,M)$ with 
\bean
L = \sum_{k=1}^{K} L^{(k)}, \quad \| \bd \|_{\infty} \leq 2 \max_{k \in [K]} \|\bd^{(k)}\|_{\infty}, \quad s \leq 2 \sum_{k=1}^{K} s^{(k)}, \quad M = \max_{k \in [K]} M^{(k)}
\eean
such that $f(\bx) = (f_K \circ \cdots \circ f_{1})(\bx) $ for any $\bx \in \bbR^{d_1}$.
\end{lemma}
\begin{proof}
This is a re-statement of Remark 13 in \citet{nakada2020adaptive}.
\end{proof}

\begin{lemma}[Parallelization]
\label{secnn:par}
Let $K \in \bbN$ be given. Consider $L^{(k)} \in \bbN_{\geq 2}$, $ s^{(k)}, M^{(k)} > 0$, $\bd^{(k)} = (d_1^{(k)},\ldots,d_{L^{(k)}}^{(k)})^{\top} \in \bbN^{L^{(k)}}$ for $k \in [K]$.
For any neural networks $f_1,\ldots,f_{K}$ with 
\bean
f_k \in \cF_{\rm NN} (L^{(k)}, \bd^{(k)}, s^{(k)}, M^{(k)} ), \quad k \in [K],
\eean
there exists a neural network $f \in \cF_{\rm NN}(L,\bd,s,M)$ with 
\bean
&& L = \max_{k \in [K]} L^{(k)}, \quad \|\bd\|_{\infty} \leq 2 \sum_{k=1}^{K} \|\bd^{(k)}\|_{\infty},
\\
&& s \leq 2\sum_{k=1}^{K} \left(s^{(k)} + L d_{L^{(k)}}^{(k)} \right), \quad M \leq \left(\max_{k \in [K]} M^{(k)} \right) \vee 1
\eean
such that
\bean
f(\bx) = \left( f_1(\bx^{(1)}),\ldots, f_K(\bx^{(K)}) \right) \in \bbR^{d_{L^{(1)}}^{(1)} + \cdots + d_{L^{(K)}}^{(K)}}
\eean
for $\bx = (\bx^{(1)},\ldots,\bx^{(K)}) \in \bbR^{d_1^{(1)}+\cdots+d_{1}^{(K)} }$. If $L^{(1)} = \cdots = L^{(K)} = \widetilde L$ with $\widetilde L \in \bbN_{\geq 2}$, $(L,\bd,s,M)$ also satisfies
\bean
L = \widetilde L, \quad \|\bd\|_{\infty} \leq \sum_{k=1}^{K} \|\bd^{(k)}\|_{\infty}, \quad s \leq \sum_{k=1}^{(K)} s^{(k)}, \quad M \leq  \max_{k \in [K]} M^{(k)}.
\eean
\end{lemma}
\begin{proof}
This is a re-statement of Lemma F.3 in \citet{oko2023diffusion}.
\end{proof}

\begin{lemma}[Linear function]
\label{secnn:lin}
Let $W \in \bbR^{d_2 \times d_1}, \bb \in \bbR^{d_2}$ be given with $d_1,d_2 \in \bbN$. There exists a neural network $f_{\rm lin} \in \cF_{\rm NN}(L,\bd,s,M)$ with
\bean
L = 2, \quad \bd = (d_1,2d_2,d_2)^{\top}, \quad s = 2\|W\|_{0} + 2\|\bb\|_{0} + 2d_2, \quad M = \max\{ \|W\|_{\infty}, \|\bb\|_{\infty}, 1\}
\eean
such that $f_{\rm lin}(\bx) = W\bx + \bb$ for any $\bx \in \bbR^{d_1}$.
\end{lemma}
\begin{proof}
Note that $W\bx+ \bb = \rho(W\bx + \bb )-\rho(-W\bx - \bb)$ for any $\bx \in \bbR^{d_1}$. Let $f_{\rm lin}(\cdot) =  W_2 \rho ( W_1 \cdot + \bb_{1})  + \bb_2$ with
\bean
&&W_1 = \left(W^{\top}, -W^{\top} \right)^{\top} \in \bbR^{2d_2 \times d_1}, \quad \bb_{1} = \left(\bb^{\top} ,-\bb^{\top} \right)^{\top} \in \bbR^{2d_2},
\\
&& W_2 = \left(\bbI_{d_2}, -\bbI_{d_2} \right) \in \bbR^{d_2 \times 2d_2}, \quad \bb_{2} = \mathbf{0}_{d_2} \in \bbR^{d_2},
\eean
where $\bbI_{d_2}$ denotes the $d_2 \times d_2$ identity matrix. Then, the assertion is followed by a simple calculation.
\end{proof}

\begin{lemma}[Identity function]
\label{secnn:id}
For any $L \geq \bbN_{\geq 2}$ and $m \in \bbN$, there exists a neural network $f_{\rm id}^{(m,L)} \in \cF_{\rm NN}(L,\bd,s,M)$ with
\bean
\bd = (m,2m,\ldots,2m,m)^{\top}, \quad s = 2mL, \quad M = 1
\eean
such that $f_{\rm id}^{(m,L)}(\bx) = \bx$ for any $\bx \in \bbR^{m}$.
\end{lemma}
\begin{proof}
This is a re-statement of Lemma F.2 in \citet{oko2023diffusion}.
\end{proof}

\medskip

The following lemma provides the concatenation and parallelization of two neural networks, where only one network has shared weight.

\begin{lemma}[Concatenation and parallelization of weight-sharing networks]
\label{secnn:sharing}
Consider the class of weight-sharing networks $\cF_{\rm WSNN}(L,\bd,s,M,\cP_{\bfm})$ and vanilla feedforward neural networks $\cF_{\rm NN}( \widetilde L, \widetilde \bd, \widetilde s, \widetilde M)$.
For any neural networks $f \in \cF_{\rm WSNN}(L,\bd,s,M,\cP_{\bfm})$ and $\widetilde f \in \cF_{\rm NN}( \widetilde L, \widetilde \bd, \widetilde s, \widetilde M)$, there exists a neural network $f_{\rm cc} \in \cF_{\rm WSNN}( L_{\rm cc} , \bd_{\rm cc} , s_{\rm cc} , M_{\rm cc}, \bfm_{\rm cc}, \cP_{\rm cc})$ with
\bean
&& L_{\rm cc} = L + \widetilde L, \quad \|\bd_{\rm cc}\|_{\infty} \leq 2 (\|\bd\|_{\infty} \vee \| \widetilde \bd \|_{\infty}), \quad s_{\rm cc} \leq 2s + 2\widetilde s, \quad M_{\rm cc} = M \vee \widetilde M,
\eean
$\| \bfm_{\rm cc} \|_{\infty} = \| \bfm \|_{\infty}$ and the set of permutation matrices $\cP_{\rm cc}$ such that $f_{\rm cc}(\bx) = (\widetilde f \circ f)(\bx)$ for any $\bx \in \bbR^{d_1}$.
Also, there exists a neural network $f_{\rm pr} \in \cF_{\rm NN}( L_{\rm pr} , \bd_{\rm pr} , s_{\rm pr} , M_{\rm pr}, \bfm_{\rm pr}, \cP_{\rm pr})$ with
\bean
&& L_{\rm pr} = L \vee \widetilde L, \quad \|\bd_{\rm pr}\|_{\infty} \leq 2 \|\bd\|_{\infty} + 2 \| \widetilde \bd \|_{\infty},
\\
&& s_{\rm pr} \leq 2s + 2\widetilde s + 2(L \vee \widetilde L)(d_{L} + \widetilde d_{\widetilde L}), \quad M_{\rm pr} = \max(M, \widetilde M, 1),
\eean
$\| \bfm_{\rm pr} \|_{\infty} = \| \bfm \|_{\infty}$ and the set of permutation matrices $\cP_{\rm pr}$ such that
\bean
f_{\rm pr} (\bx) = \left( f(\bx_1), \widetilde f(\bx_2) \right) \in \bbR^{d_{L} + \widetilde d_{\widetilde L}}
\eean
for any $\bx =  (\bx_1,\bx_2) \in \bbR^{d_1 + \widetilde d_{1}}$.

\end{lemma}

\begin{proof}
The first assertion can be easily derived from Remark 13 of \citet{nakada2020adaptive} with 
\bean
\bfm_{\rm cc} = (\bfm, 1, \ldots, 1) \in \bbN^{L_{\rm cc} - 1} \text{\quad and \quad}
\cP_{\rm cc} = \cP \cup \left\{ \cQ_{l}, \cR_{l} \right\}_{ l \in \{L,\ldots, L_{\rm cc} - 1 \} },
\eean
where 
$\cQ_{l}$ and $\cR_{l}$ are the set of $d_{l} \times d_{l}$ and $d_{l+1} \times d_{l+1}$ identity matrix, respectively.

For the second part, let $\{ W_{l}, \bb_{l} \}_{l \in [L]}$ and
 $\{\widetilde W_{l}, \widetilde \bb_{l} \}_{l \in [\widetilde L ]}$ be the parameter matrices of $f$ and $\widetilde f$, respectively.
If $L = \widetilde L$,
let $\widetilde \bd_{\rm pr} = (d_{1}+\widetilde d_{1}, \ldots, d_{L+1}+\widetilde d_{L+1})$ and $ \bfm_{\rm pr} = \bfm$.
Also, for each $l \in [L]$, let $\cQ_{{\rm pr}, l}$ and $\cR_{{\rm pr}, l}$ be the set of permutation matrices of the form 
\bean
\begin{pmatrix}
Q_{l}^{(j)} & \  \mathbf{0}
\\
\mathbf{0} & \ \bbI_{\widetilde d_{l}}
\end{pmatrix} \text{\quad and \quad }
\begin{pmatrix}
R_{l}^{(j)} & \  \mathbf{0}
\\
\mathbf{0} & \ \bbI_{\widetilde d_{l+1}}
\end{pmatrix}
\eean
with $j \in [m_{l}]$, respectively.
Then, the assertion follows with
\bean
W_{{\rm pr}, l} = \begin{pmatrix}
W_{l} & \  \mathbf{0}
\\
\mathbf{0} & \ m_{l}^{-1} \widetilde W_{l}
\end{pmatrix}
\text{\quad and \quad}
\bb_{{\rm pr}, l} = \begin{pmatrix}
\bb_{l}
\\
m_{l}^{-1} \widetilde \bb_{l} 
\end{pmatrix}, \quad  l \in [L-1],
\eean
and $\cP_{\rm pr} = \left\{ \cQ_{{\rm pr},l}, \cR_{{\rm pr},l} \right\}_{ l \in [L-1] }$, where $\{W_{{\rm pr},l} , \bb_{{\rm pr},l} \}_{l \in [L]}$ are the parameter matrices of $f_{\rm pr}$.
If $L = \widetilde L + 1$, consider a neural network $\overline{\bff}$ with $L$-layer and parameter matrices $\{ \overline W_{l}, 
\overline \bb_{l} \}_{l \in [L]}$, where $\overline W_{l} = \widetilde W_{l}, \overline \bb_{l} = \widetilde \bb_{l}$ for $l \in [\widetilde L - 1]$, and
\bean
\overline{W}_{\widetilde L} = 
\begin{pmatrix}
\widetilde W_{\widetilde L} 
\\
-\widetilde W_{\widetilde L}
\end{pmatrix},
\quad
\overline{\bb}_{\widetilde L} = 
\begin{pmatrix}
\widetilde \bb_{\widetilde L} 
\\
-\widetilde \bb_{\widetilde L} 
\end{pmatrix},
\quad
\overline{W}_{\widetilde L + 1} = 
\begin{pmatrix}
\bbI_{\widetilde d_{\widetilde L + 1}} 
\\
-\bbI_{\widetilde d_{\widetilde L + 1}} 
\end{pmatrix},
\quad
\overline{\bb}_{\widetilde L + 1} = \mathbf{0}_{2 \widetilde d_{\widetilde L + 1}}.
\eean
We then apply the results for the case of parallelization between same layer network.
If $L > \widetilde L + 1$, consider a weight-sharing neural network $\overline \bff = \bff_{\rm id}^{(\widetilde d_{\widetilde L+1}, L - \widetilde L)} \circ \widetilde \bff$ with $L$-layer, where $\bff_{\rm id}^{(\widetilde d_{\widetilde L+1}, L - \widetilde L)}$ is the neural network in Lemma~\ref{secnn:id}.
We then apply the results for the case of parallelization between same layer network.

If $L = \widetilde L - 1$, consider a weight-sharing neural network $\overline{\bff}$ with $\widetilde L$-layer, parameter matrices $\{ \overline W_{l}, 
\overline \bb_{l} \}_{l \in [\widetilde L]}$ with $\overline W_{l} =  W_{l}, \overline \bb_{l} =  \bb_{l}$ for $l \in [L - 1]$, and
\bean
\overline{W}_{ L} = 
\begin{pmatrix}
W_{ L} 
\\
-W_{ L}
\end{pmatrix},
\quad
\overline{\bb}_{ L} = 
\begin{pmatrix}
\bb_{ L} 
\\
-\bb_{ L} 
\end{pmatrix},
\quad
\overline{W}_{L + 1} = 
\begin{pmatrix}
\bbI_{ d_{L + 1}} 
\\
-\bbI_{ d_{ L + 1}} 
\end{pmatrix},
\quad
\overline{\bb}_{L + 1} = \mathbf{0}_{2 d_{ L + 1}},
\eean
$m_{L}=1$ and the set of permutation matrices $\cP \cup \{ \cQ_{L}, \cR_{L} \}$, where $\cQ_{L}$ and $ \cR_{L} $ are the set of $d_{L} \times d_{L}$ and $2d_{L+1} \times 2d_{L+1}$  identity matrix, respectively.
We then apply the results for the case of parallelization between same layer network.
If $L < \widetilde L - 1$, consider a weight-sharing neural network $\overline \bff = \bff_{\rm id}^{( d_{ L+1}, \widetilde L - L)} \circ \bff$ with $\widetilde L$-layer and the set of permutation matrices $\cP \cup \{ \cQ_{l}, \cR_{l} \}_{L \leq l \leq \widetilde L - 1}$, where $\bff_{\rm id}^{( d_{L+1}, \widetilde L - L)}$ is the neural network in Lemma~\ref{secnn:id} and for each $l \in \{L,\ldots,\widetilde L - 1\}$, $\cQ_{l}$ and $ \cR_{l}$ are the set of $d_{l} \times d_{l}$ and $d_{l+1} \times d_{l+1}$ identity matrix, respectively.
We then apply the results for the case of parallelization between same layer network.

\end{proof}

\begin{lemma}[Multiplication]
\label{secnn:mult}
Let $m \geq 2, C \geq 1, 0 < \widetilde \epsilon \leq 1$ be given. For any $\epsilon > 0$, there exists a positive constant $C_{N,1}$ and a neural network $f_{\rm mult} \in \cF_{\rm NN}(L,\bd,s,M)$ with 
\bean
&& L \leq C_{N,1} \log m \{ \log (1/\epsilon) + m \log C \}, \quad \bd = (m,48m,\ldots,48m,1)^{\top},
\\
&& s \leq C_{N,1} m \{ \log(1/\epsilon) +  \log C \}, \quad M = C^{m}
\eean
such that
\bean
\left\vert f_{\rm mult}(\widetilde \bx) - \prod_{i=1}^{m} x_{i} \right\vert \leq \epsilon + mC^{m-1}\widetilde \epsilon, \quad \forall \bx \in [-C,C]^{m}, \widetilde \bx \in \bbR^{m} \ {\rm with} \ \|\bx-\widetilde \bx\|_{\infty} \leq \widetilde \epsilon,
\eean
$ \| f_{\rm mult} \|_{\infty} \leq C^{m}$ and $f_{\rm mult}(\widetilde \bx) = 0$ if $0 \in \{\widetilde x_1,\ldots,\widetilde x_{m}\}$.
\end{lemma}
\begin{proof}
This is a re-statement of Lemma F.6 in \citet{oko2023diffusion}.
\end{proof}

\begin{lemma}[Clipping function]
\label{secnn:clip}
Let $\underline{\bb} = (\underline{b}_1,\ldots,\underline{b}_m), \overline{\bb} = (\overline{b}_1,\ldots,\overline{b}_m) \in \bbR^{m}$ be given with $m \in \bbN$ and $\underline{b}_i \leq \overline{b}_i$ for all $i \in [m]$. Then, there exists a neural network 
\bean
f_{\rm clip}^{(\underline{\bb},\overline{\bb})} \in \cF_{\rm NN}(2,(m,2m,m)^{\top}, 7m, \| \underline{\bb} \|_{\infty} \vee \| \overline{\bb} \|_{\infty} )
\eean
such that 
\bean
f_{\rm clip}^{(\underline{\bb},\overline{\bb})} (\bx) = \left( \overline{b}_1 \wedge \{x_1 \vee \underline{b}_1 \}, \ldots, \overline{b}_m \wedge \{ x_m \vee \underline{b}_m \}  \right) \in \bbR^{m}
\eean
for $\bx = (x_1,\ldots,x_m) \in \bbR^{m}$.
\end{lemma}
\begin{proof}
This is a re-statement of Lemma F.4 in \citet{oko2023diffusion}.
\end{proof}


\begin{lemma}[Logarithm function]
\label{secnn:log}
For any $0 < \epsilon < 1/4$, there exists a positive constant $C_{N,2}$ and a neural network $f_{\log} \in \cF_{\rm NN}(L,\bd,s,M)$ with
\bean
&& L \leq C_{N,2} \{ \log(1/\epsilon)\}^{2} \log \log (1/\epsilon), \quad \|\bd\|_{\infty}\leq C_{N,2}  \{ \log(1/\epsilon)\}^{3}
\\
&& s  \leq C_{N,2}  \{ \log(1/\epsilon)\}^{5} \log \log (1/\epsilon), \quad M \leq \exp \left( 8 \{\log (1/\epsilon) \}^{2} \right)
\eean
such that
\bean
\left\vert \log x - f_{\log}(\widetilde x) \right\vert \leq \epsilon + \frac{\vert x - \widetilde x\vert}{\epsilon}
\eean
for $x \in [\epsilon,1/\epsilon]$ and $\widetilde x \in \bbR$.
\end{lemma}

\begin{proof}
Let $0 < \epsilon < 1/4, \delta = 1/8$ and $D_1 = \lfloor \frac{2 \log (1/\epsilon)}{\log (1+\delta)} -1 \rfloor +1$. Then, $[\epsilon, 1/\epsilon] \subseteq \bigcup_{i=1}^{D_1} [\underline{T}_i,\overline{T}_i] $, where $\underline{T}_i = (1+\delta)^{i-1}\epsilon$ and $\overline{T}_i = (1+\delta)^{i+1}\epsilon$ for $i \in [D_1]$. Let $D_2 = \lfloor \frac{ \log (1/\epsilon)}{\log 2} \rfloor + 3 $. For any $i \in [D_1]$ and $ x \in [\underline{T}_i,\overline{T}_i]$, Taylor's theorem yields that
\bean
\log x = P_{i}(x) + \frac{(-1)^{D_2-1} (x-\underline{T}_i)^{D_2} }{D_2\{\underline{T}_i + \xi (\overline{T}_i - \underline{T}_i)\}^{D_2}}
\eean
for a suitable $\xi \in [0,1]$, where 
\bean
P_{i}(x) = \log \underline{T}_i + \frac{x-\underline{T}_i}{\underline{T}_i} +\sum_{k=2}^{D_2-1} \frac{(-1)^{k-1}   (x -\underline{T}_i)^{k}  }{k\underline{T}_i^{k}}.
\eean
Since $x-\underline{T}_i \leq \overline{T}_i - \underline{T}_i =  \underline{T}_i (\delta^2 + 2\delta)$ and $\underline{T}_i \leq \underline{T}_i + \xi (\overline{T}_i - \underline{T}_i)$, it follows that
\be
\left\vert \log x - P_{i}(x) \right\vert \leq \frac{1}{D_2}\left( \frac{\overline{T}_i-\underline{T}_i}{ \underline{T}_i } \right)^{D_2} =  \frac{\left(\delta^2+ 2\delta \right)^{D_2}}{D_2} \leq \frac{2^{-D_2}}{D_2} \leq \frac{\epsilon}{D_2}, \quad i \in [D_1] \label{eq:nnlog}
\ee
for $x \in [\underline{T}_i,\overline{T}_i]$.
Let $N_1$ be a constant in Lemma~\ref{secnn:mult}. For $k \geq 2$, there exists a neural network $f_{\rm mult}^{(k)} \in \cF_{\rm NN}(L_{\rm mult}^{(k)}, \bd_{\rm mult}^{(k)}, s_{\rm mult}^{(k)}, M_{\rm mult}^{(k)})$ with 
\be \begin{split}
& L_{\rm mult}^{(k)} \leq N_1  (k+D_2) \log k \{ \log (1/\epsilon) + \log D_2 \}, \quad \bd_{\rm mult}^{(k)} = (k,48k,\ldots,48k,1)^{\top}, \quad
\\
& s_{\rm mult}^{(k)} \leq N_1 k(k+D_2) \{ \log(1/\epsilon) + \log D_2 \}, \quad M_{\rm mult}^{(k)} = \epsilon^{-k} \label{eq:nnmultcomp}
\end{split} \ee
such that $\vert f_{\rm mult}^{(k)}(x_1,\ldots,x_k) - \prod_{i=1}^{k} x_i \vert \leq \epsilon^{D_2}/D_2$ for any $x_1,\ldots,x_k \in [-\epsilon^{-1},\epsilon^{-1}]$.
For any $k \geq 1$ and $i \in [D_1]$, Lemma~\ref{secnn:lin} implies that there exists a neural network $f_{\rm lin}^{(i,k)} \in \cF_{\rm NN}(2,(1,2k,k)^{\top}, 6k, \underline{T}_i)$ such that $f_{\rm lin}^{(i,k)}(x) = (x-\underline{T}_i,\ldots,x-\underline{T}_i)^{\top} \in \bbR^{k}$ for any $x \in \bbR$. 
Combining Lemma~\ref{secnn:comp} with the last display, it follows that $f_{\rm pow}^{(i,k)} = f_{\rm mult}^{(k)} \circ f_{\rm lin}^{(i,k)} \in \cF_{\rm NN}(L_{\rm pow}^{(i,k)},\bd_{\rm pow}^{(i,k)},s_{\rm pow}^{(i,k)}, M_{\rm pow}^{(i,k)})$ for $i \in [D_{1}], k \geq 2$ with
\bean
L_{\rm pow}^{(i,k)} =  L_{\rm mult}^{(k)} +2, \quad \| \bd_{\rm pow}^{(i,k)}\|_{\infty} \leq 96k, \quad s_{\rm pow}^{(i,k)} = 2s_{\rm mult}^{(k)} + 12k, \quad M_{\rm pow}^{(i,k)} = \underline{T}_i \vee M_{\rm mult}^{(k)}
\eean
and
\bean
\left\vert f_{\rm pow}^{(i,k)}(x) - (x-\underline{T}_i)^{k} \right\vert \leq \frac{ \epsilon^{D_2}}{D_2},  
\eean
for $x \in [\epsilon,\epsilon^{-1}]$.
Consider functions $f_1,\ldots,f_{D_1} : \bbR \to \bbR$ such that
\bean
f_i(\cdot) = \log \underline{T}_i + \frac{f_{\rm lin}^{(i,1)}(\cdot)}{\underline{T}_i} +\sum_{k=2}^{D_2-1} \frac{(-1)^{k-1}   f_{\rm pow}^{(i,k)}(\cdot)  }{k\underline{T}_i^{k}}, \quad i \in [D_{1}].
\eean
Since $f_i - \log \underline{T}_i$ is a linear combination of $ f_{\rm lin}^{(i,1)}, f_{\rm pow}^{(i,2)}, \ldots,f_{\rm pow}^{(i,D_2-1)}$ for $ i \in [D_{1}]$, Lemma~\ref{secnn:comp}, Lemma~\ref{secnn:par} and Lemma~\ref{secnn:lin} implies that $f_i \in \cF_{\rm NN}(L^{(i)},\bd^{(i)},s^{(i)},M^{(i)})$ with
\be \begin{split}
& L^{(i)}  \leq L_{\rm pow}^{(i,D_2-1)} + 2 \leq D_3 D_2 \log D_2  \{ \log (1/\epsilon) + \log D_2 \}, 
\\
& \| \bd^{(i)} \|_{\infty}  \leq 2\max \left(2\sum_{k=2}^{D_2-1} \| \bd_{\rm pow}^{(i,k)}\|_{\infty} + 4 , D_2-1 \right) \leq D_3 D_2^2,
\\
& s^{(i)}  \leq 2\left\{ \sum_{k=2}^{D_2-1} \left( s_{\rm pow}^{(i,k)} +L_{\rm pow}^{(i,D_2-1)} + 2 \right) + L_{\rm pow}^{(i,D_2-1)} + 8 \right\} + 2D_2 + 2
\\
&  \leq D_3 D_2^{3}\log D_2  \{ \log (1/\epsilon) + \log D_2 \},
\\
& M^{(i)}  \leq \left( \max_{k \in [D_2-1]}  M_{\rm pow}^{(i,k)} \right) \vee \underline{T}_{i}^{-D_2+1} \leq \epsilon^{-D_2} \label{eq:nnlogcomp}
\end{split} \ee
for a large enough constant $D_3 = D_3(N_1)$.
Then,
\bean
&& \left\vert P_i(x) - f_i(x) \right\vert \leq \sum_{k=2}^{D_2-1} \frac{ \vert f_{\rm pow}^{(i,k)}(x) - (x-\underline{T}_i)^{k} \vert  }{k\underline{T}_i^{k}} 
\\
&& \leq \frac{D_2-2}{2 \epsilon^{D_2-1}} \max_{2 \leq k \leq D_2-1}\left\vert f_{\rm pow}^{(i,k)}(x) - (x-\underline{T}_i)^{k} \right\vert \leq \frac{\epsilon}{2}, \quad i \in [D_{1}]
\eean
for $x \in [\epsilon,\epsilon^{-1}]$, where the first inequality holds because $\underline{T}_i \geq \epsilon$.
Combining (\ref{eq:nnlog}) with the last display, we have
\be
\left\vert \log x -f_i(x) \right\vert \leq \left(\frac{1}{2} + \frac{1}{D_2} \right)\epsilon, \quad i \in [D_1] \label{eq:nnlog2}
\ee
fo $x \in [\underline{T}_i,\overline{T}_i]$.
Consider functions $f_{\rm swit}^{(1)},\ldots,f_{\rm swit}^{(D_1)} : \bbR \to [0,1]$ such that
\bean
&& f_{\rm swit}^{(1)}(\cdot) = \frac{1}{\overline{T}_{1}-\underline{T}_{2}} \rho \left(-f_{\rm clip}^{(\underline{T}_{2},\overline{T}_{1})} (\cdot) + \overline{T}_{1} \right),
\\
&& f_{\rm swit}^{(i)}(\cdot) = \frac{1}{\overline{T}_{i-1}-\underline{T}_{i}} \rho \left(f_{\rm clip}^{(\underline{T}_
{i},\overline{T}_{i-1})} (\cdot) - \underline{T}_{i} \right) - \frac{1}{\overline{T}_{i}-\underline{T}_{i+1}} \rho \left(f_{\rm clip}^{(\underline{T}_{i+1},\overline{T}_{i})} (\cdot) - \underline{T}_{i+1} \right),
\\
&& f_{\rm swit}^{(D_1)}(\cdot) = \frac{1}{\overline{T}_{D_1-1}-\underline{T}_{D_1}} \rho \left(f_{\rm clip}^{(\underline{T}_
{D_1},\overline{T}_{D_1-1})} (\cdot) - \underline{T}_{D_1} \right),
\quad 2 \leq i \leq D_1-1,
\eean
where $f_{\rm clip}^{(\underline{T}_i,\overline{T}_{i-1})} \in \cF_{\rm NN}(2, (1,2,1)^{\top}, 7, 8 \epsilon^{-1})$ denotes the neural network in Lemma~\ref{secnn:clip}.
Note that $\sum_{i=1}^{D_1} f_{\rm swit}^{(i)}(x) = 1$ for $x \in \bbR$ and $f_{\rm swit}^{(i)}(x) = 0$ for $x \in \bigcup_{j =1}^{D_1} [\underline{T}_j,\overline{T}_j] \backslash [\underline{T}_{i},\overline{T}_i], i \in [D_1]$.
Consider a function $f: \bbR \to \bbR$ such that $f(\cdot) = \sum_{i=1}^{D_1} f_{\rm mult}^{(2)}(f_{\rm swit}^{(i)}(\cdot) , f_{i}(\cdot))$. Since $\vert f_{\rm mult}^{(2)} (x_1,x_2) - x_1x_2\vert \leq \epsilon^{D_2} / D_2$ for any $x_1,x_2 \in [-\epsilon^{-1},\epsilon^{-1}]$, we have
\bean
&& \left\vert \log x - f(x) \right\vert \leq \left\vert \log x - \sum_{i=1}^{D_1} f_{\rm swit}^{(i)}(x) f_i(x) \right\vert + \frac{D_1 \epsilon^{D_2}}{D_2}   
\\
&& = \left\vert  \sum_{i=1}^{D_1} f_{\rm swit}^{(i)}(x) \left\{ \log x - f_{i}(x) \right\} \right\vert + \frac{D_1 \epsilon^{D_2}}{D_2}
\\
&& \leq \sum_{i=1}^{D_1} f_{\rm swit}^{(i)}(x) \left\vert \log x - f_{i}(x) \right\vert + \frac{D_1 \epsilon^{D_2}}{D_2} \leq \left(\frac{1}{2} + \frac{1}{D_2} \right)\epsilon\sum_{i=1}^{D_1} f_{\rm swit}^{(i)}(x) + \frac{D_1 \epsilon^{D_2}}{D_2}
\\
&& =\left(\frac{1}{2} + \frac{1+D_1 \epsilon^{D_2-1}}{D_2} \right) \epsilon  \leq \epsilon
\eean
for $x \in \bigcup_{i =1}^{D_1} [\underline{T}_i,\overline{T}_i]$, where the second inequality holds by (\ref{eq:nnlog2}).
Combining (\ref{eq:nnmultcomp}) and (\ref{eq:nnlogcomp}) with Lemma~\ref{secnn:par} and Lemma~\ref{secnn:comp}, we have $ f_{\rm mult}^{(2)}(f_{\rm swit}^{(i)}(\cdot) , f_{i}(\cdot)) \in \cF_{\rm NN}(\widetilde L^{(i)},\widetilde \bd^{(i)},\widetilde s^{(i)}, \widetilde M^{(i)})$ for $i \in [D_{1}]$ with
\bean
&& \widetilde L^{(i)} \leq (L^{(i)} \vee 2) + L_{\rm mult}^{(2)} \leq D_4 D_2 \log D_2  \{ \log (1/\epsilon) + \log D_2 \} 
\\
&& \| \widetilde \bd^{(i)} \|_{\infty} \leq 2\max \left( 2\|\bd^{(i)}\|_{\infty} + 4 ,  \|\bd_{\rm mult}^{(2)}\|_{\infty}   \right) \leq D_4 D_2^2
\\
&& \widetilde s^{(i)} \leq 4 s^{(i)} + 4 (L^{(i)} \vee 2) + 2 s_{\rm mult}^{(2)} + 28  \leq D_4 D_2^{3}\log D_2  \{ \log (1/\epsilon) + \log D_2 \} 
\\
&& \widetilde M^{(i)} \leq \max\left(8\epsilon^{-1}, M^{(i)}, M_{\rm mult}^{(2)}, 1 \right) \leq \epsilon^{-D_2},
\eean
where $D_4 = D_4(D_3)$ is a large enough constant.
Let $f_{\rm clip}^{(\epsilon,\epsilon^{-1})} \in \cF_{\rm NN}(2,(1,2,1)^{\top}, 7,\epsilon^{-1})$ be the nueral network in Lemma~\ref{secnn:clip}.
Since $f$ is a linear combination of $f_{\rm mult}^{(2)}(f_{\rm swit}^{(i)}(\cdot) , f_{i}(\cdot))$ for each $i \in [D_1]$, Lemma~\ref{secnn:comp}, Lemma~\ref{secnn:par} and Lemma~\ref{secnn:lin} implies that $f \circ f_{\rm clip}^{(\epsilon,\epsilon^{-1})} \in \cF_{\rm NN}(L,\bd,s,M)$ with
\bean
&& L \leq \max_{i \in [D_1]} \widetilde L^{(i)} + 4 \leq D_5 \{ \log(1/\epsilon)\}^{2} \log \log (1/\epsilon)
\\
&& \|\bd\|_{\infty} \leq 2\max\left( 2\sum_{i=1}^{D_1} \| \widetilde \bd^{(i)}\|_{\infty}, 4 D_1 \right) \leq D_5  \{ \log(1/\epsilon)\}^{3}
\\
&& s \leq 2\sum_{i=1}^{D_1}\left(\widetilde s^{(i)} +  \max_{j \in [D_1]} \widetilde L^{(j)} \right) + 4D_1 + 18 \leq D_5  \{ \log(1/\epsilon)\}^{5} \log \log (1/\epsilon)
\\
&& M \leq \max_{i \in [D_1]} \widetilde M^{(i)} \vee \epsilon^{-1} \leq \exp \left( 8 \{\log (1/\epsilon) \}^{2} \right)
\eean
for large enough constant $D_5 = D_5(D_4)$.
Note that $\vert (f \circ f_{\rm clip}^{(\epsilon,\epsilon^{-1})})(\widetilde x) - \log x \vert \leq \vert (f \circ f_{\rm clip}^{(\epsilon,\epsilon^{-1})})(\widetilde x) - \log ( \epsilon^{-1} \wedge \{\widetilde x \vee \epsilon\}) \vert + \vert \log ( \epsilon^{-1} \wedge \{\widetilde x \vee \epsilon\}) - \log x \vert \leq \epsilon + \epsilon^{-1} \vert x-\widetilde x\vert$ for any $x \in [\epsilon,\epsilon^{-1}]$ and $\widetilde x \in \bbR$.
Then, the assertion follows by re-defining the constant.
\end{proof}

\begin{lemma}[Negative exponential function]
\label{secnn:exp}
For any $0 < \epsilon < 2^{-4e+2}$, there exists a positive constant $C_{N,3}$ and a neural network $f_{\exp} \in \cF_{\rm NN}(L,\bd,s,M)$ with
\bean
&& L \leq C_{N,3} \log(1/\epsilon) \log \log (1/\epsilon), \quad \|\bd\|_{\infty} \leq C_{N,3}  \{ \log(1/\epsilon)\}^{3}
\\
&& s \leq C_{N,3}  \{ \log(1/\epsilon)\}^{4}, \quad M \leq C_{N,3} \epsilon^{-1}
\eean
such that
\bean
\left\vert e^{-x} - f_{\exp}(\widetilde x) \right\vert \leq \epsilon + \vert x-\widetilde x\vert
\eean
for any $x \geq 0$ and $\widetilde x \in \bbR$.
\end{lemma}

\begin{proof}
Let $0 < \epsilon < 2^{-4e+2}, D_1 = \lfloor \log (4/\epsilon) \rfloor +1, D_2 = \lfloor \log (4/\epsilon) / \log 2  \rfloor +1$ and $\underline{T}_i = i-1, \overline{T}_i = i+1$ for $i \in [D_1]$. Then, Taylor's theorem yields that for any $i \in [D_1]$ and $x \in [\underline{T}_i, \overline{T}_i]$,
\bean
e^{-x} = e^{-\underline{T}_i} e^{-(x-\underline{T}_i)} = e^{-\underline{T}_i} \left\{ P_i(x)
 + \frac{(-1)^{D_2} e^{-\xi(x-\underline{T}_i)} (x-\underline{T}_i)^{D_2} }{D_2 !}
\right\}
\eean
for a suitable $\xi \in [0,1]$, where
\bean
P_i(x) = 1 - (x-\underline{T}_i) + \sum_{k=2}^{D_2-1} \frac{(-1)^{k} (x-\underline{T}_i)^{k}} {k!}.
\eean
Since $0 \leq x-\underline{T}_i \leq 2$ and $\underline{T}_i \geq 0$, it follows that
\be
\left\vert e^{-x} - e^{-\underline{T}_i} P_i(x) \right\vert  \leq \frac{2^{D_2}}{D_2 !} \leq \left(\frac{2e}{D_2} \right)^{D_2} \leq \left(\frac{1}{2} \right)^{D_2} \leq \frac{\epsilon}{4}, \label{eq:expremainder}
\ee
where the second inequality holds because $k! \geq k^{k} e^{-k}$ for any $k \in \bbN$.
Let $N_1$ be a constant in Lemma~\ref{secnn:mult}. For $k \geq 2$, there exists a neural network $f_{\rm mult}^{(k)} \in \cF_{\rm NN}(L_{\rm mult}^{(k)}, \bd_{\rm mult}^{(k)}, s_{\rm mult}^{(k)}, M_{\rm mult}^{(k)})$ with 
\be \begin{split}
& L_{\rm mult}^{(k)} \leq N_1  \log k \{ \log (4D_2/\epsilon^2) +  k \log 2 \}, \quad \bd_{\rm mult}^{(k)} = (k,48k,\ldots,48k,1)^{\top}, \quad
\\
& s_{\rm mult}^{(k)} \leq N_1 k \{ \log(4D_2/\epsilon^2) + \log 2 \}, \quad M_{\rm mult}^{(k)} = 2^{k} \label{eq:nnmultexpcomp}
\end{split} \ee
such that $\vert f_{\rm mult}^{(k)}(x_1,\ldots,x_k) - \prod_{i=1}^{k} x_i \vert \leq \epsilon^2/(4D_2)$ for any $x_1,\ldots,x_k \in [-2,2]$.
For any $k \geq 1$ and $i \in [D_1]$, Lemma~\ref{secnn:lin} implies that there exists a neural network 
\bean
f_{\rm lin}^{(i,k)} \in \cF_{\rm NN}(2,(1,2k,k)^{\top}, 6k, \underline{T}_i)
\eean
such that $f_{\rm lin}^{(i,k)}(x) = (x-\underline{T}_i,\ldots,x-\underline{T}_i)^{\top} \in \bbR^{k}$ for any $x \in \bbR$. 
Combining Lemma~\ref{secnn:comp} with the last display, it follows that $f_{\rm pow}^{(i,k)} = f_{\rm mult}^{(k)} \circ f_{\rm lin}^{(i,k)} \in \cF_{\rm NN}(L_{\rm pow}^{(i,k)},\bd_{\rm pow}^{(i,k)},s_{\rm pow}^{(i,k)}, M_{\rm pow}^{(i,k)})$ for $i \in [D_{1}], k \geq 2$ with
\bean
L_{\rm pow}^{(i,k)} =  L_{\rm mult}^{(k)} +2, \quad \| \bd_{\rm pow}^{(i,k)}\|_{\infty} \leq 96k, \quad s_{\rm pow}^{(i,k)} = 2s_{\rm mult}^{(k)} + 12k, \quad M_{\rm pow}^{(i,k)} = \underline{T}_i \vee M_{\rm mult}^{(k)}
\eean
and
\bean
\left\vert f_{\rm pow}^{(i,k)}(x) - (x-\underline{T}_i)^{k} \right\vert \leq \frac{ \epsilon^2}{4D_2}
\eean
for $x \in [\underline{T}_i,\overline{T}_i]$.
Consider functions $f_1,\ldots,f_{D_1} : \bbR \to \bbR$ such that
\bean
f_i(\cdot) = 1 - f_{\rm lin}^{(i,1)}(\cdot) +\sum_{k=2}^{D_2-1} \frac{(-1)^{k}   f_{\rm pow}^{(i,k)}(\cdot)  }{k!}, \quad i \in [D_1].
\eean
Since $f_i - 1$ is a linear combination of $ f_{\rm lin}^{(i,1)}, f_{\rm pow}^{(i,2)}, \ldots,f_{\rm pow}^{(i,D_2-1)}$, Lemma~\ref{secnn:comp}, Lemma~\ref{secnn:par} and Lemma~\ref{secnn:lin} implies that $f_i \in \cF_{\rm NN}(L^{(i)},\bd^{(i)},s^{(i)},M^{(i)})$ with
\be \begin{split}
& L^{(i)} \leq L_{\rm pow}^{(i,D_2-1)} + 2 \leq D_3 \log D_2  \{ \log (1/\epsilon) + D_2 \} 
\\
& \| \bd^{(i)} \|_{\infty} \leq 2\max \left(2\sum_{k=2}^{D_2-1} \| \bd_{\rm pow}^{(i,k)}\|_{\infty} + 4 , D_2-1 \right) \leq D_3 D_2^2
\\
& s^{(i)} \leq 2\left\{ \sum_{k=2}^{D_2-1} \left( s_{\rm pow}^{(i,k)} +L_{\rm pow}^{(i,D_2-1)} + 2 \right) + L_{\rm pow}^{(i,D_2-1)} + 8 \right\} + 2D_2 + 2
\\
& \leq D_3 D_2^{2} \{ \log (1/\epsilon) + D_2 \} 
\\
& M^{(i)} \leq \left( \max_{k \in [D_2-1]}  M_{\rm pow}^{(i,k)} \right) \vee 1 \leq (D_1+1) \vee 2^{D_2-1} \label{eq:nnexpcomp}
\end{split} \ee
for a large enough constant $D_3 = D_3(N_1)$.
Then, 
\bean
&& \left\vert P_i(x) - f_{i}(x) \right\vert \leq \sum_{k=2}^{D_2-1} \frac{\vert f_{\rm pow}^{(i,k)}(x) - (x-\underline{T}_i)^{k}  \vert}{k!}
\\
&& \leq D_2 \max_{2 \leq k \leq D_2-1}\left\vert f_{\rm pow}^{(i,k)}(x) - (x-\underline{T}_i)^{k} \right\vert \leq \frac{\epsilon^2}{4}, \quad i \in [D_{1}]
\eean
for $x \in [\underline{T}_i,\overline{T}_i]$.
Combining (\ref{eq:expremainder}) with the last display, we have
\be
\left\vert e^{-x} - e^{-\underline{T}_i} f_i(x) \right\vert \leq \left\vert e^{-x} - e^{-\underline{T}_i} P_i(x) \right\vert + \left\vert P_i(x) - f_{i}(x) \right\vert \leq \frac{\epsilon}{4}+ \frac{\epsilon^2}{4} \leq \frac{\epsilon}{2}, \quad i \in [D_1] \label{eq:nnexp2}
\ee
for $x \in [\underline{T}_i,\overline{T}_i]$, where the first inequality holds because $e^{-\underline{T}_i} \leq 1$.
Consider functions $f_{\rm swit}^{(1)},\ldots,f_{\rm swit}^{(D_1+1)} : \bbR \to [0,1]$ such that
\bean
&& f_{\rm swit}^{(1)}(\cdot) = \frac{1}{\overline{T}_{1}-\underline{T}_{2}} \rho \left(-f_{\rm clip}^{(\underline{T}_{2},\overline{T}_{1})} (\cdot) + \overline{T}_{1} \right),
\\
&& f_{\rm swit}^{(i)}(\cdot) = \frac{1}{\overline{T}_{i-1}-\underline{T}_{i}} \rho \left(f_{\rm clip}^{(\underline{T}_
{i},\overline{T}_{i-1})} (\cdot) - \underline{T}_{i} \right) - \frac{1}{\overline{T}_{i}-\underline{T}_{i+1}} \rho \left(f_{\rm clip}^{(\underline{T}_{i+1},\overline{T}_{i})} (\cdot) - \underline{T}_{i+1} \right),
\\
&& 2 \leq i \leq D_1,
\\
&& f_{\rm swit}^{(D_1+1)}(\cdot) = \frac{1}{\overline{T}_{D_1}-\underline{T}_{D_1+1}} \rho \left(f_{\rm clip}^{(\underline{T}_
{D_1+1},\overline{T}_{D_1})} (\cdot) - \underline{T}_{D_1+1} \right),
\eean
where $\underline{T}_{D_{1}+1} = D_1$ and $f_{\rm clip}^{(\underline{T}_i,\overline{T}_{i-1})} \in \cF_{\rm NN}(2, (1,2,1)^{\top}, 7, D_1)$ denotes the neural network in Lemma~\ref{secnn:clip}.
Note that $\sum_{i=1}^{D_1+1} f_{\rm swit}^{(i)}(x) = 1$ for $x \in \bbR$, and $f_{\rm swit}^{(i)}(x) = 0$ for $x \in [0,\underline{T}_i] \cup [\overline{T}_i, \infty), i \in [D_1]$ and $f_{\rm swit}^{(D_1+1)}(x) = 0$ for $x \leq \underline{T}_{D_1+1}$.
Consider a function $f: \bbR \to \bbR$ such that $f(\cdot) = \sum_{i=1}^{D_1}  e^{-\underline{T}_i}  f_{\rm mult}^{(2)}(f_{\rm swit}^{(i)}(\cdot) , f_{i}(\cdot))$.
Since $\vert f_{\rm mult}^{(2)} (x_1,x_2) - x_1x_2\vert \leq \epsilon^2/(4D_2)$ for any $x_1,x_2 \in [-2,2]$, we have
\bean
&& \left\vert e^{-x} - f(x) \right\vert \leq \left\vert e^{-x} - \sum_{i=1}^{D_1} e^{-\underline{T}_i} f_{\rm swit}^{(i)}(x) f_i(x) \right\vert + \frac{D_1\epsilon^{2}}{4D_2}  
\\
&& = \left\vert  \sum_{i=1}^{D_1} f_{\rm swit}^{(i)}(x) \left\{ e^{-x} -  e^{-\underline{T}_i}f_{i}(x) \right\} + f_{\rm swit}^{(D_1+1)}(x)e^{-x} \right\vert + \frac{D_1\epsilon^{2}}{4D_2},
\\
&& \leq \sum_{i=1}^{D_1} f_{\rm swit}^{(i)}(x) \left\vert e^{-x} - e^{-\underline{T}_i}f_{i}(x) \right\vert + \frac{\epsilon}{4} + \frac{D_1\epsilon^{2}}{4D_2} 
\\
&& \leq \frac{\epsilon}{2} \sum_{i=1}^{D_1} f_{\rm swit}^{(i)}(x) + \frac{\epsilon}{4} + \frac{D_1\epsilon^{2}}{4D_2}
\leq \left(\frac{3}{4} + \frac{(D_1+1)\epsilon}{4D_2} \right) \epsilon  \leq \epsilon
\eean
for $x \in [0,\infty)$, where the second inequality holds because  $\vert f_{\rm swit}^{(D_1+1)}(x)e^{-x} \vert \leq e^{-x} \leq \epsilon/4$ for $x \geq \underline{T}_{D_1+1}$ and the third inequality holds by  (\ref{eq:nnexp2}).
Combining (\ref{eq:nnmultexpcomp}) and (\ref{eq:nnexpcomp}) with Lemma~\ref{secnn:par} and Lemma~\ref{secnn:comp}, we have $  f_{\rm mult}^{(2)}(f_{\rm swit}^{(i)}(\cdot) , f_{i}(\cdot)) \in \cF_{\rm NN}(\widetilde L^{(i)},\widetilde \bd^{(i)},\widetilde s^{(i)}, \widetilde M^{(i)})$ for $i \in [D_{1}]$ with
\bean
&& \widetilde L^{(i)} \leq (L^{(i)} \vee 2) + L_{\rm mult}^{(2)} \leq D_4 \log D_2  \{ \log (1/\epsilon) + \log D_2 \} 
\\
&& \| \widetilde \bd^{(i)} \|_{\infty} \leq 2\max \left( 2\|\bd^{(i)}\|_{\infty} + 4 ,  \|\bd_{\rm mult}^{(2)}\|_{\infty}   \right) \leq D_4 D_2^2
\\
&& \widetilde s^{(i)} \leq 4 s^{(i)} + 4 (L^{(i)} \vee 2) + 2 s_{\rm mult}^{(2)} + 28  \leq D_4 D_2^{2}  \{ \log (1/\epsilon) +  D_2 \} 
\\
&& \widetilde M^{(i)} \leq \max\left(D_1, M^{(i)}, M_{\rm mult}^{(2)}, 1 \right),
\eean
where $D_4 = D_4(D_3)$ is a large enough constant.
Since $f$ is a linear combination of $f_{\rm mult}^{(2)}(f_{\rm swit}^{(i)}(\cdot) , f_{i}(\cdot))$ for each $i \in [D_1]$, Lemma~\ref{secnn:comp}, Lemma~\ref{secnn:par} and Lemma~\ref{secnn:lin} implies that $f \circ \rho_{1} \in \cF_{\rm NN}(L,\bd,s,M)$ with
\bean
&& L \leq \max_{i \in [D_1]} \widetilde L^{(i)} + 3 \leq D_5 \log(1/\epsilon) \log \log (1/\epsilon)
\\
&& \|\bd\|_{\infty} \leq 2\max\left( 2\sum_{i=1}^{D_1} \| \widetilde \bd^{(i)}\|_{\infty}, 4 D_1 \right) \leq D_5  \{ \log(1/\epsilon)\}^{3}
\\
&& s \leq 2\sum_{i=1}^{D_1}\left(\widetilde s^{(i)} +  \max_{j \in [D_1]} \widetilde L^{(j)} \right) + 4D_1 + 4 \leq D_5  \{ \log(1/\epsilon)\}^{4}
\\
&& M \leq \max_{i \in [D_1]} \left(\widetilde M^{(i)} \vee e^{-\underline{T}_i} \right) \leq D_5 \epsilon^{-1}
\eean
for large enough constant $D_5 = D_5(D_4)$.
Note that $\vert(f \circ \rho_{1})(\widetilde x) - e^{-x} \vert \leq \vert(f \circ \rho_{1})(\widetilde x) - e^{-(\widetilde x \vee 0)} \vert \ + \vert e^{-(\widetilde x \vee 0)} - e^{-x} \vert \leq \epsilon + \vert x-\widetilde x\vert$ for any $x \geq 0$ and $\widetilde x \in \bbR$.
Then, the assertion follows by re-defining the constant.
\end{proof}

\begin{lemma}[$\mu_t$ and $\sigma_t$]
\label{secnn:mtst}
For any $0 < \epsilon < 1/2$, there exists a positive constant $C_{N,4} = C_{N,4}(\underline{\tau},\overline{\tau})$ and neural networks $f_{\mu} \in \cF_{\rm NN}(L_{\mu},\bd_{\mu},s_{\mu},M_{\mu}), f_{\sigma} \in \cF_{\rm NN}(L_{\sigma},\bd_{\sigma},s_{\sigma},M_{\sigma})$ with
\bean
&& L_{\mu}, L_{\sigma} \leq C_{N,4} \{ \log(1/\epsilon)\}^{2}, \quad \|\bd_{\mu} \|_{\infty},  \| \bd_{\sigma}\|_{\infty} \leq C_{N,4}  \{ \log(1/\epsilon)\}^{2}
\\
&& s_{\mu}, s_{\sigma} \leq C_{N,4}  \{ \log(1/\epsilon)\}^{3}, \quad M_{\mu}, M_{\sigma} \leq C_{N,4} \log ( 1/\epsilon)
\eean
such that
\bean
\left\vert \mu_{t_1} - f_{\mu}(t_1) \right\vert \leq \epsilon \quad \quad \text{and} \quad \quad \left\vert \sigma_{t_2} - f_{\sigma}(t_2) \right\vert \leq \epsilon 
\eean
for any $t_1 \geq 0$ and $t_2 \geq \epsilon$.
\end{lemma}
\begin{proof}
This is a re-statement of Lemma B.1 in \citet{oko2023diffusion}.
\end{proof}

\begin{lemma}[Reciprocal function]
\label{secnn:rec}
For any $0 < \epsilon < 1$, there exists a positive constant $C_{N,5}$ and a neural network $f_{\rm rec} \in \cF_{\rm NN}(L,\bd,s,M)$ with
\bean
&& L \leq C_{N,5} \{ \log(1/\epsilon)\}^{2}, \quad \|\bd\|_{\infty}\leq C_{N,5}  \{ \log(1/\epsilon)\}^{3}
\\
&& s  \leq C_{N,5}  \{ \log(1/\epsilon)\}^{4}, \quad M \leq C_{N,5} \epsilon^{-2}
\eean
such that
\bean
\left\vert \frac{1}{x} - f_{\rm rec}(\widetilde x) \right\vert \leq \epsilon + \frac{\vert x - \widetilde x \vert}{\epsilon^2}
\eean
for any $x \in [\epsilon,1/\epsilon]$ and $\widetilde x \in \bbR$.
\end{lemma}
\begin{proof}
This is a re-statement of Lemma F.7 in \citet{oko2023diffusion}.
\end{proof}

\section{Proofs for the Approximation Theory}
\label{secapp:approx}

In this section, we provide the proof of Theorem~\ref{secthm:1}.
We begin by outlining the crucial lemmas and propositions.

For $n \in \bbN$, let $P_{n}$ be the Legendre polynomial of degree $n$ defined as
\bean
  P_{n}(x) = \left( \frac{1}{2^{n}{n!}} \right) \frac{\d^{n}}{\d x^{n}} (x^2-1)^{n}
\eean
for $x \in \bbR$.
It is well-known (page 114 of \citet{arnold2004lectures}) that equation $P_{n}=0$ has $n$ distinct roots $\widetilde x_1^{(n)} ,\ldots,\widetilde x_{n}^{(n)}$ satisfying $-1< \widetilde x_1^{(n)} < \cdots < \widetilde x_{n}^{(n)} <1$.
Let $\{ \widetilde w_{1}^{(n)},\ldots, \widetilde w_{n}^{(n)} \}$ be the Gauss-Legendre quadrature weights, that is,
\bean
\widetilde w_j^{(n)} = \begin{dcases}
  \int_{-1}^{1} \prod_{\substack{k = 1 \\ k \neq j}}^{n} \left( \frac{x-\widetilde x_k^{(n)}}{\widetilde x_j^{(n)} - \widetilde x_k^{(n)}} \right) \d x, & \text{if }
       n \geq 2
  \\
  2, & \text{if } n=1.
  \end{dcases}
\eean
Let $n_{\beta}$ be the largest integer strictly smaller than $\beta \vee 2$.
For simplicity, we denote the vectors $(\widetilde x_1^{(n_\beta)} ,\ldots,\widetilde x_{n_\beta}^{(n_\beta)} )$ and $( \widetilde w_{1}^{(n_\beta)},\ldots, \widetilde w_{n_\beta}^{(n_\beta)} )$ as  $( \widetilde x_1 ,\ldots,\widetilde x_{n_\beta} )$ and $(\widetilde w_{1},\ldots, \widetilde w_{n_\beta} )$, respectively.
The following lemma provides an error bound for the $m$-points quadrature rule to approximate a one-dimensional integral.

\begin{lemma}[$1$-dimensional $m$-point quadrature rule]
\label{sec:qde}
Let $A < B$ and $\beta, K > 0$ be given.
For every $m \in {n_{\beta}} \bbN$, there exists $(w_i, x_i)_{i \in [m]}$ with $w_i > 0 $ and $x_i \in (A, B)$ such that
\bean
&& \left\vert \int_{A}^{B} g(x) \d x -  \sum_{i=1}^{m} w_i g(x_i) \right\vert \leq \left\{ \frac{{{n_{\beta}}}^{\beta}}{2^{\beta-\lfloor \beta \rfloor} \lfloor \beta \rfloor !} \right\} K (B-A)^{\beta + 1 } m^{-\beta},
\eean
for every $g \in \cH^{\beta, K}_{1}([A,B])$.
More specifically, one can choose
\bean
&& w_i = \frac{(B-A){n_{\beta}}}{2m}\widetilde w_{i- n_{\beta} \lfloor i / {n_{\beta}} \rfloor},
\\
&& x_i = A + \frac{(B-A){n_{\beta}}}{2m} \left\{\widetilde x_{i- n_{\beta} \lfloor i / {n_{\beta}}  \rfloor} + 2 \left\lfloor i / {n_{\beta}} \right\rfloor +1 \right\}.
\eean
\end{lemma}

Let $\phi$ be the $one$-dimensional standard normal density.
The following lemma provides a bound for the \Holder-norm of a function multiplied by $\phi$ and its derivative $\phi'$.

\begin{lemma}[Preservation of \Holder \ continuity]
\label{sec:normalholder}
Let $\beta, K > 0,  a,b \in \bbR $ be given and $g \in \cH_{1}^{\beta,K}([a,b])$. Then, there exists a positive constant $C_{G,1} = C_{G,1}(\beta)$ such that $g \phi \in \cH_{1}^{\beta,KC_{G,1}}([a,b])$ and $g \phi' \in  \cH_{1}^{\beta,KC_{G,1}}([a,b])$.
\end{lemma}

For $\mu, \sigma > 0$, define $p_{\mu,\sigma}(\cdot)$ as
\bean
p_{\mu,\sigma}(\bx) = \int_{\| \by \|_{\infty} \leq 1} p_0(\by)  \phi_{\sigma}(\bx - \mu \by) \d \by = \mu^{-D} \int_{\| \bx + \sigma \by \|_{\infty} \leq \mu}  p_0 \left( \frac{\bx + \sigma \by}{\mu} \right) \prod_{i=1}^{D} \phi(y_i) \d \by.
\eean
Since $p_0$ is $\beta$-smooth under the (\bS) assumption, we can approximate the integral using the quadrature method with Lemma~\ref{sec:qde} and Lemma~\ref{sec:normalholder}.
Note however that $\by$ in RHS ranges over a large set for small $\sigma$, and the error bound given in Lemma~\ref{sec:qde} depends polynomially on the size of the interval.
Since the tail of $\phi$ decays very quickly, one can control the numerical error as in the following lemma.

\begin{lemma}[Quadrature rule for $p_{\mu,\sigma}(\bx)$ and $\nabla p_{\mu,\sigma}(\bx)$]
\label{sec:qdeapprox}
Let $\beta, K > 0$ be given and suppose that true density $p_0$ belongs to $\cH^{\beta,K}([-1,1]^{D})$.
For $\tau_{\rm bd},\tau_{\rm tail}, \mu, \sigma > 0$, $m \in n_{\beta} \bbN$, $i \in [D]$, $j \in [m]$ and $\bx = (x_1, \ldots, x_D)^\top \in \bbR^D$, let
\bean
&& y_j^{(i)} =  2\sqrt{2 \tau_{\rm tail}} \{\log (1/\sigma) \}^{\tau_{\rm bd} + \frac{1}{2}} 
 \left\{  - x_i-\mu + \frac{{n_{\beta}} \mu }{m}
 \left(  \widetilde x_{j- n_{\beta} \lfloor  j / n_{\beta}  \rfloor} + 2 \left\lfloor   j / n_{\beta}  \right\rfloor +1\right)  \right\} , 
\\
&& w_j =  2\sqrt{2 \tau_{\rm tail}} {n_{\beta}} \widetilde w_{j- n_{\beta} \lfloor   j / n_{\beta}  \rfloor} \{\log (1/\sigma) \}^{\tau_{\rm bd} + \frac{1}{2}}.
\eean
For $\bj = (j_1,\ldots,j_{D})^\top\in [m]^{D}$, let ${\widetilde \by}_{\bj} = (y_{j_1}^{(1)}, \ldots, y_{j_{D}}^{(D)} )^{\top} \in \bbR^{D}$.
Then,
\bean
&& \left\| \frac{\bx + \sigma \widetilde \by_{\bj} }{\mu} \right\|_{\infty} \leq 1- \frac{\{\log(1/\sigma) \}^{-\tau_{\rm bd}}}{2} , 
\quad \quad \bj \in [m]^{D},
\\
&& \mu^{D} \left\vert  p_{\mu,\sigma}(\bx) - \frac{1}{m^{D}}  \sum_{\bj \in [m]^{D}} \left\{ \prod_{i=1}^{D}  w_{j_i} \phi(y_{j_{i}}^{(i)}) \right\} p_0\left( \frac{\bx + \sigma \widetilde \by_{\bj} } {\mu}   \right) \right\vert \leq \epsilon  \quad \quad \text{and}
\\
&& \mu^{D} \left\| \sigma\nabla p_{\mu,\sigma}(\bx) -   \frac{1}{m^{D}} 
\sum_{\bj \in [m]^{D}} \widetilde \by_{\bj} \left\{ \prod_{i=1}^{D}  w_{j_i} \phi(y_{j_{i}}^{(i)}) \right\} p_0\left( \frac{\bx + \sigma \widetilde \by_{\bj} } {\mu} \right) \right\|_{\infty} \leq \epsilon
\eean
for every $\| \bx \|_{\infty} \leq \mu- \mu\{\log(1/\sigma)\}^{-\tau_{\rm bd}}, \mu \in [1/2,1], \sigma \in (0,\widetilde C_{2}]$, where $ \widetilde C_{1} = \widetilde C_{1}(\beta,D,\tau_{\rm tail}), $ $ \widetilde C_{2} = \widetilde C_{2}(\beta,D,\tau_{\rm bd}, \tau_{\rm tail})$ and
\bean
\epsilon = \widetilde C_{1} K  \left(\sigma^{\tau_{\rm tail}}  + m^{-\beta}  \{\log(1/\sigma) \}^{{ (\tau_{\rm bd} +\frac{1}{2})(\beta + 1 ) }}  \right)\{\log(1/\sigma) \}^{(\tau_{\rm bd}+\frac{1}{2})(D-1) }.
\eean
\end{lemma}

We can approximate the maps $(\bx,t) \mapsto p_{t}(\bx)$ and $(\bx,t) \mapsto \nabla p_{t} (\bx)$ using deep ReLU networks by replacing $(\mu, \sigma)$  in Lemma~\ref{sec:qdeapprox} with $(\mu_{t}, \sigma_{t})$.
As discussed in Section~\ref{sec:approx}, a weight-sharing network is used to reduce the number of distinct network parameters.
The approximation result is provided in the following proposition.

\begin{proposition}[Approximation at the interior of near-support]
\label{secpt:1}
Suppose the true density $p_0$ satisfies the assumption (\bS) and 
\bean
\tau_{\rm bd} \in 1/2 + \bbN,
\quad \tau_{\rm tail} > 0, 
\quad \tau_{\rm min} \geq \frac{4\beta}{d(\beta \wedge 1)}.
\eean
Then, for every $m \geq \widetilde C_{5}$, there exists a class of permutation matrices $\cP = \{ \cQ_{i}, \cR_{i} \}_{ i \in [L-1]}$ and weight-sharing  network $\bff \in \cF_{\rm WSNN}(L,\bd,s,M,\cP_{\bfm})$
with 
\bean
&& L \leq \widetilde C_{3} (\log m )^{2} \log \log m, \quad \|\bd\|_{\infty} \leq  \widetilde C_{3} m^{D+1},
\\
&& s \leq  \widetilde C_{3} m (\log m)^{5} \log \log m, \quad M \leq \exp \left(  \widetilde C_{3} \{ \log m \}^{2} \right), 
\\
&& \| \bfm \|_{\infty} \leq \widetilde C_{3} m^{D}
\eean
satisfying
\bean
&& \left\| \begin{pmatrix}
\sigma_t \nabla p_t(\bx) 
\\
 p_t(\bx)
\end{pmatrix}
-  \bff(\bx,t) \right\|_{\infty} \leq \widetilde C_{4} \left( \log m \right)^{(\tau_{\rm bd}+\frac{1}{2}) (D-1)  } \left\{
t^{\frac{\tau_{\rm tail}}{2}} + m^{-\frac{\beta}{d}} \left( \log m \right)^{(\tau_{\rm bd}+\frac{1}{2})(\beta + 1 )} \right\} 
\eean
for every $\bx \in \bbR^{D}$ with $\|\bx\|_{\infty} \leq \mu_t - \mu_t \{\log (1/\sigma_t)\}^{-\tau_{\rm bd}}$ and $m^{-\tau_{\rm min}} \leq t \leq \overline{\tau}^{-1} (\widetilde C_{2}^2 \wedge 1/2) $.

Here, $\widetilde C_{3} = \widetilde C_{3} ( \beta, d, D, K, \overline{\tau}, \underline{\tau}, \tau_{\rm bd}, \tau_{\rm tail}, \tau_{\rm min} )$, 
$\widetilde C_{4} = \widetilde C_{4} ( \beta, d, D, K, \underline{\tau}, \tau_{\rm bd}, \tau_{\rm tail}, \tau_{\rm min} )$,
$ \widetilde C_{5} = \widetilde C_{5} ( \beta, d, \underline{\tau}, \tau_{\rm bd}, \tau_{\rm tail}, \tau_{\min} )$
and
$\widetilde C_{2} = \widetilde C_{2}( \beta,D,\tau_{\rm bd},\tau_{\rm tail} )$ be the constant in Lemma~\ref{sec:qdeapprox}, 
\end{proposition}

\medskip

As $t \to 0$, $p_t$ is not lower bounded near the boundary of the support of $p_0$ due to the lower bound condition, making the approximation of $\nabla \log p_t$ challenging.
With the assumption (\bB), $p_0$ is infinitely smooth so one can approximate $p_0$ efficiently with local polynomials by applying Taylor's theorem in the low-density region.
Since a Gaussian density can also be efficiently approximated with local polynomials, one can calculate the integral in $p_t$ closed form, and approximate the output with vanilla feedforward neural networks.
The following proposition provides the approximation result, and our main proof strategy follows the proofs of Lemma B.2--Lemma B.5 from \citet{oko2023diffusion}, with modifications for simplification.

\begin{proposition}[Approximation at the boundary of near-support]
\label{secpt:2}
Let $K, \tau_{\rm bd}, \tau_{\rm x} > 0$, $ 0 <  \tau_{\rm t} < 1 $, $0 < \widetilde \tau_{\rm bd} < \tau_{\rm bd}$ be given and suppose the true density $p_0$ satisfies that $ \| p_0 \|_{\infty} \leq K $.
Then, for $0 < \delta \leq \widetilde C_{8}$ and $p_0$ satisfying
\bean
\sup_{\balpha \in \bbN^{D}} \sup_{ 1- \{ \log (1 / \delta ) \}^{- \widetilde \tau_{\rm bd}} \leq \| \bx \|_{\infty} \leq 1 } \left\vert ( \D^{\balpha} p_0 )(\bx) \right\vert \leq K,
\eean
there exists a network $\bff \in \cF_{\rm NN} ( L, \bd, s, M )$ with
\bean
&& L \leq \widetilde C_{6} \{ \log(1/\delta)\}^{4}, \quad \| \bd\|_{\infty} \leq \widetilde C_{6} \{ \log(1/\delta)\}^{ 7 + D \widetilde \tau_{\rm bd} + D },
\\
&& s \leq \widetilde C_{6}  \{ \log(1/\delta)\}^{11+ D \widetilde \tau_{\rm bd} + D }, \quad M \leq \exp \left( \widetilde C_{6} \{\log(1/\delta) \}^{2} \right),
\eean
satisfying
\bean
&& \left\| \begin{pmatrix}
\sigma_t \nabla p_t(\bx) 
\\
 p_t(\bx)
\end{pmatrix}
-  \bff(\bx,t) \right\|_{\infty} \leq \widetilde C_{7}  \delta \{\log(1/\delta) \}^{D}
\eean
for every $\bx \in \bbR^{D}$ with $\mu_{t} -  \tau_{\rm x} \{ \log (1/\sigma_{t}) \}^{-\tau_{\rm bd}}  \leq \|\bx \|_{\infty} \leq \mu_{t} + \tau_{\rm x} \sigma_{t} \sqrt{\log (1/\delta)}$ and $\delta \leq t \leq \delta^{\tau_{\rm t}}$.

Here, $\widetilde C_{6} = \widetilde C_{6}( D, K, \overline{\tau}, \underline{\tau}, \tau_{\rm x} )$, $\widetilde C_{7} = \widetilde C_{7} ( D , K , \underline{\tau} )$,  $ \widetilde C_{8} = \widetilde C_{8}( D,  \overline{\tau}, \tau_{\rm bd}, \tau_{\rm x}, \tau_{\rm t}, \widetilde \tau_{\rm bd} )$ are positive constants.

\end{proposition}

For $t_{*} \geq 0$ and $t > 0$,  we have
\bean
p_{t_* + t}(\bx) = \int_{\bbR^{D}} p_{t_*}(\by) \phi_{\sigma_{t}}(\bx-\mu_{t} \by) \d \by, \quad \bx \in \bbR^{D}
\eean
due to the Markov property of the process $(\bX_{t})_{t \geq 0}$. 
Note that the map $\bx \mapsto p_{t_*}(\bx)$ is infinitely differentiable and its norm is bounded as  $\| \D^{\bk} p_{t_*} (\cdot) \|_{\infty} \lesssim t_{*}^{-k./2}$ for any $\bk \in \bbN^{D}$; see Lemma~\ref{secsc:ptderivativebound}).
Then, one can approximate $p_{t_*}$ with a local Taylor expansion, yielding an error $O ( m^{-k. / D}  t_{*}^{-k./2} )$ using grid points bounded by $ O ( m )$, both up to a poly-logarithmic factor.
Similar to the proof of Proposition~\ref{secpt:2}, one can approximate the map $ (\bx,t) \mapsto p_{t_*+t}(\bx) $ with vanilla feedforward neural networks.
The following proposition provides the approximation result, and our main proof strategy follows the proof of Lemma B.7 from \citet{oko2023diffusion}, with modifications for simplification.

\begin{proposition}[Approximation for large $t$]
\label{secpt:3}
Let $K, \tau_1, \tau_{\rm x} > 0, \tau_{\rm sm} \in \bbN, \tau_{\rm low} \in (0,1)$ be given and suppose the true density $p_0$ satisfies that $\tau_1 \leq p_0(\bx) \leq K$ for any $\bx \in [-1,1]^{D}$.   
Then, for $m \geq \widetilde C_{11}$, there exists a neural network $\bff \in \cF_{\rm NN}(L,\bd,s,M)$ with
\bean
&& L \leq \widetilde C_{9} ( \log m )^{4} , \quad \| \bd\|_{\infty} \leq \widetilde C_{9} m (\log m)^{9},
\\
&& s \leq \widetilde C_{9}  m (\log m)^{9}, \quad M \leq \exp(\widetilde C_{9} (\log m)^2 )
\eean
such that
\bean
&& \left\| \begin{pmatrix}
\sigma_t \nabla p_{t_*+t}(\bx) 
\\
 p_{t_*+t}(\bx)
\end{pmatrix}
-  \bff(\bx,t) \right\|_{\infty} \leq \widetilde C_{10} m^{ - \frac{ \tau_{\rm low} \tau_{\rm sm} - ( D + 1 - \tau_{\rm low} ) D  }{D ( 1 + D  )}}   
 (\log m)^{D(\frac{\tau_{\rm sm}}{2} + 1)}
\eean
for every $\bx \in \bbR^{D}$ with $\|\bx\|_{\infty} \leq \mu_t + \tau_{\rm x} \sigma_t \sqrt{\log (1/\delta)}$, $\delta \leq t \leq \overline{\tau}^{-1} \log (1/\delta)$, where
\bean
t_* = m^{-\frac{ 2-2\tau_{\rm low} }{D} } 
\quad \text{and} \quad \delta = m^{- \frac{\tau_{\rm low} \tau_{\rm sm} + D + 1 - \tau_{\rm low}  }{D (1 + D  ) }}.
\eean
Here, $\widetilde C_{9}, \widetilde C_{10}, \widetilde C_{11}$ are positive constants depending on $( D, K, \overline{\tau},  \underline{\tau}, \tau_{\rm x},  \tau_{\rm sm}, \tau_{\rm low} )$.
\end{proposition}

\subsection{Proofs of Lemma~\ref{sec:qde} to \ref{sec:qdeapprox}}

In this subsection, we provide the proof of Lemma~\ref{sec:qde}, Lemma~\ref{sec:normalholder}, and Lemma~\ref{sec:qdeapprox}.

By the definition, we have $\exp(-\overline{\tau} t) \leq \mu_t \leq \exp(-\underline{\tau} t)$ and $1-\exp(-2\underline{\tau} t) \leq \sigma_t^2 \leq 1- \exp(-2\overline{\tau} t)$ for $t \geq 0$. Since  $x/2 \leq 1- e^{-x}$ for $0 \leq x \leq 1$ and $1 - e^{-x} \leq x$ for $x \geq 0$, we have
\be
\begin{split}
& \mu_t \leq 1- \frac{\underline{\tau}t}{2} \leq 1,
\quad
\sigma_t \leq \sqrt{ 1 - \exp (-2 \overline{\tau} t ) } \leq  \sqrt{2\overline{\tau}t},
\quad \forall t \geq 0,  
\quad \text{and}
\\
&
\mu_t  \geq 1- \overline{\tau}t \geq \frac{1}{2},
\quad
\sigma_t \geq \sqrt{1 - \exp(-2 \underline{\tau} t) } \geq  \sqrt{\underline{\tau}t },
\quad \forall 0 \leq t \leq (2\overline{\tau})^{-1}.
\label{eqthm:mtstbd}
\end{split}
\ee
The last display is widely used in the following proofs.

\subsubsection{Proof of Lemma~\ref{sec:qde}}

\begin{proof}
Let $m_0 = \frac{2m}{(B-A){n_{\beta}}}$ and consider a function $g \in \cH_{1}^{\beta,K}([A,B])$.
Simple calculation yields that
\be
\int_{A}^B g(x) \d x = \sum_{i=1}^{\frac{m_0(B-A)}{2}} \int_{A+\frac{2(i-1)}{m_0}}^{A+\frac{2i}{m_0}} g (x) \d x = \sum_{i=1}^{\frac{m_0(B-A)}{2}} \int_{-\frac{1}{m_0}}^{\frac{1}{m_0}} g_i(x) \d x, \label{eqn:partition}
\ee
where $g_i(x) = g(x +A+\frac{2i-1}{m_0} )$.
For each $i \in \{1,\ldots, \frac{m_0(B-A)}{2}\}$, let $L_i$ be the Lagrange interpolating polynomial of degree ${n_{\beta}}-1$ that agrees with the function $g_i$ at knots $\{\widetilde x_1 / m_0, \ldots, \widetilde x_{{n_{\beta}}} / m_0 \}$, defined as 
\bean
L_i(x) = \sum_{j=1}^{{n_{\beta}}} g_i(\widetilde x_j/m_0) l_j(x),
\eean
where \bean
l_j(x) = \begin{dcases}
  \prod_{\substack{k = 1 \\ k \neq j}}^{\lfloor \beta \rfloor} \frac{m_0 x-\widetilde x_k}{\widetilde x_j - \widetilde x_k}, & \text{if }
       \beta > 2,
\\
  1, & \text{if } 0<\beta \leq 2,
\end{dcases}
\eean
for $j \in [{n_{\beta}}]$.
Note that $L_i(\widetilde x_j /m_0) = g_i(\widetilde x_j / m_0)$ for any $j \in [{n_{\beta}}]$.

If $0 < \beta \leq 1$, $\widetilde x_1 = 0$ and $L_i(x) = g_i(0)$. Since $g_i \in \cH^{\beta,K}_1([-1/m_0,1/m_0])$, we have
\be \begin{split}
& \left\vert \int_{-\frac{1}{m_0}}^{\frac{1}{m_0}} \left\{  g_i(x) - L_i(x) \right\} \d x \right\vert
 \leq
\int_{-\frac{1}{m_0}}^{\frac{1}{m_0}} \vert g_i(x) - L_i(x) \vert \d x
\\
& \leq \int_{-\frac{1}{m_0}}^{\frac{1}{m_0}} K \vert x \vert^{\beta} \d x \leq \frac{2}{m_0} K m_0^{-\beta}  = 2K m_0^{-(\beta+1)}. \label{eqn:qde1}
\end{split} \ee
If $\beta >1$, fix $\widetilde x_0 \in [-1,1]$ satisfying $\widetilde x_0 \neq \widetilde x_j$ for $j \in [\lfloor \beta \rfloor]$. Consider a function $h : [-1/m_0,1/m_0] \to \bbR$ such that 
\bean
h(x) = g_i(x) - L_i(x) - \left\{ g_i\left(\frac{\widetilde x_0}{ m_0}\right) - L_i\left(\frac{\widetilde x_0}{ m_0}\right) \right\}  \prod_{j=1}^{\lfloor \beta \rfloor } \left( \frac{m_0 x - \widetilde x_j}{\widetilde x_0 - \widetilde x_j} \right).
\eean
Then, $h(\widetilde x_j/m_0) = 0$ for $j \in \{0,\ldots, \lfloor \beta \rfloor\}$ and $g$ is $\lfloor \beta \rfloor$-times differentiable on $(-1/m_0,1/m_0)$.  Generalized Rolle's Theorem (see Theorem 1.10 of \citet{burden2010numerical}) implies that there exists a constant $\xi_{\widetilde x_0} \in (-1/m_0,1/m_0)$ such that $(\D^{\lfloor \beta \rfloor}h)(\xi_{\widetilde x_0}) = 0$. Since $L_i$ is the polynomial of degree less than $\lfloor \beta \rfloor$ and $(\D^{\lfloor \beta \rfloor} L_i)(\xi_{\widetilde x_0}) = 0$, a simple calculation yields that
\bean
g_i\left(\frac{\widetilde x_0}{ m_0}\right) = L_i\left(\frac{\widetilde x_0}{ m_0}\right) + \frac{(\D^{\lfloor \beta \rfloor} g_i )(\xi_{\widetilde x_0})}{\lfloor \beta \rfloor !} \prod_{j=1}^{\lfloor \beta \rfloor} \left(\frac{\widetilde x_0 - \widetilde x_j}{ m_0} \right).
\eean
Note that $g_i(\widetilde x_j / m_0) = L_i(\widetilde x_j /m_0)$ for $j \in [\lfloor \beta \rfloor]$. Combining with the last display, there exists a function $\xi : [-1/m_0,1/m_0] \to (-1/m_0,1/m_0)$ such that
\bean
g_i(x) = L_i(x) + \frac{(\D^{\lfloor \beta \rfloor} g_i )(\xi(x))}{\lfloor \beta \rfloor !} \prod_{j=1}^{\lfloor \beta \rfloor} \left(x - \frac{ \widetilde x_j}{ m_0} \right), \quad x \in \left[-\frac{1}{m_0},\frac{1}{m_0} \right],
\eean
where $\xi(x) = 0$ for $x \in \{\frac{\widetilde x_1}{m_0},\ldots,\frac{\widetilde x_{\lfloor \beta \rfloor}}{m_0} \}$. For $x \in [-1/m_0,1/m_0]$, we have
\bean
&& \left\vert g_i(x) - L_i(x) - \frac{(\D^{\lfloor \beta \rfloor} g_i )(0)}{\lfloor \beta \rfloor !} \prod_{j=1}^{\lfloor \beta \rfloor} \left(x - \frac{ \widetilde x_j}{ m_0} \right) \right\vert 
\\
&& =
\left\vert \left\{  \frac{(\D^{\lfloor \beta \rfloor}g_i )(\xi(x)) - (\D^{\lfloor \beta \rfloor}g_i )(0)}{\lfloor \beta \rfloor !} \right\} \prod_{j=1}^{\lfloor \beta \rfloor} \left(x - \frac{ \widetilde x_j}{ m_0} \right) \right\vert
\\
&& \leq \frac{K \vert \xi(x) \vert^{\beta - \lfloor \beta \rfloor}}{\lfloor \beta \rfloor !} \prod_{j=1}^{\lfloor \beta \rfloor} \left(\frac{2}{ m_0} \right) \leq \frac{K 2^{\lfloor \beta \rfloor}}{\lfloor \beta \rfloor !} m_0^{-\beta},
\eean
where the first inequality holds because $g_i \in \cH^{\beta,K}_{1}([-1/m_0,1/m_0])$. Since $\{\widetilde x_1,\ldots,\widetilde x_{\lfloor \beta \rfloor } \}$ are the roots of the Legendre polynomial, its orthogonality implies that $\int_{-1/m_0}^{1/m_0} \prod_{j=1}^{\lfloor \beta \rfloor} (x - \frac{\widetilde x_j}{m_0}) \d x = 0$. Combining with the last display, it follows that
\be \begin{split} 
& \left\vert \int_{-\frac{1}{m_0}}^{\frac{1}{m_0}} \left\{  g_i(x) - L_i(x) \right\} \d x \right\vert
\\
& = \left\vert \int_{-\frac{1}{m_0}}^{\frac{1}{m_0}} \left\{ g_i(x) - L_i(x) - \frac{(\D^{\lfloor \beta \rfloor}g_i )(0)}{\lfloor \beta \rfloor !} \prod_{j=1}^{\lfloor \beta \rfloor} \left(x - \frac{ \widetilde x_j}{ m_0} \right) \right\} \d x \right\vert 
\\
&\leq  \int_{-\frac{1}{m_0}}^{\frac{1}{m_0}} \left\vert g_i(x) - L_i(x) - \frac{(\D^{\lfloor \beta \rfloor}g_i )(0)}{\lfloor \beta \rfloor !} \prod_{j=1}^{\lfloor \beta \rfloor} \left(x - \frac{ \widetilde x_j}{ m_0} \right) \right\vert \d x \leq \frac{K 2^{\lfloor \beta \rfloor+1}}{\lfloor \beta \rfloor !} m_0^{-(\beta+1)} \label{eqn:qde2}.
\end{split} \ee

A simple calculation yields that
\bean
\int_{-\frac{1}{m_0}}^{\frac{1}{m_0}} L_i(x)
= \sum_{j=1}^{{n_{\beta}}}  g_i \left(\frac{\widetilde x_j}{m_0} \right) \left\{ \int_{-\frac{1}{m_0}}^{\frac{1}{m_0}} l_j(x) \d x \right\}
= \sum_{j=1}^{{n_{\beta}}} \frac{\widetilde w_j}{m_0} g_i \left(\frac{\widetilde x_j}{m_0} \right).
\eean
Combining (\ref{eqn:partition}), (\ref{eqn:qde1}) and (\ref{eqn:qde2}) with the last display, we have
\bean
&& \left\vert \int_{A}^B g(x) \d x - \sum_{i=1}^{\frac{m_0(B-A)}{2}} \sum_{j=1}^{{n_{\beta}}} \frac{\widetilde w_j}{m_0} g_i \left(\frac{\widetilde x_j}{m_0} \right) \right\vert \leq \sum_{i=1}^{\frac{m_0(B-A)}{2}} \left\vert \int_{-\frac{1}{m_0}}^{\frac{1} {m_0}} \left\{ g_i(x) - L_i(x) \right\}  \d x \right\vert
\\
&& \leq \frac{K(B-A) 2^{\lfloor \beta \rfloor}}{\lfloor \beta \rfloor !} m_0^{-\beta} =  \frac{{( \lfloor \beta \rfloor  \vee 1)}^{\beta}}{2^{\beta-\lfloor \beta \rfloor} \lfloor \beta \rfloor !} (B-A)^{\beta + 1} m^{-\beta}.
\eean
 Then, the assertion follows because $\sum_{i=1}^m w_i g(x_i) = \sum_{i=1}^{m_0(B-A)/2} \sum_{j=1}^{{n_{\beta}}} \frac{\widetilde w_j}{m_0} g_i \left(\frac{\widetilde x_j}{m_0} \right) $.
\end{proof}


\subsubsection{Proof of Lemma~\ref{sec:normalholder}}

\begin{proof}
For any $n \in \bbZ_{\geq 0}$, it is well-known (see \citep{indritz1961inequality}) that 
\bean
\left\vert \frac{\d^n}{\d x^n}  e^{-x^2} \right\vert \leq \sqrt{2^{n} n!} e^{-\frac{x^2}{2}}, \quad x \in \bbR
\eean
and moreover,
\be
\| \D^{n}\phi \|_{\infty} = \frac{1}{\sqrt{2\pi}}  \left\| \frac{\d^n}{\d x^n}  e^{-\frac{x^2}{2}} \right\|_{\infty} \leq \frac{\sqrt{n!}}{\sqrt{2\pi}} \left\| e^{-\frac{x^2}{4}} \right\|_{\infty} \leq \frac{\sqrt{n!}}{\sqrt{2\pi}}, \label{eqn:gdbound}
\ee
where the first inequality holds by the chain rule.
Then, 
\be \begin{split}
& \sum_{\alpha =0}^{\lfloor \beta \rfloor} \left\| \D^{\alpha}(g\phi) \right\|_{\infty} = \sum_{\alpha = 0}^{\lfloor \beta \rfloor} \left\| \sum_{r=0}^{\alpha} \binom{\alpha}{r} (\D^{r}g)(\D^{\alpha-r}\phi) \right\|_{\infty}
\\
& \leq \sum_{\alpha = 0}^{\lfloor \beta \rfloor} \sum_{r=0}^{\alpha} \binom{\alpha}{r} \left\| \D^{r} g \right\|_{\infty}  \left\| \D^{\alpha-r} \phi \right\|_{\infty} \leq \sum_{\alpha = 0}^{\lfloor \beta \rfloor} \sum_{r=0}^{\alpha} \alpha^{\alpha}   \left\| \D^{r} g \right\|_{\infty} \frac{\sqrt{(\alpha-r)!}}{\sqrt{2\pi}}
\\
& \leq \sum_{\alpha = 0}^{\lfloor \beta \rfloor}  K \alpha^{\alpha} \frac{\sqrt{\alpha!}}{\sqrt{2\pi}}
\leq K \left(\lfloor \beta \rfloor +1 \right)^{\lfloor \beta \rfloor +1} \frac{\sqrt{\lfloor \beta \rfloor!}}{\sqrt{2\pi}}.
\label{eqn:prodbd1}
\end{split} \ee
Similarly, we have
\be 
\sum_{\alpha =0}^{\lfloor \beta \rfloor} \left\| \D^{\alpha}(g\phi') \right\|_{\infty}
\leq \sum_{\alpha = 0}^{\lfloor \beta \rfloor} \sum_{r=0}^{\alpha} \binom{\alpha}{r} \left\| \D^{r} g \right\|_{\infty}  \left\| \D^{\alpha-r+1} \phi \right\|_{\infty} \leq K \left(\lfloor \beta \rfloor +1 \right)^{ \lfloor \beta \rfloor +1} \frac{\sqrt{({\lfloor \beta \rfloor}+1)!}}{\sqrt{2\pi}}. \label{eqn:prodbd4}
\ee
Let $A = [a,b]$.
For any differentiable function $h : A \subseteq \bbR \to \bbR$ and $0 < \gamma \leq 1$, we have
\be \begin{split}
& \sup_{\substack{x,y \in A \\ x \neq y }} \frac{\left\vert h(x)-h(y) \right\vert}{\vert x-y \vert^{\gamma}} 
 \leq
\sup_{\substack{x,y \in A \\ x \neq y \\ \vert x-y \vert \leq 1}} \frac{\left\vert h(x)-h(y) \right\vert}{\vert x-y \vert^{\gamma}} 
+
\sup_{\substack{x,y \in A \\ \vert x-y \vert \geq 1 }} \frac{\left\vert h(x)-h(y) \right\vert}{\vert x-y \vert^{\gamma}} 
\\
& \leq
\sup_{\substack{x,y \in A \\ x \neq y \\ \vert x-y \vert \leq 1}} \frac{\left\vert h(x)-h(y) \right\vert}{\vert x-y \vert}  +\sup_{\substack{x,y \in A \\ \vert x-y \vert \geq 1 }} {\left\vert h(x)-h(y) \right\vert}
 \leq  \left\| h' \right\|_{\infty} + 2 \left\|  h \right\|_{\infty}. \label{eqn:holderbound}
\end{split} \ee
If $\beta > 1$, it follows that
\bean
&& \frac{\left\vert(\D^{\lfloor \beta \rfloor}(g \phi))(x) - (\D^{\lfloor \beta \rfloor}(g \phi))(y) \right\vert }{\vert x-y\vert^{\beta-\lfloor \beta \rfloor}}
\\
&& =
\frac{\left\vert \sum_{\alpha=0}^{\lfloor \beta \rfloor} \binom{\lfloor \beta \rfloor}{\alpha} \left\{ (\D^{\alpha} g)(x) (\D^{\lfloor \beta \rfloor-\alpha} \phi)(x) - (\D^{\alpha} g)(y) (\D^{\lfloor \beta \rfloor-\alpha} \phi)(y) \right\} \right\vert }{\vert x-y\vert^{\beta-\lfloor \beta \rfloor}}
\\
&& \leq \sum_{\alpha=0}^{\lfloor \beta \rfloor} \binom{\lfloor \beta \rfloor}{\alpha}  \left\vert (\D^{\alpha}g)(x) \right\vert \left(  \frac{ \left\vert (\D^{\lfloor \beta \rfloor-\alpha}\phi)(x) - (\D^{\lfloor \beta \rfloor-\alpha}\phi)(y) \right\vert }{\vert x-y\vert^{\beta-\lfloor \beta \rfloor}} \right) 
\\
&& \quad +  \sum_{\alpha=0}^{\lfloor \beta \rfloor} \binom{\lfloor \beta \rfloor}{\alpha}
\left\vert(\D^{\lfloor \beta \rfloor-\alpha}\phi)(y) \right\vert \left( \frac { \left\vert (\D^{\alpha}g)(x) - (D^{\alpha}g)(y) \right\vert }{\vert x-y\vert^{\beta-\lfloor \beta \rfloor}} \right)
\\
&& \leq \sum_{\alpha=0}^{\lfloor \beta \rfloor} \binom{\lfloor \beta \rfloor}{\alpha}  \left\| D^{\alpha} g \right\|_{\infty} \left(  \left\| \D^{ \lfloor \beta \rfloor -\alpha +1} \phi \right\|_{\infty} + 2 \left\| \D^{ \lfloor \beta \rfloor -\alpha } \phi
\right\|_{\infty} \right) 
\\
&& \quad + \sum_{\alpha=0}^{\lfloor \beta \rfloor-1} \binom{\lfloor \beta \rfloor}{\alpha}
 \left\| \D^{\lfloor \beta \rfloor-\alpha}\phi \right\|_{\infty} \left(  \left\| \D^{\alpha + 1} g \right\|_{\infty} + 2 \left\| \D^{\alpha }g
\right\|_{\infty} \right) 
\\
&& \quad + \|\phi\|_{\infty} \left( \frac { \left\vert (\D^{\lfloor \beta \rfloor}g)(x) - (\D^{\lfloor \beta \rfloor}g)(y) \right\vert }{\vert x-y\vert^{\beta-\lfloor \beta \rfloor}} \right)
\eean
for any $x,y \in A$ with $x \neq y$, where the last inequality holds by (\ref{eqn:holderbound}).
Combining (\ref{eqn:gdbound}) with the last display, we have
\bean
&& \sup_{\substack{x,y \in A \\ x \neq y }} \frac{\left\vert(\D^{\lfloor \beta \rfloor}(g \phi))(x) - (\D^{\lfloor \beta \rfloor}(g \phi))(y) \right\vert }{\vert x-y\vert^{\beta-\lfloor \beta \rfloor}}
\\
&& \leq  \sum_{\alpha=0}^{\lfloor \beta \rfloor} \binom{\lfloor \beta \rfloor}{\alpha}  \left\| D^{\alpha} g \right\|_{\infty} \left(  \left\| \D^{ \lfloor \beta \rfloor -\alpha +1} \phi \right\|_{\infty} + 2 \left\| \D^{ \lfloor \beta \rfloor -\alpha } \phi
\right\|_{\infty} \right) 
\\
&& \quad + \sum_{\alpha=0}^{\lfloor \beta \rfloor-1} \binom{\lfloor \beta \rfloor}{\alpha}
 \left\| \D^{\lfloor \beta \rfloor-\alpha}\phi \right\|_{\infty} \left(  \left\| \D^{\alpha + 1} g \right\|_{\infty} + 2 \left\| \D^{\alpha }g
\right\|_{\infty} \right)
\\
&& \quad + \|\phi\|_{\infty} \left( \sup_{\substack{x,y \in A \\ x \neq y }} \frac { \left\vert (\D^{\lfloor \beta \rfloor}g)(x) - (D^{\lfloor \beta \rfloor}g)(y) \right\vert }{\vert x-y\vert^{\beta-\lfloor \beta \rfloor}} \right)
\\
&& \leq \sum_{\alpha=0}^{\lfloor \beta \rfloor} {\lfloor \beta \rfloor}^{\alpha}  \left\| D^{\alpha} g \right\|_{\infty}   \left(  \frac{\sqrt{(\lfloor \beta \rfloor - \alpha +1)!}}{\sqrt{2\pi}} + \frac{2\sqrt{(\lfloor \beta \rfloor - \alpha)!}}{\sqrt{2\pi}}  \right)
\\
&& \quad + \sum_{\alpha = 0}^{\lfloor \beta \rfloor -1 } {\lfloor \beta \rfloor}^{\alpha} \left(  \left\| \D^{\alpha + 1} g \right\|_{\infty} + 2 \left\| \D^{\alpha }g
\right\|_{\infty} \right) \left(  \frac{\sqrt{(\lfloor \beta \rfloor - \alpha)!}}{\sqrt{2\pi}}  \right)
\\
&& \quad +  \frac{1}{\sqrt{2\pi}}\left( \sup_{\substack{x,y \in A \\ x \neq y }} \frac { \left\vert (\D^{\lfloor \beta \rfloor}g)(x) - (\D^{\lfloor \beta \rfloor}g)(y) \right\vert }{\vert x-y\vert^{\beta-\lfloor \beta \rfloor}} \right).
\eean
Moreover, the last display is bounded by
\be \begin{split}
& \leq 3 \lfloor \beta \rfloor^{\lfloor \beta \rfloor} \left( \frac{\sqrt{(\lfloor \beta \rfloor +1)!}}{\sqrt{2\pi}} \right) \left( \sum_{\alpha = 0}^{\lfloor \beta \rfloor}  \left\| \D^{\alpha} g \right\|_{\infty} \right)
\\
& \quad + 3 \lfloor \beta \rfloor^{\lfloor \beta \rfloor-1} \left( \frac{\sqrt{\lfloor \beta \rfloor !}}{\sqrt{2\pi}} \right) \left( \sum_{\alpha = 0}^{\lfloor \beta \rfloor}  \left\| \D^{\alpha} g \right\|_{\infty} + \sup_{\substack{x,y \in A \\ x \neq y }} \frac { \left\vert (\D^{\lfloor \beta \rfloor}g)(x) - (\D^{\lfloor \beta \rfloor}g)(y) \right\vert }{\vert x-y\vert^{\beta-\lfloor \beta \rfloor}} \right)
\\
& \leq 6K \lfloor \beta \rfloor^{\lfloor \beta \rfloor} \left( \frac{\sqrt{(\lfloor \beta \rfloor +1)!}}{\sqrt{2\pi}} \right). \label{eqn:prodbd2}
\end{split} \ee
Similarly, we have
\be \begin{split}
& \sup_{\substack{x,y \in A \\ x \neq y }} \frac{\left\vert(\D^{\lfloor \beta \rfloor}(g \phi'))(x) - (\D^{\lfloor \beta \rfloor}(g \phi'))(y) \right\vert }{\vert x-y\vert^{\beta-\lfloor \beta \rfloor}}
\\
& \leq  \sum_{\alpha=0}^{\lfloor \beta \rfloor} \binom{\lfloor \beta \rfloor}{\alpha}  \left\| \D^{\alpha} g \right\|_{\infty} \left(  \left\| \D^{ \lfloor \beta \rfloor -\alpha +2} \phi \right\|_{\infty} + 2 \left\| \D^{ \lfloor \beta \rfloor -\alpha +1 } \phi
\right\|_{\infty} \right) 
\\
& \quad + \sum_{\alpha=0}^{\lfloor \beta \rfloor-1} \binom{\lfloor \beta \rfloor}{\alpha}
 \left\| \D^{\lfloor \beta \rfloor-\alpha+1}\phi \right\|_{\infty} \left(  \left\| \D^{\alpha + 1} g \right\|_{\infty} + 2 \left\| \D^{\alpha }g
\right\|_{\infty} \right)
\\
& \quad + \|\phi'\|_{\infty} \left( \sup_{\substack{x,y \in A \\ x \neq y }} \frac { \left\vert (\D^{\lfloor \beta \rfloor}g)(x) - (\D^{\lfloor \beta \rfloor}g)(y) \right\vert }{\vert x-y\vert^{\beta-\lfloor \beta \rfloor}} \right)
\\
& \leq \sum_{\alpha=0}^{\lfloor \beta \rfloor} { \lfloor \beta \rfloor }^{\alpha}  \left\| \D^{\alpha} g \right\|_{\infty}   \left(  \frac{\sqrt{(\lfloor \beta \rfloor - \alpha +2)!}}{\sqrt{2\pi}} + \frac{2\sqrt{(\lfloor \beta \rfloor - \alpha +1)!}}{\sqrt{2\pi}}  \right)
\\
& \quad + \sum_{\alpha = 0}^{\lfloor \beta \rfloor -1 } {\lfloor \beta \rfloor}^{\alpha} \left(  \left\| \D^{\alpha + 1} g \right\|_{\infty} + 2 \left\| \D^{\alpha }g
\right\|_{\infty} \right) \left(  \frac{\sqrt{(\lfloor \beta \rfloor - \alpha +1)!}}{\sqrt{2\pi}}  \right) 
\\
& \quad +  \frac{1}{\sqrt{2\pi}}\left( \sup_{\substack{x,y \in A \\ x \neq y }} \frac { \left\vert (\D^{\lfloor \beta \rfloor}g)(x) - (\D^{\lfloor \beta \rfloor}g)(y) \right\vert }{\vert x-y\vert^{\beta-\lfloor \beta \rfloor}} \right)
\\
& \leq 6K \lfloor \beta \rfloor^{\lfloor \beta \rfloor} \left( \frac{\sqrt{(\lfloor \beta \rfloor +2)!}}{\sqrt{2\pi}} \right). \label{eqn:prodbd5}
\end{split} \ee
If $0 < \beta \leq 1$, we have
\bean
&& \sup_{\substack{x,y \in A \\ x \neq y }} \frac{\left\vert (g\phi)(x) - (g\phi)(y) \right\vert }{\vert x-y\vert^{\beta}} \leq \sup_{\substack{x,y \in A \\ x \neq y }} \frac{\left\vert g(x)\phi(x) - g(x)\phi(y) + g(x)\phi(y) - g(y)\phi(y) \right\vert }{\vert x-y\vert^{\beta}} 
\\
&& \leq \|g\|_{\infty} \left( \sup_{\substack{x,y \in A \\ x \neq y }} \frac{\left\vert \phi(x) - \phi(y) \right\vert }{\vert x-y\vert^{\beta}} \right) 
+
\|\phi\|_{\infty} \left( \sup_{\substack{x,y \in A \\ x \neq y }} \frac{\left\vert g(x) - g(y) \right\vert }{\vert x-y\vert^{\beta}} \right)
\\
&& \leq \left( \|g\|_{\infty} + \sup_{\substack{x,y \in A \\ x \neq y }} \frac{\left\vert g(x) - g(y) \right\vert }{\vert x-y\vert^{\beta}} \right) \left( \left\| \phi' \right\|_{\infty} + 2 \left\|  \phi \right\|_{\infty} \right) \leq \frac{3K}{\sqrt{2\pi}} 
\eean
and
\bean
&& \sup_{\substack{x,y \in A \\ x \neq y }} \frac{\left\vert (g\phi')(x) - (g\phi')(y) \right\vert }{\vert x-y\vert^{\beta}}
\leq \left( \|g\|_{\infty} + \sup_{\substack{x,y \in A \\ x \neq y }} \frac{\left\vert g(x) - g(y) \right\vert }{\vert x-y\vert^{\beta}} \right) \left( \left\| \phi'' \right\|_{\infty} + 2 \left\|  \phi' \right\|_{\infty} \right)
\\
&& \leq \frac{3\sqrt{2}K}{\sqrt{2\pi}},
\eean
where the second and third inequality holds by (\ref{eqn:holderbound}) and (\ref{eqn:gdbound}), respectively.
Combining (\ref{eqn:prodbd1}), (\ref{eqn:prodbd4}), (\ref{eqn:prodbd2}), (\ref{eqn:prodbd5}) with the last display, we have $g\phi \in \cH_{1}^{\beta, KD_1}(A)$ and $g \phi' \in \cH_{1}^{\beta,KD_1}$, where 
\bean
D_1 =\left(\lfloor \beta \rfloor +1 \right)^{\lfloor \beta \rfloor +1} \frac{\sqrt{(\lfloor \beta \rfloor +1)!}}{\sqrt{2\pi}} + 6\lfloor \beta \rfloor^{\lfloor \beta \rfloor} \left( \frac{\sqrt{(\lfloor \beta \rfloor +2)!}}{\sqrt{2\pi}} \right).
\eean
The assertion follows by re-defining the constant.
\end{proof}

\subsubsection{Proof of Lemma~\ref{sec:qdeapprox}}

\begin{proof}
Consider $\tau_{\rm bd},\tau_{\rm tail} > 0, m \in n_{\beta} \bbN$ and
real-valued $D$-dimensional vectors $\bx = (x_1,\ldots,x_{D})^{\top}, \by = (y_1,\ldots,y_{D})^{\top}$ such that $\| \bx \|_{\infty} \leq \mu - \mu\{\log(1/\sigma)\}^{-\tau_{\rm bd}} $ and  $\|\bx +\sigma \by \|_{\infty} \leq \mu$.
Let $\by_{<i+1} =  (y_1,\ldots,y_{i})^{\top} \in \bbR^{i}$ and $\by_{>D-i} =(y_{D-i+1},\ldots,y_{D})^{\top} \in \bbR^{i}$ for $i \in [D]$.
For $y \in \bbR$ and each $i$, denote $(\by_{<i},y,\by_{>i})$ as a $D$-dimensional vector that is identical to $\by$ except for the $i$-th component, which is replaced by $y$.

Let $D_{\sigma} = 2\sqrt{2 \tau_{\rm tail}} \{\log (1/\sigma) \}^{(\tau_{\rm bd} + \frac{1}{2})}$ and fix $i \in [D]$. Then, 
\bean
-x_i-\mu \leq -\mu \{\log(1/\sigma)\}^{-\tau_{\rm bd}} < 0
\text{\quad and \quad} -x_i+\mu \geq \mu \{\log(1/\sigma)\}^{-\tau_{\rm bd}} > 0.
\eean
Moreover, $(-x_i-\mu)/\sigma < (-x_i-\mu)D_{\sigma} < 0$ and $(-x_i+\mu)/\sigma > (-x_i+\mu)D_{\sigma} > 0$ for small enough $\sigma$ so that $0 < D_{\sigma} \leq (2\sigma)^{-1}$.
Consider a one-dimensional real-valued function $g_i$ such that
\bean
g_i \left(y; \by_{<i},\by_{>i} \right) =  p_0\left( \frac{\bx + \sigma (\by_{<i},y,\by_{>i}) }{\mu}   \right) \phi(y), \quad y \in \left[(-x_i-\mu)/\sigma,(-x_i+\mu)/\sigma \right].
\eean
Then,
\be \begin{split}
& \left\vert \int_{ \frac{-x_i-\mu}{\sigma}}^{\frac{-x_i+\mu}{\sigma} } g_i \left(y; \by_{<i},\by_{>i} \right) \d y - \int_{ (-x_i-\mu)D_{\sigma}}^{ (-x_i+\mu)D_{\sigma}} g_i \left(y; \by_{<i},\by_{>i} \right) \d y \right\vert
\\
& = \int_{ \frac{-x_i-\mu}{\sigma}}^{ (-x_i-\mu)D_{\sigma}} g_i \left(y; \by_{<i},\by_{>i} \right) \d y + \int_{ (-x_i+\mu)D_{\sigma}}^{ \frac{-x_i+\mu}{\sigma} } g_i \left(y; \by_{<i},\by_{>i} \right) \d y
\\
& \leq K \int_{-\infty}^{(-x_i-\mu)D_{\sigma}  } \phi(y) \d y + K \int_{(-x_i+\mu)D_{\sigma} }^{\infty} \phi(y) \d y
\\
& \leq K \exp \left( -\frac{ (-x_i-\mu)^2 D_{\sigma}^2}{2} \right) + K \exp \left( -\frac{ (-x_i+\mu)^2 D_{\sigma}^2 }{2} \right)
\\
& \leq 2K \exp\left(-\tau_{\rm tail} \log(1/\sigma)\right) = 2K \sigma^{\tau_{\rm tail}} ,  \label{eqn:tailerror}
\end{split} \ee
where the second inequality holds by the tail probability of the standard normal distribution.
Let $C_{G,1} = C_{G,1}(\beta)$ be the constant in Lemma~\ref{sec:normalholder}. Since $\sigma/\mu < 1$ for $\sigma < 2$, Lemma~\ref{sec:normalholder} implies that $g_i \in \cH_1^{\beta,KC_{G,1}}([(-x_i-\mu)D_{\sigma},(-x_i+\mu)D_{\sigma} ])$.
Moreover, Lemma~\ref{sec:qde} implies that
\be \begin{split}
& \left\vert \int_{(-x_i-\mu)D_{\sigma}}^{(-x_i+\mu)D_{\sigma}} g_i \left(y ; \by_{<i},\by_{>i} \right) \d y -  \sum_{j=1}^{m} \widetilde v_j g_i \left(\widetilde y_j^{(i)} ; \by_{<i},\by_{>i} \right) \right\vert
\\
& \leq \left\{ \frac{2^{\lfloor \beta \rfloor + 1 }{{ n_{\beta} } }^{\beta} }{ \lfloor \beta \rfloor !} \right\} C_{G,1} K \mu^{\beta + 1}  D_{\sigma}^{\beta + 1}  m^{-\beta}, \label{eqn:qde} \end{split}
\ee
where
\bean
&& \widetilde y_j^{(i)} = (-x_i-\mu)D_{\sigma} + \frac{\mu D_{\sigma}{n_{\beta}}}{m} \left\{\widetilde x_{j- n_{\beta}  \lfloor \frac{j}{{n_{\beta}}}  \rfloor} + 2 \left\lfloor \frac{j}{{n_{\beta}}} \right\rfloor +1 \right\}, 
\\
&& \widetilde v_j = \frac{\mu D_{\sigma}{n_{\beta}}}{m}\widetilde w_{j- n_{\beta} \lfloor \frac{j}{{n_{\beta}}} \rfloor},
\eean
and $(-x_i-\mu) D_{\sigma} < \widetilde y_j^{(i)} < (-x_i+\mu) D_{\sigma}$ for $j \in [m]$.
Combining (\ref{eqn:tailerror}) and (\ref{eqn:qde}), we have
\be
\left\vert \int_{ \frac{-x_i-\mu}{\sigma}}^{\frac{-x_i+\mu}{\sigma} } p_0\left( \frac{\bx + \sigma (\by_{<i},y,\by_{>i}) }{\mu}  \right) \phi(y) \d y - \sum_{j=1}^{m} \widetilde v_j g_i \left(\widetilde y_j^{(i)} ; \by_{<i},\by_{>i} \right) \right\vert \leq \epsilon, \label{eqn:qdetot}
\ee
where $\epsilon = 2K \sigma^{\tau_{\rm tail}} +  2^{\lfloor \beta \rfloor + 1}{{n_{\beta}}}^{\beta}  C_{G,1} K \mu^{\beta+1}  D_{\sigma}^{\beta+1} m^{-\beta} / \lfloor \beta \rfloor !$.
Since $\vert x_i \vert \leq \mu - \mu \{\log(1/\sigma)\}^{-\tau_{\rm bd}}$ and $ \sigma  D_{\sigma}\leq 1/2$, we have
\bean
\frac{x_i}{2\mu} - \frac{1}{2}
\leq 
\frac{x_i+\sigma D_{\sigma} (-x_i-\mu) }{\mu} \leq \frac{x_i+\sigma \widetilde y_{j}^{(i)} }{\mu}
\leq \frac{x_i+\sigma D_{\sigma} (-x_i+\mu) }{\mu} \leq \frac{x_i}{2\mu} + \frac{1}{2}
\eean
and 
\bean
\left\vert \frac{x_i+\sigma \widetilde y_{j}^{(i)} }{\mu} \right\vert \leq 1 - \frac{\{\log(1/\sigma)\}^{-\tau_{\rm bd}}}{2} < 1.
\eean
Consider $F_{(j_1)},\ldots,F_{(j_1,\ldots,j_{D})}$ for $j_1,\ldots,j_{D}\in [m] $, defined as
\bean
&& F_{(j_1,\ldots,j_{k-1})}
\\
&& = \int_{\|\bx_{>k} + \sigma \by_{>k} \|_{\infty} \leq \mu} \left\{ \int_{ \frac{-x_k-\mu}{\sigma}}^{ \frac{-x_k+\mu}{\sigma} }   p_0\left( \frac{\bx + \sigma \left(\widetilde y_{j_1}^{(1)},\ldots,\widetilde y_{j_{k-1}}^{(k-1)},y,\by_{>k}^{\top}\right)^{\top} }{\mu}   \right) \phi(y) \d y \right\}  \prod_{i=k+1}^{D} \phi(y_i) \d \by_{>k},
\eean
for $k \in \{2,\ldots,D-1\}$,
\bean
&& F_{(j_1,\ldots,j_{D-1})} = \int_{ \frac{-x_D-\mu}{\sigma}}^{ \frac{-x_D+\mu}{\sigma} }    p_0\left( \frac{\bx + \sigma \left(\widetilde y_{j_1}^{(1)},\ldots,\widetilde y_{j_{D-1}}^{(D-1)},y\right)^{\top} }{\mu}   \right) \phi(y) \d y, \quad \text{and}
\\
&& F_{(j_1,\ldots,j_{D})} = p_0\left( \frac{\bx + \sigma \left(\widetilde y_{j_1}^{(1)},\ldots,\widetilde y_{j_{D}}^{(D)}\right)^{\top} }{\mu}   \right).
\eean
For any $k \in \{2,\ldots,D-1\}$ and $j_1,\ldots,j_{D} \in [m]$, we have
\be \begin{split}
& \left\vert F_{(j_1,\ldots,j_{k-1})}- \sum_{j_k = 1}^{m} \widetilde v_{j_k} \phi\left(\widetilde y_{j_k}^{(k)}\right) F_{(j_1,\ldots,j_k)} \right\vert 
\\
& \leq \epsilon \int_{\|\bx_{>k} + \sigma \by_{>k} \|_{\infty} \leq \mu}  \left\{ \prod_{i=k+1}^{D} \phi(y_i)  \right\} \d \by_{>k} \leq \epsilon \label{eqn:coordwise}
\end{split} \ee
and
\bean
\left\vert F_{(j_1,\ldots,j_{D-1})}- \sum_{j_D = 1}^{m} \widetilde v_{j_D} \phi\left(\widetilde y_{j_D}^{(D)}\right) F_{(j_1,\ldots,j_D)} \right\vert \leq \epsilon,
\eean
where the first and last inequality holds by (\ref{eqn:qdetot}).
Note that 
\bean
\mu^{D} p_{\mu,\sigma}(\bx) = \int_{\|\bx_{>1} + \sigma \by_{>1} \|_{\infty} \leq \mu} \left[ \int_{ \frac{-x_1-\mu}{\sigma}}^{ \frac{-x_1+\mu}{\sigma} } \left\{   p_0\left( \frac{\bx + \sigma (y,\by_{>1}^{\top})^{\top} }{\mu}   \right) \phi(y)\right\} \d y \right]  \prod_{i=2}^{D} \phi(y_i) \d \by_{>1}.
\eean
Then, we also have
\bean
\left\vert \mu^{D} p_{\mu,\sigma}(\bx) - \sum_{j_1 = 1}^{m} \widetilde v_{j_1} \phi\left(\widetilde y_{j_1}^{(1)}\right) F_{(j_1)} \right\vert \leq \epsilon \int_{\|\bx_{>1} + \sigma \by_{>1} \|_{\infty} \leq \mu}  \left\{ \prod_{i=2}^{D} \phi(y_i)  \right\} \d \by_{>1} \leq \epsilon,
\eean
where the first inequality holds by (\ref{eqn:qdetot}).
Combining (\ref{eqn:coordwise}) with the last display, we have
\be \begin{split}
& \left\vert \mu^{D} p_{\mu,\sigma}(\bx) - \sum_{j_1,\ldots,j_{D}=1}^{m}
\prod_{k=1}^{D}  \left\{  \widetilde v_{j_k} \phi \left( \widetilde y_{j_k}^{(k)} \right) \right\}  F_{(j_1,\ldots,j_{D})}  \right\vert
\\
& \leq \left\vert \mu^{D} p_{\mu,\sigma}(\bx) - \sum_{j_1 = 1}^{m} \widetilde v_{j_1} \phi\left(\widetilde y_{j_1}^{(1)}\right) F_{(j_1)} \right\vert
\\
& \quad + \sum_{i=2}^{D} \left\vert \sum_{j_1,\ldots,j_{i-1} = 1}^{m}  
\prod_{k=1}^{i-1}  \left\{ \widetilde v_{j_k} \phi \left( \widetilde y_{j_k}^{(k)} \right) \right\}
\left\{ F_{(j_1,\ldots,j_{i-1})} -  \sum_{j_i = 1}^{m} \widetilde v_{j_i} \phi\left(\widetilde y_{j_i}^{(i)}\right) F_{(j_1,\ldots,j_{i})}  \right\} \right\vert
\\
& \leq \epsilon  \left(1+ \sum_{i=2}^{D}  \sum_{j_1,\ldots  j_{i-1} = 1}^{m} \left\vert 
\prod_{k=1}^{i-1} \left\{  \widetilde v_{j_k} \phi \left( \widetilde y_{j_k}^{(k)} \right) \right\} \right\vert \right)
\\
& \leq
\epsilon  \left(1+ \sum_{i=2}^{D}  \sum_{j_1,\ldots  j_{i-1} = 1}^{m} \frac{\left\vert \prod_{k=1}^{i-1} \widetilde v_{j_k} \right\vert}{(2\pi)^{\frac{i-1}{2}}} \right), \label{eqn:qdecoord}
\end{split} \ee
where the last inequality holds because $\vert \phi \vert \leq 1/\sqrt{2\pi}$.
For each $j \in [m]$, we have 
\bean
\vert \widetilde v_{j} \vert \leq \left\{ \frac{ \mu D_{\sigma}{n_{\beta}}}{m} \right\}  \max \left(\vert \widetilde w_1 \vert,\ldots, \vert \widetilde w_{{n_{\beta}}} \vert \right) 
= \frac{D_{1} \{\log(1/\sigma) \}^{(\tau_{\rm bd}+\frac{1}{2} )} }{ m},
\eean
where the last inequality holds because $\mu \leq 1$ and $D_{1} =  2\sqrt{2 \tau_{\rm tail}}{n_{\beta}} \max (\vert \widetilde w_1 \vert,\ldots,\vert \widetilde w_{{n_{\beta}}} \vert)$.
Then,
\bean
&& 1+ \sum_{i=2}^{D}  \sum_{j_1,\ldots,j_{i-1} = 1}^{m} \frac{\vert \prod_{k=1}^{i-1} \widetilde v_{j_k} \vert}{(2\pi)^{\frac{i-1}{2}}} \leq 1 + \sum_{i=2}^{D} \sum_{j_1,\ldots,j_{i-1} = 1}^{m} \left( \frac{ D_{1} \{\log(1/\sigma) \}^{(\tau_{\rm bd}+\frac{1}{2} )} }{m\sqrt{2\pi}} \right)^{i-1}
\\ 
&& = 1 + \sum_{i=2}^{D} \left( \frac{ D_{1} \{\log(1/\sigma) \}^{(\tau_{\rm bd}+\frac{1}{2} )} }{\sqrt{2\pi}} \right)^{i-1}
\leq D\left( \frac{ D_{1} \{\log(1/\sigma) \}^{(\tau_{\rm bd}+\frac{1}{2} )} }{\sqrt{2\pi}} \right)^{D-1},
\eean
where the last inequality holds for small enough $\sigma$ so that $D_{1} \{\log(1/\sigma) \}^{(\tau_{\rm bd}+\frac{1}{2} )} \geq \sqrt{2\pi}$.
Also, there exists a constant $D_{2} = D_{2}(\beta,\tau_{\rm tail},C_{G,1})$ such that 
\bean
\epsilon \leq D_{2} K \left( \sigma^{\tau_{\rm tail}} + m^{-\beta} \{\log(1/\sigma) \}^{(\tau_{\rm bd}+\frac{1}{2} )(\beta + 1 ) } \right).
\eean
Hence,
\be \begin{split}
& \left\vert  \mu^{D} p_{\mu,\sigma}(\bx) - \sum_{j_1,\ldots,j_{D}=1}^{m}  \prod_{k=1}^{D} \left\{ \widetilde v_{j_k} \phi \left( \widetilde y_{j_k}^{(k)} \right) \right\}  F_{(j_1,\ldots,j_{D})}   \right\vert
\\
& \leq \epsilon \left(1+ \sum_{i=2}^{D}  \sum_{j_1,\ldots,j_{i-1} = 1}^{m} \frac{\vert \prod_{k=1}^{i-1} \widetilde v_{j_k} \vert}{(2\pi)^{\frac{i-1}{2}}} \right)
\\
& \leq D_{3} K \left( \sigma^{\tau_{\rm tail}} + m^{-\beta} \{\log(1/\sigma) \}^{(\tau_{\rm bd}+\frac{1}{2} )(\beta + 1 )} \right)\{\log(1/\sigma) \}^{(\tau_{\rm bd}+\frac{1}{2} )(D-1)}, \label{eqn:assert1} \end{split} \ee
where $D_{3} =   D D_{2} (D_{1}/\sqrt{2\pi})^{D-1}$.

Note that
\bean
&& \nabla p_{\mu, \sigma} (\bx) = \int_{\| \by \|_{\infty} \leq 1} \left( \frac{\mu \by - \bx }{\sigma^2} \right) p_0(\by)  \phi_{\sigma}(\bx - \mu \by) \d \by 
\\
&& =  \sigma^{-1} \mu^{-D} \int_{\| \bx + \sigma \by \|_{\infty} \leq \mu} (y_1,\ldots,y_{D})^{\top}  p_0 \left( \frac{\bx + \sigma \by}{\mu} \right) \prod_{i=1}^{D} \phi(y_i) \d \by.
\eean
For $i \in [D]$, consider a one-dimensional real-valued function $\widetilde g_i$ such that
\bean
\widetilde g_i \left(y; \by_{<i},\by_{>i} \right) =  y g_i \left(y; \by_{<i},\by_{>i} \right), \quad y \in \left[(-x_i-\mu)/\sigma,(-x_i+\mu)/\sigma \right].
\eean
Then,
\be \begin{split}
& \left\vert \int_{ \frac{-x_i-\mu}{\sigma}}^{\frac{-x_i+\mu}{\sigma} } \widetilde g_i \left(y; \by_{<i},\by_{>i} \right) \d y -  \int_{ (-x_i-\mu)D_{\sigma}}^{ (-x_i+\mu)D_{\sigma}} \widetilde g_i \left(y; \by_{<i},\by_{>i} \right) \d y \right\vert
\\
& \leq \left\vert \int_{ \frac{-x_i-\mu}{\sigma}}^{(-x_i-\mu)D_{\sigma}} \widetilde g_i \left(y; \by_{<i},\by_{>i} \right) \d y  \right\vert
+ \left\vert \int^{ \frac{-x_i+\mu}{\sigma}}_{(-x_i+\mu)D_{\sigma}} \widetilde g_i \left(y; \by_{<i},\by_{>i} \right) \d y \right\vert
\\
& \leq K \int^{\infty}_{(x_i+\mu)D_{\sigma}  } y\phi(y) \d y + K \int_{(-x_i+\mu)D_{\sigma} }^{\infty} y\phi(y) \d y
\\
& = K \exp \left( -\frac{ (x_i+\mu)^2 D_{\sigma}^2}{2} \right) + K \exp \left( -\frac{ (-x_i+\mu)^2 D_{\sigma}^2 }{2} \right)
\\
& \leq 2K \exp\left(-\tau_{\rm tail} \log(1/\sigma)\right) = 2K \sigma^{\tau_{\rm tail}},   \label{eqn:tailerrorgrad}
\end{split} \ee
where the first equality holds because 
\bean
-x_i - \mu \leq -\{\log(1/\sigma)\}^{-\tau_{\rm bd}}/2 < 0
\text{\quad and \quad}
-x_i + \mu \geq 2^{-1} \{\log(1/\sigma)\}^{-\tau_{\rm bd}} > 0.
\eean
Since $\phi'(y) = -y\phi(y)$ for $y \in \bbR$ and $\sigma/\mu < 1$, Lemma~\ref{sec:normalholder} implies that $\widetilde g_i \in \cH_{1}^{\beta,KC_{G,1}}((-x_i-\mu)D_{\sigma},(-x_i+\mu)D_{\sigma}])$.
Moreover, Lemma~\ref{sec:qde} implies that
\bean
&& \left\vert \int_{(-x_i-\mu)D_{\sigma}}^{(-x_i+\mu)D_{\sigma}} \widetilde g_i \left(y ; \by_{<i},\by_{>i} \right) \d y -  \sum_{j=1}^{m} \widetilde v_j \widetilde y_j^{(i)}  g_i \left(\widetilde y_j^{(i)} ; \by_{<i},\by_{>i} \right) \right\vert
\\
&& \leq \left\{ \frac{2^{\lfloor \beta \rfloor +1 }{{n_{\beta}}}^{\beta} }{ \lfloor \beta \rfloor !} \right\} C_{G,1} K \mu^{\beta + 1}  D_{\sigma}^{\beta + 1}  m^{-\beta}
\eean
because $\widetilde g_i(y ; \by_{<i},\by_{>i} )= y g_i(y ; \by_{<i},\by_{>i} )$.
Combining (\ref{eqn:tailerrorgrad}) with the last display, we have
\be
\left\vert \int_{ \frac{-x_i-\mu}{\sigma}}^{\frac{-x_i+\mu}{\sigma} } p_0\left( \frac{\bx + \sigma (\by_{<i},y,\by_{>i}) }{\mu}  \right) y\phi(y) \d y  -  \sum_{j=1}^{m} \widetilde v_j \widetilde y_j^{(i)}  g_i \left(\widetilde y_j^{(i)} ; \by_{<i},\by_{>i} \right) \right\vert \leq \epsilon \label{eqn:qdetotgrad}.
\ee
Combining (\ref{eqn:qdetot}) with a simple calculation, we have
\bean 
&& \left\vert \int_{\| \bx_{-i} + \sigma \by_{-i} \|_{\infty} \leq \mu}   p_0 \left( \frac{\bx + \sigma (\by_{<i},y,\by_{>i})}{\mu} \right) \prod_{\substack{k=1 \\ k \neq i}}^{D} \phi(y_k) \d \by_{<i}  \right.
\\
&& \quad \quad - \left. \sum_{ \substack{ j_1,\ldots,j_{D} = 1  \\ {\rm w/o} \ j_i} }^{m} \prod_{\substack{k=1 \\ k \neq i}}^{D} \left\{\widetilde v_{j_k} \phi \left(\widetilde y_{j_k}^{(k)} \right) \right\}  p_0\left( \frac{ \bx + \sigma \left(\widetilde y_{j_1}^{(1)},\ldots,\widetilde y_{j_{i-1}}^{(i-1)},y,\widetilde y_{j_{i+1}}^{(i+1)},\ldots,\widetilde y_{j_{D}}^{(D)}\right)^{\top} }{\mu}   \right)  \right\vert
\\
&& \leq \epsilon  \left(1+  \sum_{h=1}^{D-1}  \sum_{ \substack{ j_1,\ldots,j_{h} = 1  \\ {\rm w/o} \ j_i} }^{m} \prod_{\substack{k=1 \\ k \neq i}}^{h} \left\{\vert \widetilde v_{j_k} \vert \phi \left(\widetilde y_{j_k}^{(k)} \right) \right\} \right)
\eean
for any $y \in \bbR$ satisfying $\vert x_i + \sigma y \vert \leq \mu$,
where $\sum_{ j_1,\ldots,j_{D} = 1 \ {\rm w/o} \ j_i }^{m}$ denotes the summation over 
\bean
1 \leq j_1,\ldots,j_{i-1},j_{i+1},\ldots,j_{D} \leq m.
\eean
Note that
\bean
&& \mu^{D}\sigma (\nabla p_{\mu, \sigma} (\bx))_{i} = \int_{\| \bx + \sigma \by \|_{\infty} \leq \mu} y_i  p_0 \left( \frac{\bx + \sigma \by}{\mu} \right) \prod_{k=1}^{D} \phi(y_k) \d \by
\\
&& = \int_{\frac{-x_i-\mu}{\sigma}}^{\frac{-x_i+\mu}{\sigma}} \left\{ \int_{\| \bx_{-i} + \sigma \by_{-i} \|_{\infty} \leq \mu}   p_0 \left( \frac{\bx + \sigma \by}{\mu} \right) \prod_{\substack{k=1 \\ k \neq i}}^{D} \phi(y_k) \d \by_{<i} 
\right\}  y_i \phi(y_i)  \d y_i.
\eean
Combining (\ref{eqn:qdetotgrad}) with the last three displays, we have
\bean
&&  \left\vert  \int_{\| \bx + \sigma \by \|_{\infty} \leq \mu} y_i  p_0 \left( \frac{\bx + \sigma \by}{\mu} \right) \prod_{k=1}^{D} \phi(y_k) \d \by  -  \sum_{j_1,\ldots,j_{D}=1}^{m} \widetilde y_{j_i}^{(i)} \prod_{k=1}^{D} \left\{ \widetilde v_{j_k} \phi \left( \widetilde y_{j_k}^{(k)} \right)  \right\}  F_{(j_1,\ldots,j_{D})} \right\vert
 \\
 && \leq  \epsilon  \left(1+  \sum_{h=1}^{D-1}  \sum_{ \substack{ j_1,\ldots,j_{h} = 1  \\ {\rm w/o} \ j_i} }^{m} \prod_{\substack{k=1 \\ k \neq i}}^{h} \left\{ \vert \widetilde v_{j_k} \vert \phi \left(\widetilde y_{j_k}^{(k)} \right) \right\} + \sum_{ \substack{ j_1,\ldots,j_{D} = 1  \\ {\rm w/o} \ j_i} }^{m} \prod_{\substack{k=1 \\ k \neq i}}^{D} \left\{\vert \widetilde v_{j_k} \vert \phi \left(\widetilde y_{j_k}^{(k)} \right) \right\} \right)   \int_{\frac{-x_i-\mu}{\sigma}}^{\frac{-x_i+\mu}{\sigma}} \vert y  \vert \phi(y) \d y 
 \\
 && \leq \frac{2\epsilon}{\sqrt{2\pi}} \left(1+  \sum_{h=1}^{D}  \sum_{ \substack{ j_1,\ldots,j_{h} = 1  \\ {\rm w/o} \ j_i} }^{m} \prod_{\substack{k=1 \\ k \neq i}}^{D} \frac{ \vert \widetilde v_{j_k} \vert}{\sqrt{2\pi}}  \right),
\eean
where the last inequality holds because $\int_{-\infty}^{\infty} \vert y \vert \phi(y) = 2 \int_{0}^{\infty} y \phi(y) \d y = \frac{2}{\sqrt{2\pi}}$ and $\vert \phi \vert \leq 1/\sqrt{2\pi}$.
Combining with (\ref{eqn:assert1}), the last display is bounded by 
\bean
\frac{2D_3K}{\sqrt{2\pi}} \left( \sigma^{\tau_{\rm tail}} + m^{-\beta} \{\log(1/\sigma) \}^{(\tau_{\rm bd}+\frac{1}{2} )(\beta + 1 )} \right)\{\log(1/\sigma) \}^{(\tau_{\rm bd}+\frac{1}{2} )(D-1)},
\eean
and the assertion follows by re-defining constants.
\end{proof}


\subsection{Proof of Proposition~\ref{secpt:1}}

\begin{proof}
Let $\delta > 0$ be a small enough value as described below.
There exists neural networks  $f_{\mu} \in \cF_{\rm NN}(L_{\mu},\bd_{\mu},s_{\mu},M_{\mu}), f_{\sigma} \in \cF_{\rm NN}(L_{\sigma},\bd_{\sigma},s_{\sigma},M_{\sigma})$ with
\bean
&& L_{\mu}, L_{\sigma} \leq C_{N,4} \{ \log(1/\delta)\}^{2}, \quad \|\bd_{\mu} \|_{\infty},  \| \bd_{\sigma}\|_{\infty} \leq C_{N,4}  \{ \log(1/\delta)\}^{2}
\\
&& s_{\mu}, s_{\sigma} \leq C_{N,4}  \{ \log(1/\delta)\}^{3}, \quad M_{\mu}, M_{\sigma} \leq C_{N,4} \log ( 1/\delta)
\eean
such that
\be
\left\vert \mu_{t} - f_{\mu}(t) \right\vert \leq \delta \quad \quad \text{and} \quad \quad \left\vert \sigma_{t} - f_{\sigma}(t) \right\vert \leq \delta \label{eqthm1:nnmtst}
\ee
for $ t \geq \delta$, where $C_{N,4}$ is the constant in Lemma~\ref{secnn:mtst}.
for any $0 \leq t \leq (2\overline{\tau})^{-1}$.

Since $\log (1/x) = -\log x$ for any $x > 0$, Lemma~\ref{secnn:log} implies that there exists a positive constant $D_{1} = D_{1}(\underline{\tau})$ and neural network $f_{\log} \in \cF_{\rm NN}(L_{\log}, \bd_{\log}, s_{\log}, M_{\log})$ with
\bean
&& L_{\log} \leq D_{1} \{ \log(1/\delta)\}^{2} \log \log (1/\delta), \quad \|\bd_{\log}\|_{\infty}\leq D_{1}  \{ \log(1/\delta)\}^{3}
\\
&& s_{\log}  \leq D_{1}  \{ \log(1/\delta)\}^{5} \log \log (1/\delta), \quad M_{\log} \leq \exp \left( D_{1} \{\log (1/\delta) \}^{2} \right)
\eean
such that $
\left\vert \log (1/x) - f_{\log}(\widetilde x) \right\vert \leq \sqrt{\underline{\tau} \delta}/2 + (2/\sqrt{\underline{\tau} \delta}) \vert x - \widetilde x \vert
$
for $\sqrt{\underline{\tau} \delta}/2 \leq x \leq (\sqrt{\underline{\tau} \delta}/2)^{-1}$ and $\widetilde x \in \bbR$.
Combining with (\ref{eqthm1:nnmtst}), we have
\be \begin{split}
 \left\vert \log(1/\sigma_{t}) - f_{\rm log}(f_{\sigma}(t)  ) \right\vert \leq   \frac{2 \sqrt{\delta}}{\sqrt{\underline{\tau}}} + \frac{\sqrt{\underline{\tau} \delta}}{2} = \left(\frac{2}{\sqrt{\underline{\tau}}} + \frac{\sqrt{\underline{\tau}}}{2}   \right) \sqrt{\delta} \label{eqthm1:logbd}
\end{split} \ee
for $\delta \leq t \leq (2 \overline{\tau})^{-1}$.
Let 
\bean
\tau_{\rm mult} = \tau_{\rm bd} + 1
\eean
Lemma~\ref{secnn:mult} implies that for $k \geq 2$, there exists a neural network $f_{\rm mult}^{(k)} \in \cF_{\rm NN}(L_{\rm mult}^{(k)}, \bd_{\rm mult}^{(k)}, s_{\rm mult}^{(k)}, M_{\rm mult}^{(k)})$ with 
\bean
&& L_{\rm mult}^{(k)} \leq C_{N,1} (k \tau_{\rm mult} +1)  \log k  \log (1/\delta) , \quad \| \bd_{\rm mult}^{(k)} \|_{\infty} = 48k, \quad
\\
&& s_{\rm mult}^{(k)} \leq  C_{N,1}k(\tau_{\rm mult}+1)   \log(1/\delta) , \quad M_{\rm mult}^{(k)} = \{\log (1/\delta) \}^{k\tau_{\rm mult}}  \label{eqthm1:nnmultcomp}
\eean
such that 
\be
\left\vert f_{\rm mult}^{(k)}(\widetilde x_1,\ldots, \widetilde x_k) - \prod_{i=1}^{k} x_i \right\vert \leq \delta + k \{\log (1/\delta)\}^{(k-1)\tau_{\rm mult}}\widetilde \epsilon \label{eqthm1:mult}
\ee
for any $\bx = (x_1,\ldots,x_k) \in \bbR^{k}$ with $ \|\bx\|_{\infty} \leq \{\log (1/\delta)\}^{\tau_{\rm mult} }$ and $\widetilde \bx = (\widetilde x_1,\ldots,\widetilde x_{k}) \in \bbR^{k}$ with $\| \bx - \widetilde \bx\|_{\infty} \leq \widetilde \epsilon$,
where $0 < \widetilde \epsilon \leq 1 $ and $C_{N,1}$ is the constant in Lemma~\ref{secnn:mult}.
Combining Lemma~\ref{secnn:lin} and Lemma~\ref{secnn:comp} with the last display, there exists a neural network $f_{\rm pow}^{(k)} \in \cF_{\rm NN}(L_{\rm pow}^{(k)}, \bd_{\rm pow}^{(k)},  s_{\rm pow}^{(k)} ,  M_{\rm pow}^{(k)}  )$ with
\bean
L_{\rm pow}^{(k)} =  L_{\rm mult}^{(k)} +2, \quad \| \bd_{\rm pow}^{(k)}\|_{\infty} \leq 96k, \quad s_{\rm pow}^{(k)} \leq 2s_{\rm mult}^{(k)} + 8k, \quad M_{\rm pow}^{(k)} =  M_{\rm mult}^{(k)} \vee 1
\eean
such that
\be
\left\vert f_{\rm pow}^{(k)}(\widetilde x) - x^{k} \right\vert \leq \delta + k \{ \log (1/\delta)\}^{(k-1)\tau_{\rm mult}}\widetilde \epsilon
\label{eqthm1:nnpow}
\ee
for any $\vert x \vert \leq \{\log (1/\delta)\}^{\tau_{\rm mult}}$ and $\widetilde x \in \bbR$ with $\vert x - \widetilde x \vert \leq \widetilde \epsilon$.
For $\delta \leq t \leq (2\overline{\tau})^{-1}$, we have
\be
\vert \log (1/\sigma_{t}) \vert \leq \log (1/\sqrt{\underline{\tau} \delta}) \leq \log(1/\delta), \label{eqthm1:logstbd}
\ee
where the first inequality holds by (\ref{eqthm:mtstbd}) and the last inequality holds with small enough $\delta$.
Combining (\ref{eqthm1:logbd}) with (\ref{eqthm1:nnpow}), it follows that
\be \begin{split}
& \left\vert \{ \log(1/\sigma_t)\}^{\tau_{\rm bd}+\frac{1}{2}} -   f_{\rm pow}^{(\tau_{\rm bd}+\frac{1}{2})} \left( f_{\log} \left( f_{\sigma}(t) \right) \right) \right\vert
\\
& \leq \delta + \left(\tau_{\rm bd}+\frac{1}{2} \right)\left(\frac{2}{\sqrt{\underline{\tau}}} + \frac{\sqrt{\underline{\tau}}}{2}   \right) \sqrt{\delta} \left\{ \log(1/\delta) \right\}^{(\tau_{\rm bd}-\frac{1}{2} )\tau_{\rm mult}} 
\\
& \leq D_{2} \sqrt{\delta} \{ \log(1/\delta) \}^{(\tau_{\rm bd}-\frac{1}{2} )\tau_{\rm mult}} \label{eqthm1:nnwj1}
\end{split} \ee
for $\delta \leq t \leq (2\overline{\tau})^{-1}$, where $D_{2} =1 + (\tau_{\rm bd}+\frac{1}{2} )(\frac{2}{\sqrt{\underline{\tau}}} + \frac{\sqrt{\underline{\tau}}}{2}  )$.
Combining (\ref{eqthm1:mult})0 with the last two displays, we have
\bean
&& \left\vert \{ \log(1/\sigma_t)\}^{\tau_{\rm bd}+\frac{1}{2}} x -   f_{\rm mult}^{(2)}\left( f_{\rm pow}^{(\tau_{\rm bd}+\frac{1}{2})} \left( f_{\log} \left( f_{\sigma}(t) \right) \right), x \right) \right\vert 
\\
&& \leq \delta + 2 D_{2} \sqrt{\delta}  \{\log (1/\delta)\}^{(\tau_{\rm bd}+\frac{1}{2} )\tau_{\rm mult}} \leq D_{3} \sqrt{\delta}  \{\log (1/\delta)\}^{(\tau_{\rm bd}+\frac{1}{2} )\tau_{\rm mult}}
\eean
for $\delta \leq  t \leq (2\overline{\tau})^{-1}$ and $\vert x \vert \leq 1$, where $D_{3} = 1+2D_{2}$.
Let $m \in {n_{\beta}} \bbN$ be a large enough value as described below.
Then, consider functions $g_{\rm y}^{(1)},\ldots,g_{\rm y}^{(m)} : [-1,1] \times [0,\infty) \to \bbR$ and $ f_{\rm y}^{(1)},\ldots,f_{\rm y}^{(m)} : \bbR \times \bbR \to \bbR$  such that
\bean
&& g_{\rm y}^{(j)}(x,t) 
\\
&& = 2\sqrt{2 \tau_{\rm tail}} \{\log (1/\sigma_t) \}^{\tau_{\rm bd} + \frac{1}{2}} 
 \left\{  - x-\mu_t + \frac{\mu_t{n_{\beta}}}{m} \left(  \widetilde x_{j- n_{\beta} \lfloor \frac{j}{{n_{\beta}}}  \rfloor} + 2 \left\lfloor \frac{j}{{n_{\beta}}} \right\rfloor +1\right)  \right\}, \quad j \in [m]
 \eean
 for $x \in [-1,1], t \in [0,\infty)$ and
 \bean
 f_{\rm y}^{(j)}(x,t) && =    2\sqrt{2 \tau_{\rm tail}}  \left\{ -  f_{\rm mult}^{(2)}\left( f_{\rm pow}^{(\tau_{\rm bd}+\frac{1}{2})} \left( f_{\log} \left( f_{\sigma}(t) \right) \right), x \right)  - f_{\rm mult}^{(2)}\left( f_{\rm pow}^{(\tau_{\rm bd}+\frac{1}{2})} \left( f_{\log} \left( f_{\sigma}(t) \right) \right), f_{\mu}(t) \right) \right.
\\
&& \left. \quad   + 
\frac{{n_{\beta}}}{m} \left(  \widetilde x_{j- n_{\beta} \lfloor \frac{j}{{n_{\beta}}}  \rfloor} + 2 \left\lfloor \frac{j}{{n_{\beta}}} \right\rfloor +1\right)
f_{\rm mult}^{(2)}\left( f_{\rm pow}^{(\tau_{\rm bd}+\frac{1}{2})} \left( f_{\log} \left( f_{\sigma}(t) \right) \right), f_{\mu}(t) \right) \right\}
,\quad j \in [m]
\eean
for $x,t \in \bbR$, where $\{(\widetilde x_j,\widetilde w_j) : j \in [{n_{\beta}}] \}$ are the constants in Lemma~\ref{sec:qdeapprox}.
Then,
\bean
&& \left\vert  g_{\rm y}^{(j)}(x,t) - f_{\rm y}^{(j)}(x,t)  \right\vert
 \\
 && \leq 2\sqrt{2\tau_{\rm tail}} \left\{ D_{3} + \frac{D_{3}{n_{\beta}}}{m} \left( \left\vert  \widetilde x_{j- n_{\beta} \lfloor \frac{j}{{n_{\beta}}}  \rfloor} \right\vert + 2 \left\lfloor \frac{j}{{n_{\beta}}} \right\rfloor +1\right) +1  \right\} \sqrt{\delta}  \{\log (1/\delta)\}^{(\tau_{\rm bd}+\frac{1}{2} )\tau_{\rm mult}}
\eean
and
\bean
\left\vert  g_{\rm y}^{(j)}(x,t) \right\vert  
 \leq 2\sqrt{2\tau_{\rm tail}} \left\{ 2 + \frac{2{n_{\beta}}}{m} \left( \left\vert  \widetilde x_{j- n_{\beta}  \lfloor \frac{j}{{n_{\beta}}}  \rfloor} \right\vert + 2 \left\lfloor \frac{j}{{n_{\beta}}} \right\rfloor + 1 \right) \right\} \{ \log(1/\delta) \}^{\tau_{\rm bd}+\frac{1}{2} }
\eean
for $\vert x \vert \leq 1$ and $\delta \leq t \leq (4 \overline{\tau})^{-1}$, where the last inequality holds by (\ref{eqthm1:logstbd}).
Then, there exists a constant $D_{4} = D_{4}(\beta, \tau_{\rm tail}, \tau_{\rm bd}, D_{3})$ such that
\be \begin{split}
& \left\vert f_{\rm y}^{(j)}(x,t)  \right\vert \leq \left\vert  g_{\rm y}^{(j)}(x,t) - f_{\rm y}^{(j)}(x,t)  \right\vert + \left\vert  g_{\rm y}^{(j)}(x,t) \right\vert \leq D_{4}  \{ \log(1/\delta) \}^{\tau_{\rm bd}+\frac{1}{2}}, \quad \text{and}
\\
& \left\vert  g_{\rm y}^{(j)}(x,t) - f_{\rm y}^{(j)}(x,t)  \right\vert \leq D_{4} \sqrt{\delta}  \{\log (1/\delta)\}^{(\tau_{\rm bd}+\frac{1}{2} )\tau_{\rm mult}}
\label{eqthm1:nnqd}
\end{split} \ee
for $\vert x \vert \leq 1$ and $\delta \leq t \leq (2\overline{\tau})^{-1}$ with small enough $\delta$ so that $\sqrt{\delta}  \{\log (1/\delta)\}^{(\tau_{\rm bd}+\frac{1}{2} )\tau_{\rm mult}} \leq \{ \log(1/\delta) \}^{\tau_{\rm bd}+\frac{1}{2}}$.
Lemma~\ref{secnn:comp} implies that $f_{\rm pow}^{(\tau_{\rm bd}+\frac{1}{2})} \circ f_{\log} \circ f_{\sigma} \in \cF_{\rm NN}(L_{\rm  pl\sigma},\bd_{\rm  pl\sigma},s_{\rm  pl\sigma},M_{\rm  pl\sigma})$ with
\bean
&& L_{\rm pl\sigma} = L_{\rm pow}^{(\tau_{\rm bd} + \frac{1}{2})} + L_{\log} + L_{\sigma} \leq D_{5} \{\log(1/\delta)\}^{2} \log \log (1/\delta),
\\
&& \|\bd_{\rm  pl\sigma} \|_{\infty} \leq 2 \max\left(  \| \bd_{\rm pow}^{(\tau_{\rm bd} + \frac{1}{2})} \|_{\infty}, \| \bd_{\log} \|_{\infty},   \|\bd_{\sigma}\|_{\infty} \right) \leq D_{5} \{\log(1/\delta)\}^{3},
\\
&& s_{\rm  pl\sigma}  \leq 2\left(  s_{\rm pow}^{(\tau_{\rm bd} + \frac{1}{2})} +  s_{\log} + s_{\sigma} \right) \leq D_{5} \{\log(1/\delta)\}^{5} \log \log (1/\delta),
\\
&& M_{\rm  pl\sigma}  = \max\left(  M_{\rm pow}^{(\tau_{\rm bd} + \frac{1}{2})},  M_{\log}, M_{\sigma} \right) \leq \exp\left(D_{5} \{\log(1/\delta)\}^{2} \right),
\eean
where $D_{5} = D_{5}(\tau_{\rm bd}, C_{N,1}, C_{N,4}, D_{1})$.
Then, Lemma~\ref{secnn:comp}, Lemma~\ref{secnn:par} and Lemma~\ref{secnn:lin} imply that
$f_{\rm y}^{(j)} \in \cF_{\rm NN}(L_{\rm x}, \bd_{\rm x}, s_{\rm x}, M_{\rm x}^{(j)})$ for $j \in [m]$ with
\be \begin{split}
& L_{\rm x} \leq D_{6} \{\log(1/\delta)\}^{2} \log \log (1/\delta),
\quad
\| \bd_{\rm x}\|_{\infty} \leq D_{6} \{\log(1/\delta)\}^{3},
\\
& s_{\rm x} \leq D_{6} \{\log(1/\delta)\}^{5}\log \log (1/\delta),
\quad
M_{\rm x}^{(j)} \leq \exp\left(D_{6} \{\log(1/\delta)\}^{2} \right), 
\label{eqthm1:nnxcomp}
\end{split} \ee
where $D_{6} = D_{6}(\beta,\tau_{\rm tail}, C_{N,1}, C_{N,4}, D_{3}, D_{5})$.
Let $C_{N,5}$ be the constant in Lemma~\ref{secnn:rec}. Then, there exists a neural network $f_{\rm rec} \in \cF_{\rm NN}(L_{\rm rec},\bd_{\rm rec},s_{\rm rec},M_{\rm rec})$ with
\bean
&& L_{\rm rec} \leq C_{N,5} \{ \log(1/\delta)\}^{2}, \quad \|\bd_{\rm rec}\|_{\infty}\leq C_{N,5}  \{ \log(1/\delta)\}^{3}
\\
&& s_{\rm rec}  \leq C_{N,5}  \{ \log(1/\delta)\}^{4}, \quad M_{\rm rec} \leq C_{N,5} \delta^{-2}
\eean
such that $\vert x^{-1} - f_{\rm rec}(x) \vert \leq \delta$ for $x \in [\delta, 1/\delta]$.
Combining with (\ref{eqthm1:nnmtst}), we have
\be \begin{split}
& \left\vert \frac{1}{\mu_t} - f_{\rm rec}(f_{\mu}(t)) \right\vert
\leq
\left\vert \frac{1}{\mu_t} - \frac{1}{f_{\mu}(t)} \right\vert + \left\vert \frac{1}{f_{\mu}(t)} - f_{\rm rec}(f_{\mu}(t)) \right\vert
\\
& \leq \left( \mu_{t}  \wedge  f_{\mu}(t)  \right)^{-2} \vert \mu_t - f_{\mu}(t) \vert + \delta \leq 17 \delta \label{eqthm1:nnrec}
\end{split} \ee 
for $\delta \leq t \leq (2\overline{\tau})^{-1}$, where the second inequality holds because $1/4 \leq 1/2 - \delta \leq   f_{\mu}(t) $ with $\delta \leq 1/4$.
Consider functions $\widetilde{f}_{\rm y}^{(1)}, \ldots, \widetilde f_{\rm y}^{(m)} : \bbR \times \bbR \to \bbR$ such that
\bean
\widetilde f_{\rm y}^{(j)}(x,t) = f_{\rm mult}^{(2)}\left(f_{\rm rec}(f_{\mu}(t)), x + f_{\rm mult}^{(2)}\left(f_{\rm y}^{(j)}(x,t) , f_{\sigma}(t) \right) \right), \quad j \in [m]
\eean
for $x,t \in \bbR$.
Then, Lemma~\ref{secnn:comp}, Lemma~\ref{secnn:par}, Lemma~\ref{secnn:lin} and Lemma~\ref{secnn:id} imply that
$\widetilde f_{\rm y}^{(j)} \in \cF_{\rm NN}(\widetilde L_{\rm x}, \widetilde \bd_{\rm x},\widetilde  s_{\rm x}, \widetilde M_{\rm x}^{(j)})$ for $j \in [m]$ with
\be \begin{split}
& \widetilde L_{\rm x} \leq D_{7} \{\log(1/\delta)\}^{2} \log \log (1/\delta),
\quad
\| \widetilde \bd_{\rm x} \|_{\infty}  \leq D_{7} \{\log(1/\delta)\}^{3},
\\
& \widetilde s_{\rm x} \leq D_{7} \{\log(1/\delta)\}^{5}\log \log (1/\delta),
\quad
\widetilde M_{\rm x}^{(j)} \leq \exp\left(D_{7} \{\log(1/\delta)\}^{2} \right), \label{eqthm1:nntildexcomp}
\end{split} \ee
where $D_{7} = D_{7}(C_{N,4}, C_{N,5}, D_{3}, D_{6})$.
Note that both $\vert \sigma_{t} \vert$ and $\vert g_{\rm y}^{(j)}(x,t) \vert$ are less than $ \{\log(1/\delta) \}^{\tau_{\rm mult}} $ for $\vert x \vert \leq 1$ and $\delta \leq t \leq (2\overline{\tau})^{-1}$ with small enough $\delta$, due to the (\ref{eqthm:mtstbd}) and (\ref{eqthm1:nnqd}).
Then,
\bean
\left\vert \sigma_t g_{\rm y}^{(j)}(x,t) - f_{\rm mult}^{(2)}\left(f_{\rm y}^{(j)}(x,t) , f_{\sigma}(t) \right) \right\vert \leq \delta + 2 D_{4} \sqrt{\delta}  \{\log (1/\delta)\}^{(\tau_{\rm bd}+\frac{3}{2} )\tau_{\rm mult} }
\eean
for $x \in [-1,1]$ and $ \delta \leq t \leq (2\overline{\tau})^{-1}$, where the inequality holds by combining (\ref{eqthm1:nnmtst}) and (\ref{eqthm1:nnqd}) with (\ref{eqthm1:mult}).
Also, both $\vert x+\sigma_t g_{\rm y}^{(j)}(x,t) \vert$ and $\vert \mu_{t}^{-1} \vert$ are less than $ \{\log(1/\delta) \}^{\tau_{\rm mult}} $ for $\vert x \vert \leq 1$ and $\delta \leq t \leq (2\overline{\tau})^{-1}$ with small enough $\delta$, due to the  (\ref{eqthm:mtstbd}) and (\ref{eqthm1:nnqd}). 
Combining (\ref{eqthm1:nnrec}) and (\ref{eqthm1:mult}) with the last display, we have
\be \begin{split}
& \left\vert \frac{x+\sigma_t g_{\rm y}^{(j)}(x,t)  }{\mu_t} -  \widetilde f_{\rm y}^{(j)}(x,t)  \right\vert
\\
& \leq
\delta + 2 \{\log (1/\delta) \}^{\tau_{\rm mult}} \max\left(17 \delta, \delta + 2 D_{4} \sqrt{\delta}  \{\log (1/\delta)\}^{(\tau_{\rm bd}+\frac{3}{2} )\tau_{\rm mult}  } \right)
\\
& \leq D_{8} \sqrt{\delta} \{\log (1/\delta)\}^{(\tau_{\rm bd}+\frac{5}{2} )\tau_{\rm mult} } \label{eqthm1:nninput}
\end{split} \ee
for $x \in [-1,1]$ and $\delta \leq t \leq (2\overline{\tau})^{-1}$, where $D_{8} = 35 + 2D_{4}$.
Consider functions $g_{\rm w}^{(1)},\ldots,g_{\rm w}^{(m)} : [0,\infty) \to \bbR$ and $f_{\rm w}^{(1)}, \ldots, f_{\rm w}^{(m)} : \bbR \to \bbR$ such that
\bean
g_{\rm w}^{(j)}(t) = 2\sqrt{2\tau_{\rm tail}  } {n_{\beta}} \widetilde w_{j- n_{\beta}  \lfloor \frac{j}{{n_{\beta}}} \rfloor} \{ \log(1/\sigma_t)\}^{\tau_{\rm bd}+\frac{1}{2}} 
\eean
for $t \in [0,\infty)$ and
\bean
f_{\rm w}^{(j)}(t) =  2\sqrt{2\tau_{\rm tail}  } {n_{\beta}} \widetilde w_{j- n_{\beta}  \lfloor \frac{j}{{n_{\beta}}} \rfloor} 
 f_{\rm pow}^{(\tau_{\rm bd}+\frac{1}{2})} \left( f_{\log} \left( f_{\sigma}(t) \right) \right)
\eean
for $t \in \bbR$. 
By (\ref{eqthm1:nnwj1}), we have
\be \begin{split}
\left\vert g_{\rm w}^{(j)}(t) - f_{\rm w}^{(j)}(t) \right\vert \leq D_{9} \sqrt{\delta}  \{\log (1/\delta)\}^{(\tau_{\rm bd}-\frac{1}{2} )\tau_{\rm mult}  }, \label{eqthm1:nnwj}
\end{split} \ee
for $\delta \leq t \leq (2\overline{\tau})^{-1}$, where $D_{9} = 2\sqrt{2\tau_{\rm tail}  } D_{2} {n_{\beta}} \max_{j \in [{n_{\beta}}]} \widetilde w_{j}   $.
Also, Lemma~\ref{secnn:comp} and Lemma~\ref{secnn:par} implies that
$ f_{\rm w}^{(j)} \in \cF_{\rm NN}( L_{\rm w},  \bd_{\rm w},  s_{\rm w}, M_{\rm w}^{(j)})$ for $j \in [m]$ with
\be \begin{split}
&  L_{\rm w} \leq   L_{\rm mult}^{(2)} +  L_{\rm  pl\sigma} \vee L_{\mu}  \leq  D_{10} \{\log(1/\delta)\}^{2} \log \log (1/\delta),
\\
& \|  \bd_{\rm w} \|_{\infty} \leq 4 \left\{
\|\bd_{\rm mult}^{(2)}\|_{\infty} \vee \left( \|\bd_{\rm  pl\sigma}\|_{\infty} \vee \|\bd_{\mu}\|_{\infty} \right)
\right\} \leq  D_{10} \{\log(1/\delta)\}^{3} ,
\\
& s_{\rm w} \leq 4 \left\{ s_{\rm mult}^{(2)} + 2\left(L_{\rm  pl\sigma} \vee L_{\mu} \right) + s_{\rm  pl\sigma} + s_{\mu}
\right\}  \leq  D_{10} \{\log(1/\delta)\}^{5} \log \log (1/\delta)
\\
& M_{\rm w}^{(j)} \leq \max\left\{ 2\sqrt{2\tau_{\rm tail}  } {n_{\beta}} \max_{j \in [{n_{\beta}}]} \widetilde w_{j},  M_{\rm mult}^{(2)}, M_{\rm  pl\sigma}, M_{\mu}, 1 \right\} \leq \exp\left(D_{10} \{\log(1/\delta)\}^{2} \right),
\label{eqthm1:nnwcomp}
\end{split} \ee
where $D_{10} = D_{10}(\beta, \tau_{\rm tail},C_{N,1}, C_{N,4}, D_{5})$.
Consider a function $\bff_{\rm pre}  : \bbR^{D} \times \bbR \to \bbR^{3mD}$ such that
\bean
&&  \left(  \bff_{\rm pre}(\bx,t)  \right)_{3m(i-1)+3(j-1) + 1} = f_{\rm y}^{(j)}(x_i,t),
\quad \left(  \bff_{\rm pre}(\bx,t)  \right)_{3m(i-1)+3(j-1) +2} = \widetilde f_{\rm y}^{(j)}(x_i,t),
\\
&& \left(  \bff_{\rm pre}(\bx,t)  \right)_{3m(i-1)+3(j-1)+3} = f_{\rm w}^{(j)}(t),
\quad i \in [D], j \in [m]
\eean
for $\bx \in \bbR^{D}$ and $t \in \bbR$.
Combining Lemma~\ref{secnn:par} with (\ref{eqthm1:nnxcomp}), (\ref{eqthm1:nntildexcomp}) and (\ref{eqthm1:nnwcomp}),
we have $\bff_{\rm pre} \in \cF_{\rm NN}(L_{\rm pre}, \bd_{\rm pre}, s_{\rm pre}, M_{\rm pre})$ with
\bean
&& L_{\rm pre} \leq \max\left(L_{\rm x}, \widetilde L_{\rm x}, L_{\rm w} \right) \leq D_{11}\{\log(1/\delta)\}^{2} \log \log (1/\delta),
\\
&& \|\bd_{\rm pre}\|_{\infty} \leq 2\left(mD \| \bd_{\rm x}\|_{\infty} + mD \| \widetilde \bd_{\rm x}\|_{\infty} + mD \| \bd_{\rm w}\|_{\infty} \right)
\leq D_{11} m \{\log(1/\delta)\}^{3} ,
\\
&& s_{\rm pre} \leq 2\left\{ 3mD\max\left(L_{\rm x}, \widetilde L_{\rm x}, L_{\rm w} \right) + mD s_{\rm x} + mD \widetilde s_{\rm x} + mD s_{\rm w}  \right\}
\\
&& \leq D_{11} m \{\log(1/\delta)\}^{5} \log \log (1/\delta),
\\
&& M_{\rm pre} \leq \max\left(M_{\rm x}, \widetilde M_{\rm x}, M_{\rm w}, 1 \right) \leq \exp\left(D_{11} \{\log (1/\delta)\}^2 \right),
\eean
where $D_{11} = D_{11}(D,D_{6},D_{7},D_{10})$.

The assumption (\bS) implies that $p_0 = g_{2} \circ \bg_{1}$ for functions $\bg_1: [-1,1]^{D} \to [-K,K]^{\vert \cI \vert}$ and $g_2 : [-K,K]^{\vert \cI \vert} \to \bbR$, where
$\bg_1 = (g_{11},\ldots, g_{1 \vert \cI \vert})$ with $g_{1i} \in \cH^{\beta,K}([-1,1]^{\vert I \vert}), I \in \cI$ and $g_2(x_1,\ldots,x_{\vert \cI \vert}) = \prod_{i=1}^{\vert \cI \vert} x_i$ for $x_1,\ldots,x_{\vert \cI \vert} \in [-K,K]$. A simple calculation yields that $g_2 \in \cH^{\gamma, \widetilde K}([-K,K]^{\vert \cI \vert})$ with $\widetilde K = (2K)^{\vert \cI \vert}$ for any $\gamma \geq \vert \cI \vert + 1$.
Since $\vert \cI \vert \leq 2^{D}$, Lemma 5 of \citet{chae2023likelihood} implies that there exists neural networks $f_{p_0} \in \cF_{\rm NN}(L_{p_0},\bd_{p_0},s_{p_0},M_{p_0})$ with
\bean
L_{p_0} \leq D_{12} \log m, \quad \| \bd_{p_0}\|_{\infty} \leq D_{12} m, \quad s_{p_0} \leq D_{12} m \log m, \quad M_{p_0} \leq 1
\eean
such that $ 
\vert p_0(\bx) - f_{p_0}(\bx) \vert \leq m^{-\frac{\beta}{d}}$
for $\| \bx \|_{\infty} \leq 1$, where $D_{12} = D_{12}(\beta,d,D,K)$. 
Since $p_0 \in \cH^{\beta,K}([-1,1]^{D})$, we have $\vert p_0(\bx) - p_0(\widetilde \bx)\vert \leq K D \| \bx - \widetilde \bx \|_{\infty}^{\beta \wedge 1}$ for any $\bx, \widetilde \bx \in [-1,1]^{D}$.
Let $\widetilde C_{2}$ be the constant in Lemma~\ref{sec:qdeapprox}.
Then, we have
\bean
\left\vert \frac{x+\sigma_t g_{\rm y}^{(j)}(x,t)  }{\mu_t} \right\vert \leq 1-\frac{ \{ \log(1/\sigma_{t}) \}^{-\tau_{\rm bd}}}{2} < 1
\eean
for $\vert x \vert \leq \mu_{t} - \mu_{t} \{ \log (1/\sigma_t)\}^{-\tau_{\rm bd}}$ and $\delta \leq t \leq \overline{\tau}^{-1}(\widetilde C_{2}^2 \wedge 1/4)$.
Combining with (\ref{eqthm1:nninput}), we have
\be \begin{split}
& \left\vert p_0 \left( \frac{\bx + \sigma_{t} \left( g_{\rm y}^{(j_1)}(x_1,t),\ldots,g_{\rm y}^{(j_D)} (x_D,t)  \right)^{\top} }{\mu_t} \right) - f_{p_0} \left( \widetilde f_{\rm y}^{(j_1)}(x_1,t),\ldots,\widetilde f_{\rm y}^{(j_D)}(x_D,t) \right) \right\vert
\\
& \leq K D D_{8}^{\beta \wedge 1} {\delta}^{\frac{\beta \wedge 1}{2}} \{\log(1/\delta)\}^{(\beta \wedge 1)(\tau_{\rm bd}+\frac{5}{2} )\tau_{\rm mult}}  + m^{-\frac{\beta}{d}} \defeq \epsilon_{p_0}, \quad j_{1}, \ldots, j_{D} \in [m] \label{eqthm1:nnp0}
\end{split}
\ee
for $\| \bx\|_{\infty} \leq \mu_t - \mu_t \{ \log (1/\sigma_{t})\}^{-\tau_{\rm bd}}$ and $\delta \leq t \leq \overline{\tau}^{-1}(\widetilde C_{2}^2 \wedge 1/4)$.
Let $C_{N,3}$ be the constant in Lemma~\ref{secnn:exp}. Then, there exists a neural network $f_{\rm exp} \in \cF_{\rm NN} ( L_{\rm exp} , \bd_{\rm exp} , s_{\rm exp} , M_{\rm exp} )$ with
\bean
&& L_{\exp} \leq C_{N,3} \log(1/\delta) \log \log (1/\delta), \quad \|\bd_{\exp}\|_{\infty} \leq C_{N,3} \{\log(1/\delta)\}^{3},
\\
&& s_{\exp} \leq C_{N,3} \{\log(1/\delta)\}^{4}, \quad M_{\exp} \leq C_{N,3} \delta^{-1},
\eean
such that $\vert e^{-x} - f_{\exp}(\widetilde x) \vert \leq \delta + \vert x - \widetilde x \vert$ for any $x \geq 0$ and $\widetilde x \in \bbR$.
Consider a function $f_{\phi} : \bbR^{D} \to \bbR$ such that
\bean
f_{\phi}(\bx) = (2\pi)^{-\frac{D}{2}}f_{\exp}\left( \sum_{i=1}^{D} \frac{f_{\rm pow}^{(2)}(x_i)}{2}  \right)
\eean
for $\bx = (x_1,\ldots,x_{D}) \in \bbR^{D}$.
Then, Lemma~\ref{secnn:comp}, Lemma~\ref{secnn:par} and Lemma~\ref{secnn:lin} imply that
$f_{\phi} \in \cF_{\rm NN}(L_{\phi}, \bd_{\phi}, s_{\phi}, M_{\phi} )$ with
\bean 
&& L_{\phi} \leq D_{13} \log(1/\delta) \log \log (1/\delta), 
\quad \| \bd_{\phi} \|_{\infty}
\leq  D_{13} \{ \log(1/\delta) \}^{3}
\\
&& s_{\phi} 
\leq  D_{13} \{ \log(1/\delta) \}^{4}, 
\quad M_{\phi}  \leq   D_{13} \delta^{-1},
\eean
where $D_{13} = D_{13}(D, C_{N,1}, C_{N,3})$.
Combining (\ref{eqthm1:nnqd}) with (\ref{eqthm1:nnpow}), it follows that
\be \begin{split}
& \left\vert \prod_{i=1}^{D} \phi \left(g_{\rm y}^{(j_i)}(x_i,t) \right)  - f_{\phi}\left(f_{\rm y}^{(j_1)}(x_1,t),\ldots,f_{\rm y}^{(j_D)}(x_D,t)  \right) \right\vert
\\
& \leq (2\pi)^{-\frac{D}{2}} \left[ \delta + \left\vert \sum_{i=1}^{D} \frac{ \left\{g_{\rm y}^{(j_i)}(x_i,t) \right\}^2 - f_{\rm pow}^{(2)}\left( f_{\rm y}^{(j_i)}(x_i,t) \right)  }{2} \right\vert \right]
\\
& \leq (2\pi)^{-\frac{D}{2}} \left[  \delta + \frac{1}{2} \sum_{i=1}^{D} \left( \delta + 2\{ \log(1/\delta)\}^{\tau_{\rm mult}} \left\vert g_{\rm y}^{(j_i)}(x_i,t) - f_{\rm y}^{(j_i)}(x_i,t) \right\vert   \right) \right]
\\
& \leq D_{14} \sqrt{\delta} \{ \log(1/\delta)\}^{(\tau_{\rm bd}+\frac{3}{2} )\tau_{\rm mult}}  \defeq \epsilon_{\phi}, \quad j_1,\ldots,j_{D} \in [m] \label{eqthm1:nnpsi}
\end{split}
\ee
for $\|\bx\|_{\infty} \leq 1$ and $\delta \leq t \leq (2\overline{\tau})^{-1}$, where $D_{14} = D_{14}(D,D_{4})$.
Combining (\ref{eqthm1:nnwj}) with (\ref{eqthm1:mult}), we have
\be \begin{split}
& \left\vert \prod_{i=1}^{D} g_{\rm w}^{(j_i)}(t) - f_{\rm mult}^{(D)}\left( f_{\rm w}^{(j_1)}(t),\ldots,f_{\rm w}^{(j_D)}(t) \right)  \right\vert
\\
& \leq \delta + D D_{9} \sqrt{\delta} \{\log (1/\delta) \}^{( \tau_{\rm bd} + D - \frac{3}{2}) \tau_{\rm mult} }
\\
&
 \leq D_{15}  \sqrt{\delta} \{\log (1/\delta) \}^{ ( \tau_{\rm bd} + D - \frac{3}{2}) \tau_{\rm mult}  } \defeq \epsilon_{\rm w} \label{eqthm1:nnwprod}
\end{split} \ee
for $\delta \leq t \leq (2\overline{\tau})^{-1}$ with small enough $\delta$ so that $\vert g_{\rm w}^{(j_i)}(t) \vert \leq \{ \log (1/\delta)\}^{\tau_{\rm bd}}$ for each $j_i \in [m]$, where $D_{15} = 1+DD_{9}$.

Note that $\vert p_0(\bx) \vert \leq K$ for any $\|\bx\|_{\infty} \leq 1$, $\vert \phi(x) \vert \leq  1/\sqrt{2\pi}$ for any $x \in \bbR$ and $\vert g_{\rm y}^{(j)}(\bx,t)\vert, \vert g_{\rm w}^{(j)}(t) \vert \leq D_{16} \{\log(1/\delta)\}^{\tau_{\rm bd}+\frac{1}{2}}$ for $j \in [m]$, $\|\bx\|_{\infty} \leq 1, \delta \leq t \leq (4 \overline{\tau})^{-1}$, where $D_{16} = 2\sqrt{2\tau_{\rm tail}} (2 + {n_{\beta}}) \max_{j \in [{n_{\beta}}]} \widetilde w_j$.
Let
$\widetilde{f}_{\rm mult}^{(2)} \in \cF_{\rm NN}(\widetilde{L}_{\rm mult}^{(2)} ,\widetilde{\bd}_{\rm mult}^{(2)} ,\widetilde{s}_{\rm mult}^{(2)} ,\widetilde{M}_{\rm mult}^{(2)} )$ be the neural network in Lemma~\ref{secnn:mult}, 
with
\bean
&& \widetilde{L}_{\rm mult}^{(2)} \leq C_{N,1} \log 2 \left\{ (2D\tau_{\rm bd}+D +2) \log(1/\delta)+ D \log D_{16} \right\},
\\
&& \| \widetilde{\bd}_{\rm mult}^{(2)}\|_{\infty} \leq 96,
\\
&& \widetilde{s}_{\rm mult}^{(2)} \leq 2 C_{N,1} \left\{(D\tau_{\rm bd}+D/2+1) \log(1/\delta) + D\log D_{16}   \right\},
\\
&& \widetilde{M}_{\rm mult}^{(2)} = D_{16}^{2D} \{ \log(1/\delta) \}^{2D(\tau_{\rm bd}+\frac{1}{2})}
\eean
such that 
\bean
\vert \widetilde{f}_{\rm mult}^{(2)}(\widetilde \bx) - x_1 x_2 \vert \leq \delta + 2D_{16}^{D} \{ \log(1/\delta) \}^{D(\tau_{\rm bd}+\frac{1}{2})}  \widetilde \epsilon
\eean
for all $\| \bx\|_{\infty} \leq D_{16}^{D} \{ \log(1/\delta) \}^{D(\tau_{\rm bd}+\frac{1}{2})}$, $\widetilde \bx \in \bbR^2$ with $\| \bx - \widetilde \bx \|_{\infty} \leq \widetilde \epsilon$.
Also, let 
\bean
\overline{f}_{\rm mult}^{(3)} \in \cF_{\rm NN} ( \overline{L}_{\rm mult}^{(3)} ,  \overline{\bd}_{\rm mult}^{(3)} , \overline{s}_{\rm mult}^{(3)} ,\overline{M}_{\rm mult}^{(3)} )
\eean
be the neural network in Lemma~\ref{secnn:mult} with
\bean
&& \overline{L}_{\rm mult}^{(3)} \leq C_{N,1} \log 3 \left\{\log(1/\delta) + 3 \log  (K\vee1) \right\}, \quad \| \overline{\bd}_{\rm mult}^{(3)} \|_
{\infty} \leq 144,
\\
&& \overline{s}_{\rm mult}^{(3)} \leq 3 C_{N,1} \left\{\log(1/\delta) +  \log  (K\vee1) \right\}, \quad \overline{M}_{\rm mult}^{(3)} = K^{3} \vee 1
\eean
such that 
\bean
\vert \overline{f}_{\rm mult}^{(3)}(\widetilde \bx) - x_1 x_2 x_3 \vert \leq \delta + 3(K^2 \vee 1)\widetilde \epsilon
\eean
for all $\|\bx\|_{\infty} \leq K \vee 1$, $\widetilde \bx \in \bbR^3$ with $\| \bx - \widetilde \bx \|_{\infty} \leq \widetilde \epsilon$.
Consider functions $\bff_{\rm main} : \bbR^{D} \times \bbR^{D} \times \bbR^{D} \to \bbR^{D+1}$ such that
\bean
&& \left( \bff_{\rm main}(\bx,\widetilde \bx, \bw) \right)_{i} = \widetilde{f}_{\rm mult}^{(2)}\left( \overline{ f}_{\rm mult}^{(3)} \left(x_i, f_{p_0}(\widetilde \bx), f_{\phi}(\bx)\right), f_{\rm mult}^{(D)}(\bw)\right), \quad i \in [D],
\\
&& \left( \bff_{\rm main}(\bx,\widetilde \bx, \bw) \right)_{D+1} = \widetilde{f}_{\rm mult}^{(2)} \left( \overline{f}_{\rm mult}^{(3)}\left(1, f_{p_0}(\widetilde \bx), f_{\phi}(\bx) \right), f_{\rm mult}^{(D)}(\bw) \right)
\eean
for $\bx = (x_1,\ldots,x_{D}) \in \bbR^{D}$ and $\widetilde \bx, \bw \in \bbR^{D}$. Lemma~\ref{secnn:comp}, Lemma~\ref{secnn:par} and Lemma~\ref{secnn:id} implies that $\bff_{\rm main} \in \cF_{\rm NN}(L_{\rm main}, \bd_{\rm main}, s_{\rm main}, M_{\rm main})$ with
\bean
&& L_{\rm main} \leq D_{17} \left[ \log m + \log(1/\delta) \log \log (1/\delta) \right], \quad \| \bd_{\rm main} \|_{\infty} \leq D_{17} \left[ m + \{\log(1/\delta)\}^{3}  \right]
\\
&& s_{\rm main} \leq D_{17} \left[ m\log m +   \{\log(1/\delta)\}^{4}  \right], \quad M_{\rm main} \leq D_{17} \delta^{-1},
\eean
where $D_{17} = D_{17}(D,K,C_{N,1},D_{10}, D_{12}, D_{13}, D_{16})$.
For $\bj = ( j_{1},\ldots,j_{D}) \in [m]^{D}$, consider a function $\bff^{(\bj)} : \bbR^{D} \times \bbR \to \bbR^{3D}$ such that
\bean
&& \bff^{(\bj)}(\bx,t) 
\\
&& = \left(f_{\rm y}^{(j_1)}(x_1,t),\ldots,f_{\rm y}^{(j_D)}(x_D,t)
, \widetilde f_{\rm y}^{(j_1)}(x_1,t),\ldots,\widetilde f_{\rm y}^{(j_D)}(x_D,t)
, f_{\rm w}^{(j_1)}(t),\ldots,f_{\rm w}^{(j_D)}(t) \right)^{\top}.
\eean
By (\ref{eqthm1:nnp0}), (\ref{eqthm1:nnpsi}) and (\ref{eqthm1:nnwprod}), we have
\bean
&& \left\vert  p_0 \left( \frac{\bx + \sigma_{t} \left( g_{\rm y}^{(j_1)}(x_1,t),\ldots,g_{\rm y}^{(j_D)} (x_D,t)  \right)^{\top} }{\mu_t} \right) \prod_{i=1}^{D} \phi\left( g_{\rm y}^{(j_i)}(x_i,t) \right)g_{\rm w}^{(j_i)}(t) \right.
\\
&& \quad - \left. \left(  \bff_{\rm main}\left(  \bff^{(\bj)}(\bx,t) \right)\right)_{D+1}   \right\vert
\\
&& \leq 
\delta + 2D_{16}^{D} \{ \log (1/\delta) \}^{D(\tau_{\rm bd}+\frac{1}{2})} \max \left\{ \delta + 3(K^2 \vee 1) \left(\epsilon_{p_0} \vee \epsilon_{\phi} \right)   , 
\epsilon_{\rm w} \right\}
\\
&& \leq D_{18} \{ \log (1/\delta) \}^{D(\tau_{\rm bd}+\frac{1}{2})}\left[ m^{-\frac{\beta}{d}} + {\delta}^{\frac{\beta \wedge 1}{2}} \{\log(1/\delta)\}^{( \tau_{\rm bd} + D + \frac{3}{2}) \tau_{\rm mult}  }  \right]
\eean
and
\bean
&& \left\vert  g_{\rm y}^{(j_k)}(x_k,t)  p_0 \left( \frac{\bx + \sigma_{t} \left( g_{\rm y}^{(j_1)}(x_1,t),\ldots,g_{\rm y}^{(j_D)} (x_D,t)  \right)^{\top} }{\mu_t} \right) \prod_{i=1}^{D} \phi\left( g_{\rm y}^{(j_i)}(x_i,t) \right)g_{\rm w}^{(j_i)}(t) \right.
\\
&& \quad \left. - \left(\bff_{\rm main}\left( \bff^{(\bj)}(\bx,t) \right)\right)_{k}   \right\vert
\\
&& \leq 
\delta + 2D_{16}^{D} \{ \log (1/\delta) \}^{D(\tau_{\rm bd}+\frac{1}{2})} \max \left\{ \delta + 3(K^2 \vee 1) \left(\epsilon_{p_0} \vee \epsilon_{\phi} \right)   , 
\epsilon_{\rm w} \right\}
\\
&& \leq D_{18} \{ \log (1/\delta) \}^{D(\tau_{\rm bd}+\frac{1}{2})}\left[ m^{-\frac{\beta}{d}} + {\delta}^{\frac{\beta \wedge 1}{2}} \{\log(1/\delta)\}^{ ( \tau_{\rm bd} + D + \frac{3}{2}) \tau_{\rm mult} }  \right]
, \quad k \in [D], \ j_1,\ldots,j_{D} \in [m]
\eean
for $\|\bx\|_{\infty} \leq \mu_{t} - \mu_{t} \{ \log (1/\sigma_t) \}^{-\tau_{\rm bd}}$ and $\delta \leq t \leq \overline{\tau}^{-1}(\widetilde C_{2}^2 \wedge 1/4)$, where 
\bean
D_{18} = D_{18} ( \beta, D, K, D_{8}, D_{14}, D_{15}, D_{16} ).
\eean
Let $\widetilde C_{1}$ be the constant in Lemma~\ref{sec:qdeapprox}. 
It follows that
\bean
&& \mu_{t}^{D} \left\| \frac{1}{m^{D}} \sum_{\bj \in [m]^{D}} \bff_{\rm main}\left( \bff^{(\bj)}(\bx,t) \right) 
- \begin{pmatrix}
\sigma_t  \nabla p_t(\bx)
\\
  p_t(\bx)
\end{pmatrix}
\right\|_{\infty}
\\
&& \leq \widetilde C_{1} K \{ \log (1/\sigma_{t})\}^{(\tau_{\rm bd} + \frac{1}{2} )(D-1)}  \left( \sigma_{t}^{\tau_{\rm tail}} + m^{-\beta} \{ \log (1/\sigma_{t})\}^{(\tau_{\rm bd} + \frac{1}{2} )(\beta + 1 )} \right)
\\
&& \quad + D_{18}  \{ \log (1/\delta) \}^{D(\tau_{\rm bd}+\frac{1}{2})}\left[ m^{-\frac{\beta}{d}} + {\delta}^{\frac{\beta \wedge 1}{2}} \{\log(1/\delta)\}^{ ( \tau_{\rm bd} + D + \frac{3}{2}) \tau_{\rm mult} }  \right]
\eean
for $\delta \leq t \leq \overline{\tau}^{-1} (\widetilde C_{2}^2 \wedge 1/2)$.
Let 
\bean
m = n_{\beta} \lfloor \widetilde m \rfloor
\quad \text{and} \quad
\delta = \widetilde m^{- \tau_{\min}}
\eean
with large enough $\widetilde m > 0$.
Since $\tau_{\rm min} \geq \frac{4\beta}{d(\beta \wedge 1)}$, we have ${\delta}^{\frac{\beta \wedge 1}{2}} \{\log(1/\delta)\}^{( \tau_{\rm bd} + D + \frac{3}{2}) \tau_{\rm mult}} \leq \widetilde m^{-\frac{\beta}{d}}$ for large enough $\widetilde m$. Then,
\be \begin{split}
& \left\| \frac{1}{m^{D}} \sum_{\bj \in [m]^{D}} \bff_{\rm main}\left( \bff^{(\bj)}(\bx,t) \right) 
- \begin{pmatrix}
\sigma_t  \nabla p_t(\bx)
\\
  p_t(\bx)
\end{pmatrix}
\right\|_{\infty}
\\
& \leq D_{19} \left( \log \widetilde m \right)^{(\tau_{\rm bd}+\frac{1}{2}) (D-1)  } \left\{
t^{\frac{\tau_{\rm tail}}{2}} + \widetilde m^{-\frac{\beta}{d}} \left( \log \widetilde m \right)^{(\tau_{\rm bd}+\frac{1}{2})(\beta + 1 )} \right\}, \label{eqthm1:error}
\end{split} \ee
where $D_{19} = D_{19}(\beta, K, D, d, \tau_{\rm bd}, \tau_{\rm tail}, 
\tau_{\rm min}, \overline{\tau}, \underline{\tau}, \widetilde C_{1}, D_{18})$.

  Recall that $\bff_{\rm pre} : \bbR^{D} \times \bbR \to \bbR^{3mD} $ and $\bff_{\rm main} : \bbR^{3D} \to \bbR^{D+1}$.
Define a function $\widetilde \bff_{\rm pre} : \bbR^{D} \times \bbR \to \bbR^{6mD}$ as
\bean
\widetilde \bff_{\rm pre}(\bx,t) = \rho \begin{pmatrix}
\bff_{\rm pre}(\bx,t)
\\
-\bff_{\rm pre}(\bx,t)
\end{pmatrix}
\eean
for $\bx \in \bbR^{D}$ and $t \in \bbR$.
For each $\bj = (j_1,\ldots,j_{D}) \in [m]^{D}$, there exists a $6mD \times 6mD $ permutation matrix $Q_{1}^{(\bj)} $ such that 
\bean
Q_{1}^{(\bj)} \widetilde  \bff_{\rm pre} (\bx,t)
=
\rho \begin{pmatrix}
 \bff^{(\bj)}(\bx,t) \\
- \bff^{(\bj)}(\bx,t) \\
\mathbf{0}_{6mD - 6D }
 \end{pmatrix} \in \bbR^{6mD }.
\eean
Let $W_{1},\bb_{1}, \ldots, W_{L_{\rm main}},\bb_{L_{\rm main}} $ be the weight matrices and shift vectors of the neural network $\bff_{\rm main}$, where $\bd_{\rm main} = (d_1,\ldots,d_{L_{\rm main}+1})$ with $d_1 = 3D$, $d_{L_{\rm main}+1} = D + 1$ and $W_{l} \in \bbR^{d_{l+1} \times d_{l}}$, $\bb_{l} \in \bbR^{d_{l+1}}$ for $l \in [L_{\rm main}]$.
Since $x \vee 0 -( -x \vee 0) = x$ for any $x \in \bbR$, we have
\bean
 \widetilde W_{1} Q_{1}^{(\bj)} \widetilde  \bff_{\rm pre}(\bx,t) + \widetilde \bb_{1}
=  
\begin{pmatrix}
W_1 \bff^{(\bj)}(\bx,t)  + \bb_{1}
\\
\mathbf{0}_{m^{D} d_{2} - d_{2} }
\end{pmatrix}
\in \bbR^{m^{D} d_{2} }
,
\quad \forall \bj \in [m]^{D},
 \eean
 where
 \bean
 && \widetilde W_{1} = 
\begin{pmatrix}
W_1 \ & -W_1 \
& \mathbf{0}
\\
\mathbf{0} \ & \mathbf{0} \ & \mathbf{0}
\end{pmatrix} \in \bbR^{m^{D} d_{2} \times 6mD},
\qquad
\widetilde \bb_{1} =
\begin{pmatrix}
\bb_{1}
\\
\mathbf{0} 
\end{pmatrix} \in \bbR^{m^{D} d_{2} }.
 \eean
For each $\bj \in [m]^{D}$, there exists a $m^{D} d_{2} \times m^{D} d_{2}$ permutation matrix $R_{1}^{(\bj)}$ such that
\bean
R_{1}^{(\bj)}  \left(  \widetilde W_{1} Q_{1}^{(\bj)} \widetilde  \bff_{\rm pre}(\bx,t) + \widetilde \bb_{1} \right) 
=  
\begin{pmatrix}
\mathbf{0}_{ d_{2} \sum_{i=1}^{D}  m^{i-1}(j_{i}-1) }
\medskip
\\
W_1 \bff^{(\bj)}(\bx,t)  + \bb_{1}
\\
\medskip
\mathbf{0}_{ m^{D} d_{2} - d_{2} \sum_{i=1}^{D}  m^{i-1}(j_{i}-1) -
 d_{2} }
\end{pmatrix}
\in \bbR^{m^{D} d_{2}}.
\eean
It follows that
\bean
&&  \left( \sum_{ \bj \in [m]^{D}} R_{1}^{(\bj)}  \left(  \widetilde W_{1} Q_{1}^{(\bj)} \widetilde  \bff_{\rm pre}(\bx,t) + \widetilde \bb_{1} \right) \right)_{ d_{2} \sum_{i=1}^{D}  m^{i-1}(\widetilde j_{i}-1) +1 : d_{2} \sum_{i=1}^{D}  m^{i-1}(\widetilde j_{i}-1) + d_{2} }
\\
&& =
W_1 \bff^{(\widetilde \bj)}(\bx,t)  + \bb_{1}, \quad \forall  \widetilde \bj = ( \widetilde j_{1},\ldots,\widetilde j_{D} ) \in [m]^{D}.
\eean

Let $\widetilde \bd = (\widetilde d_{1},\ldots, \widetilde d_{L_{\rm main +2}})$, where $\widetilde d_{1} = 2md_{1} = 6 m D$, $\widetilde d_{l} = m^{D} d_{l}, l \in \{2,\ldots, L_{\rm main}\}, \widetilde d_{L_{\rm main}+1} = 2m^{D} d_{L_{\rm main}+1} = 2m^{D} (D + 1) $ and $\widetilde d_{L_{\rm main}+2} = 2 d_{L_{\rm main}+1} = 2 ( D + 1)$.
For each $l \in \{ 2, \ldots, L_{\rm main} - 1 \}$, define block-sparse matrix and vector as
\bean
&& \widetilde W_{l} = 
\begin{pmatrix}
W_{l} & \  \mathbf{0}
\\
\mathbf{0} & \ \mathbf{0}
\end{pmatrix} \in \bbR^{\widetilde d_{l+1} \times \widetilde d_{l}}
\quad \text{and} \quad
\widetilde \bb_{l} = 
\begin{pmatrix}
\bb_{l} 
\\
\mathbf{0}
\end{pmatrix} \in \bbR^{\widetilde d_{l+1}}.
\eean
Moreover, for each  $l \in \{ 2, \ldots, L_{\rm main} - 1 \}$ and $\bj \in [m]^{D}$, there exist $\widetilde d_{l} \times \widetilde d_{l}$ permutation matrix $Q_{l}^{(\bj)}$  and  $\widetilde d_{l+1} \times \widetilde d_{l+1}$ permutation matrix $R_{l}^{(\bj)}$ such that
\bean
R_{l}^{(\bj)}  \left(  \widetilde W_{l} Q_{l}^{(\bj)} \bz + \widetilde \bb_{l} \right) 
=
\rho \begin{pmatrix}
\mathbf{0}_{ d_{l+1} \sum_{i=1}^{D}  m^{i-1}(j_{i}-1) }
\medskip
\\
W_l \bz_{d_{l} \sum_{i=1}^{D}  m^{i-1}(j_{i}-1) + 1: 
d_{l} \sum_{i=1}^{D}  m^{i-1}(j_{i}-1) + d_{l}} + \bb_{l}
\\
\medskip
\mathbf{0}_{ m^{D} d_{l+1} - d_{l+1} \sum_{i=1}^{D}  m^{i-1}(j_{i}-1) -
 d_{l+1} }
\end{pmatrix},
\quad \bz \in \bbR^{\widetilde d_{l}},
\eean
where $\bz_{n_1:n_2} \in \bbR^{n_{2}-n_{1}+1}$ denotes the subvector of $\bz$ from $n_1$-th component to $n_2$-th component.
It follows that
\bean
&& \left( \widetilde \bff_{L_{\rm main}-1} \circ  \cdots  \circ  \widetilde \bff_{1} \circ \widetilde \bff_{\rm pre}(\bx,t) \right)_{ d_{L_{\rm main}} \sum_{i=1}^{D}  m^{i-1} ( j_{i}-1) +1 : d_{L_{\rm main}}  \sum_{i=1}^{D}  m^{i-1}( j_{i}-1) +d_{L_{\rm main}}  }
\\
&& \quad = \rho \left(  W_{L_{\rm main}-1} \cdot +  \bb_{L_{\rm main}-1} \right) \circ \cdots \circ
\rho \left(  W_{1} \cdot  +  \bb_{1} \right) \circ \bff^{( \bj )}(\bx,t) ,
\quad \forall \bj \in [m]^{D},
\eean
where $\widetilde \bff_{l} : \bbR^{\widetilde d_{l}} \rightarrow \bbR^{\widetilde d_{l+1}}$ is a vector-valued function defined as
\bean
\widetilde \bff_{l}(\bz) = \rho \left( \sum_{ \bj \in [m]^{D}} R_{l}^{(\bj)}  \left( 
\widetilde W_{l} Q_{l}^{(\bj)} \bz + \widetilde \bb_{l}
\right) \right), \quad l \in [L_{{\rm main}}-1].
\eean

Let $\widetilde W_{L_{\rm main}}$ and $\widetilde \bb_{L_{\rm main}}$ be the block-sparse matrix and vector, respectively, defined as
\bean
\widetilde W_{L_{\rm main}} = 
\begin{pmatrix}
W_{L_{\rm main}} & \  \mathbf{0}
\\
-W_{L_{\rm main}} & \ \mathbf{0}
\\
\mathbf{0} & \ \mathbf{0}
\end{pmatrix} \in \bbR^{\widetilde d_{L_{\rm main}+1 } \times \widetilde d_{L_{\rm main}}}
\quad \text{and} \quad
\widetilde \bb_{L_{\rm main}} = 
\begin{pmatrix}
\bb_{L_{\rm main}} 
\\
-\bb_{L_{\rm main}} 
\\
\mathbf{0}
\end{pmatrix} \in \bbR^{\widetilde d_{L_{\rm main}+1 } }.
\eean
For each $\bj \in [m]^{D}$, there exist $\widetilde d_{L_{\rm main}} \times \widetilde d_{L_{\rm main}}$ permutation matrix $Q_{L_{\rm main}}^{(\bj)}$  and  $\widetilde d_{L_{\rm main}+1} \times \widetilde d_{L_{\rm main}+1}$ permutation matrix $R_{L_{\rm main}}^{(\bj)}$ such that
\bean
&& \left(\widetilde \bff_{L_{\rm main}} \circ  \cdots  \widetilde \bff_{1} \circ \widetilde \bff_{\rm pre}(\bx,t) \right)   _{ 2d_{L_{\rm main}+1} \sum_{i=1}^{D}  m^{i-1} ( j_{i}-1) +1 : 2d_{L_{\rm main}+1}  \sum_{i=1}^{D}  m^{i-1}( j_{i}-1) +d_{L_{\rm main}+1}  }
\\
&& \quad =  \rho \left( W_{L_{\rm main}} \cdot +  \bb_{L_{\rm main}} \right) \circ  \rho \left(  W_{L_{\rm main}-1} \cdot +  \bb_{L_{\rm main}-1} \right) \circ \cdots \circ
\rho \left(  W_{1} \cdot  +  \bb_{1} \right) \circ \bff^{( \bj )}(\bx,t)
\\
&& \quad = \rho \left( \bff_{\rm main}\left( \bff^{(\bj)}(\bx,t) \right)\right)
\eean
and
\bean
&& \left( \widetilde \bff_{L_{\rm main}} \circ  \cdots \circ  \widetilde \bff_{1} \circ \widetilde \bff_{\rm pre}(\bx,t) \right)_{ 2d_{L_{\rm main}+1} \sum_{i=1}^{D}  m^{i-1} ( j_{i}-1) + d_{L_{\rm main}+1} +1 : 2d_{L_{\rm main}+1}  \sum_{i=1}^{D}  m^{i-1}( j_{i}-1) + 2d_{L_{\rm main}+1}  }
\\
&& \quad  =  \rho \left( - W_{L_{\rm main}} \cdot -  \bb_{L_{\rm main}} \right) \circ  \rho \left(  W_{L_{\rm main}-1} \cdot +  \bb_{L_{\rm main}-1} \right) \circ \cdots \circ
\rho \left(  W_{1} \cdot  +  \bb_{1} \right) \circ \bff^{( \bj )}(\bx,t)
\\
&& \quad = \rho \left( - \bff_{\rm main}\left( \bff^{(\bj)}(\bx,t) \right) \right),
\eean
where
 $\widetilde \bff_{L_{\rm main}} : \bbR^{\widetilde d_{L_{\rm main}}} \rightarrow \bbR^{\widetilde d_{L_{\rm main}+1}}$ is a vector-valued function defined as
\bean
\widetilde \bff_{L_{\rm main}}(\bz) = \rho \left( \sum_{ \bj \in [m]^{D}} R_{L_{\rm main}}^{(\bj)}  \left( 
\widetilde W_{L_{\rm main}} Q_{L_{\rm main}}^{(\bj)} \bz + \widetilde \bb_{L_{\rm main}}
\right) \right).
\eean
Let $\widetilde W_{L_{\rm main} + 1}$ be a block-sparse matrix, defined as
\bean
\widetilde W_{L_{\rm main}+1}
=
\begin{pmatrix}
m^{-D} \bbI_{D+1} \ & -m^{-D} \bbI_{D+1} \
& \mathbf{0}
\\
\mathbf{0} \ & \mathbf{0} \ & \mathbf{0}
\end{pmatrix} \in \bbR^{ \widetilde d_{L_{\rm main}+2 } \times \widetilde d_{L_{\rm main}+1 }}.
\eean
For each $\bj \in [m]^{D}$, there exist  $\widetilde d_{L_{\rm main}+1} \times \widetilde d_{L_{\rm main} + 1} $ permutation matrix $Q_{L_{\rm main} + 1}^{(\bj)}$ and $\widetilde d_{L_{\rm main}+2} \times \widetilde d_{L_{\rm main} + 2}$ permutation matrix $R_{L_{\rm main} + 1}^{(\bj)}$ such that 
\bean
\left( \bff_1 (\bx,t) \right)_{1:D+1} =  \rho \left( m^{-D} \sum_{\bj \in [m]^{D}} \bff_{\rm main}\left( \bff^{(\bj)}(\bx,t) \right)  \right)
\eean
and
\bean
\left( \bff_1 (\bx,t) \right)_{D+2:2D+2} =  \rho \left( - m^{-D} \sum_{\bj \in [m]^{D}} \bff_{\rm main}\left( \bff^{(\bj)}(\bx,t) \right)  \right),
\eean
where 
\bean
\bff_{1} (\bx,t) = \rho \left( \sum_{ \bj \in [m]^{D} } R_{L_{\rm main} + 1}^{(\bj)} \widetilde W_{L_{\rm main}+1} Q_{L_{\rm main}+1}^{(\bj)}  \widetilde \bff_{L_{\rm main}}   \circ  \cdots \circ  \widetilde \bff_{1} \circ \widetilde \bff_{\rm pre}(\bx,t) \right).
\eean
Combining with (\ref{eqthm1:error}), we have
\be \begin{split}
& \left\| \bff(\bx,t)
- \begin{pmatrix}
\sigma_t \nabla p_t(\bx)
\\
 p_t(\bx)
\end{pmatrix}
\right\|_{\infty} 
\\
& \leq D_{19} \left( \log \widetilde m \right)^{(\tau_{\rm bd}+\frac{1}{2}) (D-1)  } \left\{
t^{\frac{\tau_{\rm tail}}{2}} + \widetilde m^{-\frac{\beta}{d}} \left( \log \widetilde m \right)^{(\tau_{\rm bd}+\frac{1}{2})(\beta + 1 )} \right\}
\label{eqthm1:errorfin}
\end{split}
\ee
for $\| \bx \|_{\infty} \leq \mu_{t} - \mu_{t}\{ \log(1/\sigma_{t}) \}^{-\tau_{\rm bd} }$  and $\widetilde m^{- \tau_{\min}} \leq t \leq \overline{\tau}^{-1} (\widetilde C_{2}^2 \wedge 1/2)$,
where $\bff: \bbR^{D} \times \bbR \rightarrow \bbR^{D+1}$ is a vector-valued function, defined as
\bean
\bff(\bx,t) = \begin{pmatrix}
\bbI_{D+1} & \ -\bbI_{D+1}
\end{pmatrix} \bff_{1}(\bx,t).
\eean
Note that $\| \widetilde {W}_{l} \|_{0} = \| W_{l} \|_{0}$, $\| \widetilde \bb_{l} \|_{0} = \| \bb_{l} \|_{0}$ for $l \in \{2,\ldots,L_{\rm main}-1\}$,
$\| \widetilde W_{1} \|_{0} = 2 \| W_{1} \|_{0}$, $\| \widetilde \bb_{1} \|_{0} = 2 \| \bb_{1} \|_{0}$, $\| \widetilde W _{L_{\rm main}} \|_{0} = 2 \| W_{L_{\rm main}} \|_{0}$,
$\| \widetilde \bb_{L_{\rm main}} \|_{0} = 2 \| \bb_{L_{\rm main}} \|_{0}$ and $\| \widetilde W_{L_{\rm main}+1} \|_{0} = 2 D $.
Let $
\overline{L} = L_{\rm pre} + L_{\rm main} + 2 $ and $\overline{\bd} = 
(\overline{d}_1,\ldots, \overline{d}_{ \overline{L} +1}) \in \bbN^{ \overline{L} +1} $ with
\bean
(\overline{d}_1,\ldots,\overline{d}_{L+1}) =
(d_{\rm pre}^{(1)},\ldots,d_{\rm pre}^{(L_{\rm pre})},\widetilde d_{1},\ldots,   \widetilde d_{L_{\rm main}+2}, d_{L_{\rm main}+1}
),
\eean
where $\bd_{\rm pre} = (d_{\rm pre}^{(1)},\ldots,d_{\rm pre}^{(L_{\rm pre}+1)})$.
For $1 \leq i \leq L_{\rm pre}$, let 
$\cQ_{i}$ and $\cR_{i}$ be the set of $\overline{d}_{i} \times \overline{d}_{i}$ and $\overline{d}_{i+1} \times \overline{d}_{i+1}$ identity matrix, respectively. For $ L_{\rm pre} < i \leq L-1$, let 
\bean
\cQ_{i} = \left( Q_{i-L_{\rm pre}}^{(\bj)} \right)_{\bj \in [m]^{D}} \quad \text{and} \quad
\cR_{i} = \left( R_{i-L_{\rm pre}}^{(\bj)} \right)_{\bj \in [m]^{D}}.
\eean
Let $\bfm = (m_1,\ldots ,m_{ \overline{L} -1})$ with $ m_{i} = 1$ for $i \leq L_{\rm pre}$ and $m_{i} = m^{D}$ for $ i > L_{\rm pre}$, and $\cP = ((\cQ_{l}, \cR_{l}))_{l \in [\overline{L}-1]}$.
Then, $\bff \in \cF_{\rm WSNN}(\overline L, \overline \bd, \overline s, \overline M,\cP_{\bfm})$ with
\bean
\overline s \leq 2s_{\rm pre} + 2s_{\rm main} + 4 d_{L_{\rm main}+1} \quad \text{ and } \quad
\overline M = \max\left(M_{\rm pre}, M_{\rm main}, 1, m^{-D} \right).
\eean
Recall that $m = n_{\beta} \lfloor \widetilde m \rfloor $ and $\delta = \widetilde m^{- \tau_{\min}}$.
Thus,
\bean
&& \overline L \leq D_{20} (\log \widetilde m )^{2} \log \log \widetilde m, \quad \| \overline \bd\|_{\infty} \leq D_{20} \widetilde m^{D+1},
\\
&& \overline s \leq D_{20} \widetilde m (\log \widetilde m)^{5} \log \log \widetilde m^{D+1}, \quad \overline M \leq \exp \left( D_{20} \{ \log \widetilde m \}^{2} \right),
\eean
where $D_{20} = D_{20}(\tau_{\rm min}, D_{11}, D_{17}, n_{\beta})$.
Combining (\ref{eqthm1:errorfin}) with the last display, the assertion follows by re-defining the constants.
\end{proof}


\subsection{Proof of Proposition~\ref{secpt:2}}

\begin{proof}

Let $\tau_{\rm tail} = 2 \vee \sqrt{ (D+3)/(2e) }$.
Given small enough $\delta > 0$ as described below, we have
\bean
&& \int_{ \substack{ \| \bx-\mu_{t}\by \|_{\infty} \geq  \tau_{\rm tail} \sigma_{t}\sqrt{ \log ( 1 / \delta ) } \\ \| \by \|_{\infty} \leq 1   }} \ p_0(\by) \phi_{\sigma_{t}}(\bx-\mu_{t} \by) \d \by
\\
&& \leq K \int_{  \| \bx-\mu_{t}\by \|_{\infty} \geq  \tau_{\rm tail} \sigma_{t}\sqrt{ \log ( 1 / \delta ) }   } \ \phi_{\sigma_{t}}(\bx-\mu_{t} \by) \d \by  = K \mu_{t}^{-D} \int_{  \| \bz \|_{\infty} \geq  \tau_{\rm tail}\sqrt{ \log ( 1 / \delta ) }   } \phi_{1}(\bz) \d \bz
\\
&& \leq K \mu_{t}^{-D} \sum_{i=1}^{D} \int_{\vert z_i \vert \geq  \tau_{\rm tail} \sqrt{ \log ( 1 / \delta ) }   } \phi(z_i) \d z_i \leq 2 K D \mu_t^{-D} \delta^{ \tau_{\rm tail}^2 / 2}
\eean
for $ \bx \in \bbR^{D}$ and $t > 0$, where the last inequality holds by the tail probability of the standard normal distribution.
Also, for $i \in [D] $, we have
\bean
&& \left\vert \int_{ \substack{ \| \bx-\mu_{t}\by \|_{\infty} \geq \tau_{\rm tail} \sigma_{t} \sqrt{ \log (1 / \delta ) } \\ \| \by \|_{\infty} \leq 1  } } \  \left( \frac{\mu_{t} y_i - x_i }{\sigma_t} \right) p_0(\by) \phi_{\sigma_{t}}(\bx-\mu_{t} \by) \d \by \right\vert
\\
&& \leq K \mu_{t}^{-D}   \int_{ \| \bz\|_{\infty} \geq \tau_{\rm tail} \sqrt{ \log ( 1 / \delta ) } } \  \left\vert  z_i \right\vert \phi_{1}(\bz) \d \bz
\leq K \mu_{t}^{-D} \sum_{j=1}^{D}  \int_{\vert z_j \vert \geq  \tau_{\rm tail} \sqrt{ \log ( 1 / \delta ) }   } \ \vert z_i\vert \phi_{1}(\bz) \d \bz
\\
&& = K  \mu_{t}^{-D} \left\{  (D-1) \bbE[\vert Z \vert]  \int_{\vert z \vert \geq \tau_{\rm tail} \sqrt{ \log ( 1 / \delta ) }   } \phi(z) \d z + \int_{\vert z \vert \geq \tau_{\rm tail} \sqrt{ \log ( 1 / \delta ) }   } \vert z \vert \phi(z) \d z  \right\}
\\
&& \leq K  \mu_{t}^{-D} \left\{ 2  (D-1) \sqrt{2/\pi} \delta^{ \tau_{\rm tail}^2 / 2 } + \sqrt{\bbE[Z^2]} \sqrt{ \bbP \left( \vert Z \vert \geq \tau_{\rm tail} \sqrt{ \log ( 1 / \delta ) } \right) }   \right\} 
\\
&&  \leq  K \mu_t^{-D} \left\{ 2 (D-1) \sqrt{2 / \pi} + \sqrt{2}  \right\} \delta^{\tau_{\rm tail}^2 / 4},
\eean
for $ \bx \in \bbR^{D}$ and $t > 0$, where $Z$ denotes the one-dimensional standard normal random variable and the second inequality holds by the Cauchy-Schwarz inequality.
Since $\mu_t^{-D} \leq 2^{D}$ for $0 \leq t \leq (2\overline{\tau})^{-1}$,
\be \begin{split}
 & \left\vert  p_t(\bx) - \int_{ \substack{ \| \bx-\mu_{t}\by \|_{\infty} \leq \tau_{\rm tail} \sigma_{t} \sqrt{ \log ( 1 / \delta ) } \\ \| \by \|_{\infty} \leq 1  } } \ p_0(\by) \phi_{\sigma_{t}}(\bx-\mu_{t} \by) \d \by \right\vert 
 \leq  K D  2^{D+1}  \delta^{ \tau_{\rm tail}^2 / 2 },
 \\
 & \left\| \sigma_{t} \nabla p_t(\bx) -  \int_{ \substack{ \| \bx-\mu_{t}\by \|_{\infty} \leq \tau_{\rm tail} \sigma_{t} \sqrt{ \log ( 1 / \delta ) } \\ \| \by \|_{\infty} \leq 1  } } \ \left( \frac{\mu_t \by - \bx}{\sigma_t} \right) p_0(\by) \phi_{\sigma_{t}}(\bx-\mu_{t} \by) \d \by \right\|_{\infty}
 \\
 & \leq K 2^{D} \left\{ 2 (D-1) \sqrt{2 / \pi} + \sqrt{2}  \right\} \delta^{\tau_{\rm tail}^2 /  4}
 \label{eqprop2:tailbd}
\end{split} \ee
for $\bx \in \bbR^{D}$ and $0 < t \leq (2\overline{\tau})^{-1}$.
For $0 < t \leq (2 \overline{\tau})^{-1}$ and $\bx,\by \in \bbR^{D}$ with $\|\bx\|_{\infty} \geq \mu_{t} - \tau_{\rm x} \{ \log (1/\sigma_{t}) \}^{-\tau_{\rm bd}} $ and $\| \bx-\mu_{t}\by \|_{\infty} \leq \tau_{\rm tail} \sigma_{t} \sqrt{ \log ( 1 / \delta) } $, we have
\bean
&& \|\by \|_{\infty} \geq \frac{ \|\bx\|_{\infty} - \| \bx - \mu_{t}\by \|_{\infty} }{\mu_{t}} \geq 1 - \left( \frac{ \tau_{\rm x} \{ \log (1/\sigma_{t}) \}^{-\tau_{\rm bd}} + \tau_{\rm tail} \sigma_{t} \sqrt{ \log ( 1 / \delta ) } }{\mu_{t}}  \right)
 \\
 && \geq 1- 2\left[ \tau_{\rm x} \left\{ \log ( 1/ \sqrt{2 \overline{\tau} t} ) \right\}^{-\tau_{\rm bd}} + \tau_{\rm tail} \sqrt{ 2\overline{\tau} t \log ( 1 / \delta ) } \right],
\eean
where the last inequality holds by (\ref{eqthm:mtstbd}).
For $0 < t \leq \delta^{\tau_{\rm t}}$ and small enough $\delta$ so that $\delta^{ \tau_{\rm t} } \leq (2 \overline {\tau} )^{-1}$, the last display is bounded by
\bean
&& 1- 2\left[ \tau_{\rm x} \left\{ \log ( 1/ \sqrt{2 \overline{\tau} \delta^{\tau_{\rm t}} } ) \right\}^{-\tau_{\rm bd}} + \tau_{\rm tail} \sqrt{ 2\overline{\tau} \delta^{\tau_{\rm t}} } \sqrt{ \log ( 1 / \delta ) } \right].
\eean
Moreover, the last display is lower bounded by $1 - 2 D_{1}$ for small enough $\delta$, where
\bean
D_{1} = \left( 2 \left\lfloor 2 \left\{ \log ( 1 / \delta) \right\}^{\widetilde \tau_{\rm bd}} \vee 4 \right\rfloor + 2 \right)^{-1}.
\eean
Then, $ D_{1} \leq ( \{ \log ( 1 / \delta ) \}^{- \widetilde \tau_{\rm bd}}  /2) \wedge (1/4)$ and $D_{1}^{-1} \in 2\bbN$.
Let $\by^{(1)},\ldots,\by^{(D_2)} \in \bbR^{D}$ be distinct vectors satisfying that
\bean
\left\{\by^{(1)},\ldots,\by^{(D_{2})} \right\} = \left\{ D_{1} (n_1,\ldots,n_{D})^{\top}  : n_i \in \bbZ, i \in [D] \right\} \cap \left\{\by \in \bbR^{D} : \|\by\|_{\infty} = 1- D_{1} \right\},
\eean
where $D_{2} = (2 / D_{1} - 1)^{D} - (2 / D_{1} - 3)^{D}$.
Let $\cY_{i} = \{\by \in \bbR^{D} : \|\by - \by^{(i)} \|_{\infty} \leq D_{1} \}$ for $i \in [D_{2}]$. 
For any $i \in [D_2]$ and $\by \in \cY_{i}$, we have $1- 2 D_{1} \leq \| \by \|_{\infty} \leq 1$ and $1- 2 D_{1} \leq \|\by^{(i)} \|_{\infty} \leq 1$. 
Assume that the density $p_0$ satisfies
\bean
\sup_{\balpha \in \bbN^{D} } \sup_{ 1 - \{ \log ( 1 / \delta) \}^{- \widetilde \tau_{\rm bd} } \leq \| \bx \|_{\infty} \leq 1 } \left\vert ( \D^{\balpha} p_0 (\bx) ) \right\vert \leq K.
\eean
With small enough $\delta$ so that $D_{3} = \lfloor \log_{4/e} (1/\delta) \rfloor +1 \geq 2$, Taylor's theorem for multivariate function implies that
\bean
p_0(\by) = \sum_{0\leq k. < D_{3}} \frac{(\D^{\bk} p_0 )(\by^{(i)}) }{\bk !} (\by - \by^{(i)})^{\bk} +  \sum_{k. = D_{3}} \frac{(\D^{\bk} p_0 )(\xi\by^{(i)} + (1-\xi) \by ) }{\bk !} (\by - \by^{(i)})^{\bk}
\eean
for a suitable $\xi \in [0,1]$ and $i \in [D_2], \by \in \cY_i$, where $(\by-\by^{(i)})^{\bk} = \prod_{j=1}^{D} (y_j - y^{(i)}_j)^{k_j}$, $\bk ! = \prod_{j=1}^{D} k_j!$ and $k. = \| \bk \|_{1}$.
A simple calculation yields that
\bean
&& \left\vert p_0(\by) -  \sum_{0\leq k. < D_{3}} \frac{(\D^{\bk} p_0 )(\by^{(i)}) }{\bk !} (\by - \by^{(i)})^{\bk} \right\vert \cdot 1\{ \|\by - \by^{(i)} \|_{\infty} \leq D_{1} \}
\\
&& \leq \sum_{k. = D_{3}} K \prod_{j=1}^{D} \left( \frac{e D_{1} }{k_j} \right)^{k_j} \leq K (D_{3}+1)^{D} \left( \frac{e}{4} \right)^{D_{3}}
\eean
for $\by \in  \bbR^{D}$ and $i \in [D_2]$ because  $k! \geq k^{k} e^{-k}$ for any $k \in \bbZ_{\geq 0}$ and $0 < D_{1} \leq 1/4$.
Since $D_{3} \geq \log \delta / \log (e / 4)$, $(e/4)^{D_{3}} \leq \delta$.
Then, the last display is further bounded by
\bean
K \left(   \lfloor \log_{4/e} (1/\delta) \rfloor + 2 \right)^{D} \delta \leq K3^{D} \delta \left\{ \frac{\log (1/\delta) }{\log (4/e)} \right\}^{D}.
\eean

Since $\cY_1,\ldots,\cY_{D_{2}}$ are mutually disjoint except on a set of Lebesgue measure zero and $ \bigcup_{i=1}^{D_{2}} \cY_i = \{\by \in \bbR^{D}: 1- 2 D_{1} \leq  \| \by \|_{\infty} \leq 1 \}$, we have
\bean
&& \int_{ \substack{ \| \bx-\mu_{t}\by \|_{\infty} \leq \tau_{\rm tail} \sigma_{t} \sqrt{ \log ( 1 / \delta ) }
\\ \| \by  \|_{\infty} \leq 1}} g(\by) \d \by
= \sum_{i=1}^{D_{2}} \int_{ \substack{ \| \bx-\mu_{t}\by \|_{\infty} \leq \tau_{\rm tail} \sigma_{t} \sqrt{ \log ( 1 / \delta ) }
\\ \| \by - \by^{(i)}  \|_{\infty} \leq D_{1} }} g(\by) \d \by
\\
&& = \int_{ \substack{ \| \bx-\mu_{t}\by \|_{\infty} \leq \tau_{\rm tail} \sigma_{t} \sqrt{ \log ( 1 / \delta ) }
\\ \|\by\|_{\infty} \leq 1  }}
\sum_{i=1}^{D_{2}} g(\by) \cdot 1\{ \| \by - \by^{(i)}  \|_{\infty} \leq D_{1} \} \d \by
\eean
for any continuous function $g : \bbR^{D} \to \bbR$,  $\|\bx\|_{\infty} \geq \mu_{t} - \tau_{\rm x} \{ \log (1/\sigma_{t}) \}^{-\tau_{\rm bd}}$ and $0 < t \leq \delta^{\tau_{\rm t}}$.
Combining (\ref{eqprop2:tailbd}) with the last two displays, we have
\bean
&& \left\vert p_t(\bx) - \sum_{i=1}^{D_{2}} \int_{ \substack{ \| \bx-\mu_{t}\by \|_{\infty} \leq \tau_{\rm tail} \sigma_{t} \sqrt{ \log ( 1 / \delta ) }
\\ \| \by - \by^{(i)} \|_{\infty} \leq D_{1}}}  \left\{ \sum_{0\leq k. < D_{3}} \frac{(\D^{\bk} p_0 )(\by^{(i)}) }{\bk !}   (\by - \by^{(i)})^{\bk} \right\} \phi_{\sigma_{t}}(\bx-\mu_{t} \by)  \d \by \right\vert
\\
 && \leq  K D 2^{D+1} \delta^{ \tau_{\rm tail}^2 / 2 } +  K3^{D} \delta \left\{ \frac{\log (1/\delta) }{\log (e/4)} \right\}^{D} 
 \\
 && \quad \cdot  \int_{ \substack{ \| \bx-\mu_{t}\by \|_{\infty} \leq \tau_{\rm tail} \sigma_{t} \sqrt{ \log ( 1 / \delta ) }
\\ \| \by  \|_{\infty} \leq 1}} \sum_{i=1}^{D_{2}}  \phi_{\sigma_{t}}(\bx-\mu_{t} \by) \cdot 1\{ \| \by - \by^{(i)} \|_{\infty} \leq D_{1} \}  \d \by   
\\
&& =  K D 2^{D+1} \delta^{ \tau_{\rm tail}^2 / 2 } + K  3^{D} \delta \left\{ \frac{\log (1/\delta) }{\log (4/e)} \right\}^{D} 
 \int_{ \substack{ \| \bx-\mu_{t}\by \|_{\infty} \leq \tau_{\rm tail} \sigma_{t} \sqrt{ \log ( 1 / \delta ) }
\\ \| \by  \|_{\infty} \leq 1}}  \phi_{\sigma_{t}}(\bx-\mu_{t} \by)  \d \by   
\eean
for $\|\bx\|_{\infty} \geq \mu_{t} - \tau_{\rm x} \{ \log (1/\sigma_{t}) \}^{-\tau_{\rm bd}}$ and $0 < t \leq \delta^{\tau_{\rm t}}$.
Moreover, the last display is bounded by
\be
 K D 2^{D+1} \delta^{ \tau_{\rm tail}^2 / 2 } + K  6^{D} \delta \left\{ \frac{\log (1/\delta) }{\log (4/e)} \right\}^{D} \label{eqprop2:p0errorbd}
\ee
because $\int_{\bbR^{D}} \phi_{\sigma_t}(\bx-\mu_t \by) \d \by = \mu_t^{-D}$ and $\mu_t^{-D} \leq 2^{D}$ by (\ref{eqthm:mtstbd}).
Also, we have
\bean
&& \Bigg\| \sigma_{t} \nabla p_t(\bx) -  \sum_{i=1}^{D_{2}} \int_{ \substack{ \| \bx-\mu_{t}\by \|_{\infty} \leq \tau_{\rm tail} \sigma_{t} \sqrt{ \log ( 1 / \delta ) }
\\ \| \by - \by^{(i)} \|_{\infty} \leq D_{1} } } 
\\
&& \qquad \left\{  \sum_{0\leq k. < D_{3}}  \frac{(\D^{\bk} p_0 )(\by^{(i)}) }{\bk !}    \left(\frac{ \mu_t \by - \bx  }{\sigma_t} \right)   (\by - \by^{(i)})^{\bk} \right\}
 \phi_{\sigma_{t}}(\bx-\mu_{t} \by) \d \by \Bigg\|_{\infty}
 \\
 && \leq   K 2^{D} \left\{ 2 (D-1) \sqrt{2 / \pi} + \sqrt{2}  \right\} \delta^{\tau_{\rm tail}^2 / 4 }  +      K3^{D} \delta \left\{ \frac{\log (1/\delta) }{\log (4/e)} \right\}^{D}
 \\
 && \qquad \cdot \int_{ \substack{ \| \bx-\mu_{t}\by \|_{\infty} \leq \tau_{\rm tail} \sigma_{t} \sqrt{ \log ( 1 / \delta ) }
\\ \| \by \|_{\infty} \leq 1}} \sum_{i=1}^{D_{2}}  \left\| \frac{ \mu_t \by - \bx  }{\sigma_t} \right\|_{\infty}  \phi_{\sigma_{t}}(\bx-\mu_{t} \by) \cdot 1\{ \| \by - \by^{(i)} \|_{\infty} \leq D_{1} \}    \d \by   
\eean
for $\| \bx \|_{\infty} \geq \mu_{t} - \tau_{\rm x} \{ \log(1/\sigma_t) \}^{-\tau_{\rm bd}}$ and $0 < t \leq \delta^{\tau_{\rm t}}$.
Moreover, the last display is bounded by
\be
 K 2^{D} \left\{ 2 (D-1) \sqrt{2 / \pi} + \sqrt{2}  \right\} \delta^{\tau_{\rm tail}^2 / 4 }    +  2 K  \tau_{\rm tail}    6^{D} \{ \log (4/e) \}^{-D} \delta \{ \log ( 1 / \delta) \}^{D + 1/2}
 \label{eqprop2:p0errorbdgrad}
\ee
because $\int_{\bbR^{D}} \phi_{\sigma_t}(\bx-\mu_t \by) \d \by = \mu_t^{-D}$ and $\mu_t^{-D} \leq 2^{D}$ by (\ref{eqthm:mtstbd}).

For $x,y \in \bbR$ and $t > 0$, Taylor's theorem yields that 
\bean
\left\vert \exp \left( -\frac{(x-\mu_{t}y)^2}{2 \sigma_{t}^2} \right) - \sum_{l=0}^{D_{4}-1} \frac{1}{l!} \left( -\frac{(x-\mu_{t}y)^2}{2 \sigma_{t}^2} \right)^{l} \right\vert 
\leq \frac{1}{D_{4} !} \left( \frac{(x-\mu_{t}y)^2}{2 \sigma_{t}^2} \right)^{D_{4}},
\eean
where $D_{4} = \lfloor 2 e \tau_{\rm tail}^2  \log (1/\delta)  \rfloor +1$ with small enough $\delta$ so that $D_{4} \geq 1$.
For $\vert x - \mu_{t} y \vert \leq  \tau_{\rm tail} \sigma_{t} \sqrt{ \log ( 1 / \delta ) } $ and $ t \geq 0 $, the last display is further bounded by
\bean
\left(\frac{ e \tau_{\rm tail}^2 \log ( 1 / \delta ) }{D_{4}} \right)^{D_{4}} 
\leq 2^{-D_{4}} 
\leq \delta^{2 e \tau_{\rm tail}^2 \log 2 } \leq \delta^{ e \tau_{\rm tail}^2 },
\eean
where the last inequality holds because $1/2 \leq \log 2$ and $ 0 < \delta < 1$.
Then,
\bean
&& \left\vert \int_{ \substack{\vert x - \mu_{t} y \vert \leq  \tau_{\rm tail} \sigma_{t} \sqrt{ \log ( 1 / \delta ) } \\ \vert y - \widetilde y \vert \leq \widetilde \tau } } \ \left(\frac{\mu_t y - x}{\sigma_t}\right)^{m}  (y-\widetilde y)^{k} \exp \left( -\frac{(x-\mu_{t}y)^2}{2 \sigma_{t}^2} \right)  \d y \right.
\\
&& \quad \left. - \sum_{l=0}^{D_{4}-1} \frac{1}{l!}
\int_{ \substack{ \vert x - \mu_{t} y \vert \leq  \tau_{\rm tail} \sigma_{t} \sqrt{ \log ( 1 / \delta ) } \\ \vert y - \widetilde y \vert \leq D_{1} } } \ \left(\frac{\mu_t y - x}{\sigma_t} \right)^{m}  (y-\widetilde y)^{k}  \left( -\frac{(x-\mu_{t}y)^2}{2 \sigma_{t}^2} \right)^{l} \d y \right\vert
\\
&& \leq \delta^{e \tau_{\rm tail}^2 } \int_{ \substack{ \vert x - \mu_{t} y \vert \leq  \tau_{\rm tail} \sigma_{t} \sqrt{ \log ( 1 / \delta ) } \\ \vert y - \widetilde y \vert \leq D_{1} } } \left\vert \frac{\mu_t y - x}{\sigma_t} \right\vert^{m} \vert y - \widetilde y \vert^{k} \d y 
\\
&& \leq \delta^{e \tau_{\rm tail}^2}   \int_{\vert y \vert \leq 1} D_{1}^{k} \left\{ \tau_{\rm tail}\sqrt{ \log ( 1 / \delta ) } \right\}^{m} \d y  \leq \delta^{e \tau_{\rm tail}^2} \left\{  \tau_{\rm tail} \sqrt{ \log ( 1 / \delta ) } \right\}^{m} 
\eean
for any $x, \widetilde y \in \bbR$ with $\vert \widetilde y \vert = 1-D_{1}$,$t \geq 0, k \in \bbZ_{\geq 0}$ and $m \in \{0,1\}$, where the last inequality holds because $D_{1} < 1$.
Moreover,
\bean
&& \left\vert \sum_{l=0}^{D_{4}-1} \frac{1}{l!}
\int_{ \substack{\vert x - \mu_{t} y \vert \leq  \tau_{\rm tail} \sigma_{t} \sqrt{ \log ( 1 / \delta ) }\\ \vert y - \widetilde y \vert \leq D_{1} }} \ \left(\frac{\mu_t y - x}{\sigma_t}\right)^{m}   (y-\widetilde y)^{k}  \left( -\frac{(x-\mu_{t}y)^2}{2 \sigma_{t}^2} \right)^{l} \d y \right\vert
\\
&& \leq  \left(1+ \delta^{e \tau_{\rm tail}^2} \right)\left\{ \tau_{\rm tail} \sqrt{ \log ( 1 / \delta ) }\right\}^{m} 
\eean
for any $x, \widetilde y \in \bbR$ with $\vert \widetilde y \vert = 1- D_{1} $,$t \geq 0, k \in \bbZ_{\geq 0}$ and $m \in \{0,1\}$
 because
\bean
&& \left\vert \int_{ \substack{ \vert x - \mu_{t} y \vert \leq  \tau_{\rm tail} \sigma_{t} \sqrt{ \log ( 1 / \delta ) } \\ \vert y - \widetilde y \vert \leq D_{1} } } \ \left(\frac{\mu_t y - x}{\sigma_t}\right)^{m} (y-\widetilde y)^{k} \exp \left( -\frac{(x-\mu_{t}y)^2}{2 \sigma_{t}^2} \right)  \d y \right\vert
\\
&& \leq  D_{1}^{k} \{ \tau_{\rm tail} \sqrt{ \log (1/\delta)} \}^{m} \int_{ \vert y \vert \leq 1} 1 \d y
\leq \left\{ \tau_{\rm tail} \sqrt{\log(1/\delta)} \right\}^{m} .
\eean
Note that $\vert \prod_{j=1}^{D} x_j - \prod_{j=1}^{D} \widetilde x_j \vert \leq D C^{D-1} \| \bx - \widetilde \bx\|_{\infty}$ for any $\bx, \widetilde \bx \in [-C,C]^{D}$.
Since the last two displays are bounded by $2 \{ \tau_{\rm tail} \sqrt{\log(1/\delta)} \}^{m} $, we have
\be \begin{split}
& \left\vert \int_{ \substack{ \| \bx-\mu_{t}\by \|_{\infty} \leq \tau_{\rm tail} \sigma_{t} \sqrt{ \log ( 1 / \delta ) }
\\ \| \by - \by^{(i)} \|_{\infty} \leq D_{1}  } } \ \left(\frac{\mu_t y_h - x_h}{\sigma_t}\right)^{m}    (\by - \by^{(i)})^{\bk} \phi_{\sigma_{t}}(\bx-\mu_{t} \by) \d \by  -   \right.
\\
& \prod_{j=1}^{D}  \sum_{l=0}^{D_{4}-1} \left.  \frac{1}{l! \sqrt{2\pi} \sigma_t}
\int_{ \substack{ \vert x_j-\mu_{t} y_j \vert \leq \tau_{\rm tail} \sigma_{t} \sqrt{ \log ( 1 / \delta ) }
\\ \vert y_j - y^{(i)}_j \vert \leq D_{1} }} 
\ \left(\frac{\mu_t y_h - x_h}{\sigma_t}\right)^{m \cdot 1\{h = j\} }   (y_j - y^{(i)}_j)^{k_j} \left( -\frac{(x_j -\mu_{t} y_j)^2}{2 \sigma_{t}^2} \right)^{l}  \d y_j \right\vert
\\
& \leq (2\pi \sigma_{t}^2)^{-\frac{D}{2}} D 2^{D-1} \delta^{e \tau_{\rm tail}^2 } \left\{ \tau_{\rm tail} \sqrt{\log(1/\delta)} \right\}^{mD}
\label{eqprop2:experrorbd}
\end{split} \ee
for $\bx \in \bbR^{D}, t \geq 0, h \in [D], i \in [D_{2}], j \in [D], \bk \in \bbZ_{\geq 0}^{D}$ and $m \in \{0,1\}$ with small enough $\delta$ so that $\tau_{\rm tail} \sqrt{\log(1/\delta)} > 1 $.
With $m = 0$ in the last display, the last integral satisfies that
\bean
&& \frac{1}{l! \sqrt{2\pi} \sigma_t}\int_{ \substack{ \vert x_j-\mu_{t} y_j \vert \leq \tau_{\rm tail} \sigma_{t} \sqrt{ \log ( 1 / \delta ) }
\\ \vert y_j - y^{(i)}_j \vert \leq D_{1} } } \   \left(y_j - y^{(i)}_j \right)^{k_j} \left( -\frac{(x_j -\mu_{t} y_j)^2}{2 \sigma_{t}^2} \right)^{l}  \d y_j 
\\
&& = \frac{ \mu_{t}^{-1 } }{l!  \sqrt{2\pi}}\int_{ \substack{ \vert z_j \vert \leq \tau_{\rm tail} \sqrt{ \log ( 1 / \delta ) }
\\ \vert \mu_t^{-1}\sigma_t z_j + \mu_t^{-1} x_j - y_j^{(i)} \vert \leq D_{1}  } } \   \left(\mu_t^{-1}\sigma_t z_j + \mu_t^{-1} x_j - y_j^{(i)} \right)^{k_j} z_j^{2l} (-2)^{-l}  \d z_j 
\\
&& = \frac{(-2)^{-l}  \mu_{t}^{-1 }  }{l! \sqrt{2\pi}}\int_{ \substack{ \vert z_j \vert \leq \tau_{\rm tail}  \sqrt{ \log ( 1 / \delta ) }
\\ \vert \mu_t^{-1}\sigma_t z_j + \mu_t^{-1} x_j - y_j^{(i)} \vert \leq D_{1}  } }  \ \sum_{r_j = 0}^{k_j} 
\binom{k_j}{r_j}
\left(\mu_t^{-1} \sigma_t \right)^{r_j}  \left(\mu_t^{-1} x_j - y_j^{(i)} \right)^{k_j-r_j} z_j^{r_j+2l} \ \d z_j
\\
&& = \frac{(-2)^{-l} \mu_t^{-k_j - 1}  }{l! \sqrt{2\pi}}  \sum_{r_j = 0}^{k_j} 
\binom{k_j}{r_j}
 \sigma_t^{r_j}  \left( x_j - \mu_t y_j^{(i)} \right)^{k_j-r_j}  \left( \frac{ \overline{z}_{i,j}^{r_j+2l+1} - \underline{z}_{i,j}^{r_j+2l+1} }{r_j+2l+1} \right) \defeq P_{i,j,k_j,l}(x_j, t),
\eean
where
\bean
&& \overline{z}_{i,j} = \min \left( \max \left( \frac{ \mu_t(y_j^{(i)} + D_{1} ) - x_j }{\sigma_{t}}, -\tau_{\rm tail} \sqrt{ \log ( 1 / \delta ) } \right), \tau_{\rm tail} \sqrt{ \log ( 1 / \delta ) } \right),
\\
&& \underline{z}_{i,j} =\min \left( \max \left( \frac{ \mu_t(y_j^{(i)} - D_{1} ) - x_j }{\sigma_{t}}, -\tau_{\rm tail} \sqrt{ \log ( 1 / \delta ) } \right), \tau_{\rm tail} \sqrt{ \log ( 1 / \delta ) } \right).
\eean
Combining (\ref{eqthm:mtstbd}), (\ref{eqprop2:p0errorbd}) and (\ref{eqprop2:experrorbd}) with the last two displays, we have
\bean
&&  \left\vert p_t(\bx) - g_t(\bx) \right\vert
 \\
&& \leq 
 K D 2^{D+1} \delta^{\tau_{\rm tail}^2 / 2 }
 + K    6^{D} \delta \left\{ \frac{\log (1/\delta) }{\log (4/e)} \right\}^{D}
 \\
&& \quad + \sum_{i=1}^{D_2} 
\sum_{0 \leq k. <D_3} \left\vert \frac{ (\D^{\bk} p_0)(\by^{(i)})  }{\bk!} \right\vert (2\pi \sigma_{t}^2)^{-\frac{D}{2}} D 2^{D-1} \delta^{e \tau_{\rm tail}^2 }  
\\
&& \leq
 D_{5} \left[ \delta^{\tau_{\rm tail}^2 / 2} + \delta \{ \log (1/\delta) \}^{D} +  \delta^{e\tau_{\rm tail}^2 - D/2} \{ \log ( 1 / \delta) \}^{D \widetilde \tau_{\rm bd} + D  }  \right]
\eean
for $\|\bx\|_{\infty} \geq \mu_{t} -\tau_{\rm x} \{ \log (1/\sigma_{t}) \}^{-\tau_{\rm bd}}$ and $\delta \leq t \leq \delta^{\tau_{\rm t}}$,
where $D_{5} = D_{5}(D, K, \underline{\tau}, \tau_{\rm tail})$ and $g_{t} : \bbR^{D} \to \bbR$ is a function such that
\bean
 g_t(\bx)=
 \sum_{i=1}^{D_{2}} \sum_{0 \leq k. < D_{3}} \left\{ 
 (\D^{\bk} p_0)(\by^{(i)})  \right\} \prod_{j=1}^{D} \sum_{l=0}^{D_{4}-1} \frac{  P_{i,j,k_j,l}(x_j, t) }{k_j !},
\quad \bx \in \bbR^{D}.
\eean
Since $\tau_{\rm tail} \geq 2 \vee \sqrt{(D+3)/(2e)}$, we have
\be
\left\vert p_t(\bx) - g_t(\bx) \right\vert
\leq 3 D_{5} \delta \{ \log ( 1 / \delta) \}^{D} 
\label{eqprop2:ptpoly}
\ee
for small enough $\delta$.
Similarly, with $m = 1$ and $h = j$ in (\ref{eqprop2:experrorbd}), the last integral in (\ref{eqprop2:experrorbd}) satisfies that
\bean
&& \frac{1}{l! \sqrt{2\pi } \sigma_t} \int_{ \substack{ \vert x_j-\mu_{t} y_j \vert \leq \tau_{\rm tail} \sigma_{t} \sqrt{ \log ( 1 / \delta ) }
\\ \vert y_j - y^{(i)}_j \vert \leq D_{1} } } \
\left(\frac{\mu_t y_j - x_j}{\sigma_t}\right)
(y_j - y^{(i)}_j)^{k_j} \left( -\frac{(x_j -\mu_{t} y_j)^2}{2 \sigma_{t}^2} \right)^{l}  \d y_j  
\\
&& = \frac{ \mu_{t}^{-1} }{l! \sqrt{2\pi}}\int_{ \substack{ \vert z_j \vert \leq \tau_{\rm tail}  \sqrt{ \log ( 1 / \delta ) }
\\ \vert \mu_t^{-1}\sigma_t z_j + \mu_t^{-1} x_j - y_j^{(i)} \vert \leq D_{1}  } } \   \left(\mu_t^{-1}\sigma_t z_j + \mu_t^{-1} x_j - y_j^{(i)} \right)^{k_j} z_j^{2l+1} (-2)^{-l}  \d z_j
\\
&& = \frac{(-2)^{-l} \mu_t^{-k_j - 1} }{l! \sqrt{2\pi}}  \sum_{r_j = 0}^{k_j} 
\binom{k_j}{r_j}
 \sigma_t^{r_j}  \left( x_j - \mu_t y_j^{(i)} \right)^{k_j-r_j}  \left( \frac{ \overline{z}_{i,j}^{r_j+2l+2} - \underline{z}_{i,j}^{r_j+2l+2} }{r_j+2l+2} \right) \defeq \widetilde P_{i,j,k_j,l}(x_j, t).
\eean
Combining (\ref{eqthm:mtstbd}), (\ref{eqprop2:p0errorbdgrad}) and (\ref{eqprop2:experrorbd}) with the last display, we have
\bean
 && \left\vert \sigma_t  \left( \nabla p_t(\bx) \right)_{h} - \widetilde g_t^{(h)}(\bx) \right\vert
 \\
 && \leq 
 K 2^{D} \left\{ 2 (D-1) \sqrt{2 / \pi} + \sqrt{2}  \right\} \delta^{\tau_{\rm tail}^2 / 4 }    +  2 K  \tau_{\rm tail}    6^{D} \{ \log (4/e) \}^{-D} \delta \{ \log ( 1 / \delta) \}^{D + 1/2}
\\
&& \quad + \sum_{i=1}^{D_2} 
\sum_{0 \leq k. <D_3} \left\vert \frac{ (\D^{\bk} p_0)(\by^{(i)})  }{\bk!} \right\vert (2\pi \sigma_{t}^2)^{-\frac{D}{2}} D 2^{D-1} \delta^{e \tau_{\rm tail}^2 } \left\{ \tau_{\rm tail} \sqrt{\log(1/\delta)} \right\}^{D}
\\
&& \leq
 D_{6} \left[ \delta^{ \tau_{\rm tail}^2 / 4 } + \delta \{ \log (1/\delta) \}^{D+1/2 } +   \delta^{e\tau_{\rm tail}^2 - D/2} \{ \log ( 1 / \delta) \}^{D ( \widetilde \tau_{\rm bd} + 3/2) } \right], \quad h \in [D]
\eean
for $\|\bx\|_{\infty} \geq \mu_{t} - \tau_{\rm x} \{ \log (1/\sigma_{t}) \}^{-\tau_{\rm bd}} $ and $\delta \leq t \leq \delta^{\tau_{\rm t}}$,
where $D_{6} = D_{6}(D, K, \tau_{\rm tail}, \underline{\tau})$ and $\widetilde g_{t}^{(h)} : \bbR^{D} \to \bbR, h \in [D]$ is a function such that
\bean
&& \widetilde g_t^{(h)}(\bx) = \sum_{i=1}^{D_{2}} \sum_{0 \leq k. < D_{3}} \left\{ 
 (\D^{\bk} p_0)(\by^{(i)})  \right\} \left\{  \prod_{\substack{j=1 \\ j \neq h}}^{D} \sum_{l=0}^{D_{4}-1} \frac{  P_{i,j,k_j,l}(x_j, t) }{k_j !} \right\}
 \left\{ \sum_{l=0}^{D_{4}-1} \frac{ \widetilde P_{i,h,k_h,l} (x_h, t) }{k_h !} \right\}.
\eean
Since $\tau_{\rm tail} \geq 2 \vee \sqrt{(D+3)/(2e)}$, we have
\be
\left\vert \sigma_t  \left( \nabla p_t(\bx) \right)_{h} - \widetilde g_t^{(h)}(\bx) \right\vert
\leq 3 D_{6} \delta \{ \log ( 1 / \delta) \}^{D}
, \quad h \in [D]
 \label{eqprop2:nablapoly}
\ee
for small enough $\delta$.

Let $0 < \widetilde \delta < \delta$ be a small enough value as described below.
With ${\widetilde \delta}^2 < 1/2$, Lemma~\ref{secnn:mtst} implies that
there exist neural networks  $f_{\mu} \in \cF_{\rm NN}(L_{\mu},\bd_{\mu},s_{\mu},M_{\mu}), f_{\sigma} \in \cF_{\rm NN}(L_{\sigma},\bd_{\sigma},s_{\sigma},M_{\sigma})$ with
\bean
&& L_{\mu}, L_{\sigma} \leq C_{N,4} \{ \log(1/\widetilde \delta)\}^{2}, \quad \|\bd_{\mu} \|_{\infty},  \| \bd_{\sigma}\|_{\infty} \leq C_{N,4}  \{ \log(1/\widetilde \delta)\}^{2}
\\
&& s_{\mu}, s_{\sigma} \leq C_{N,4}  \{  \log(1/\widetilde \delta)\}^{3}, \quad M_{\mu}, M_{\sigma} \leq C_{N,4}  \log ( 1/\widetilde \delta)
\eean
such that $
\vert \mu_{t} - f_{\mu}(t) \vert \leq \widetilde \delta$ and $ \vert \sigma_{t} - f_{\sigma}(t) \vert \leq \widetilde \delta$
for $ t \geq \widetilde \delta$, where $C_{N,4}$ is the constant in Lemma~\ref{secnn:mtst}.
Also, Lemma~\ref{secnn:rec} implies that there exist a neural network $f_{\rm rec} \in \cF_{\rm NN}(L_{\rm rec}, \bd_{\rm rec}, s_{\rm rec}, M_{\rm rec} )$ with
\bean
&& L_{\rm rec} \leq C_{N,5} \{ \log(1/ \widetilde \delta )\}^{2}, \quad \|\bd_{\rm rec}\|_{\infty} \leq C_{N,5}  \{ \log (1/ \widetilde \delta )\}^{3}
\\
&& s_{\rm rec}  \leq C_{N,5}  \{ \log(1/ \widetilde \delta )\}^{4}, \quad M_{\rm rec} \leq C_{N,5} \widetilde \delta^{-2}
\eean
such that $\vert 1/x - f_{\rm rec}(x) \vert \leq  \widetilde \delta $ for any $x \in [ \widetilde \delta,1/ \widetilde \delta]$, where $C_{N,5}$ is the constant in Lemma~\ref{secnn:rec}.
Since $  \sigma_t-\widetilde \delta \leq f_{\sigma}(t)  \leq \sigma_t + \widetilde \delta$ for $t \geq \widetilde \delta$ and $\sqrt{\underline{\tau} t } \leq \sigma_t \leq 1$ for $t \geq \delta$, we have $\delta \leq f_{\sigma}(t)  \leq 2$ for $ t \geq \delta $ with small enough $\widetilde \delta$ so that $\widetilde \delta \leq \sqrt{\underline{\tau} 
\delta} - \delta$ and $\widetilde \delta \leq 1$.
A simple calculation yields that
\be \begin{split}
& \left\vert 1/\sigma_t - f_{\rm rec}(f_{\sigma}(t)) \right\vert \leq
\left\vert 1/\sigma_t - 1/ f_{\sigma}(t) \right\vert  + \left\vert  1/ f_{\sigma}(t) - f_{\rm rec}(f_{\sigma}(t)) \right\vert 
\\
& \leq \{ \sigma_t \wedge f_{\sigma}(t)\}^{-2} \vert \sigma_t - f_{\sigma}(t) \vert + \widetilde \delta \leq (1+\delta^{-2}) \widetilde \delta  \label{eqprop2:nnlogrecst}
\end{split} \ee
for $t \geq \delta$. 
Lemma~\ref{secnn:mult} implies that there exists a neural network 
\bean
\widetilde f_{\rm mult}^{(k)} \in \cF_{\rm NN} ( \widetilde L_{\rm mult}^{(k)}, \widetilde \bd_{\rm mult}^{(k)}, \widetilde s_{\rm mult}^{(k)}, \widetilde M_{\rm mult}^{(k)} ), \quad k \geq 2
\eean
with 
\bean
&& \widetilde L_{\rm mult}^{(k)} \leq C_{N,1} \log k  \{\log(1/ \widetilde \delta) + \log(1/\delta) \}, \quad \widetilde d_{\rm mult}^{(k)} = (k,48k,\ldots,48k,1)^{\top},
\\
&& \widetilde s_{\rm mult}^{(2)} \leq  C_{N,1} k \{\log(1/ \widetilde \delta) 
+  \log(1/ \delta) \}, \quad  \widetilde M_{\rm mult}^{(k)} = \delta^{-k}
\eean
such that 
\be
\left\vert \widetilde f_{\rm mult}^{(k)}(\widetilde x_1,\ldots, \widetilde x_k) - \prod_{i=1}^{k} x_i \right\vert \leq \widetilde \delta + k \delta^{-(k-1)} \widetilde \epsilon \label{eqprop2:nnmult}
\ee
for any $\bx = (x_1,\ldots,x_k) \in \bbR^{k}$ with $ \|\bx\|_{\infty} \leq \delta^{-1}$ and $\widetilde \bx = (\widetilde x_1,\ldots,\widetilde x_{k}) \in \bbR^{k}$ with $\| \bx - \widetilde \bx\|_{\infty} \leq \widetilde \epsilon$,
where $0 < \widetilde \epsilon \leq 1 $ and $C_{N,1}$ is the constant in Lemma~\ref{secnn:mult}.
Let 
\bean
f_{\rm clip} \in \cF_{\rm NN}(2,(1,2,1)^{\top}, 7, \tau_{\rm tail} \sqrt{\log(1/\delta)} )
\eean
be the neural network in Lemma~\ref{secnn:clip} such that
$f_{\rm clip}(x) = (x \vee -\tau_{\rm tail}\sqrt{\log(1/\delta)}  ) \wedge \tau_{\rm tail} \sqrt{\log(1/\delta)}$ for $x \in \bbR$.
For $i \in [D_{2}]$ and $ j \in [D]$, consider functions $\overline{f}_{i,j}, \underline{f}_{i,j} : \bbR \times \bbR \to \bbR$ such that
\bean
&& \overline{f}_{i,j} (x,t) = f_{\rm clip}\left( \widetilde f_{\rm mult}^{(2)}\left( f_{\rm rec}\left(f_{\sigma}(t) \right), \{ y_j^{(i)} +  D_{1} \} f_{\mu}(t) - x \right)  \right),
 \\
 && \underline{f}_{i,j} (x,t) = f_{\rm clip}\left( \widetilde f_{\rm mult}^{(2)}\left( f_{\rm rec}\left(f_{\sigma}(t) \right), \{ y_j^{(i)} -  D_{1} \} f_{\mu}(t) - x \right) \right),
\eean
for $x,t \in \bbR$.
Note that $\vert y_j^{(i)}+ D_{1} \vert \leq 1 $ and $\vert y_j^{(i)} - D_{1} \vert \leq 1 $ for $i \in [D_2]$.
Combining (\ref{eqprop2:nnlogrecst}) and (\ref{eqprop2:nnmult}) with the last display, both $\vert \overline{z}_{i,j} - \overline{f}_{i,j} (x,t) \vert$ and $\vert \underline{z}_{i,j} - \underline{f}_{i,j} (x,t) \vert$ are bounded by
\be
\widetilde \delta + 2 \delta^{-1} ( 1 + \delta^{-2} ) \widetilde \delta
\leq 5 \delta^{-3} \widetilde \delta \label{eqprop2:nnzij}
\ee
for $ \vert x \vert \leq  \mu_{t} + \tau_{\rm x} \sigma_{t} \sqrt{\log (1/\delta)}$ and $\delta \leq t \leq \delta^{\tau_{\rm t}}$ with small enough $\delta$ so that $\vert (y_j^{(i)} + D_{1} ) \mu_t - x \vert $, $\vert (y_j^{(i)} - D_{1} ) \mu_t - x \vert $ and $ 1/\sigma_t$ are upper bounded by $\delta^{-1}$.
Since $  \mu_t -\widetilde \delta \leq f_{\mu}(t)  \leq \mu_t + \widetilde \delta$ for $t \geq \widetilde \delta$, we have $1/4 \leq 1/2  - \widetilde \delta \leq f_{\mu}(t)  \leq 1+\widetilde \delta \leq 2$ and $1/4 \leq \mu_t \leq 2$ for $\delta \leq t \leq D_{1}$ with small enough $\widetilde \delta$ by (\ref{eqthm:mtstbd}).
A simple calculation yields that
\be \begin{split}
& \left\vert 1/\mu_t - f_{\rm rec}( f_{\mu}(t) ) \right\vert \leq 
\left\vert 1/\mu_t - 1/f_{\mu}(t) \right\vert + \left\vert   1/f_{\mu}(t)  - f_{\rm rec}( f_{\mu} (t)  ) \right\vert
\\
& \leq \{ \mu_t \wedge f_{\mu}(t)\}^{-2} \vert \mu_t - f_{\mu}(t) \vert + \widetilde \delta \leq 17 \widetilde \delta \label{eqprop2:nnmtrec}
\end{split} \ee
for $\delta \leq t \leq \delta^{\tau_{\rm t}}$.
For any $i \in [D_2], j \in [D], k \in \{0,\ldots,D_{3}-1\}$, consider functions $f_{i,j,k}, \widetilde f_{i,j,k} : \bbR \times \bbR \to \bbR$ such that
\bean
&& f_{i,j,k}(x,t)  = \sum_{l=0}^{D_{4}-1} \sum_{r=0}^{k} \binom{k}{r} \left\{ \frac{(-2)^{-l}}{k! l! \sqrt{2\pi} (r+2l+1)  } \right\} \left\{ \overline f_{i,j,k,l,r,r+2l+1} -  \underline f_{i,j,k,l,r,r+2l+1} \right\},
\\
&& \widetilde f_{i,j,k}(x,t)  =  \sum_{l=0}^{D_{4}-1} \sum_{r=0}^{k} \binom{k}{r} \left\{ \frac{(-2)^{-l}}{k! l! \sqrt{2\pi} (r+2l+2)  } \right\} \left\{ \overline f_{i,j,k,l,r,r+2l+2} -  \underline f_{i,j,k,l,r,r+2l+2} \right\},
\eean
for $x,t \in \bbR$, where
\bean
 && \overline f_{i,j,k,l,r,s}  = \widetilde f_{\rm mult}^{(2k+s+1)} \left(  f_{\rm rec}\left(f_{\mu}(t) \right) \cdot \mathbf{1}_{k+1}, f_{\sigma}(t) \cdot \mathbf{1}_{r}, \left\{ x - f_{\mu}(t) y_j^{(i)} \right\} \cdot \mathbf{1}_{k-r}, \overline{f}_{i,j}(x,t) \cdot \mathbf{1}_{s}    \right),
 \\
&& \underline f_{i,j,k,l,r,s}  = \widetilde f_{\rm mult}^{(2k+s+1)} \left(  f_{\rm rec}\left(f_{\mu}(t) \right) \cdot \mathbf{1}_{k+1}, f_{\sigma}(t) \cdot \mathbf{1}_{r}, \left\{ x - f_{\mu}(t) y_j^{(i)} \right\} \cdot \mathbf{1}_{k-r}, \underline{f}_{i,j}(x,t) \cdot \mathbf{1}_{s}    \right),
\eean
for $ s \in \{r+2l+1, r+2l+2 \}$.
Combining (\ref{eqprop2:nnmtrec}), (\ref{eqprop2:nnlogrecst}), (\ref{eqprop2:nnzij}) and (\ref{eqprop2:nnmult}) with the last two displays, the definition of $P_{i,j,k,l}(x, t)$ and $\widetilde P_{i,j,k,l}(x, t)$ implies that
\bean
&& \left\vert \sum_{l=0}^{D_{4}-1 }  \frac{ P_{i,j,k,l}(x, t) }{k!} -  f_{i,j,k}(x,t)   \right\vert 
\\
&& \leq  \sum_{l=0}^{D_{4}-1} \sum_{r=0}^{k} \binom{k}{r} \left\{ \frac{2^{-l+1}}{k! l! \sqrt{2\pi} (r+2l+2)  } \right\} \left\{ \widetilde \delta + 5(2k+r+2l+1) \delta^{-2k-r-2l-4} \widetilde \delta  \right\}
\eean
and
\bean
&& \left\vert \sum_{l=0}^{D_{4}-1 }  \frac{ \widetilde P_{i,j,k,l}(x, t) }{k!} -  \widetilde f_{i,j,k}(x,t)   \right\vert 
\\
&& \leq  \sum_{l=0}^{D_{4}-1} \sum_{r=0}^{k} \binom{k}{r} \left\{ \frac{2^{-l+1}}{k! l! \sqrt{2\pi} (r+2l+3)  } \right\} \left\{ \widetilde \delta + 5(2k+r+2l+2) \delta^{-2k-r-2l-5} \widetilde \delta  \right\}
\eean
for $i \in [D_2], j \in [D], k \in \{0,\ldots,D_{3}-1\}, \vert x \vert \leq  \mu_{t} + \tau_{\rm x} \sigma_{t} \sqrt{\log (1/\delta)}$ and $\delta \leq t \leq \delta^{\tau_{\rm t}}$ with small enough $\delta$ so that $ \sigma_{t}, \mu_{t}^{-1}, \vert x - y_j^{(i)} \vert, \tau_{\rm tail} \sqrt{ \log (1/\delta) } $ are all upper bounded by $ \delta^{-1}$.
Since $\sum_{r=0}^{k} \binom{k}{r} = 2^{k}$ and $2^{k}/k! \leq 2$ for all $k \in \bbN$,
the last two displays are bounded by
\be
 \sum_{l = 0}^{D_{4}-1} \left( \frac{ 2^{-l+2}}{l! \sqrt{2\pi} ( 2l + 1) } \right) \left\{ 1 + 5( 3 D_{4} + 2l 1 ) \delta^{- 3D_{3} - 2l } \right\} \widetilde \delta
\leq \delta^{-D_{7} \log(1/\delta)  } \widetilde \delta, \label{eqprop2:nnpijk}
\ee
where $D_{7} = D_{7}(\tau_{\rm tail})$.
For any $i \in [D_2], j \in [D], k \in \{0,\ldots,D_{3}-1\}, l \in \{0,\ldots,D_{4}-1\}, \vert x \vert \leq \mu_t + \tau_{\rm x} \sigma_{t} \sqrt{ \log(1/\delta) } $ and $\delta \leq t \leq \delta^{\tau_{\rm t}}$, we have
\bean
\left\vert \frac{ P_{i,j,k,l}(x, t) }{k!} \right\vert \leq  \left\{\frac{2^{-l-k}  }{k! l! \sqrt{2\pi}(r+2l+1)} \sum_{r = 0}^{k} \binom{k}{r} \right\}  \left\{ 2\mu_{t} + \tau_{\rm x} \sigma_{t} \sqrt{\log (1/\delta)} \right\}^{k} \left\{ \tau_{\rm tail} \sqrt{ \log (1/\delta) } \right\}^{k+2l+1}
\eean
and
\bean
\left\vert \frac{ \widetilde P_{i,j,k,l}(x, t) }{k!} \right\vert \leq  \left\{\frac{2^{-l-k}  }{k! l! \sqrt{2\pi}(r+2l+1)} \sum_{r = 0}^{k} \binom{k}{r} \right\}  \left\{ 2\mu_{t} + \tau_{\rm x} \sigma_{t} \sqrt{\log (1/\delta)}\right\}^{k} \left\{ \tau_{\rm tail} \sqrt{ \log (1/\delta) } \right\}^{k+2l+2}
\eean
by the definition of $P_{i,j,k,l}(x)$ and $\widetilde P_{i,j,k,l}(x)$.
Since $\sum_{r=0}^{k} \binom{k}{r} = 2^{k}$ and $2^{k}/k! \leq 2$ for all $k \in \bbN$,
the last two displays are bounded by $\{D_{8} \log (1/\delta) \}^{k+l+1}$, where $D_{8} = D_{8}(\tau_{\rm tail}, \tau_{\rm x})$.
Then,
\bean
&& \left\vert \sum_{l=0}^{D_{4}-1 }  \frac{ P_{i,j,k,l}(x) }{k!}   \right\vert \leq D_{4}\{D_{8} \log (1/\delta) \}^{D_{3} + D_{4} -1 } \leq \{ \log ( 1 / \delta) \}^{ D_{9} \log(1/\delta)  } \quad \text{and}
\\
&&  \left\vert \sum_{l=0}^{D_{4}-1 }  \frac{ \widetilde P_{i,j,k,l}(x) }{k!}   \right\vert \leq D_{4}\{D_{8} \log (1/\delta) \}^{ D_{3} + D_{4} - 1 } \leq \{ \log ( 1 / \delta) \}^{ D_{9} \log(1/\delta)  }
\eean
for $i \in [D_2], j \in [D], k \in \{0,\ldots,D_{3}-1\}, \vert x \vert \leq \mu_{t} + \tau_{\rm x} \sigma_{t} \sqrt{\log (1/\delta)} $ and $\delta \leq t \leq \delta^{\tau_{\rm t}}$,
where $D_{9} = D_{9}(\tau_{\rm tail}, D_{8})$.
Consider functions $f, \widetilde f^{(1)},\ldots, \widetilde f^{(D)} : \bbR^{D} \times \bbR \to \bbR$ such that
\bean
&& f(\bx,t) = \sum_{i=1}^{D_{2}}  \sum_{0 \leq k. < D_{3}} \left\{ (\D^{\bk} p_0)(\by^{(i)})  \right\}
f_{\rm mult}^{(D)} \left( f_{i,1,k_1}(x_1,t),\ldots, f_{i,D,k_D}(x_D, t)  \right) \quad \text{and}
\\
&& \widetilde f^{(h)}(\bx,t)= \sum_{i=1}^{D_{2}} \sum_{0 \leq k. < D_{3}} \left\{(\D^{\bk} p_0)(\by^{(i)})  \right\} f_{\rm mult}^{(D)}\left(\widetilde f_{i,h,k_h}(x_h, t), \underbrace{   f_{i,1,k_1}(x_1,t),\ldots, f_{i,D,k_D}(x_D, t)}_{{\rm without} \   f_{i,h,k_h}(x_h,t)}    \right),
\eean
where $ f_{\rm mult}^{(D)} \in \cF_{\rm NN}( L_{\rm mult}^{(D)},  \bd_{\rm mult}^{(D)}, s_{\rm mult}^{(D)},  M_{\rm mult}^{(D)})$ is the neural network in Lemma~\ref{secnn:mult} with
\bean
&&  L_{\rm mult}^{(D)} \leq C_{N,1} \log D  [ \log(1/ \widetilde \delta) + D D_{9} \log(1/\delta) \log \log (1/\delta) ], \quad  d_{\rm mult}^{(D)} = (D,48D,\ldots,48D,1)^{\top},
\\
&&  s_{\rm mult}^{(D)} \leq   C_{N,1} D [\log(1/ \widetilde \delta) 
+  D_{9} \log(1/\delta) \log \log (1/\delta) ], \quad   M_{\rm mult}^{(D)} = \{ \log ( 1 / \delta) \}^{ D D_{9} \log(1/\delta)  }
\eean
such that 
\bean
\left\vert f_{\rm mult}^{(D)}(\widetilde x_1,\ldots, \widetilde x_{D}) - \prod_{i=1}^{D} x_i \right\vert \leq \widetilde \delta + D \{ \log ( 1 / \delta) \}^{ (D-1) D_{9} \log(1/\delta)  } \widetilde \epsilon 
\eean
for any $\bx = (x_1,\ldots,x_{D}) \in \bbR^{D}$ with $ \|\bx\|_{\infty} \leq \{ \log ( 1 / \delta) \}^{ D_{9} \log(1/\delta)  } $ and $\widetilde \bx = (\widetilde x_1,\ldots,\widetilde x_{D}) \in \bbR^{D}$ with $\| \bx - \widetilde \bx\|_{\infty} \leq \widetilde \epsilon$.
Combining (\ref{eqprop2:nnpijk}) with the last display, we have
\bean
&& \left\vert f(\bx,t) - g_t(\bx) \right\vert
\\
&& \leq  \sum_{i=1}^{D_{2}}  \sum_{0 \leq k. < D_{3}} \left\vert (\D^{\bk} p_0)(\by^{(i)})  \right\vert \left\{ 1 + D \delta^{-D_{7} \log(1/\delta)  } \{ \log ( 1 / \delta) \}^{ (D-1) D_{9} \log(1/\delta)  }  \right\} \widetilde \delta
\\
&& \leq K D_{2} D_{3}^{D} \left\{ 1 + D \delta^{-D_{7} \log(1/\delta)  } \{ \log ( 1 / \delta) \}^{ (D-1) D_{9} \log(1/\delta)  }  \right\} \widetilde \delta 
\leq \delta^{- D_{10} \log(1/\delta)  } \widetilde \delta
\eean
for $\| \bx \|_{\infty} \leq \mu_{t} + \tau_{\rm x} \sigma_{t} \sqrt{\log (1/\delta)} $ and $\delta \leq t \leq D_{1}$, where $D_{10} = D_{10}(D,K,\widetilde \tau_{\rm bd}, D_{7},D_{9})$ is a large enough constant. Similarly, we have
\bean
\left\vert \widetilde f^{(h)}(\bx,t) - \widetilde g^{(h)}_t(\bx) \right\vert 
\leq \delta^{- D_{10} \log(1/\delta) } \widetilde \delta, \quad h \in [D]
\eean
for $\| \bx \|_{\infty} \leq \mu_{t} + \tau_{\rm x} \sigma_{t} \sqrt{\log (1/\delta)} $ and $\delta \leq t \leq \delta^{\tau_{\rm t}}$.
Let $\widetilde \delta = \delta^{D_{10} \log(1/\delta) +1 }$. Combining (\ref{eqprop2:ptpoly}) and (\ref{eqprop2:nablapoly}) with the last two displays, we have
\bean
&& \left\vert p_t(\bx) - f(\bx,t) \right\vert
\leq \delta + 3 D_{5} \delta \{ \log (1/\delta)\}^{D}
\leq \left( 1 + 3 D_{5} \right) \delta \{ \log (1/\delta)\}^{D}
\qquad \text{and}
\\
&& \left\vert \sigma_t  \left( \nabla p_t(\bx) \right)_{h} - \widetilde f^{(h)}(\bx,t) \right\vert 
\leq \delta + 3 D_{6} \delta \{ \log (1/\delta)\}^{D}
\leq \left( 1 + 3 D_{6} \right) \delta \{ \log (1/\delta)\}^{D}, \quad h \in [D]
\eean
for $\mu_t - \tau_{\rm x}  \{ \log(1/\sigma_t) \}^{-\tau_{\rm bd}} \leq \| \bx \|_{\infty} \leq \mu_{t} + \tau_{\rm x} \sigma_{t} \sqrt{\log (1/\delta)} $ and $\delta \leq t \leq \delta^{\tau_{\rm t}}$.
Note that
\bean
D_{2} \leq D_{13} \{ \log ( 1 / \delta ) \}^{D \widetilde \tau_{\rm bd}},
\quad D_{3} \leq D_{13} \log (1 / \delta),
\quad D_{4} \leq D_{13} \log (1 / \delta), 
\eean
where $D_{13} = D_{13}(D, \tau_{\rm tail})$.
Consider a function $ \bff : \bbR^{D} \times \bbR \to \bbR^{D+1}$ such that
\bean
\bff(\bx,t) = ( \widetilde f^{(1)}(\bx,t) , \ldots , \widetilde f^{(D)}(\bx,t), f(\bx,t) )^{\top}
\eean
for $\bx \in \bbR^{D}$ and $ t \in \bbR$. Lemma~\ref{secnn:comp}, Lemma~\ref{secnn:par}, Lemma~\ref{secnn:lin} and Lemma~\ref{secnn:id}  implies that $\bff \in \cF_{\rm NN}(L,\bd,s,M)$ with
\bean
&& L \leq D_{11} \{ \log(1/\delta)\}^{4}, \quad \| \bd\|_{\infty} \leq D_{11} \{ \log(1/\delta)\}^{7+D \widetilde \tau_{\rm bd} + D},
\\
&& s \leq D_{11}  \{ \log(1/\delta)\}^{11 + D \widetilde \tau_{\rm bd} + D }, \quad M \leq \exp( D_{11} \{\log(1/\delta) \}^{2}),
\eean
where $D_{11} = D_{11}( D, K, \tau_{\rm tail}, C_{N,1},  C_{N,4}, C_{N,5}, D_{9} , D_{13})$.
The assertion follows by re-defining the constants.

\end{proof}

\subsection{Proof of Proposition~\ref{secpt:3}}

\begin{proof}
Let
\bean
\tau_{\rm t} = \overline{\tau}^{-1}
\quad \text{and} \quad
\tau_{\rm tail}
= \left\{  4(D \overline{\tau} \tau_{\rm t} + 1 ) \vee \left( \frac{D+1}{e} \right) \right\}^{\frac{1}{2}}.
\eean
Let $t_* > 0$ and $0 < \delta < 1$ be small enough values as described below.
By the Markov property of $(\bX_{t})_{t \geq 0}$, we have
\bean
p_{t_*+t}(\bx) = \int_{\bbR^{D}} p_{t_*}(\by) \phi_{\sigma_{t}}(\bx-\mu_{t} \by) \d \by
\eean
for any $\bx \in \bbR^{D}$ and $t \geq 0$.
Let $C_{S,1} = C_{S,1}(D,K,\tau_1)$ be the constant in Lemma~\ref{secsc:ptbound}. Since $\vert p_{t_*}(\bx) \vert \leq C_{S,1}$ for any $\bx \in \bbR^{D}$, we have
\bean
&& \int_{  \| \bx-\mu_{t}\by \|_{\infty} \geq \tau_{\rm tail} \sigma_{t} \sqrt{\log(1/\delta)} } \ p_{t_*}(\by) \phi_{\sigma_{t}}(\bx-\mu_{t} \by) \d \by
\\
&& \leq C_{S,1} \int_{  \| \bx-\mu_{t}\by \|_{\infty} \geq \tau_{\rm tail} \sigma_{t} \sqrt{\log(1/\delta)}    } \ \phi_{\sigma_{t}}(\bx-\mu_{t} \by) \d \by  = C_{S,1} \mu_{t}^{-D} \int_{  \| \bz \|_{\infty} \geq \tau_{\rm tail}  \sqrt{\log(1/\delta)}   } \phi_{1}(\bz) \d \bz
\\
&& \leq C_{S,1} \mu_{t}^{-D} \sum_{i=1}^{D} \int_{\vert z_i \vert \geq \tau_{\rm tail}  \sqrt{\log(1/\delta)}    } \phi(z_i) \d z_i \leq 2 C_{S,1} D \mu_t^{-D} \delta^{\frac{\tau_{\rm tail}^2}{2}}
\eean
for $ \bx \in \bbR^{D}$ and $t \geq 0$, where the last inequality holds by the tail probability of the standard normal distribution.
Also, 
\bean
&& \left\vert \int_{  \| \bx-\mu_{t}\by \|_{\infty} \geq  \tau_{\rm tail} \sigma_{t} \sqrt{\log(1/\delta)}  } \  \left( \frac{\mu_{t} y_i - x_i }{\sigma_t} \right) p_{t_*}(\by) \phi_{\sigma_{t}}(\bx-\mu_{t} \by) \d \by \right\vert
\\
&& \leq C_{S,1} \mu_{t}^{-D}  \int_{ \| \bz\|_{\infty} \geq \tau_{\rm tail} \sqrt{\log(1/\delta)} } \  \left\vert  z_i \right\vert \phi_{1} (\bz) \d \bz
\leq C_{S,1} \mu_{t}^{-D}  \sum_{j=1}^{D}  \int_{ \vert z_i \vert \geq \tau_{\rm tail} \sqrt{\log(1/\delta)} } \ \vert z_i \vert \phi_{1}(\bz) \d \bz
\\
&& =  C_{S,1}  \mu_{t}^{-D} \left\{  (D-1)  \bbE [\vert Z \vert] \int_{\vert z \vert \geq \tau_{\rm tail} \sqrt{\log(1/\delta)}    } \phi(z) \d z + \int_{\vert z \vert \geq \tau_{\rm tail} \sqrt{\log(1/\delta)}    } \vert z \vert \phi(z) \d z  \right\}
\\
&& \leq 2  C_{S,1}  \mu_{t}^{-D} \left\{  (D-1) \sqrt{2/\pi} \delta^{\tau_{\rm tail}^2 /2} + \sqrt{\bbE[Z^2]} \delta^{\tau_{\rm tail}^2 /4}  \right\}
\\
&&  \leq 2  C_{S,1} D \mu_t^{-D}  \delta^{\frac{\tau_{\rm tail}^2}{4}}, \quad i \in [D]
\eean
for $ \bx \in \bbR^{D}$ and $t \geq 0$, where $Z$ denotes the one-dimensional standard normal random variable and the second inequality holds by the Cauchy-Schwarz inequality.
Note that $\exp(-\overline{\tau} t) \leq \mu_t \leq \exp(-\underline{\tau} t)$ for $t \geq 0$. 
Then,
\be \begin{split}
 & \left\vert  p_{t_*+t}(\bx) - \int_{  \| \bx-\mu_{t}\by \|_{\infty} \leq \tau_{\rm tail} \sigma_{t} \sqrt{\log (1/\delta)} } \ p_{t_*}(\by) \phi_{\sigma_{t}}(\bx-\mu_{t} \by) \d \by \right\vert
 \leq 2 C_{S,1} D  \delta^{\frac{\tau_{\rm tail}^2}{2} - D\overline{\tau} \tau_{\rm t}},
 \\
 & \left\| \sigma_{t} \nabla p_{t_*+t}(\bx) -  \int_{  \| \bx-\mu_{t}\by \|_{\infty} \leq \tau_{\rm tail} \sigma_{t} \sqrt{\log (1/\delta)}  } \ \left( \frac{\mu_t \by - \bx}{\sigma_t} \right)^{\top} p_{t_*}(\by) \phi_{\sigma_{t}}(\bx-\mu_{t} \by) \d \by \right\|_{\infty}
 \\
 & \leq 2  C_{S,1} D  \delta^{\frac{\tau_{\rm tail}^2}{4} - D\overline{\tau} \tau_{\rm t}}
 \label{eqprop3:tailbd}
\end{split} \ee
for $\bx \in \bbR^{D}$ and $0 \leq t \leq \tau_{\rm t} \log (1/\delta)$.
Lemma~\ref{secsc:ptbound} implies that $p_{t_*}(\bx) \leq C_{S,1} \delta^{\tau_{\rm tail}^2/2} $ for $\|\bx\|_{\infty} \geq \mu_{t_*} + \tau_{\rm tail} \sigma_{t_*} \sqrt{\log(1/\delta)}$. Then,
\bean
&& \left\vert \int_{ \substack{  \| \bx-\mu_{t}\by \|_{\infty} \leq \tau_{\rm tail} \sigma_{t} \sqrt{\log (1/\delta)} \\  \|\by\|_{\infty} \geq \mu_{t_*} + \tau_{\rm tail} \sigma_{t_*} \sqrt{\log(1/\delta)} }} p_{t_*}(\by) \phi_{\sigma_{t}}(\bx-\mu_{t} \by) \d \by  \right\vert
\\
&& \leq 2 C_{S,1}   \delta^{\frac{\tau_{\rm tail}^2}{2}}  \int_{ \substack{  \| \bx-\mu_{t}\by \|_{\infty} \leq \tau_{\rm tail} \sigma_{t} \sqrt{\log (1/\delta)} \\  \|\by\|_{\infty} \geq \mu_{t_*} + \tau_{\rm tail} \sigma_{t_*} \sqrt{\log(1/\delta)} }} \phi_{\sigma_{t}}(\bx-\mu_{t} \by) \d \by 
\\
&& \leq  C_{S,1}  \mu_t^{-D} \delta^{\frac{\tau_{\rm tail}^2}{2}} \int_{\bbR^{D}} \phi_{1}(\bz) \d \bz \leq  C_{S,1} \delta^{\frac{\tau_{\rm tail}^2}{2} - D \overline{\tau} \tau_{\rm t} }
\eean
for $\bx \in \bbR^{D}$ and $0 \leq t \leq \tau_{\rm t} \log (1/\delta)$, where the last inequaltiy holds because $\mu_t^{-D} \leq \exp(D\overline{\tau}t)$.
Also,
\bean
&& \left\vert \int_{ \substack{  \| \bx-\mu_{t}\by \|_{\infty} \leq \tau_{\rm tail} \sigma_{t} \sqrt{\log (1/\delta)} \\  \|\by\|_{\infty} \geq \mu_{t_*} + \tau_{\rm tail} \sigma_{t_*} \sqrt{\log(1/\delta)} }} \left( \frac{\mu_t y_i - x_i}{\sigma_t} \right) p_{t_*}(\by) \phi_{\sigma_{t}}(\bx-\mu_{t} \by) \d \by  \right\vert
\\
&& \leq C_{S,1} \delta^{\frac{\tau_{\rm tail}^2}{2}} \mu_t^{-D} \int_{\bbR^{D}} \vert z_i \vert \phi_{1}(\bz) \d \bz \leq \left(\sqrt{\frac{2}{\pi}} \right) C_{S,1}  \delta^{\frac{\tau_{\rm tail}^2}{2} - D\overline{\tau} \tau_{\rm t}}
\eean
for $\bx \in \bbR^{D}$ and $0 \leq t \leq \tau_{\rm t} \log (1/\delta)$.
Combining with (\ref{eqprop3:tailbd}), we have
\be \begin{split}
 & \left\vert  p_{t_*+t}(\bx) - \int_{ \substack{  \| \bx-\mu_{t}\by \|_{\infty} \leq \tau_{\rm tail} \sigma_{t} \sqrt{\log (1/\delta)} \\  \|\by\|_{\infty} \leq \mu_{t_*} + \tau_{\rm tail} \sigma_{t_*} \sqrt{\log(1/\delta)} }} \ p_{t_*}(\by) \phi_{\sigma_{t}}(\bx-\mu_{t} \by) \d \by \right\vert
 \\
 & \leq  C_{S,1} (2D+1)  \delta^{\frac{\tau_{\rm tail}^2}{2} - D\overline{\tau} \tau_{\rm t}}
  \label{eqprop3:tailbd2}
\end{split} \ee
and
\bean
 && \left\| \sigma_{t} \nabla p_{t_*+t}(\bx) -  \int_{ \substack{  \| \bx-\mu_{t}\by \|_{\infty} \leq \tau_{\rm tail} \sigma_{t} \sqrt{\log (1/\delta)} \\  \|\by\|_{\infty} \leq \mu_{t_*} + \tau_{\rm tail} \sigma_{t_*} \sqrt{\log(1/\delta)} }} \ \left( \frac{\mu_t \by - \bx}{\sigma_t} \right)^{\top} p_{t_*}(\by) \phi_{\sigma_{t}}(\bx-\mu_{t} \by) \d \by \right\|_{\infty}
 \\
 && \leq C_{S,1} (2D+ \sqrt{2/\pi})  \delta^{\frac{\tau_{\rm tail}^2}{4} - D\overline{\tau} \tau_{\rm t}}
\eean
for $\bx \in \bbR^{D}$ and $0 \leq t \leq \tau_{\rm t} \log (1/\delta)$.
Let $m_* \in \bbN_{\geq 2}$ be a large enough value as described below and $\by^{(1)},\ldots,\by^{(D_1)} \in \bbR^{D}$ be distinct vectors satisfying that
\bean
&& \left\{\by^{(1)},\ldots,\by^{(D_1)} \right\} = \left\{ \tau_* (n_1,\ldots,n_{D})^{\top} : n_i \in \bbZ, i \in [D] \right\} \cap \left\{ \by \in \bbR^{D} : \| \by \|_{\infty} \leq (m_*-1)\tau_* \right\},
\eean
where $\tau_* = \{ \mu_{t_*} + \tau_{\rm tail} \sigma_{t_*} \sqrt{\log(1/\delta)}\} / m_*$ and $D_{1} = (2m_*-1)^{D}$. Let $\cY_i = \{ \by \in \bbR^{D} : \| \by-\by^{(i)} \|_{\infty} \leq \tau_* \}$ for $i \in [D_1]$.
Taylor's theorem for multivariate function implies that
\bean
p_{t_*}(\by) = \sum_{0\leq k. < \tau_{\rm sm}} \frac{(\D^{\bk} p_{t_*} )(\by^{(i)}) }{\bk !} (\by - \by^{(i)})^{\bk} +  \sum_{k. = \tau_{\rm sm}} \frac{(\D^{\bk} p_{t_*} )(\xi\by^{(i)} + (1-\xi) \by ) }{\bk !} (\by - \by^{(i)})^{\bk}
\eean
for a suitable $\xi \in [0,1]$ and $i \in [D_1], \by \in \bbR^{D}$, where $(\by-\by^{(i)})^{\bk} = \prod_{j=1}^{D} (y_j - y^{(i)}_j)^{k_j}$, $\bk ! = \prod_{j=1}^{D} k_j!$ and $k. = \| \bk \|_{1}$.
Combining with Lemma~\ref{secsc:ptderivativebound}, we have
\bean
&& \left\vert p_{t_*}(\by) -  \sum_{0\leq k. <\tau_{\rm sm}} \frac{(\D^{\bk} p_{t_*} )(\by^{(i)}) }{\bk !} (\by - \by^{(i)})^{\bk} \right\vert \cdot 1\{ \|\by-\by^{(i)} \|_{\infty} \leq \tau_* \}
\\
&& \leq \sum_{k. = \tau_{\rm sm}} C_{S,3} \sigma_{t_*}^{-\tau_{\rm sm}} \prod_{j=1}^{D} \left( \frac{e \tau_*}{k_j} \right)^{k_j} \leq C_{S,3} (\tau_{\rm sm}+1)^{D} \left( \frac{e \tau_*}{\sigma_{t_*}} \right)^{\tau_{\rm sm}}
\eean
for $\by \in  \bbR^{D}$ and $i \in [D_1]$ because  $k! \geq k^{k} e^{-k}$ for any $k \in \bbZ_{\geq 0}$, where $C_{S,3} = C_{S,3}(D,K,\tau_{\rm sm}, \overline{\tau}, \underline{\tau} )$ is the constant in Lemma~\ref{secsc:ptderivativebound}.
Since $\cY_1,\ldots,\cY_{D_{1}}$ are mutually disjoint except on a set of Lebesgue measure zero and $ \bigcup_{i=1}^{D_{1}} \cY_i = \{\by \in \bbR^{D}:  \| \by \|_{\infty} \leq \mu_{t_*} + \tau_{\rm tail} \sigma_{t_*} \sqrt{\log(1/\delta)} \}$, 
we have
\bean
&& \int_{   \|\by\|_{\infty} \leq \mu_{t_*} + \tau_{\rm tail} \sigma_{t_*} \sqrt{\log(1/\delta)} } \ g(\by) \d \by = \sum_{i=1}^{D_1} \int_{   \|\by - \by^{(i)} \|_{\infty} \leq \tau_* }  g(\by) \d \by 
\\
&& = \int_{   \|\by\|_{\infty} \leq \mu_{t_*} + \tau_{\rm tail} \sigma_{t_*} \sqrt{\log(1/\delta)} } \ \sum_{i=1}^{D_1} g(\by) \cdot 1\{ \|\by-\by^{(i)} \|_{\infty} \leq \tau_* \} \d \by
\eean
for any continuous function $g : \bbR^{D} \to \bbR$.
Combining (\ref{eqprop3:tailbd2}) with the last two displays, we have
\be \begin{split}
& \left\vert  p_{t_*+t}(\bx) - \sum_{i=1}^{D_1}   \int_{ \substack{  \| \bx-\mu_{t}\by \|_{\infty} \leq \tau_{\rm tail} \sigma_{t} \sqrt{\log (1/\delta)} \\  \| \by - \by^{(i)} \|_{\infty} \leq \tau_* }} \left\{ \sum_{0\leq k. <\tau_{\rm sm}} \frac{(\D^{\bk} p_{t_*} )(\by^{(i)}) }{\bk !}(\by-\by^{(i)})^{\bk} \right\} \phi_{\sigma_{t}}(\bx-\mu_{t} \by) \d \by \right\vert
\\
& \leq  C_{S,1} (2 D +1)  \delta^{\frac{\tau_{\rm tail}^2}{2} - D\overline{\tau} \tau_{\rm t}} 
\\
& \quad +   C_{S,3} (\tau_{\rm sm}+1)^{D} \left( \frac{e \tau_*}{\sigma_{t_*}} \right)^{\tau_{\rm sm}} \int_{ \substack{  \| \bx-\mu_{t}\by \|_{\infty} \leq \tau_{\rm tail} \sigma_{t} \sqrt{\log (1/\delta)} \\  \|\by\|_{\infty} \leq \mu_{t_*} + \tau_{\rm tail} \sigma_{t_*} \sqrt{\log(1/\delta)} }} \ \sum_{i=1}^{D_1}  \phi_{\sigma_{t}}(\bx-\mu_{t} \by) \cdot 1\{ \|\by-\by^{(i)} \|_{\infty} \leq \tau_* \} \d \by
\\
& \leq C_{S,1} (2 D +1)  \delta^{\frac{\tau_{\rm tail}^2}{2} - D\overline{\tau} \tau_{\rm t}} +  C_{S,3} (\tau_{\rm sm}+1)^{D} \left( \frac{e \tau_* }{\sigma_{t_*}} \right)^{\tau_{\rm sm}} \delta^{ -D\overline{\tau} \tau_{\rm t} } 
\label{eqprop3:p0errorbd}
\end{split} \ee
for $\bx \in \bbR^{D}$ and $0 \leq t \leq \tau_{\rm t} \log(1/\delta)$, where the last inequality holds because $\int_{\bbR^{D}} \phi_{\sigma_t}(\bx-\mu_t \by) \d \by = \mu_t^{-D}$ and $\mu_t^{-D} \leq \exp(D\overline{\tau}t) \leq \delta^{-D\overline{\tau}\tau_{\rm t} }$.
Also, we have
\be \begin{split}
& \left\| \sigma_t \nabla p_{t_*+t}(\bx) - \sum_{i=1}^{D_1} \int_{ \substack{  \| \bx-\mu_{t}\by \|_{\infty} \leq \tau_{\rm tail} \sigma_{t} \sqrt{\log (1/\delta)} \\  \| \by - \by^{(i)} \|_{\infty} \leq \tau_* }}    \right.
\\
& \left. \qquad \left\{ \sum_{0\leq k. <\tau_{\rm sm}} \frac{(\D^{\bk} p_{t_*} )(\by^{(i)}) }{\bk !}  \ \left( \frac{\mu_t \by - \bx}{\sigma_t} \right)^{\top} (\by-\by^{(i)})^{\bk} \right\} \phi_{\sigma_{t}}(\bx-\mu_{t} \by) \d \by \right\|_{\infty}
\\
& \leq    C_{S,1} (2D + \sqrt{2 / \pi})  \delta^{\frac{\tau_{\rm tail}^2}{4} - D\overline{\tau} \tau_{\rm t}}  +  C_{S,3} (\tau_{\rm sm}+1)^{D} \left( \frac{e \tau_*}{\sigma_{t_*}} \right)^{\tau_{\rm sm}}
\\
& \qquad \cdot  \int_{ \substack{  \| \bx-\mu_{t}\by \|_{\infty} \leq \tau_{\rm tail} \sigma_{t} \sqrt{\log (1/\delta)} \\  \|\by\|_{\infty} \leq \mu_{t_*} + \tau_{\rm tail} \sigma_{t_*} \sqrt{\log(1/\delta)} }} \sum_{i=1}^{D_1}  \left\| \frac{\mu_t \by - \bx}{\sigma_t} \right\|_{\infty}  \phi_{\sigma_{t}}(\bx-\mu_{t} \by) \cdot 1\{ \|\by-\by^{(i)} \|_{\infty} \leq \tau_* \} \d \by
\\
& \leq    C_{S,1} (2D + \sqrt{2 / \pi}) \delta^{\frac{\tau_{\rm tail}^2}{4} - D\overline{\tau} \tau_{\rm t}} + C_{S,3} \tau_{\rm tail} (\tau_{\rm sm}+1)^{D} \left( \frac{e \tau_*}{\sigma_{t_*}} \right)^{\tau_{\rm sm}} \delta^{ -D\overline{\tau} \tau_{\rm t} }  \sqrt{\log(1/\delta)} \label{eqprop3:p0errorbdgrad}
\end{split} \ee
for $\bx \in \bbR^{D}$ and $0 \leq t \leq \tau_{\rm t} \log(1/\delta)$, where the last inequality holds because $\int_{\bbR^{D}} \phi_{\sigma_t}(\bx-\mu_t \by) \d \by = \mu_t^{-D}$ and $\mu_t^{-D} \leq \exp(D\overline{\tau}t) \leq \delta^{-D\overline{\tau}\tau_{\rm t} }$.

For $x,y \in \bbR$ and $t > 0$, Taylor's theorem yields that 
\bean
\left\vert \exp \left( -\frac{(x-\mu_{t}y)^2}{2 \sigma_{t}^2} \right) - \sum_{l=0}^{D_{2}-1} \frac{1}{l!} \left( -\frac{(x-\mu_{t}y)^2}{2 \sigma_{t}^2} \right)^{l} \right\vert 
\leq \frac{1}{D_{2} !} \left( \frac{(x-\mu_{t}y)^2}{2 \sigma_{t}^2} \right)^{D_{2}},
\eean
where $D_{2} = \lfloor 2e \tau_{\rm tail}^2 \log(1/\delta) \rfloor +1$ with small enough $\delta$ so that $D_{2} \geq 1$.
For $\vert x - \mu_{t} y \vert \leq  \tau_{\rm tail} \sigma_{t} \sqrt{\log(1/\delta)}$ and $t \geq 0$, the last display is further bounded by
\bean
\left(\frac{e \tau_{\rm tail}^2 \log (1/\delta)  }{D_{2}} \right)^{D_{2}} \leq 2^{-D_{2}} \leq \delta^{2   e\tau_{\rm tail}^2  \log 2} \leq \delta^{e\tau_
{\rm tail}^2},
\eean
where the last inequality holds because $1/2 \leq \log 2$ and $0 < \delta < 1$.
Then,
\bean
&& \left\vert \int_{ \substack{\vert x - \mu_{t} y \vert \leq  \tau_{\rm tail} \sigma_{t} \sqrt{\log(1/\delta)} \\ \vert y - \widetilde y \vert \leq \tau_* } } \ \left(\frac{\mu_t y - x}{\sigma_t}\right)^{u}  (y-\widetilde y)^{k} \exp \left( -\frac{(x-\mu_{t}y)^2}{2 \sigma_{t}^2} \right)  \d y \right.
\\
&& \quad \left. - \sum_{l=0}^{D_{2}-1} \frac{1}{l!}
\int_{ \substack{ \vert x - \mu_{t} y \vert \leq \tau_{\rm tail} \sigma_{t} \sqrt{\log(1/\delta)} \\ \vert y - \widetilde y \vert \leq \tau_* } } \ \left(\frac{\mu_t y - x}{\sigma_t} \right)^{u}  (y-\widetilde y)^{k}  \left( -\frac{(x-\mu_{t}y)^2}{2 \sigma_{t}^2} \right)^{l} \d y \right\vert
\\
&& \leq \delta^{e\tau_{\rm tail}^2} \int_{ \substack{ \vert x - \mu_{t} y \vert \leq  \tau_{\rm tail} \sigma_{t} \sqrt{\log(1/\delta)} \\ \vert y - \widetilde y \vert \leq \tau_* } } \left\vert \frac{\mu_t y - x}{\sigma_t} \right\vert^{u} \vert y - \widetilde y \vert^{k} \d y 
\\
&& \leq \delta^{e\tau_{\rm tail}^2} \int_{\vert y \vert \leq m_* \tau_*} \left\{ \tau_{\rm tail} \sqrt{\log(1/\delta)} \right\}^{u}    \tau_*^{k}   \ \d y
= m_* \tau_*^{k+1} \left\{ \tau_{\rm tail} \sqrt{\log(1/\delta)} \right\}^{u} 
 \delta^{e\tau_{\rm tail}^2}
\eean
for any $x, \widetilde y \in \bbR$ with $\vert \widetilde y \vert = (m_*-1)\tau_*$,$t \geq 0, k \in \bbZ_{\geq 0}$ and $u \in \{0,1\}$.
Note that 
$ \mu_{t_*} \leq 1- \underline{\tau}t_* / 2 \leq 1$ and $\sqrt{\underline{\tau} t_*} \leq \sigma_{t_*} \leq \sqrt{2\overline{\tau}t_*} \leq 1$ with $t_* \leq (2\overline{\tau})^{-1}$.
For $k \leq \tau_{\rm sm}$, the last display is bounded by 
\bean
&& (1+\tau_{\rm tail})^{k+1} m_*^{-k}    \delta^{e\tau_{\rm tail}^2} \{ \log(1/\delta) \}^{\frac{k+1}{2}} \left\{ \tau_{\rm tail} \sqrt{\log(1/\delta)} \right\}^{u} 
\\
&& \leq \tau_{\rm tail} (1+\tau_{\rm tail})^{\tau_{\rm sm}+1} \delta^{e\tau_{\rm tail}^2}  \{\log(1/\delta)\}^{\frac{\tau_{\rm sm}}{2}+1} \
\eean
because $m_* \tau_* = \mu_{t_*}  + \sigma_{t_*} \tau_{\rm tail} \sqrt{ \log (1/\delta)}$ and  $\mu_{t_*}  + \sigma_{t_*} \tau_{\rm tail} \sqrt{ \log (1/\delta)} \leq (1+\tau_{\rm tail}) \sqrt{\log(1/\delta)}$ with small enough $\delta$ so that $\delta \leq 1/e$ and $\tau_{\rm tail} \sqrt{\log(1/\delta)} \geq 1$.
Moreover,
\bean
&& \left\vert \sum_{l=0}^{D_{2}-1} \frac{1}{l!}
\int_{ \substack{ \vert x - \mu_{t} y \vert \leq \tau_{\rm tail} \sigma_{t} \sqrt{\log(1/\delta)} \\ \vert y - \widetilde y \vert \leq \tau_* } } \ \left(\frac{\mu_t y - x}{\sigma_t} \right)^{u}  (y-\widetilde y)^{k}  \left( -\frac{(x-\mu_{t}y)^2}{2 \sigma_{t}^2} \right)^{l} \d y \right\vert
\\
&& \leq  \tau_{\rm tail}\left(1+ \delta^{e\tau_{\rm tail}^2} \right)
 (1+\tau_{\rm tail})^{ \tau_{\rm sm} +1}  \{ \log(1/\delta) \}^{\frac{ \tau_{\rm sm}}{2}+1} 
\eean
for any $x, \widetilde y \in \bbR$ with $\vert \widetilde y \vert = (m_*-1)\tau_*$,$t \geq 0, 0 \leq k \leq \tau_{\rm sm}$ and $u \in \{0,1\}$
 because
\bean
&& \left\vert \int_{ \substack{\vert x - \mu_{t} y \vert \leq  \tau_{\rm tail} \sigma_{t} \sqrt{\log(1/\delta)} \\ \vert y - \widetilde y \vert \leq \tau_* } } \ \left(\frac{\mu_t y - x}{\sigma_t}\right)^{u}  (y-\widetilde y)^{k} \exp \left( -\frac{(x-\mu_{t}y)^2}{2 \sigma_{t}^2} \right)  \d y \right\vert 
\\
&& \leq  \int_{\vert y \vert \leq m_* \tau_*} \left\{ \tau_{\rm tail} \sqrt{\log(1/\delta)} \right\}^{u} \tau_*^k \d y \leq m_* \tau_*^{k+1} \tau_{\rm tail} \sqrt{\log(1/\delta)} 
\\
&& \leq \tau_{\rm tail} (1+\tau_{\rm tail})^{\tau_{\rm sm}+1}  \{ \log(1/\delta) \}^{\frac{\tau_{\rm sm}}{2}+1}.
\eean
Note that $\vert \prod_{j=1}^{D} x_j - \prod_{j=1}^{D} \widetilde x_j \vert \leq D C^{D-1} \| \bx - \widetilde \bx\|_{\infty}$ for any $\bx, \widetilde \bx \in [-C,C]^{D}$.
Since the last two displays are bounded by
$
2\tau_{\rm tail}(1+\tau_{\rm tail})^{\tau_{\rm sm}+1}  \{\log(1/\delta)\}^{\frac{\tau_{\rm sm}}{2} +1},
$
we have
\be \begin{split}
& \left\vert \int_{ \substack{ \| \bx-\mu_{t}\by \|_{\infty} \leq \tau_{\rm tail} \sigma_{t} \sqrt{\log (1/\delta)}
\\ \| \by - \by^{(i)} \|_{\infty} \leq \tau_*  } } \ \left(\frac{\mu_t y_h - x_h}{\sigma_t}\right)^{u}    (\by - \by^{(i)})^{\bk} \phi_{\sigma_{t}}(\bx-\mu_{t} \by) \d \by  -   \right.
\\
& \prod_{j=1}^{D}  \sum_{l=0}^{D_{2}-1} \left.  \frac{1}{l! \sqrt{2\pi } \sigma_t}
\int_{ \substack{ \vert x_j-\mu_{t} y_j \vert \leq \tau_{\rm tail} \sigma_{t} \sqrt{\log (1/\delta)}
\\ \vert y_j - y^{(i)}_j \vert \leq \tau_* } } 
\ \left(\frac{\mu_t y_h - x_h}{\sigma_t}\right)^{u \cdot 1\{h = j\} }   (y_j - y^{(i)}_j)^{k_j} \left( -\frac{(x_j -\mu_{t} y_j)^2}{2 \sigma_{t}^2} \right)^{l}  \d y_j \right\vert
\\
& \leq (2\pi \sigma_{t}^2)^{-\frac{D}{2}} D 2^{D-1} \tau_{\rm tail}^{D}(1+\tau_{\rm tail})^{D(\tau_{\rm sm}+1)}  \delta^{e \tau_{\rm tail}^2}  \{\log(1/\delta)\}^{D(\frac{\tau_{\rm sm}}{2} +  1)} \label{eqprop3:experrorbd}
 \end{split} \ee
for $\bx \in \bbR^{D}, t > 0, h \in [D], i \in [D_{1}], j \in [D], \bk \in \bbZ_{\geq 0}^{D}$ and $u \in \{0,1\}$ with $k. \leq \tau_{\rm sm}$.
With $u = 0$ in the last display, the second integral satisfies that
\bean
&& \frac{1}{l! \sqrt{2\pi} \sigma_t}\int_{ \substack{ \vert x_j-\mu_{t} y_j \vert \leq \tau_{\rm tail} \sigma_{t} \sqrt{\log (1/\delta)}
\\ \vert y_j - y^{(i)}_j \vert \leq \tau_* } } \   \left(y_j - y^{(i)}_j \right)^{k_j} \left( -\frac{(x_j -\mu_{t} y_j)^2}{2 \sigma_{t}^2} \right)^{l}  \d y_j 
\\
&& = \frac{\mu_{t}^{-1}}{l!  \sqrt{2\pi}}\int_{ \substack{ \vert z_j \vert \leq \tau_{\rm tail}  \sqrt{\log (1/\delta)}
\\ \vert \mu_t^{-1}\sigma_t z_j + \mu_t^{-1} x_j - y_j^{(i)} \vert \leq \tau_*  } } \   \left(\mu_t^{-1}\sigma_t z_j + \mu_t^{-1} x_j - y_j^{(i)} \right)^{k_j} z_j^{2l} (-2)^{-l}  \d z_j 
\\
&& = \frac{(-2)^{-l} \mu_{t}^{-1}  }{l! \sqrt{2\pi}}\int_{ \substack{ \vert z_j \vert \leq \tau_{\rm tail}  \sqrt{\log (1/\delta)}
\\ \vert \mu_t^{-1}\sigma_t z_j + \mu_t^{-1} x_j - y_j^{(i)} \vert \leq \tau_*  } }  \ \sum_{r_j = 0}^{k_j} 
\binom{k_j}{r_j}
\left(\mu_t^{-1} \sigma_t \right)^{r_j}  \left(\mu_t^{-1} x_j - y_j^{(i)} \right)^{k_j-r_j} z_j^{r_j+2l} \ \d z_j
\\
&& = \frac{(-2)^{-l} \mu_t^{-k_j-1}  }{l! \sqrt{2\pi} }  \sum_{r_j = 0}^{k_j} 
\binom{k_j}{r_j}
 \sigma_t^{r_j}  \left( x_j - \mu_t y_j^{(i)} \right)^{k_j-r_j}  \left( \frac{ \overline{z}_{i,j}^{r_j+2l+1} - \underline{z}_{i,j}^{r_j+2l+1} }{r_j+2l+1} \right) \defeq P_{i,j,k_j,l}(x_j, t),
\eean
where
\bean
&& \overline{z}_{i,j} = \min \left( \max \left( \frac{ \mu_t (y_j^{(i)} + \tau_* ) - x_j }{\sigma_{t}}, -\tau_{\rm tail} \sqrt{\log (1/\delta)} \right), \tau_{\rm tail} \sqrt{\log (1/\delta)} \right),
\\
&& \underline{z}_{i,j} =\min \left( \max \left( \frac{ \mu_t (y_j^{(i)} - \tau_* ) - x_j }{\sigma_{t}}, -\tau_{\rm tail} \sqrt{\log (1/\delta)} \right), \tau_{\rm tail} \sqrt{\log (1/\delta)} \right).
\eean
Combining (\ref{eqprop3:p0errorbd}) and (\ref{eqprop3:experrorbd}) with the last two displays, we have
\bean
&&  \left\vert p_{t_*+t}(\bx) - g_{t_*,t}(\bx) \right\vert
 \\
 &&\leq 
C_{S,1} (2D +1)  \delta^{\frac{\tau_{\rm tail}^2}{2} - D\overline{\tau} \tau_{\rm t}} +  C_{S,3} (\tau_{\rm sm}+1)^{D} \left( \frac{e \tau_* }{\sigma_{t_*}} \right)^{\tau_{\rm sm}} \delta^{ -D\overline{\tau} \tau_{\rm t} } 
 \\
 && \quad + \sum_{i=1}^{D_1} 
\sum_{0 \leq k. < \tau_{\rm sm}} \left\vert \frac{ (\D^{\bk} p_{t_*})(\by^{(i)})  }{\bk!} \right\vert (2\pi \sigma_{t}^2)^{-\frac{D}{2}} D 2^{D-1} \tau_{\rm tail}^{D}(1+\tau_{\rm tail})^{D(\tau_{\rm sm}+1)}  \delta^{e \tau_{\rm tail}^2}  \{\log(1/\delta)\}^{D(\frac{\tau_{\rm sm}}{2} +  1)}
\\
\eean
for $\bx \in \bbR^{D}$ and $0 < t \leq \tau_{\rm t} \log(1/\delta)$, where $g_{t_*,t} : \bbR^{D} \to \bbR$ is a function such that
\bean
 g_{t_*,t}(\bx)=
 \sum_{i=1}^{D_{1}} \sum_{0 \leq k. < \tau_{\rm sm}} \left\{ 
\frac{ (\D^{\bk} p_{t_*})(\by^{(i)}) }{\bk! } \right\} \prod_{j=1}^{D} \sum_{l=0}^{D_{2}-1}  P_{i,j,k_j,l}(x_j, t)
\eean
for $\bx \in \bbR^{D}$. 
Let 
\bean
m_* = \left\lfloor m^{1/D} +1 \right\rfloor
\quad \text{and} \quad
t_* = m^{ - \frac{ 2-2\tau_{\rm low}}{D}}
\eean
with large enough $m \in \bbR$ so that $t_* \leq (2\overline{\tau})^{-1}$ and $m_* \in \bbN_{\geq 2}$.
Since $\sigma_t^2 \geq 1-\exp(-2\underline{\tau} t)$ for $ t \geq 0$ and $1-\exp(-x) \geq x/2$ for $0 \leq x \leq 1$, we have $\sigma_t^{-D} \leq (\underline{\tau} \delta)^{-D/2}$ for $t \geq \delta$ with $\delta \leq \underline{\tau}^{-1}$.
Then,
\bean
&&  \left\vert p_{t_*+t}(\bx) - g_{t_*,t}(\bx) \right\vert
\\
&& \leq  D_{3} \left[   \delta^{\frac{D\tau_{\rm tail}^2}{2} - D\overline{\tau} \tau_{\rm t}} + m_*^{-\tau_{\rm sm}} t_{*}^{-\frac{\tau_{\rm sm}}{2}}   \delta^{ -D\overline{\tau} \tau_{\rm t} } \{ \log(1/\delta)\}^{\frac{\tau_{\rm sm}}{2}}  + m_*^{D}t_{*}^{-\frac{\tau_{\rm sm}}{2}} \delta^{e \tau_{\rm tail}^2 - D}  \{\log(1/\delta)\}^{D(\frac{\tau_{\rm sm}}{2} +  1)}  \right]
\\
&& \leq D_{3} \left[ \delta^{\frac{D\tau_{\rm tail}^2}{2} - D\overline{\tau} \tau_{\rm t}} + m^{-\frac{\tau_{\rm low} \tau_{\rm sm} }{D}}\delta^{-D\overline{\tau} \tau_{\rm t}  }  \{\log(1/\delta)\}^{\frac{\tau_{\rm sm}}{2} } + 2^{D} m^{\frac{D+1-\tau_{\rm low}}{D}} \delta^{e \tau_{\rm tail}^2 - D}  \{\log(1/\delta)\}^{D(\frac{\tau_{\rm sm}}{2} +  1)} \right]
\eean
for $\bx \in \bbR^{D}$ and $\delta \leq t \leq \tau_{\rm t} \log(1/\delta)$,
where $D_{3} = D_{3}(D,\tau_{\rm sm},\underline{\tau},\tau_{\rm tail},C_{S,1},C_{S,3})$ and the last inequality holds because $m^{1/D} \leq m_{*} \leq 2m^{1/D}$.
Since $\tau_{\rm tail}^2 = 4(D \overline{\tau} \tau_{\rm t} + 1 ) \vee (D+1)/e$, we have
\bean
\frac{\tau_{\rm tail}^2}{2} -D \overline{\tau} \tau_{\rm t} \geq 1
\quad \text{and} \quad e \tau_{\rm tail}^2 - D \geq 1.
\eean
Then,
\be \begin{split}
& \left\vert p_{t_*+t}(\bx) - g_{t_*,t}(\bx) \right\vert
\\
& \leq D_{3} \left[ \delta + m^{-\frac{\tau_{\rm low} \tau_{\rm sm} }{D}}\delta^{-D\overline{\tau} \tau_{\rm t}  }  \{\log(1/\delta)\}^{\frac{\tau_{\rm sm}}{2} } + 2^{D} m^{\frac{D+1-\tau_{\rm low}}{D}} \delta \{\log(1/\delta)\}^{D(\frac{\tau_{\rm sm}}{2} +  1)} \right]
\label{eqprop3:ptpoly}
\end{split} \ee
for $\bx \in \bbR^{D}$ and $\delta \leq t \leq \tau_{\rm t} \log(1/\delta)$.
Similarly, with $u = 1$ and $h = j$ in (\ref{eqprop3:experrorbd}), the last integral in (\ref{eqprop3:experrorbd}) satisfies that
\bean
&& \frac{1}{l! \sqrt{2\pi }\sigma_t} \int_{ \substack{ \vert x_j-\mu_{t} y_j \vert \leq \tau_{\rm tail} \sigma_{t} \sqrt{\log (1/\delta)} 
\\ \vert y_j - y^{(i)}_j \vert \leq \tau_* } } \
\left(\frac{\mu_t y_j - x_j}{\sigma_t}\right)
(y_j - y^{(i)}_j)^{k_j} \left( -\frac{(x_j -\mu_{t} y_j)^2}{2 \sigma_{t}^2} \right)^{l}  \d y_j  
\\
&& = \frac{\mu_{t}^{-1}}{l! \sqrt{2\pi}}\int_{ \substack{ \vert z_j \vert \leq \tau_{\rm tail}  \sqrt{\log (1/\delta)}
\\ \vert \mu_t^{-1}\sigma_t z_j + \mu_t^{-1} x_j - y_j^{(i)} \vert \leq \widetilde \tau_*  } } \   \left(\mu_t^{-1}\sigma_t z_j + \mu_t^{-1} x_j - y_j^{(i)} \right)^{k_j} z_j^{2l+1} (-2)^{-l}  \d z_j
\\
&& = \frac{(-2)^{-l} \mu_t^{-k_j-1}  }{l! \sqrt{2\pi}}  \sum_{r_j = 0}^{k_j} 
\binom{k_j}{r_j}
 \sigma_t^{r_j}  \left( x_j - \mu_t y_j^{(i)} \right)^{k_j-r_j}  \left( \frac{ \overline{z}_{i,j}^{r_j+2l+2} - \underline{z}_{i,j}^{r_j+2l+2} }{r_j+2l+2} \right) \defeq \widetilde P_{i,j,k_j,l}(x_j, t).
\eean
Combining (\ref{eqprop3:p0errorbdgrad}) and (\ref{eqprop3:experrorbd}) with the last display, we have
\bean
&& \left\vert \sigma_t  \left( \nabla p_{t_*+t}(\bx) \right)_{h} - \widetilde g_{t_*,t}^{(h)}(\bx) \right\vert
 \\
&& \leq 
   C_{S,1} (2D + \sqrt{2/\pi})  \delta^{\frac{ \tau_{\rm tail}^2}{4} - D\overline{\tau} \tau_{\rm t}} + C_{S,3} \tau_{\rm tail} (\tau_{\rm sm}+1)^{D} \left( \frac{e \tau_*}{\sigma_{t_*}} \right)^{\tau_{\rm sm}} \delta^{ -D\overline{\tau} \tau_{\rm t} }  \sqrt{\log(1/\delta)}
 \\
&& \quad + \sum_{i=1}^{D_1} 
\sum_{0 \leq k. < \tau_{\rm sm}} \left\vert \frac{ (\D^{\bk} p_{t_*})(\by^{(i)})  }{\bk!} \right\vert (2\pi \sigma_{t}^2)^{-\frac{D}{2}} D 2^{D-1} \tau_{\rm tail}^{D}(1+\tau_{\rm tail})^{D(\tau_{\rm sm}+1)}  \delta^{e \tau_{\rm tail}^2}  \{\log(1/\delta)\}^{D(\frac{\tau_{\rm sm}}{2} +  1)}
\\
&&  \leq D_{4} \left[   \delta^{\frac{\tau_{\rm tail}^2}{4} - D\overline{\tau} \tau_{\rm t}} + m_*^{-\tau_{\rm sm}} t_{*}^{-\frac{\tau_{\rm sm}}{2}}   \delta^{ -D\overline{\tau} \tau_{\rm t} } \{ \log(1/\delta)\}^{\frac{\tau_{\rm sm}+1}{2}}  + m_*^{D}t_{*}^{-\frac{\tau_{\rm sm}}{2}} \delta^{e \tau_{\rm tail}^2-D}  \{\log(1/\delta)\}^{D(\frac{\tau_{\rm sm}}{2} +  1)}  \right]
\\
&& = D_{4} \left[ \delta^{\frac{\tau_{\rm tail}^2}{4} - D\overline{\tau} \tau_{\rm t}} + m^{-\frac{\tau_{\rm low} \tau_{\rm sm} }{D}}\delta^{-D \overline{\tau} \tau_{\rm t}}  \{\log(1/\delta)\}^{\frac{\tau_{\rm sm} +1}{2} } + 2^{D} m^{\frac{D+1-\tau_{\rm low}}{D}} \delta^{e \tau_{\rm tail}^2 - D}  \{\log(1/\delta)\}^{D(\frac{\tau_{\rm sm}}{2} +  1)} \right]
\eean
for $h \in [D]$, $\bx \in \bbR^{D}$ and $\delta \leq t \leq \tau_{\rm t} \log(1/\delta)$,
where $D_{4} = D_{4}(D,\tau_{\rm sm},\underline{\tau},\tau_{\rm tail},C_{S,1},C_{S,3})$ and $\widetilde g_{t_*,t}^{(h)} : \bbR^{D} \to \bbR, h \in [D]$ is a function such that
\bean
\widetilde g_{t_*,t}^{(h)}(\bx)
= \sum_{i=1}^{D_{1}} \sum_{0 \leq k. < \tau_{\rm sm}} \left\{ 
 (\D^{\bk} p_{t_*})(\by^{(i)})  \right\} \left\{  \prod_{\substack{j=1 \\ j \neq h}}^{D} \sum_{l=0}^{D_{2}-1}   \frac{ P_{i,j,k_j,l}(x_j, t) }{k_j !} \right\}
 \left\{ \sum_{l=0}^{D_{2}-1} \frac{ \widetilde P_{i,h,k_h,l} (x_h, t) }{k_h !} \right\}.
\eean
Since $\tau_{\rm tail}^2 = 4(D \overline{\tau} \tau_{\rm t} + 1 ) \vee (D+1)/e$, we have
\bean
\frac{\tau_{\rm tail}^2}{4} -D \overline{\tau} \tau_{\rm t} \geq 1
\quad \text{and} \quad e \tau_{\rm tail}^2 - D \geq 1.
\eean
Then,
\be \begin{split}
& \left\vert \sigma_t  \left( \nabla p_{t_*+t}(\bx) \right)_{h} - \widetilde g_{t_*,t}^{(h)}(\bx) \right\vert
\\
& \leq D_{4} \left[ \delta + m^{-\frac{\tau_{\rm low} \tau_{\rm sm} }{D}}\delta^{-D \overline{\tau} \tau_{\rm t}}  \{\log(1/\delta)\}^{\frac{\tau_{\rm sm} +1}{2} } + 2^{D} m^{\frac{D+1-\tau_{\rm low}}{D}} \delta  \{\log(1/\delta)\}^{D(\frac{\tau_{\rm sm}}{2} +  1)} \right], \quad h \in [D]
 \label{eqprop3:nablapoly}
\end{split} \ee
for $\bx \in \bbR^{D}$ and $\delta \leq t \leq \tau_{\rm t} \log(1/\delta)$.
Let
\bean
\delta = m^{- \frac{\tau_{\rm low} \tau_{\rm sm} + D + 1 - \tau_{\rm low}  }{D (1 + D \overline{\tau} \tau_{\rm t} ) }}
\eean
for large enough $m \in \bbR$.
Combining (\ref{eqprop3:ptpoly}) and (\ref{eqprop3:nablapoly}) with the last display, we have
\be \begin{split}
& \left\vert p_{t_*+t}(\bx) - g_{t_*,t}(\bx) \right\vert
\leq D_{5} m^{ - \frac{ \tau_{\rm low} \tau_{\rm sm} - (D + 1 - \tau_{\rm low}) D \overline{\tau} \tau_{\rm t} }{D ( 1 + D \overline{\tau} \tau_{\rm t} ) } } ( \log m )^{D (\frac{\tau_{\rm sm}}{2} + 1) },
\\
& \left\vert \sigma_t  \left( \nabla p_{t_*+t}(\bx) \right)_{h} - \widetilde g_{t_*,t}^{(h)}(\bx) \right\vert
\leq D_{5} m^{ - \frac{ \tau_{\rm low} \tau_{\rm sm} - (D + 1 - \tau_{\rm low}) D \overline{\tau} \tau_{\rm t} }{D ( 1 + D \overline{\tau} \tau_{\rm t} ) } } ( \log m )^{D (\frac{\tau_{\rm sm}}{2} + 1) }, \quad h \in [D]
\label{eqprop3:bothpoly}
\end{split} \ee
for $\bx \in \bbR^{D}$ and $\delta \leq t \leq \tau_{\rm t} \log(1/\delta)$, where $D_{5} = D_{5} ( D, \overline{\tau}, \tau_{\rm t}, \tau_{\rm sm}, \tau_{\rm low}, D_{3}, D_{4} )$.

Let $0 < \widetilde \delta < \delta$ be a small enough value as described below.
With ${\widetilde \delta}^2 < 1/2$, Lemma~\ref{secnn:mtst} implies that
there exist neural networks  $f_{\mu} \in \cF_{\rm NN}(L_{\mu},\bd_{\mu},s_{\mu},M_{\mu}), f_{\sigma} \in \cF_{\rm NN}(L_{\sigma},\bd_{\sigma},s_{\sigma},M_{\sigma})$ with
\be \begin{split}
& L_{\mu}, L_{\sigma} \leq C_{N,4} \{ \log(1/\widetilde \delta)\}^{2}, \quad \|\bd_{\mu} \|_{\infty},  \| \bd_{\sigma}\|_{\infty} \leq C_{N,4}  \{ \log(1/\widetilde \delta)\}^{2}
\\
& s_{\mu}, s_{\sigma} \leq C_{N,4}  \{  \log(1/\widetilde \delta)\}^{3}, \quad M_{\mu}, M_{\sigma} \leq C_{N,4}  \log ( 1/\widetilde \delta) \label{eqprop3:nnmtst}
\end{split} \ee
such that $
\vert \mu_{t} - f_{\mu}(t) \vert \leq \widetilde \delta$ for $t \geq 0$ and $ \vert \sigma_{t} - f_{\sigma}(t) \vert \leq \widetilde \delta$
for $ t \geq \widetilde \delta$, where $C_{N,4}$ is the constant in Lemma~\ref{secnn:mtst}.
Also, Lemma~\ref{secnn:rec} implies that there exist a neural network $f_{\rm rec} \in \cF_{\rm NN}(L_{\rm rec}, \bd_{\rm rec}, s_{\rm rec}, M_{\rm rec} )$ with
\bean
&& L_{\rm rec} \leq C_{N,5} \{ \log(1/ \widetilde \delta )\}^{2}, \quad \|\bd_{\rm rec}\|_{\infty} \leq C_{N,5}  \{ \log (1/ \widetilde \delta )\}^{3}
\\
&& s_{\rm rec}  \leq C_{N,5}  \{ \log(1/ \widetilde \delta )\}^{4}, \quad M_{\rm rec} \leq C_{N,5} \widetilde \delta^{-2}
\eean
such that $\vert 1/x - f_{\rm rec}(x) \vert \leq  \widetilde \delta $ for any $x \in [ \widetilde \delta,1/ \widetilde \delta]$, where $C_{N,5}$ is the constant in Lemma~\ref{secnn:rec}.
Since $  \sigma_t-\widetilde \delta \leq f_{\sigma}(t)  \leq \sigma_t + \widetilde \delta$ for $t \geq \widetilde \delta$ and $\sqrt{\underline{\tau} \delta} \leq \sigma_t \leq 1$ for $t \geq \delta$ , we have $\delta \leq f_{\sigma}(t)  \leq 2$ for $t \geq \delta$ with small enough $\widetilde \delta$ so that $\widetilde \delta \leq \sqrt{\underline{\tau} \delta} - \delta$ and $\widetilde \delta \leq 1$.
Then,
\be \begin{split}
& \left\vert 1/\sigma_t - f_{\rm rec}(f_{\sigma}(t)) \right\vert \leq
\left\vert 1/\sigma_t - 1/ f_{\sigma}(t) \right\vert  + \left\vert  1/ f_{\sigma}(t) - f_{\rm rec}(f_{\sigma}(t)) \right\vert 
\\
& \leq \{ \sigma_t \wedge f_{\sigma}(t)\}^{-2} \vert \sigma_t - f_{\sigma}(t) \vert + \widetilde \delta \leq (1+\delta^{-2}) \widetilde \delta  \label{eqprop3:nnrecst}
\end{split} \ee
 for $t \geq \delta$. 
Since $\mu_t - \widetilde \delta \leq f_{\mu}(t) \leq \mu_t + \widetilde \delta $ for $t \geq \widetilde \delta$ and $\delta^{\overline{\tau} \tau_{\rm t} } \leq \mu_t \leq 1$ for $0 \leq t \leq \tau_{\rm t} \log (1/\delta)$, we have $\delta^{\overline{\tau} \tau_{\rm t} }/2 \leq f_{\mu}(t) \leq 2$  for $\delta \leq t \leq \tau_{\rm t} \log (1/\delta)$ with small enough $\widetilde \delta$ so that $\widetilde \delta \leq \delta^{\overline{\tau} \tau_{\rm t} }/2 $ and $\widetilde \delta \leq 1$. 
Then,
\be \begin{split}
& \left\vert 1/\mu_t - f_{\rm rec}(f_{\mu}(t)) \right\vert \leq
\left\vert 1/\mu_t - 1/ f_{\mu}(t) \right\vert  + \left\vert  1/ f_{\mu}(t) - f_{\rm rec}(f_{\mu}(t)) \right\vert 
\\
& \leq \{ \mu_t \wedge f_{\mu}(t)\}^{-2} \vert \mu_t - f_{\mu}(t) \vert + \widetilde \delta \leq (1+ 4 \delta^{-2\overline{\tau} \tau_{\rm t} })  \widetilde \delta
\label{eqprop3:nnrecmt}
\end{split} \ee
for $\delta \leq t \leq \tau_{\rm t} \log (1/\delta)$.
Lemma~\ref{secnn:mult} implies that for $k \geq 2$, there exists a neural network $\widetilde f_{\rm mult}^{(k)} \in \cF_{\rm NN}(\widetilde L_{\rm mult}^{(k)}, \widetilde \bd_{\rm mult}^{(k)},\widetilde s_{\rm mult}^{(k)}, \widetilde M_{\rm mult}^{(k)})$ with 
\bean
&& \widetilde L_{\rm mult}^{(k)} \leq C_{N,1} \log k  \{\log(1/ \widetilde \delta) + \overline{\tau} \tau_{\rm t} \log(1/\delta) \}, \quad \widetilde d_{\rm mult}^{(k)} = (k,48k,\ldots,48k,1)^{\top},
\\
&& \widetilde s_{\rm mult}^{(2)} \leq  C_{N,1} k \{\log(1/ \widetilde \delta) 
+  \overline{\tau} \tau_{\rm t} \log(1/ \delta) \}, \quad  \widetilde M_{\rm mult}^{(k)} = \delta^{- \overline{\tau} \tau_{\rm t} k}
\eean
such that 
\be
\left\vert \widetilde f_{\rm mult}^{(k)}(\widetilde x_1,\ldots, \widetilde x_k) - \prod_{i=1}^{k} x_i \right\vert \leq \widetilde \delta + k \delta^{-\overline{\tau} \tau_{\rm t} (k-1)} \widetilde \epsilon \label{eqprop3:nnmult}
\ee
for any $\bx = (x_1,\ldots,x_k) \in \bbR^{k}$ with $ \|\bx\|_{\infty} \leq \delta^{-1}$ and $\widetilde \bx = (\widetilde x_1,\ldots,\widetilde x_{k}) \in \bbR^{k}$ with $\| \bx - \widetilde \bx\|_{\infty} \leq \widetilde \epsilon$,
where $0 < \widetilde \epsilon \leq 1 $ and $C_{N,1}$ is the constant in Lemma~\ref{secnn:mult}.
Let 
\bean
f_{\rm clip} \in \cF_{\rm NN}(2,(1,2,1)^{\top}, 7, \tau_{\rm tail} \sqrt{\log(1/\delta)} )
\eean
be the neural network in Lemma~\ref{secnn:clip} such that
$f_{\rm clip}(x) = (x \vee -\tau_{\rm tail}\sqrt{\log(1/\delta)}  ) \wedge \tau_{\rm tail} \sqrt{\log(1/\delta)}$ for $x \in \bbR$.
For $i \in [D_{1}]$ and $ j \in [D]$, consider functions $\overline{f}_{i,j}, \underline{f}_{i,j} : \bbR \times \bbR \to \bbR$ such that
\bean
&& \overline{f}_{i,j} (x,t) = f_{\rm clip}\left( \widetilde f_{\rm mult}^{(2)}\left( f_{\rm rec}\left(f_{\sigma}(t) \right), \{ y_j^{(i)} +  \tau_* \} f_{\mu}(t) - x \right)  \right),
 \\
 && \underline{f}_{i,j} (x,t) = f_{\rm clip}\left( \widetilde f_{\rm mult}^{(2)}\left( f_{\rm rec}\left(f_{\sigma}(t) \right), \{ y_j^{(i)} -  \tau_* \} f_{\mu}(t) - x \right) \right),
\eean
for $x,t \in \bbR$.
Note that both $\vert y_j^{(i)}+ \tau_* \vert$ and $\vert y_j^{(i)} - \tau_*\vert $ are upper bounded by $ m_* \tau_* \leq (1+\tau_{\rm tail} )\sqrt{\log(1/\delta)}$.
Then, for $\vert x \vert \leq 1 + \tau_{\rm x}  \sqrt{ \log (1/\delta) }$, both $\vert (y_j^{(i)}+ \tau_*)\mu_t - x \vert$ and $\vert ( y_j^{(i)} - \tau_* ) \mu_t - x \vert $ are upper bounded by $\delta^{- \overline{\tau} \tau_{\rm t} }$ with small enough $\delta$.
Combining (\ref{eqprop3:nnmtst}), (\ref{eqprop3:nnrecst}) and (\ref{eqprop3:nnmult}) with the last display, both $\vert \overline{z}_{i,j} - \overline{f}_{i,j} (x,t) \vert $ and $\vert \underline{z}_{i,j} - \underline{f}_{i,j} (x,t) \vert $ are bounded by 
\be
\widetilde \delta + 2 \delta^{- \overline{\tau} \tau_{\rm t} } ( 1 + \delta^{ -2 \overline{\tau} \tau_{\rm t} } /4 ) \widetilde \delta
\leq 5 \delta^{-3 \overline{\tau} \tau_{\rm t} } \widetilde \delta ,  \label{eqprop3:nnzij}
\ee
for $ \vert x \vert \leq \mu_t +  \tau_{\rm x} \sigma_{t} \sqrt{ \log(1/\delta) }$ and $\delta \leq t \leq \tau_{\rm t} \log (1/\delta)$.
Since $  \mu_t -\widetilde \delta \leq f_{\mu}(t)  \leq \mu_t + \widetilde \delta$ for $t \geq \widetilde \delta$, we have $1/4 \leq 1/2  - \widetilde \delta \leq f_{\mu}(t)  \leq 1+\widetilde \delta \leq 2$ and $1/4 \leq \mu_t \leq 2$ for $\delta \leq t \leq D_{1}$ with small enough $\widetilde \delta$ by (\ref{eqthm:mtstbd}).
For any $i \in [D_2], j \in [D], k \in \{0,\ldots,\tau_{\rm sm}-1\}$, consider functions $f_{i,j,k}, \widetilde f_{i,j,k} : \bbR \times \bbR \to \bbR$ such that
\bean
&& f_{i,j,k}(x,t)  = \sum_{l=0}^{D_{2}-1} \sum_{r=0}^{k} \binom{k}{r} \left\{ \frac{(-2)^{-l}}{k! l! \sqrt{2\pi} (r+2l+1)  } \right\} \left\{ \overline f_{i,j,k,l,r,r+2l+1} -  \underline f_{i,j,k,l,r,r+2l+1} \right\},
\\
&& \widetilde f_{i,j,k}(x,t)  =  \sum_{l=0}^{D_{2}-1} \sum_{r=0}^{k} \binom{k}{r} \left\{ \frac{(-2)^{-l}}{k! l! \sqrt{2\pi} (r+2l+2)  } \right\} \left\{ \overline f_{i,j,k,l,r,r+2l+2} -  \underline f_{i,j,k,l,r,r+2l+2} \right\},
\eean
for $x,t \in \bbR$, where
\bean
 && \overline f_{i,j,k,l,r,s}  = \widetilde f_{\rm mult}^{(2k+s+1)} \left(  f_{\rm rec}\left(f_{\mu}(t) \right) \cdot \mathbf{1}_{k+1}, f_{\sigma}(t) \cdot \mathbf{1}_{r}, \left\{ x- y_j^{(i)}f_{\mu}(t) \right\} \cdot \mathbf{1}_{k-r}, \overline{f}_{i,j}(x,t) \cdot \mathbf{1}_{s}    \right),
 \\
&& \underline f_{i,j,k,l,r,s}  = \widetilde f_{\rm mult}^{(2k+s+1)} \left(  f_{\rm rec}\left(f_{\mu}(t) \right) \cdot \mathbf{1}_{k+1}, f_{\sigma}(t) \cdot \mathbf{1}_{r}, \left\{ x- y_j^{(i)}f_{\mu}(t) \right\} \cdot \mathbf{1}_{k-r} , \underline{f}_{i,j}(x,t) \cdot \mathbf{1}_{s}    \right),
\eean
for $ s \in \{r+2l+1, r+2l+2 \}$.
Combining (\ref{eqprop3:nnmtst}), (\ref{eqprop3:nnrecmt}), (\ref{eqprop3:nnzij}) and (\ref{eqprop3:nnmult}) with the last two displays, we have
\bean
&& \left\vert \sum_{l=0}^{D_{2}-1 }  \frac{ P_{i,j,k,l}(x, t) }{k!} -  f_{i,j,k}(x,t)   \right\vert 
\\
&& \leq  \sum_{l=0}^{D_{2}-1} \sum_{r=0}^{k} \binom{k}{r} \left\{ \frac{2^{-l+1}}{k! l! \sqrt{2\pi} (2r+r+2l+1)  } \right\} \left\{ \widetilde \delta + 5 (2k+r+2l+2) \delta^{-(2k+r+2l+4)\overline{\tau} \tau_{\rm t}} \widetilde \delta  \right\}
\eean
and
\bean
&& \left\vert \sum_{l=0}^{D_{2}-1 }  \frac{ P_{i,j,k,l}(x, t) }{k!} -  f_{i,j,k}(x,t)   \right\vert 
\\
&& \leq  \sum_{l=0}^{D_{2}-1} \sum_{r=0}^{k} \binom{k}{r} \left\{ \frac{2^{-l+1}}{k! l! \sqrt{2\pi} (r+2l+2)  } \right\} \left\{ \widetilde \delta + 5 (2k+r+2l+3) \delta^{-(2k+r+2l+5)\overline{\tau} \tau_{\rm t}} \widetilde \delta  \right\}
\eean
for $i \in [D_1], j \in [D], k \in \{0,\ldots,\tau_{\rm sm}-1\}, \vert x \vert \leq \mu_t + \tau_{\rm x} \sigma_t \sqrt{ \log (1/\delta) }$ and $\delta \leq t \leq \tau_{\rm t} \log(1/\delta)$ with small enough $\delta$ so that $\sigma_t, \vert x - \mu_t y_j^{(i)} \vert, \tau_{\rm tail} \sqrt{ \log ( 1 / \delta) }$ are bounded by $\delta^{-\overline{\tau} \tau_{\rm t}  }$.
Since $\sum_{r=0}^{k} \binom{k}{r} = 2^{k}$ and $2^{k}/k! \leq 2$ for all $k \in \bbN$,
the last two displays are bounded by
\be
\sum_{l = 0}^{D_{2} - 1} \left( \frac{2^{-l+2}}{l! \sqrt{2\pi} (2l + 1) }  \right) \left\{ 1 + 5(3 \tau_{\rm sm} + 2 l  ) \delta^{-(3 \tau_{\rm sm} + 2 l + 1 ) \overline{\tau} \tau_{\rm t} } \right\} \widetilde \delta
\leq \delta^{-D_{6} \log(1/\delta)  } \widetilde \delta, \label{eqprop3:nnpijk}
\ee
where $D_{6} = D_{6}(\overline{\tau}, \tau_{\rm tail}, \tau_{\rm sm}, \tau_{\rm t})$.
For any $i \in [D_1], j \in [D], k \in \{0,\ldots,\tau_{\rm sm}-1\}, l \in \{0,\ldots,D_{2}-1\}, \vert x \vert \leq \mu_t + \tau_{\rm x} \sigma_t \sqrt{ \log(1/\delta)}  $ and $\delta \leq t \leq \tau_{\rm t} \log (1/\delta)$, we have
\bean
&& \left\vert \frac{ P_{i,j,k,l}(x, t) }{k!} \right\vert
\\
&& \leq  \left\{\frac{2^{-l+1} \mu_t^{-k-1}  }{k! l! \sqrt{2\pi}(r+2l+1)} \sum_{r = 0}^{k} \binom{k}{r} \right\}  \left\{ ( 2 + \tau_{\rm x} + \tau_{\rm tail} ) \sqrt{\log (1/\delta)}  \right\}^{k} \left\{ \tau_{\rm tail} \sqrt{ \log (1/\delta) } \right\}^{k+2l+1}
\eean
and
\bean
&& \left\vert \frac{ \widetilde P_{i,j,k,l}(x, t)  }{k!} \right\vert
\\
&& \leq  \left\{\frac{2^{-l+1} \mu_t^{-k-1}  }{k! l! \sqrt{2\pi}(r+2l+2)} \sum_{r = 0}^{k} \binom{k}{r} \right\}  \left\{ ( 2 + \tau_{\rm x} + \tau_{\rm tail} ) \sqrt{\log (1/\delta)}  \right\}^{k} \left\{ \tau_{\rm tail} \sqrt{ \log (1/\delta)} \right\}^{k+2l+2}
\eean
because $\mu_t \leq 1, \sigma_t \leq 1$, and $y_j^{(i)} \leq m_* \tau_* \leq (1+\tau_{\rm tail}) \sqrt{\log (1/\delta)}$.
Since $\sum_{r=0}^{k} \binom{k}{r} = 2^{k}$ and $2^{k}/k! \leq 2$ for all $k \in \bbN$,
the last two displays are bounded by 
\bean
\left( \frac{4}{\sqrt{2\pi}} \right) \delta^{- k \overline{\tau} \tau_{\rm t} }  \left\{  ( 2 + \tau_{\rm x} + \tau_{\rm tail})^{2} \tau_{\rm tail}^2 \log (1/\delta) \right\}^{k+l+1}.
\eean
Then, both $\vert \sum_{l=0}^{D_{2}-1 } P_{i,j,k,l}(x, t) / k!  \vert$ and $\vert \sum_{l=0}^{D_{2}-1 }   \widetilde P_{i,j,k,l}(x, t) / k!  \vert$ are upper bounded by
\bean
&&  D_{2}\left( \frac{4}{\sqrt{2\pi}} \right) \delta^{- ( \tau_{\rm sm} - 1 ) \overline{\tau} \tau_{\rm t} }  \left\{  ( 2 + \tau_{\rm x} + \tau_{\rm tail})^{2} \tau_{\rm tail}^2 \log (1/\delta) \right\}^{ \tau_{\rm sm} + D_{2} - 1 } \leq \{ \log ( 1 / \delta) \}^{D_{7} \log(1/\delta)  }
\eean
for $i \in [D_1], j \in [D], k \in \{0,\ldots,\tau_{\rm sm}-1\}, \vert x \vert \leq \mu_t + \tau_{\rm x} \sigma_t \sqrt{\log (1/\delta)}$ and $\delta \leq t \leq \tau_{\rm t} \log (1/\delta)$,
where $D_{7} = D_{7}(\tau_{\rm tail},\tau_{\rm sm}, \tau_{\rm x}, \tau_{\rm t}, \overline{\tau})$.
Consider functions $f_{t_*}, \widetilde f_{t_*}^{(1)},\ldots, \widetilde f_{t_*}^{(D)} : \bbR^{D} \times \bbR \to \bbR$ such that
\bean
&& f_{t_*}(\bx,t) = \sum_{i=1}^{D_{1}}  \sum_{0 \leq k. < \tau_{\rm sm}} \left\{ (\D^{\bk} p_{t_*})(\by^{(i)})  \right\}
f_{\rm mult}^{(D)}\left( f_{i,1,k_1}(x_1,t),\ldots, f_{i,D,k_D}(x_D, t)   \right) \quad \text{and}
\\
&& \widetilde f_{t_*}^{(h)}(\bx,t)= \sum_{i=1}^{D_{1}} \sum_{0 \leq k. <  \tau_{\rm sm}} \left\{(\D^{\bk} p_{t_*})(\by^{(i)})  \right\} f_{\rm mult}^{(D)}\left(\widetilde f_{i,h,k_h}(x_h, t), \underbrace{   f_{i,1,k_1}(x_1,t),\ldots, f_{i,D,k_D}(x_D, t)}_{{\rm without} \   f_{i,h,k_h}(x_h,t)}    \right),
\eean
where $ f_{\rm mult}^{(D)} \in \cF_{\rm NN}( L_{\rm mult}^{(D)},  \bd_{\rm mult}^{(D)}, s_{\rm mult}^{(D)},  M_{\rm mult}^{(D)})$ is the neural network in Lemma~\ref{secnn:mult} with
\bean
&&  L_{\rm mult}^{(D)} \leq C_{N,1} \log D  [ \log(1/ \widetilde \delta) + D D_{7}  \log(1/\delta) \log \log ( 1/ \delta) ], \quad  d_{\rm mult}^{(D)} = (D,48D,\ldots,48D,1)^{\top},
\\
&&  s_{\rm mult}^{(D)} \leq   C_{N,1} D [\log(1/ \widetilde \delta) 
+  D_{7}  \log(1/\delta) \log \log ( 1 / \delta) ], \quad   M_{\rm mult}^{(D)} = \{ \log (1/\delta) \} ^{DD_{7} \log(1/\delta)  }
\eean
such that 
\bean
\left\vert \widetilde f_{\rm mult}^{(D)}(\widetilde x_1,\ldots, \widetilde x_{D}) - \prod_{i=1}^{D} x_i \right\vert
\leq \widetilde \delta + D \{ \log (1/\delta) \} ^{ 
(D-1) D_{7} \log(1/\delta)  } \widetilde \epsilon 
\eean
for any $\bx = (x_1,\ldots,x_{D}) \in \bbR^{D}$ with $ \|\bx\|_{\infty} \leq \{ \log (1/\delta) \} ^{DD_{7} \log(1/\delta)  }$ and $\widetilde \bx = (\widetilde x_1,\ldots,\widetilde x_{D}) \in \bbR^{D}$ with $\| \bx - \widetilde \bx\|_{\infty} \leq \widetilde \epsilon$.
Combining (\ref{eqprop3:nnpijk}) with the last display, we have
\bean
&& \left\vert f_{t_*}(\bx,t) - g_t(\bx) \right\vert
\\
&& \leq  \sum_{i=1}^{D_{1}}  \sum_{0 \leq k. < \tau_{\rm sm}} \left\vert (\D^{\bk} p_{t_*})(\by^{(i)})  \right\vert \left\{ 1 + D \delta^{ - D_{6} \log ( 1 / \delta) } \{ \log (1/\delta) \} ^{ (D-1) D_{7} \log(1/\delta)  } \right\} \widetilde \delta
\\
&& \leq  D_{1}  \tau_{\rm sm}^{D} C_{S,3} \sigma_{t_*}^{-\tau_{\rm sm}} \left\{ 1 + D \delta^{ - D_{6} \log ( 1 / \delta) } \{ \log (1/\delta) \} ^{ (D-1) D_{7} \log(1/\delta)  } \right\} \widetilde \delta
\\
&& \leq m^{\frac{\tau_{\rm sm}(1-\tau_{\rm low}) +D }{D}} \delta^{- D_{8} \log(1/\delta)  }  \widetilde \delta 
\eean
for $\| \bx \|_{\infty} \leq \mu_t + \tau_{\rm x} \sigma_t \sqrt{\log (1/\delta)}$ and $\delta \leq t \leq \tau_{\rm t} \log(1/\delta)$, where $D_{8} = D_{8} (D , \underline{\tau} , \tau_{\rm sm} , C_{S,3},D_{6}, D_{7})$.
Similarly, we have
\bean
&& \left\vert \widetilde f_{t_*}^{(h)}(\bx,t) - \widetilde g^{(h)}_t(\bx) \right\vert
\\
&& \leq  \sum_{i=1}^{D_{1}}  \sum_{0 \leq k. < \tau_{\rm sm}} \left\vert (\D^{\bk} p_{t_*})(\by^{(i)})  \right\vert  \left\{ 1 + D \delta^{ - D_{6} \log ( 1 / \delta) } \{ \log (1/\delta) \} ^{ (D-1) D_{7} \log(1/\delta)  } \right\} \widetilde \delta
\\
&& \leq m^{\frac{\tau_{\rm sm}(1-\tau_{\rm low}) +D }{D}} \delta^{- D_{8} \log(1/\delta)  }  \widetilde \delta, \quad h \in [D]
\eean
for $\| \bx \|_{\infty} \leq \mu_t + \tau_{\rm x } \sigma_t \sqrt{\log (1/\delta)}$ and $\delta \leq t \leq \tau_{\rm t} \log(1/\delta)$.
Let 
\bean
\widetilde \delta = \delta^{D_{8} \log (1/\delta)} m^{- \frac{\tau_{\rm sm} + D}{D}} 
\eean
with large enough $m$.
Combining (\ref{eqprop3:bothpoly}) with the second last display,  we have
\bean
&& \left\vert p_{t_*+t}(\bx) - f_{t_*}(\bx,t) \right\vert
\leq (1 + D_{5}) m^{ - \frac{ \tau_{\rm low} \tau_{\rm sm} - ( D + 1 - \tau_{\rm low} ) D \overline{\tau} \tau_{\rm t}   }{D ( 1 + D \overline{\tau} \tau_{\rm t}  )}}   
 (\log m)^{D(\frac{\tau_{\rm sm}}{2} + 1)},
\\
&& \left\vert \sigma_t  \left( \nabla p_{t_*+t}(\bx) \right)_{h} - \widetilde f_{t_*}^{(h)}(\bx,t) \right\vert
\leq (1 + D_{5}) m^{ - \frac{ \tau_{\rm low} \tau_{\rm sm} - ( D + 1 - \tau_{\rm low} ) D \overline{\tau} \tau_{\rm t}  }{D ( 1 + D \overline{\tau} \tau_{\rm t}  )}}   
 (\log m)^{D(\frac{\tau_{\rm sm}}{2} + 1)},
 \quad h \in [D]
\eean
for $\| \bx \|_{\infty} \leq \mu_t + \tau_{\rm x} \sigma_t \sqrt{\log (1/\delta)}$ and $\delta \leq t \leq \tau_{\rm t}\log(1/\delta)$.
Consider a function $ \bff_{t_*} : \bbR^{D} \times \bbR \to \bbR^{D+1}$ such that
$ \bff_{t_*}(\bx,t) = ( f_{t_*}(\bx,t), \widetilde f_{t_*}^{(1)}(\bx,t) , \ldots , \widetilde f_{t_*}^{(D)}(\bx,t) )^{\top}$ for $\bx \in \bbR^{D}$ and $ t \in \bbR$. Lemma~\ref{secnn:comp}, Lemma~\ref{secnn:par}, Lemma~\ref{secnn:lin} and Lemma~\ref{secnn:id}  implies that $\bff_{t_*} \in \cF_{\rm NN}(L,\bd,s,m)$ with
\bean
&& L \leq D_{9} ( \log m )^{4} , \quad \| \bd\|_{\infty} \leq D_{9} m (\log m)^{9},
\\
&& s \leq D_{9}  m (\log m)^{9}, \quad M \leq \exp(D_{9} (\log m)^2 ),
\eean
where $D_{9} = D_{9}(D , \overline{\tau}, \underline{\tau}, \tau_{\rm t}, \tau_{\rm x}, \tau_{\rm sm}, \tau_{\rm low}, \tau_{\rm tail}, C_{S,3}, C_{N,4}, C_{N,5}, D_{7}, D_{8} )$ is a large enough constant.
The assertion follows by re-defining the constants.

\end{proof}


\subsection{Proof of Theorem~\ref{secthm:1}}

In this subsection, we provide the proof of Theorem~\ref{secthm:1} by combining  Propositions \ref{secpt:1} to \ref{secpt:3}.

\textit{Proof of Theorem~\ref{secthm:1}.}
Let $m > 0$ be a large enough value as described below and 
\bean
D_{1} =\sqrt{ \frac{ 8\beta}{d} }  \vee \sqrt{2\tau_{\rm min} + \frac{ 4\beta }{d}}.
\eean
We will approximate $\nabla \log p_t(\bx)$ for $\|\bx\|_{\infty} \leq  \mu_t + \sigma_t D_{1} \sqrt{\log m}$ and $m^{-\tau_{\rm min}} \leq t \leq \tau_{\rm max} \log m$ by neural networks by dividing the analysis into  the following three cases:
\ben
\item (Interior of near-support) $\|\bx\|_{\infty} \leq \mu_t- \{ \log (1/\sigma_t)\}^{-3/2}$ and $m^{-\tau_{\rm min}} \leq t \leq  3m^{-\frac{1}{2D}}$
\item (Boundary of near-support) $\mu_t-  \tau_{\rm min}^{3/2} \{    (4D)^{3/2} +3  \} \{ \log (1/\sigma_t)\}^{-3/2} \leq  \|\bx\|_{\infty} \leq \mu_t + \sigma_t D_{1} \sqrt{\log m}$ 
\\
and $m^{-\tau_{\rm min}} \leq  t \leq 3m^{-\frac{1}{2D}}$
\item (large $t$) $\|\bx\|_{\infty} \leq  \mu_t + \sigma_t D_{1} \sqrt{\log m}$ and $2 m^{-\frac{1}{2D}} \leq t \leq \tau_{\rm max}  \log m$.
\een
Then, we combine the networks into a single network and derive the approximation error over the entire region $(\bx,t) \in \bbR^{D} \times  [m^{-\tau_{\rm min}},\tau_{\rm max} \log m]$.

\subsubsection{Interior of Near-Support}
Let $\tau_{\rm tail}^{(1)} = 4D\beta / d $ and $\tau_{\rm bd}^{(1)} = 3/2$.
Let 
\bean
&& \widetilde C_{3} = \widetilde C_{3} ( \beta, d, D, K, \overline{\tau}, \underline{\tau}, \tau_{\rm bd}^{(1)}, \tau_{\rm tail}^{(1)}, \tau_{\rm min} ), 
\\
&& \widetilde C_{4} = \widetilde C_{4} ( \beta, d, D, K,  \underline{\tau}, \tau_{\rm bd}^{(1)}, \tau_{\rm tail}^{(1)}, \tau_{\rm min} ),
\\
&& \widetilde C_{5} = \widetilde C_{5} ( \beta, d, \underline{\tau}, \tau_{\rm bd}^{(1)}, \tau_{\rm tail}^{(1)}, \tau_{\min} )
\eean
be the constants in Proposition~\ref{secpt:1}, where $(\tau_{\rm tail}, \tau_{\rm bd})$ is replaced by $(\tau_{\rm tail}^{(1)}, \tau_{\rm bd}^{(1)} )$.
Also, let $\widetilde C_{2} = \widetilde C_{2}(\beta,D, \tau_{\rm bd}^{(1)}, \tau_{\rm tail}^{(1)})$ be the constant in Lemma~\ref{sec:qdeapprox}, where $( \tau_{\rm tail}, \tau_{\rm bd})$ is replaced by $(\tau_{\rm tail}^{(1)}, \tau_{\rm bd}^{(1)} )$.
For large enough $m$ so that $m \geq \widetilde C_{5}$ and $3m^{-1/(2D)} \leq \overline{\tau}^{-1} (\widetilde C_{2}^2 \wedge 1/2)$, Proposition~\ref{secpt:1} implies that
there exists a collection of permutation matrices $\cP^{(1)} = ( \cQ_{i}^{(1)}, \cR_{i}^{(1)} )_{i \in [L-1]} $ and weight-sharing neural networks 
\bean
\bff^{(1)} = (f_1^{(1)},\ldots,f_{D+1}^{(1)})^{\top} \in \cF_{\rm NN}( L^{(1)}, \bd^{(1)}, s^{(1)}, M^{(1)}, \cP^{(1)}_{\bfm^{(1)}} )
\eean
with 
\bean
&& L^{(1)} \leq \widetilde C_{3} (\log m )^{2} \log \log m, \quad \|\bd^{(1)}\|_{\infty} \leq  \widetilde C_{3} m^{D+1},
\\
&& s^{(1)} \leq  \widetilde C_{3} m (\log m)^{5} \log \log m, \quad M^{(1)} \leq \exp (  \widetilde C_{3} \{ \log m \}^{2} ), 
\\
&& \| \bfm^{(1)} \|_{\infty} \leq \widetilde C_{3} m^{D}
\eean
satisfying
\bean
&& \left\| \begin{pmatrix}
\sigma_t \nabla p_t(\bx) 
\\
 p_t(\bx)
\end{pmatrix}
-  \bff^{(1)}(\bx,t) \right\|_{\infty} \leq  \widetilde C_{4} \left( 3^{\frac{2D\beta}{d}} + 1 \right)  m^{-\frac{\beta}{d}} (\log m)^{2D + 2\beta} \defeq \epsilon_{1}
\eean
for $\|\bx\|_{\infty} \leq \mu_t- \{ \log (1/\sigma_t)\}^{-3/2}$ and $m^{-\tau_{\rm min}} \leq t \leq  3 m^{-\frac{1}{2D}}$ because $ \mu_{t} \leq 1$.
Note that $C_{S,1}^{-1} \leq p_t(\bx) \leq C_{S,1}$ for $\| \bx \|_{\infty} \leq \mu_t$, $t \geq 0$, and $ \| \sigma_t \nabla p_t(\bx) \|_{\infty} \leq C_{S,3}$ for $\bx \in \bbR^{D}, t \geq 0$, where $C_{S,1} = C_{S,1}(D , K, \tau_1)$ and $ C_{S,3} = C_{S,3}(D,K, \overline{\tau}, \underline{\tau} )$ are the constants in Lemma~\ref{secsc:ptbound} and Lemma~\ref{secsc:ptderivativebound}, respectively.
Also, $C_{S,1}^{-1}/2 \leq p_t(\bx) - \epsilon_{1} \leq f_{D+1}^{(1)}(\bx,t) \leq  p_t(\bx) + \epsilon_{1} \leq 2 C_{S,1}$  for $\|\bx\|_{\infty} \leq \mu_t-\mu_t \{ \log (1/\sigma_t)\}^{-3/2}$ and $m^{-\tau_{\rm min}} \leq t \leq  3 m^{-\frac{1}{2D}}$ with large enough $m$ so that $\epsilon_{1} \leq C_{S,1}^{-1}/2$.
Let $f_{{\rm rec}}^{(\rm in)} \in \cF_{\rm NN}(L_{{\rm rec}}^{(\rm in)},\bd_{{\rm rec}}^{(\rm in)},s_{{\rm rec}}^{(\rm in)},M_{{\rm rec}}^{(\rm in)})$ be the neural networks in Lemma~\ref{secnn:rec} with
\bean
&& L_{{\rm rec}}^{(\rm in)} \leq C_{N,5}  \{ \beta \log m /d  \}^2 , \quad \| \bd_{{\rm rec}}^{(\rm in)} \|_{\infty} \leq C_{N,5}  \{ \beta \log m /d  \}^3 ,
\\
&& s_{{\rm rec}}^{(\rm in)} \leq C_{N,5}  \{ \beta \log m /d  \}^4, \quad M_{{\rm rec}}^{(\rm in)} \leq  C_{N,5} m^{\frac{2\beta}{d}}
\eean
such that $\vert 1/x - f_{{\rm rec}}^{(\rm in)}(x) \vert \leq m^{-\beta/d}$ for $x \in [m^{-\beta/d}, m^{\beta/d} ]$.
For large enough $m$ so that $m^{-\beta/d} \leq C_{S,1}^{-1}/2$,
\bean
&& \left\vert \frac{1}{p_t(\bx)} - f_{{\rm rec}}^{(\rm in)}\left(f_{D+1}^{(1)}(\bx,t)\right) \right\vert \leq \left\vert \frac{1}{p_t(\bx)} - \frac{1}{f_{D+1}^{(1)}(\bx,t)} \right\vert + \left\vert \frac{1}{f_{D+1}^{(1)}(\bx,t)} - f_{{\rm rec}}^{(\rm in)}\left(f_{D+1}^{(1)}(\bx,t)\right) \right\vert
\\
&& \leq \{ p_t(\bx) \wedge f_{D+1}^{(1)}(\bx,t)  \}^{-2} \epsilon_1 + m^{-\frac{\beta}{d}} \leq (4 C_{S,1}^{2} + 1) \epsilon_1 
\eean
for $\|\bx\|_{\infty} \leq \mu_t- \{ \log (1/\sigma_t)\}^{-3/2}$ and $m^{-\tau_{\rm min}} \leq t \leq  3 m^{-\frac{1}{2D}}$. 
Let 
\bean
f_{{\rm mult}}^{(\rm in)} \in \cF_{\rm NN}(L_{{\rm mult}}^{(\rm in)},\bd_{{\rm mult}}^{(\rm in)},s_{{\rm mult}}^{(\rm in)},M_{{\rm mult}}^{(\rm in)})
\eean
be the neural networks in Lemma~\ref{secnn:mult} with
\bean
&& L_{{\rm mult}}^{(\rm in)} \leq C_{N,1} \log 2 \{ \beta \log m /d + 2 \log ( C_{S,1} \vee C_{S,3}) + 2\log 2 \}, \quad \bd_{{\rm mult}}^{(\rm in)} = (2,96,\ldots,96,1)^{\top},
\\
&& s_{{\rm mult}}^{(\rm in)} \leq C_{N,1} 2 \{ \beta \log m /d  +  \log ( C_{S,1} \vee C_{S,3}) + \log 2 \}, \quad M_{{\rm mult}}^{(\rm in)} = ( C_{S,1} \vee C_{S,3})^{2}
\eean
such that $\vert f_{{\rm mult}}^{(\rm in)} (\widetilde \bx) - x_1 x_2 \vert \leq m^{-\beta/d} + 2( C_{S,1} \vee C_{S,3}) \widetilde \epsilon$ for all $\|\bx\|_{\infty} \leq 2( C_{S,1} \vee C_{S,3})$, $\widetilde \bx \in \bbR^{2}$ with $\|\bx-\widetilde \bx \|_{\infty} \leq \widetilde \epsilon$ and $\vert f_{{\rm mult}}^{(\rm in)} (\bx) \vert \leq   ( C_{S,1} \vee C_{S,3})^{2} $ for all $\bx \in \bbR^2$, where $C_{N,1}$ is the constant in Lemma~\ref{secnn:mult}.
Then,
\be
\begin{split}
& \left\vert \frac{\sigma_t (\nabla p_t(\bx))_{i}}{p_t(\bx)} - f_{{\rm mult}}^{(\rm in)} \left( f_i^{(1)}(\bx,t) ,  f_{\rm rec}^{(\rm in)}\left(f_{D+1}^{(1)}(\bx,t) \right) \right)  \right\vert 
\\
& \leq m^{-\frac{\beta}{d}} + 2( C_{S,1} \vee C_{S,3})(4 C_{S,1}^{2} + 1) \epsilon_1 \leq  D_{2} m^{-\frac{\beta}{d}} (\log m)^{2D + 2\beta}, \quad i \in [D]
\label{eqthm4:nn1}
\end{split} \ee
for $\|\bx\|_{\infty} \leq \mu_t- \{ \log (1/\sigma_t)\}^{-3/2}$ and $m^{-\tau_{\rm min}} \leq t \leq  3 m^{-\frac{1}{2D}}$, where $D_{2} = D_{2}( \beta, D, C_{S,1}, C_{S,3}, \widetilde C_{4} )$.
Let $0 < \delta \leq \underline{T} = m^{-\tau_{\rm min}}$  be a small enough value as described below.
With ${\delta} < 1/2$, Lemma~\ref{secnn:mtst} implies that
there exist neural networks  $f_{\sigma} \in \cF_{\rm NN}(L_{\sigma},\bd_{\sigma},s_{\sigma},M_{\sigma})$ with
\be \begin{split}
& L_{\sigma} \leq C_{N,4} \{ \log(1/ \delta)\}^{2}, \quad   \| \bd_{\sigma}\|_{\infty} \leq C_{N,4}  \{ \log(1/ \delta)\}^{2}
\\
&  s_{\sigma} \leq C_{N,4}  \{  \log(1/ \delta)\}^{3}, \quad  M_{\sigma} \leq C_{N,4}  \log ( 1/ \delta) \label{eqthm4:nnst}
\end{split} \ee
such that $ \vert \sigma_{t} - f_{\sigma}(t) \vert \leq  \delta$
for $ t \geq  \delta$, where $C_{N,4}$ is the constant in Lemma~\ref{secnn:mtst}.
Also, Lemma~\ref{secnn:rec} implies that there exists a neural network $f_{\rm rec} \in \cF_{\rm NN}(L_{\rm rec}, \bd_{\rm rec}, s_{\rm rec}, M_{\rm rec} )$ with
\bean
&& L_{\rm rec} \leq C_{N,5} \{ \log(1/  \delta )\}^{2}, \quad \|\bd_{\rm rec}\|_{\infty} \leq C_{N,5}  \{ \log (1/  \delta )\}^{3}
\\
&& s_{\rm rec}  \leq C_{N,5}  \{ \log(1/  \delta )\}^{4}, \quad M_{\rm rec} \leq C_{N,5}  \delta^{-2}
\eean
such that $\vert 1/x - f_{\rm rec}(x) \vert \leq   \delta $ for any $x \in [  \delta,1/  \delta]$.
Since $  \sigma_t- \delta \leq f_{\sigma}(t)  \leq \sigma_t +  \delta$ for $t \geq \delta$ and $\sqrt{\underline{\tau}  \underline{T}} \leq \sigma_t \leq 1$ for $t \geq  \underline{T}$ , we have $ \underline{T} \leq f_{\sigma}(t)  \leq 2$ for $t \geq  \underline{T}$ with large enough $m$ so that $  \underline{T} \leq \sqrt{\underline{\tau}  \underline{T}} -  \underline{T}$ and $ \underline{T} \leq 1$.
Then,
\be \begin{split}
& \left\vert 1/\sigma_t - f_{\rm rec}(f_{\sigma}(t)) \right\vert \leq
\left\vert 1/\sigma_t - 1/ f_{\sigma}(t) \right\vert  + \left\vert  1/ f_{\sigma}(t) - f_{\rm rec}(f_{\sigma}(t)) \right\vert 
\\
& \leq \{ \sigma_t \wedge f_{\sigma}(t)\}^{-2} \vert \sigma_t - f_{\sigma}(t) \vert +  \delta \leq (1+ \underline{T}^{-2})  \delta  = (1 + m^{2\tau_{\rm min}}) \delta  \label{eqthm4:nnrecst}
\end{split} \ee
 for $t \geq  \underline{T}$. 
Lemma~\ref{secnn:mult} implies that there exists a neural network 
\bean
f_{\rm mult} \in \cF_{\rm NN}( L_{\rm mult},  \bd_{\rm mult}, s_{\rm mult},  M_{\rm mult})
\eean
with 
\bean
&&  L_{\rm mult} \leq C_{N,1} \log 2  \{\log(1/  \delta) + 3DD_{1}^2 \log m \}, \quad  \bd_{\rm mult} = (2,96,\ldots,96,1)^{\top},
\\
&& s_{\rm mult} \leq  C_{N,1} 2 \{\log(1/  \delta) 
+ 3DD_{1}^2 \log m \}, \quad   M_{\rm mult} =  m^{6DD_{1}^2 }
\eean
such that 
\be
\left\vert  f_{\rm mult}(\widetilde x_1, \widetilde x_2) - x_1 x_2\right\vert \leq  \delta + 2  m^{3DD_{1}^2} \widetilde \epsilon \label{eqthm4:nnmult}
\ee
for any $\bx = (x_1,x_2) \in \bbR^{2}$ with $ \|\bx\|_{\infty} \leq 
m^{3DD_{1}^2}$ and $\widetilde \bx = (\widetilde x_1,\widetilde x_{2}) \in \bbR^{k}$ with $\| \bx - \widetilde \bx\|_{\infty} \leq \widetilde \epsilon$,
where $0 < \widetilde \epsilon \leq 1 $.
Consider functions $\widetilde f_{1}^{(1)},\ldots,\widetilde f_{D}^{(1)} : \bbR^{D} \times \bbR \to \bbR$ such that
\bean
\widetilde f^{(1)}_{i}(\bx,t) =  f_{\rm mult} \left( f_{{\rm mult}}^{(\rm in)} \left( f_i^{(1)}(\bx,t) ,  f_{\rm rec}^{(\rm in)}\left(f_{D+1}^{(1)}(\bx,t) \right) \right) , f_{\rm rec} ( f_{\sigma}(t)) \right), \quad i \in [D]
\eean
for $\bx \in \bbR^{D}$ and $t \in \bbR$.
For large enough $m$ so that $m^{3DD_{1}^2} \geq  4(C_{S,1} \vee C_{S,3})^2$, we have
\bean
&& \left\vert \sigma_{t}^{-1} f_{{\rm mult}}^{(\rm in)} \left( f_i^{(1)}(\bx,t) ,  f_{\rm rec}^{(\rm in)}\left(f_{D+1}^{(1)}(\bx,t) \right) \right) - \widetilde f^{(1)}_{i}(\bx,t) \right\vert
\\
&& \leq \delta +  2  m^{3DD_{1}^2} (1+  m^{2\tau_{\rm min}}) \delta \leq 5 m^{5D D_{1}^2} \delta, \quad i \in [D]
\eean
for $\bx \in \bbR^{D}$ and $t \geq \underline{T}$ because $D_{1}^2 \geq \tau_{\rm min}$.
Combining (\ref{eqthm4:nn1}) with the last display, we have
\be \begin{split}
& \left\vert (\nabla \log p_t(\bx))_{i} - \widetilde f_{i}^{(1)}(\bx,t)  \right\vert \sigma_{t}
\\
& \leq \left\vert \frac{\sigma_t (\nabla p_t(\bx))_{i}}{p_t(\bx)} - f_{{\rm mult}}^{(\rm in)} \left( f_i^{(1)}(\bx,t) ,  f_{\rm rec}^{(\rm in)}\left(f_{D+1}^{(1)}(\bx,t) \right) \right)  \right\vert + 4 m^{5DD_{1}^2} \delta \sigma_{t}
\\
& \leq D_{2} m^{-\frac{\beta}{d}} (\log m)^{2D+2\beta } +  4 m^{5DD_{1}^2} \delta, \quad i \in [D] \label{eqthm4:nn1err}
\end{split} \ee
for $\|\bx\|_{\infty} \leq \mu_t- \{ \log (1/\sigma_t)\}^{-3/2}$ and $m^{-\tau_{\rm min}} \leq t \leq  3 m^{-\frac{1}{2D}}$.

\subsubsection{Boundary of Near-Support}
Let $\delta^{(2)} = m^{-(3D D_{1}^2 + 2\beta / d)}$ and
\bean
&& \tau_{\rm bd}^{(2)} = 3/2,
\quad \widetilde \tau_{\rm bd}^{(2)} = 1,
\\
&& \tau_{\rm x}^{(2)} = \left\{ \left(3D D_{1}^2 + \frac{2\beta}{d} \right)^{-\frac{1}{2}} \vee 1 \right\}
\left[ D_{1} \vee \tau_{\rm min}^{\frac{3}{2}} \left\{    (4D)^{\frac{3}{2}} +3  \right\} \right],
\\
&& \tau_{\rm t}^{(2)} = \left\{ 3D \left( 3D D_{1}^2 + \frac{2\beta}{d} \right) \right\}^{-1} \wedge \frac{1}{2}.
\eean
Let $\widetilde C_{6} = \widetilde C_{6}( D, K, \overline{\tau}, \underline{\tau}, \tau_{\rm x}^{(2)} )$, $\widetilde C_{7} = \widetilde C_{7} ( D , K , \underline{\tau} )$,  $ \widetilde C_{8} = \widetilde C_{8}( D,  \overline{\tau}, \tau_{\rm bd}^{(2)}, \tau_{\rm x}^{(2)}, \tau_{\rm t}^{(2)}, \widetilde \tau_{\rm bd}^{(2)} )$
be the constants in Proposition~\ref{secpt:2}, where $(\tau_{\rm bd}, \widetilde \tau_{\rm bd}, \tau_{\rm x},\tau_{\rm t})$ is replaced by
$(\tau_{\rm bd}^{(2)}, \widetilde \tau_{\rm bd}^{(2)}, \tau_{\rm x}^{(2)},\tau_{\rm t}^{(2)} )$.
For large enough $m$, we have 
\bean
\delta^{(2)}\leq \widetilde C_{8},
\quad 3m^{-\frac{1}{2D}} \leq \left\{ \delta^{(2)} \right\}^{\tau_{\rm t}^{(2)}},
\quad \left\{  \log \left( 1 / \delta^{(2)} \right) \right\}^{- \widetilde \tau_{\rm bd}^{(2)} } \leq \tau_2,
\quad
2 C_{S,1} m^{D D_{1}^2 }  \leq \left\{ \delta^{(2)} \right\}^{-1}.
\eean
Also, a simple calculation yields that
\bean
\delta^{(2)} \leq m^{-\tau_{\rm min}},
\quad \tau_{\rm x}^{(2)} \geq  \tau_{\rm min}^{\frac{3}{2}} \left\{  (4D)^{\frac{3}{2}} +3  \right\},
\quad
D_{1} \sqrt{\log m}
\leq \tau_{\rm x}^{(2)} \sqrt{ \log \left( 1 / \delta^{(2)} \right) }.
\eean
Then, Proposition~\ref{secpt:2} implies that
that there exists a neural network 
\bean
\bff^{(2)} = (f_1^{(2)},\ldots,f_{D+1}^{(2)})^{\top} \in \cF_{\rm NN} ( L^{(2)}, \bd^{(2)}, s^{(2)}, M^{(2)} )
\eean
with 
\bean
&& L \leq \widetilde C_{6} \left\{ \log \left(1/\delta^{(2)} \right) \right\}^{4}, \quad \| \bd\|_{\infty} \leq \widetilde C_{6} \left\{ \log \left(1/\delta^{(2)} \right) \right\}^{ 7 + 2D },
\\
&& s \leq \widetilde C_{6}  \left\{ \log \left(1/\delta^{(2)} \right) \right\}^{11 + 2D }, \quad M \leq \exp \left( \widetilde C_{6} \left\{ \log \left(1/\delta^{(2)} \right) \right\}^{2} \right),
\eean
such that
\bean
\left\| \begin{pmatrix}
\sigma_t \nabla p_t(\bx) 
\\
 p_t(\bx)
\end{pmatrix}
-  \bff^{(2)}(\bx,t) \right\|_{\infty} 
\leq  \widetilde C_{7} \delta^{(2)}  \left\{ \log \left( 1 / \delta^{(2)} \right) \right\}^{D} \defeq \epsilon_{2}
\eean
for $\mu_t-   \tau_{\rm min}^{3/2} \{    (4D)^{3/2} +3  \} \{ \log (1/\sigma_t)\}^{-3/2} \leq  \|\bx\|_{\infty} \leq \mu_t + \sigma_t D_{1} \sqrt{\log m}$ and $m^{-\tau_{\rm min}} \leq  t \leq 3m^{-\frac{1}{2D}}$.
Lemma~\ref{secsc:ptbound} implies that $C_{S,1}^{-1} m^{-D D_{1}^2 } \leq  p_t(\bx) \leq C_{S,1} $
for $\|\bx\|_{\infty} \leq \mu_t + \sigma_{t} D_{1} \sqrt{\log m}$ and $t \geq 0$.
Note that $C_{S,1}^{-1} m^{-D D_{1}^2 } / 2 \leq p_t(\bx) - \epsilon_{2} \leq f_{D+1}^{(2)}(\bx,t) \leq  p_t(\bx) + \epsilon_{2} \leq 2 C_{S,1} $  for $\|\bx\|_{\infty} \leq \mu_t +  \sigma_t D_{1} \sqrt{\log m}$ and $m^{-\tau_{\rm min}} \leq t \leq  3 m^{-\frac{1}{2D}}$ with large enough $m$ so that $\epsilon_{2} \leq C_{S,1}^{-1} m^{-D D_{1}^2 } /2 $.
Let 
\bean
f_{{\rm rec}}^{(\rm bd)} \in \cF_{\rm NN}(L_{{\rm rec}}^{(\rm bd)},\bd_{{\rm rec}}^{(\rm bd)}, s_{{\rm rec}}^{(\rm bd)},M_{{\rm rec}}^{(\rm bd)})
\eean
be the neural networks in Lemma~\ref{secnn:rec} with
\bean
&& L_{{\rm rec}}^{(\rm bd)} \leq C_{N,5}  \{ \log ( 1 / \delta^{(2)} )  \}^2 , \quad \| \bd_{{\rm rec}}^{(\rm bd)} \|_{\infty} \leq  C_{N,5}  \{ \log ( 1 / \delta^{(2)} )  \}^3 ,
\\
&& s_{{\rm rec}}^{(\rm bd)} \leq  C_{N,5}  \{ \log ( 1 / \delta^{(2)} )  \}^4, \quad M_{{\rm rec}}^{(\rm bd)} \leq  C_{N,5}    \{ \delta^{(2)}   \}^{-2} 
\eean
such that $\vert 1/x - f_{{\rm rec}}^{(\rm bd)}(x) \vert \leq \delta^{(2)} $ for $x \in [\delta^{(2)}, 1 / \delta^{(2)}  ]$.
Then,
\bean
&& \left\vert \frac{1}{p_t(\bx)} - f_{{\rm rec}}^{(\rm bd)} \left(f_{D+1}^{(2)}(\bx,t)\right) \right\vert \leq \left\vert \frac{1}{p_t(\bx)} - \frac{1}{f_{D+1}^{(2)}(\bx,t)} \right\vert
+ \left\vert \frac{1}{f_{D+1}^{(2)}(\bx,t)} - f_{{\rm rec}}^{(\rm bd)} \left(f_{D+1}^{(2)}(\bx,t)\right) \right\vert
\\
&& \leq \{ p_t(\bx) \wedge f_{D+1}^{(2)}(\bx,t)  \}^{-2} \epsilon_2 +  \delta^{(2)}
\leq 4 C_{S,1}^2 m^{2D D_{1}^2} \epsilon_2 + \delta^{(2)}
\\
&& \leq ( 4 C_{S,1}^2 +1 ) \epsilon_{2} m^{2 D D_{1}^2}
\eean
for $\mu_t-   \tau_{\rm min}^{3/2} \{    (4D)^{3/2} +3  \} \{ \log (1/\sigma_t)\}^{-3/2} \leq  \|\bx\|_{\infty} \leq \mu_t + \sigma_t D_{1} \sqrt{\log m}$ and $m^{-\tau_{\rm min}} \leq t \leq  3 m^{-\frac{1}{2D}}$. 
Let 
\bean
f_{{\rm mult}}^{(\rm bd)} \in \cF_{\rm NN}( L_{{\rm mult}}^{(\rm bd)}, \bd_{{\rm mult}}^{(\rm bd)}, s_{{\rm mult}}^{(\rm bd)}, M_{{\rm mult}}^{(\rm bd)} )
\eean
be the neural networks in Lemma~\ref{secnn:mult} with
\bean
&& L_{{\rm mult}}^{(\rm bd)} \leq C_{N,1} \log 2 \{ ( 2\beta / d + 1 + 2 D D_{1}^2) \log m  + 2 \log (C_{S,1} \vee C_{S,3} ) + 2\log 2  \},
\\
&& \bd_{{\rm mult}}^{(\rm bd)} = (2,96,\ldots,96,1)^{\top},
\\
&& s_{{\rm mult}}^{(\rm bd)} \leq C_{N,1} 2 \{ ( 2\beta / d + 1 + D D_{1}^2 ) \log m  +  \log ( C_{S,1} \vee C_{S,3}) + \log 2 \},
\\
&& M_{{\rm mult}}^{(\rm bd)} = ( C_{S,1} \vee C_{S,3})^{2}  m^{2 D D_{1}^2}
\eean
such that $\vert f_{{\rm mult}}^{(\rm bd)} (\widetilde \bx) - x_1 x_2 \vert \leq m^{-2\beta /d - 1} + 2( C_{S,1} \vee C_{S,3}) m^{D D_{1}^2 } \widetilde \epsilon$ for all $\|\bx\|_{\infty} \leq ( C_{S,1} \vee C_{S,3}) m^{D D_{1}^2 }$, $\widetilde \bx \in \bbR^{2}$ with $\|\bx-\widetilde \bx \|_{\infty} \leq \widetilde \epsilon$ and $\vert f_{{\rm mult}}^{(\rm bd)} (\bx) \vert \leq   ( C_{S,1} \vee C_{S,3})^{2} m^{2 D D_{1}^2} $ for all $\bx \in \bbR^2$.
Then,
\be
\begin{split}
& \left\vert \frac{\sigma_t (\nabla p_t(\bx))_{i}}{p_t(\bx)} - f_{{\rm mult}}^{(\rm bd)} \left( f_i^{(2)}(\bx,t) ,  f_{\rm rec}^{(\rm bd)}\left(f_{D+1}^{(2)}(\bx,t) \right) \right)  \right\vert 
\\
& \leq m^{-2\beta /d - 1} + 2( C_{S,1} \vee C_{S,3})(4 C_{S,1}^2 + 1) \epsilon_{2} m^{3 D D_{1}^2}   \leq  D_{3} m^{-2\beta / d} (\log m)^{D}, \quad i \in [D]
\label{eqthm4:nn2}
\end{split} \ee
for $\mu_t-   \tau_{\rm min}^{3/2} \{    (4D)^{3/2} +3  \} \{ \log (1/\sigma_t)\}^{-3/2} \leq  \|\bx\|_{\infty} \leq \mu_t + \sigma_t D_{1} \sqrt{\log m}$ and $m^{-\tau_{\rm min}} \leq t \leq  3 m^{-\frac{1}{2D}}$,
where $D_{3} = D_{3} (\beta, d, D, \widetilde C_{7}, C_{S,1}, C_{S,3}, D_{1} )$.
Consider functions $\widetilde f_{1}^{(2)},\ldots,\widetilde f_{D}^{(2)} : \bbR^{D} \times \bbR \to \bbR$ such that
\bean
\widetilde f^{(2)}_{i}(\bx,t) =  f_{\rm mult} \left( f_{{\rm mult}}^{(\rm bd)} \left( f_i^{(2)}(\bx,t) ,  f_{\rm rec}^{(\rm bd)}\left(f_{D+1}^{(2)}(\bx,t) \right) \right) , f_{\rm rec} ( f_{\sigma}(t)) \right), \quad i \in [D]
\eean
for $\bx \in \bbR^{D}$ and $t \in \bbR$.
For large enough $m$ so that $m^{3D D_{1}^2} \geq ( C_{S,1} \vee C_{S,3})^{2} m^{2 D D_{1}^2}$, (\ref{eqthm4:nnrecst}) and (\ref{eqthm4:nnmult}) implies that
\bean
&& \left\vert \sigma_{t}^{-1} f_{{\rm mult}}^{(\rm bd)} \left( f_i^{(2)}(\bx,t) ,  f_{\rm rec}^{(\rm bd)}\left(f_{D+1}^{(2)}(\bx,t) \right) \right) - \widetilde f^{(2)}_{i}(\bx,t) \right\vert
\\
&& \leq \delta +  2  m^{3DD_{1}^2} (1+  m^{2\tau_{\rm min}}) \delta \leq 5 m^{5D D_{1}^2} \delta, \quad i \in [D]
\eean
for $\bx \in \bbR^{D}$ and $t \geq m^{-\tau_{\rm min}}$.
Combining (\ref{eqthm4:nn2}) with the last display, we have
\be \begin{split}
&  \left\vert (\nabla \log p_t(\bx))_{i} - \widetilde f_{i}^{(2)}(\bx,t)  \right\vert \sigma_{t}
\\
& \leq \left\vert \frac{\sigma_t (\nabla p_t(\bx))_{i}}{p_t(\bx)} - f_{{\rm mult}}^{(\rm bd)} \left( f_i^{(2)}(\bx,t) ,  f_{\rm rec}^{(\rm bd)}\left(f_{D+1}^{(2)}(\bx,t) \right) \right)  \right\vert +  5 m^{5D D_{1}^2} \delta \sigma_{t}
\\
& \leq D_{3} m^{-2\beta/d} (\log m)^{D} + 5 m^{5D D_{1}^2} \delta \leq D_{3} m^{-\beta/d} + 5 m^{5D D_{1}^2} \delta, \quad i \in [D]
\label{eqthm4:nnpt2err}
\end{split} \ee
for $\mu_t-  \tau_{\rm min}^{3/2} \{    (4D)^{3/2} +3  \} \{ \log (1/\sigma_t)\}^{-3/2} \leq  \|\bx\|_{\infty} \leq \mu_t + \sigma_t D_{1} \sqrt{\log m}$ and $m^{-\tau_{\rm min}} \leq t \leq  3 m^{-\frac{1}{2D}}$, where the last ienquality holds for large enough $m$ so that $(\log m)^{D} \leq m^{\beta/d}$.

\subsubsection{Large $t$}
Let $m^{(3)} = \sqrt{m}, \tau_{\rm low}^{(3)} = 1/2$,
\bean
&& \tau_{\rm sm}^{(3)} = \left\lfloor \left\{ \left( \frac{2D}{\tau_{\rm low}^{(3)}} \right) \left(1 + D  \right) \left( \frac{2\beta}{d} + \frac{1}{2D} + 3 D D_1^2 \right)
+ D \left( \frac{ D + 1 - \tau_{\rm low}^{(3)} }{\tau_{\rm low}^{(3)}} \right)  \right\} \right.
\\
&& \qquad \qquad 
\left. \vee \left\{  D (1 + D) \left(   \frac{ \{ \tau_{\rm max}  \overline{\tau} \} \vee \{ 1/2D \} }{ \tau_{\rm low}^{(3)}} \right)  - \frac{ D + 1 - \tau_{\rm low}^{(3)} }{\tau_{\rm low}^{(3)}} \right\}   \right\rfloor + 1
\\
&& \tau_{\rm x}^{(3)} = D_{1} \left\{ \frac { 2 D \left( 1 + D \right) }{  \tau_{\rm low}^{(3)} \tau_{\rm sm}^{(3)} + D + 1 - \tau_{\rm low}^{(3)} }  \right\}^{\frac{1}{2}}.
\eean
Also, let 
\bean
t_* = \left\{ m^{(3)} \right\}^{- \frac{  2 - 2 \tau_{\rm low}^{(3)} }{D} },
\quad
\delta^{(3)} = \left\{ m^{(3)} \right\}^{ - \frac {  \tau_{\rm low}^{(3)} \tau_{\rm sm}^{(3)} + D + 1 - \tau_{\rm low}^{(3)} }{ D ( 1 + D  ) } },
\eean
and $\widetilde C_{9}, \widetilde C_{10}, \widetilde C_{11}$ be the constants in Proposition~\ref{secpt:3} depending on $ ( D, K, \overline{\tau},  \underline{\tau}, \tau_{\rm x}^{(3)},  \tau_{\rm sm}^{(3)}, \tau_{\rm low}^{(3)} )$, where $(\tau_{\rm x}, \tau_{\rm sm}, \tau_{\rm low})$ is replaced by $( \tau_{\rm x}^{(3)}, \tau_{\rm sm}^{(3)}, \tau_{\rm low}^{(3)} )$.
A simple calculation yields that
\bean
&& t_* = m^{-\frac{1}{2D}},
\quad
\left\{ m^{(3)} \right\}^{- \frac {  \tau_{\rm low}^{(3)} \tau_{\rm sm}^{(3)} - ( D + 1  - \tau_{\rm low}^{(3)} ) D }{ D ( 1 + D)  }}
\leq m^{-\frac{2\beta}{d} -\frac{1}{2D} - 3D D_{1}^2 },
\\
&& \tau_{\rm x}^{(3)} \sqrt{ \log \left( 1 / \delta^{(3)} \right) } = D_{1} \sqrt{\log m},
\quad
\delta^{(3)} \leq  m ^{-\frac{1}{2D}},
\quad \tau_{\rm max} \log m \leq  \overline{\tau}^{-1} \log \left( 1 / \delta^{(3)} \right).
\eean
Also, for large enough $m$, we  have
\bean
m^{(3)} \geq \widetilde C_{11}
\quad \text{and}
\quad
\left\{ \log m^{(3)} \right\}^{ \frac{D\tau_{\rm sm}^{(3)}}{2} + D  } \leq m^{\frac{1}{4D} }.
\eean
Then, Proposition~\ref{secpt:3} implies that
that there exists a neural network
\bean
\bff^{(3)} = (f_1^{(3)},\ldots,f_{D+1}^{(3)})^{\top} \in \cF_{\rm NN} (L^{(3)} , \bd^{(3)}, s^{(3)}, M^{(3)} )
\eean
with 
\bean
&& L^{(3)} \leq \widetilde C_{9} (\log m / 2 )^{4}, \quad \|\bd^{(3)}\|_{\infty} \leq  \widetilde C_{9} \sqrt{m} (\log m / 2)^{9},
\\
&& s^{(3)} \leq  \widetilde C_{9} \sqrt{m}  (\log m / 2)^{9}, \quad M^{(3)} \leq \exp (  \widetilde C_{9} \{ \log m / 2 \}^{2} ), 
\eean
such that
\bean
&& \left\| \begin{pmatrix}
\sigma_{t-t_*}  \nabla p_t(\bx) 
\\
 p_t(\bx)
\end{pmatrix}
-  \bff^{(3)}(\bx,t - t_*) \right\|_{\infty} \leq  \widetilde C_{10} 2^{-D(\tau_{\rm sm}^{(3)} / 2 + 1 )}  m^{-\frac{2\beta}{d} -\frac{1}{4D} - 3D D_{1}^2 }  \defeq \epsilon_{3}
\eean
for $\|\bx\|_{\infty} \leq  \mu_t + \sigma_t D_{1} \sqrt{\log m}$ and $2 m^{-\frac{1}{2D}} \leq t \leq \tau_{\rm max}  \log m$, where $t_* = m^{-1/(2D)}$.
Note that $C_{S,1}^{-1} m^{-D D_{1}^2 } / 2 \leq p_t(\bx) - \epsilon_{3} \leq f_{D+1}^{(3)}(\bx,t) \leq  p_t(\bx) + \epsilon_{3} \leq 2 C_{S,1}$ for $\|\bx\|_{\infty} \leq  \mu_t + \sigma_t D_{1} \sqrt{\log m}$ and $2 m^{-\frac{1}{2D}} \leq t \leq \tau_{\rm max}  \log m$ with large enough $m$ so that $\epsilon_{3} \leq C_{S,1}^{-1} m^{-D D_{1}^2 } /2 $.
Then,
\bean
&& \left\vert \frac{1}{p_t(\bx)} - f_{{\rm rec}}^{(\rm bd)}\left(f_{D+1}^{(3)}(\bx,t-t_*)\right) \right\vert
\\
&& \leq \left\vert \frac{1}{p_t(\bx)} - \frac{1}{f_{D+1}^{(3)}(\bx,t-t_*)} \right\vert 
+ \left\vert \frac{1}{f_{D+1}^{(3)}(\bx,t-t_*)} - f_{{\rm rec}}^{(\rm bd)}\left(f_{D+1}^{(3)}(\bx,t-t_*)\right) \right\vert
\\
&& \leq \{ p_t(\bx) \wedge f_{D+1}^{(3)}(\bx,t - t_*)  \}^{-2} \epsilon_3
+ \delta^{(2)}
\leq 4 C_{S,1}^2 \epsilon_{3} m^{2 D D_{1}^2} + \delta^{(2)}
\\ 
&& \leq (4 C_{S,1}^2 +\widetilde C_{10} 2^{-D(\tau_{\rm sm}^{(3)} / 2 + 1 )}) \epsilon_{3} m^{2 D D_{1}^2}
\eean
for $\|\bx\|_{\infty} \leq  \mu_t + \sigma_t D_{1} \sqrt{\log m}$ and $2 m^{-\frac{1}{2D}} \leq t \leq \tau_{\rm max}  \log m$.
The last inequality holds because $2 DD_{1}^2 \geq 4 D \tau_{\rm min} \geq 4/3 \geq 1/(4D)$.
Then,
\be
\begin{split}
& \left\vert \frac{\sigma_{t-t_*} (\nabla p_t(\bx))_{i}}{p_t(\bx)} - f_{{\rm mult}}^{(\rm bd)} \left( f_i^{(3)}(\bx,t-t_*) ,  f_{\rm rec}^{(\rm bd)}\left(f_{D+1}^{(3)}(\bx,t-t_*) \right) \right)  \right\vert 
\\
& \leq m^{-2\beta /d - 1} + 2(C_{S,1} \vee C_{S,3})(4 C_{S,1}^2 +\widetilde C_{10} 2^{-D(\tau_{\rm sm}^{(3)} / 2 + 1 )})  \epsilon_{3} m^{3 D D_{1}^2}  \leq  D_{4} m^{-\beta/d - 1/ (4D) } 
\label{eqthm4:nn3}
\end{split} \ee
for $\|\bx\|_{\infty} \leq  \mu_t + \sigma_t D_{1} \sqrt{\log m}$ and $2 m^{-\frac{1}{2D}} \leq t \leq \tau_{\rm max}  \log m$.
where $D_{4} = D_{4} (D, \tau_{\rm sm}^{(3)}, \widetilde C_{10}, C_{S,1}, C_{S,3} ) $.
Consider functions $\widetilde f_{1}^{(3)},\ldots,\widetilde f_{D}^{(3)} : \bbR^{D} \times \bbR \to \bbR$ such that
\bean
\widetilde f^{(3)}_{i}(\bx,t) =  f_{\rm mult} \left( f_{{\rm mult}}^{(\rm bd)} \left( f_i^{(3)}(\bx,t - t_*) ,  f_{\rm rec}^{(\rm bd)}\left(f_{D+1}^{(3)}(\bx,t - t_*) \right) \right) , f_{\rm rec} ( f_{\sigma}(t - t_* )) \right), \quad i \in [D]
\eean
for $\bx \in \bbR^{D}$ and $t \in \bbR$.
Combining with (\ref{eqthm4:nnrecst}) and (\ref{eqthm4:nnmult}), we have
\bean
&& \left\vert \sigma_{t - t_*}^{-1} f_{{\rm mult}}^{(\rm bd)} \left( f_i^{(3)}(\bx,t - t_*) ,  f_{\rm rec}^{(\rm bd)}\left(f_{D+1}^{(3)}(\bx,t - t_*), \right) \right) - \widetilde f^{(3)}_{i}(\bx,t) \right\vert
\\
&& \leq \delta +  2  m^{3DD_{1}^2} (1+  m^{2\tau_{\rm min}}) \delta \leq 5 m^{5D D_{1}^2} \delta, \quad i \in [D]
\eean
for $\bx \in \bbR^{D}$ and $t \geq 2m^{-1/(2D)}$.
Note that $\vert \sigma_{t} / \sigma_{t-t_*} \vert \leq \vert \sigma_{t-t_*}^{-1} \vert \leq (2\underline{\tau})^{-1/2} m^{1/(4D) }$ for $t \geq 2m^{-1/(2D)}$.
Combining (\ref{eqthm4:nn3}) with the last display, we have
\bean
&& \left\vert (\nabla \log p_t(\bx))_{i} - \widetilde f_{i}^{(3)}(\bx,t)  \right\vert \sigma_{t - t_*} 
\\
&& \leq \left\vert \frac{\sigma_{t - t_*} (\nabla p_t(\bx))_{i}}{p_t(\bx)} - f_{{\rm mult}}^{(\rm bd)} \left( f_i^{(3)}(\bx,t - t_*) ,  f_{\rm rec}^{(\rm bd)}\left(f_{D+1}^{(3)}(\bx,t - t_*) \right) \right)  \right\vert + 5 m^{5D D_{1}^2} \delta\sigma_{t - t_*}
\\
&& \leq D_{4}  m^{-\beta/d -1/(4D)}+  5 m^{5D D_{1}^2} \delta, \quad i \in [D]
\eean
and
\bean
&&  \left\vert (\nabla \log p_t(\bx))_{i} - \widetilde f_{i}^{(3)}(\bx,t)  \right\vert \sigma_{t}
\\
&& \leq  \left\vert (\nabla \log p_t(\bx))_{i} - \widetilde f_{i}^{(3)}(\bx,t)  \right\vert  \sigma_{t - t_*}  (2\underline{\tau})^{-1/2} m^{1/(4D) }
\\
&& \leq D_{4} (2\underline{\tau})^{-1/2}  \sqrt{\underline{\tau}} m^{-\beta/d }
+ 5 (2\underline{\tau})^{-1/2}  \sqrt{\underline{\tau}} m^{5 D D_{1}^2 + 1/(4D)} \delta, \quad i \in [D]
\eean
for $\|\bx\|_{\infty} \leq  \mu_t + \sigma_t D_{1} \sqrt{\log m}$ and $2 m^{-\frac{1}{2D}} \leq t \leq \tau_{\rm max}  \log m$.
Let $\delta = m^{-5 D D_{1}^2 - 1/(4D) - \beta/d}$.
Then, there exists a positive constant $D_{5} = D_{5}(\underline{\tau}, D_{2},D_{3},D_{4})$ such that (\ref{eqthm4:nn1err}), (\ref{eqthm4:nnpt2err}) and the last display are bounded by
\be
D_{5} m^{-\frac{\beta}{d}} (\log m)^{ 2D+2\beta } \defeq \epsilon_{4} \label{eqthm4:nnerrtot}
\ee

\subsubsection{Combining into a Single Function}
For large enough $m$ so that $3m^{-1/(2D)} \leq (2\overline{\tau})^{-1}$, we have
\bean
\left\{ \log (1/\sqrt{\underline{\tau}}) + \tau_{\rm min} \log m /2 \right\}^{-3/2} \leq \{ \log(1/\sigma_{t}) \} ^{-3/2} \leq \left\{  \log(1/\sqrt{6\overline{\tau}}) + \log m /(4D) \right\}^{-3/2}
\eean
for $m^{-\tau_{\rm min}} \leq t \leq 3m^{-1/(2D)}$ because $\sqrt{\underline{\tau}t }  \leq \sigma_{t} \leq \sqrt{2\overline{\tau}t}$ for $0 \leq t \leq (2\overline{\tau})^{-1}$.
For large enough $m$ so that $\log (1/\sqrt{\underline{\tau}}) \leq \tau_{\rm min} \log m /2 $, it follows that
\bean
\tau_{\rm min}^{-3/2} (\log m)^{-3/2} \leq \{ \log(1/\sigma_{t}) \} ^{-3/2} \leq (4D)^{3/2} (\log m)^{-3/2}
\eean
for $m^{-\tau_{\rm min}} \leq t \leq 3m^{-1/(2D)}$.
Then,
\be
 \mu_t-   \tau_{\rm min}^{3/2} \{    (4D)^{3/2} +3  \} \{ \log (1/\sigma_t)\}^{-3/2}
 < \mu_t-  \overline{x}  < \mu_t-    \underline{x}
< \mu_{t} - \{ \log (1/\sigma_{t}) \}^{-3/2} \label{eqthm4:clipxineq}
\ee
for $m^{-\tau_{\rm min}} \leq t \leq 3m^{-1/(2D)}$, where 
\bean
\overline{x} =  \{    (4D)^{3/2} + 2  \} (\log m)^{-3/2}
\text{\quad and \quad}
\underline{x} =  \{    (4D)^{3/2} + 1  \}   (\log m)^{-3/2}.
\eean
Consider piecewise linear functions $f_{\rm swit, x}^{(1)}, f_{\rm swit, x}^{(2)} : \bbR \to [0,1]$ such that
\bean
&& f_{\rm swit, x}^{(1)} (x) =  (\log m)^{3/2}  \rho \left( -f_{\rm clip}^{(\underline{x}, \overline{x})} (x) + \overline{x} \right) = \frac{1}{\overline{x} - \underline{x}} \max\left( - (x \vee \underline{x}) \wedge \overline{x}  + \overline{x}, 0 \right),
\\
&& f_{\rm swit, x}^{(2)} (x) =  (\log m)^{3/2}  \rho \left( f_{\rm clip}^{(\underline{x}, \overline{x})} (x) - \underline{x} \right) = \frac{1}{\overline{x} - \underline{x}} \max\left(  (x \vee \underline{x}) \wedge \overline{x}  - \underline{x}, 0 \right),
\eean
for $x \in \bbR$, where $f_{\rm clip}^{(\underline{x}, \overline{x})} \in \cF_{\rm NN}(2,(1,2,1)^{\top}, 7, \overline{x} \vee (\log m)^{3/2})$ is the neural network in Lemma~\ref{secnn:clip}.
Note that $f_{\rm swit, x}^{(1)}(x) + f_{\rm swit, x}^{(2)}(x) = 1$ for $x \in \bbR$, and $f_{\rm swit, x}^{(1)}(x) = 0$ for $x \geq \overline{x}$ and $f_{\rm swit, x}^{(2)}(x) = 0$ for $x \leq \underline{x}$.
Combining with (\ref{eqthm4:nnerrtot}) and (\ref{eqthm4:clipxineq}), we have
\bean
&& \left\vert f_{\rm swit, x}^{(1)} \left( \mu_{t} - f_{\max}(\bx) \right) \widetilde f_{i}^{(1)}(\bx,t) + f_{\rm swit, x}^{(2)} \left( \mu_{t} - f_{\max}(\bx)   \right)  \widetilde f_{i}^{(2)}(\bx,t) - (\nabla \log p_t (\bx))_{i}    \right\vert
\\
&& \leq \epsilon_{4} / \sigma_{t}, \quad i \in [D]
\eean
for $\|\bx\|_{\infty} \leq  \mu_t + \sigma_t D_{1} \sqrt{\log m}$ and $m^{-\tau_{\rm min}} \leq t \leq 3 m^{-\frac{1}{2D}}$, where $f_{\rm max} : \bbR^{D} \to \bbR$ is a function such that
\bean
&& f_{\rm max}(\bx) = \| \bx \|_{\infty}
\\
&& = \rho \left( \cdots \rho \left( \rho \left( \rho( \vert x_1 \vert - \vert x_2 \vert ) + \vert x_2 \vert  - \vert  x_3 \vert  \right) + \vert  x_{3} \vert  -  \vert x_{4} \vert \right) \cdots +  \vert x_{D-1} \vert  -  \vert x_{D} \vert  \right) +  \vert x_{D} \vert .
\eean
Note also that $\vert x \vert = \rho(x) + \rho(-x)$.
Lemma~\ref{secnn:mtst} implies that
there exist neural networks  $f_{\mu} \in \cF_{\rm NN}(L_{\mu},\bd_{\mu},s_{\mu},M_{\mu})$ with
\bean
&& L_{\mu} \leq C_{N,4} \{ \log(1/ \delta)\}^{2}, \quad ,  \| \bd_{\mu}\|_{\infty} \leq C_{N,4}  \{ \log(1/ \delta)\}^{2}
\\
&&  s_{\mu} \leq C_{N,4}  \{  \log(1/ \delta)\}^{3}, \quad  M_{\mu} \leq C_{N,4}  \log ( 1/ \delta) 
\eean
such that $ \vert \mu_{t} - f_{\mu}(t) \vert \leq  \delta$
for $ t \geq  0$.
Since $f_{\rm swit, x}^{(1)}$ and $f_{\rm swit, x}^{(2)}$ are $(\log m)^{\tau_{\rm bd}}$-Lipschitz continuous, $\vert f_{\rm swit, x}^{(i)} (\mu_{t} - \|\bx\|_{\infty}) -  f_{\rm swit, x}^{(i)} (f_\mu(t) - \|\bx\|_{\infty}) \vert \leq \delta (\log m)^{\tau_{\rm bd}} $ for each $i \in \{1,2\}$, $\bx \in \bbR^{D}$ and $t \geq 0$.
For $i \in [D]$, consider a function $f_{i}^{({\rm x})} : \bbR^{D} \times \bbR \to \bbR$ such that
\bean
&& f_{i}^{({\rm x})} (\bx,t)
\\
&& = f_{\rm mult} \left( f_{\rm swit, x}^{(1)} \left( f_{\mu}(t) - f_{\rm max}(\bx) \right),  \widetilde f_{i}^{(1)}(\bx,t) \right) + f_{\rm mult} \left( f_{\rm swit, x}^{(2)}  \left( f_\mu(t) -  f_{\rm max}(\bx)  \right)  , \widetilde f_{i}^{(2)}(\bx,t) \right)
\eean
for $\bx \in \bbR^{D}$ and $t \in \bbR$.
Combining with (\ref{eqthm4:nnmult}), we have
\be \begin{split}
& \left\vert f_{i}^{({\rm x})} (\bx,t) - (\nabla \log p_t (\bx))_{i}    \right\vert
\\
& \leq \epsilon_{4} / \sigma_{t} + 2 \delta + 4  m^{3DD_{1}^2} \delta (\log m)^{\tau_{\rm bd}} \leq \epsilon_{4} / \sigma_{t} + 6m^{3D D_{1}^2 } \delta (\log m)^{\tau_{\rm bd}}, \quad i \in [D]
\label{eqthm4:nnxclip}
\end{split} \ee
for $\|\bx\|_{\infty} \leq  \mu_t + \sigma_t D_{1} \sqrt{\log m}$ and $m^{-\tau_{\rm min}} \leq t \leq 3 m^{-\frac{1}{2D}}$.
Similarly, let $\underline{t} = 2m^{-\frac{1}{2D}}$ and $\overline{t} = 3m^{-\frac{1}{2D}}$, consider piecewise linear functions $f_{\rm swit, t}^{(1)}, f_{\rm swit, t}^{(2)} : \bbR \to [0,1]$ such that
\bean
&& f_{\rm swit, t}^{(1)} (t) =  m^{\frac{1}{2D}}  \rho \left( -f_{\rm clip}^{(\underline{t}, \overline{t})} (t) + \overline{t} \right) = \frac{1}{\overline{t} - \underline{t}} \max\left( - (t \vee \underline{t}) \wedge \overline{t}  + \overline{t}, 0 \right),
\\
&& f_{\rm swit, t}^{(2)} (t) =  m^{\frac{1}{2D}}   \rho \left( f_{\rm clip}^{(\underline{t}, \overline{t})} (t) - \underline{t} \right) = \frac{1}{\overline{t} - \underline{t}} \max\left(  (t \vee \underline{t}) \wedge \overline{t}  - \underline{t}, 0 \right),
\eean
where $f_{\rm clip}^{(\underline{t}, \overline{t})} \in \cF_{\rm NN}(2,(1,2,1)^{\top}, 7, \overline{t} \vee m^{\frac{1}{2D}} )$ is the neural network in Lemma~\ref{secnn:clip}.
Note that $f_{\rm swit, t}^{(1)}(t) + f_{\rm swit, t}^{(2)}(t) = 1$ for $t \in \bbR$, and $f_{\rm swit, t}^{(1)}(t) = 0$ for $t \geq \overline{t}$ and $f_{\rm swit, t}^{(2)}(t) = 0$ for $t \leq \underline{t}$.
Combining with (\ref{eqthm4:nnerrtot}) and (\ref{eqthm4:nnxclip}), we have
\bean
&& \left\vert f_{\rm swit, t}^{(1)}(t)  f_{i}^{({\rm x})} (\bx,t) + f_{\rm swit, t}^{(2)} \left( t  \right)  \widetilde f_{i}^{(3)}(\bx,t) - (\nabla \log p_t (\bx))_{i}    \right\vert
\\
&& \leq \epsilon_{4} / \sigma_{t} + 6m^{3D D_{1}^2 } \delta (\log m)^{\tau_{\rm bd}}, \quad i \in [D]
\eean
for $\|\bx\|_{\infty} \leq  \mu_t + \sigma_t D_{1} \sqrt{\log m}$ and $m^{-\tau_{\rm min}} \leq t \leq \tau_{\rm max} \log m$.
Consider a vector-valued function $\bff^{(\rm x,t)}  = ( f_1^{(\rm x,t)},\ldots,f_D^{(\rm x,t)} ) : \bbR^{D} \times \bbR \rightarrow \bbR$ such that
\bean
f_i^{(\rm x,t)} (\bx,t) = f_{\rm mult} \left( f_{\rm swit, t}^{(1)} \left(t \right),   f_{i}^{({\rm x})}(\bx,t) \right) + f_{\rm mult} \left( f_{\rm swit, t}^{(2)}  \left(t \right)  , \widetilde f_{i}^{(3)}(\bx,t) \right), \quad i \in [D]
\eean
for $\bx \in \bbR^{D}$ and $t \in \bbR$.
Combining with (\ref{eqthm4:nnmult}), we have
\be \begin{split}
& \left\vert f_i^{(\rm x,t)} (\bx,t) - (\nabla \log p_t (\bx))_{i}    \right\vert
\leq \epsilon_{4} / \sigma_{t} + 6m^{3D D_{1}^2 } \delta (\log m)^{\tau_{\rm bd}}  + 4\delta
\\
& \leq \epsilon_{4} / \sigma_{t} + 10 m^{3D D_{1}^2 } \delta (\log m)^{\tau_{\rm bd}} \leq 11 \epsilon_{4} / \sigma_{t}, \quad i \in [D]
\label{eqthm4:nnerrfin}
\end{split} \ee
for $\|\bx\|_{\infty} \leq  \mu_t + \sigma_t D_{1} \sqrt{\log m}$ and $m^{-\tau_{\rm min}} \leq t \leq \tau_{\rm max} \log m$, where the last inequality holds with large enough $m$ so that $m^{3D D_1^2} \delta (\log m)^{\tau_{\rm bd}} \leq \epsilon_{4}$.
Since $\| \bx \|_{\infty} \leq \|\bx\|_{2}$ for any $\bx \in \bbR^{D}$, Lemma~\ref{secsc:scorebound} implies that $\sigma_{t} \| \nabla \log p_t(\bx) \|_{\infty} \leq C_{S,2} D_{1} \sqrt{\log m} $ for $\|\bx\|_{\infty} \leq  \mu_t + \sigma_t D_{1} \sqrt{\log m}$ and $t \geq 0$ with large enough $m$ so that $D_{1} \sqrt{\log m} \geq 1$, where $C_{S,2} = C_{S,2}(D,K,\tau_{1}, \overline{\tau}, \underline{\tau})$ is the constant in Lemma~\ref{secsc:scorebound}. 
Combining with the last display, we have
\bean
\sigma_{t} \left\vert f_i^{(\rm x,t)} (\bx,t) \right\vert \leq 11\epsilon_{4} + C_{S,2} D_{1} \sqrt{\log m} \leq D_{6} \sqrt{\log m}, \quad i \in [D]
\eean
for $\|\bx\|_{\infty} \leq  \mu_t + \sigma_t D_{1} \sqrt{\log m}$ and $m^{-\tau_{\rm min}} \leq t \leq \tau_{\rm max} \log m$ with large enough $m$ so that $\epsilon_{4} \leq \sqrt{\log m}$, where $D_{6} = D_{6}(\beta,d,D,C_{S,2})$.
Consider a function $\bff = (f_1,\ldots,f_{D})^{\top} : \bbR^{D} \times \bbR \to \bbR^{D}$ such that
\bean
 f_i(\bx, t) 
= \left( f_i^{(\rm x,t)} (\bx,t) \vee - \sigma_{t}^{-1} D_6 \sqrt{\log m} \right) \wedge \sigma_{t}^{-1} D_6 \sqrt{\log m}, \quad i \in [D]
\eean
for $\bx \in \bbR^{D}$ and $t \in \bbR$.
Note that $f_i(\bx,t) = f_i^{(\rm x,t)} (\bx,t)$ for $i \in [D]$, $\|\bx\|_{\infty} \leq  \mu_t + \sigma_t D_{1} \sqrt{\log m}$ and $m^{-\tau_{\rm min}} \leq t \leq \tau_{\rm max} \log m$.
Combining with (\ref{eqthm4:nnerrfin}),
\be
\sigma_{t} \left\| \bff (\bx,t) - \nabla \log p_t (\bx)  \right\|_{\infty} 
\leq 11 \epsilon_{4} = 11 D_{5} m^{-\frac{\beta}{d}} (\log m)^{2D + 2\beta} \label{eqthm4:nnerrwhole}
\ee
for $\|\bx\|_{\infty} \leq  \mu_t + \sigma_t D_{1} \sqrt{\log m}$ and $m^{-\tau_{\rm min}} \leq t \leq \tau_{\rm max} \log m$.

\subsubsection{Outside of Near-Support}
Note that 
\bean
&& \int_{\underline{T}}^{\overline{T}} \int_{\bbR^{D}} \left\| \nabla \log p_t(\bx) - \bff(\bx,t) \right\|_2^2 p_t(\bx) \d \bx 
\\
&& \leq \int_{\underline{T}}^{\overline{T}} \int_{\|\bx\|_{\infty} \leq  \mu_t + \sigma_t D_{1} \sqrt{\log m}} \left\| \nabla \log p_t(\bx) - \bff(\bx,t) \right\|_2^2 p_t(\bx) \d \bx\\
&&  + \int_{\underline{T}}^{\overline{T}} \int_{\|\bx\|_{\infty} \ge\mu_t + \sigma_t D_{1} \sqrt{\log m}} \left\| \nabla \log p_t(\bx) - \bff(\bx,t) \right\|_2^2 p_t(\bx) \d \bx
\eean
For $t \geq 0$,
\bean
&& \int_{\|\bx\|_{\infty} \geq \mu_{t} + \sigma_{t} D_{1} \sqrt{\log m}}  p_t(\bx) \d \bx 
=  \int_{\|\bx\|_{\infty} \geq \mu_{t} + \sigma_{t}  D_{1} \sqrt{\log m}} \int_{\| \by \|_{\infty} \leq 1} p_0(\by) \phi_{\sigma_t}(\bx- \mu_{t} \by) \d \by \d \bx
\\
&& = \int_{\| \by \|_{\infty} \leq 1} p_0(\by) \int_{\|\sigma_{t}\bz + \mu_{t} \by  \|_{\infty} \geq \mu_{t} + \sigma_{t}  D_{1} \sqrt{\log m}} \phi_{1}(\bz) \d \bz \d \by
\\
&& \leq \int_{\| \by \|_{\infty} \leq 1} p_0(\by) \sum_{i=1}^{D} \int_{\vert \sigma_{t}z_i + \mu_{t} y_i  \vert \geq \mu_{t} + \sigma_{t}  D_{1} \sqrt{\log m}} \phi_{1}(\bz) \d \bz \d \by
\\
&& \leq \int_{\| \by \|_{\infty} \leq 1} p_0(\by) \sum_{i=1}^{D} \int_{\vert z_i \vert \geq  D_{1} \sqrt{\log m}} \phi_{1}(\bz) \d \bz \d \by =  D \int_{\vert z \vert \geq D_{1} \sqrt{\log m}} \phi(z) \d z,
\eean
where the last inequality holds because $\vert y_{i} \vert \leq 1$.
By the tail probability of the standard normal distribution, the last display is bounded by
\be
2 D m^{-D_{1}^2/2}. \label{eqthm4:tailbd1}
\ee
Let $ C_{S,2} = C_{S,2}(D,K,\tau_1, \overline{\tau}, \underline{\tau} )$ be the constant in Lemma~\ref{secsc:scorebound}.
Then, $\| \nabla \log p_t(\bx) \|_{2} \leq C_{S,2} (\|\bx\|_{\infty} - \mu_{t} ) /\sigma_{t}^2$ for $\|\bx\|_{\infty} \geq \mu_{t} + D_{1} \sqrt{\log m}$ with large enough $m$ so that $D_{1} \sqrt{\log m } \geq 1$.
Combining with the last display, we have
\bean
&& \int_{\|\bx\|_{\infty} \geq \mu_{t} + \sigma_{t} D_{1} \sqrt{\log m}} \| \nabla \log p_t(\bx) \|_2^2 p_t(\bx) \d \bx 
\\
&& \leq 2 C_{S,2}^2 \sigma_{t}^{-4} \int_{\|\bx\|_{\infty} \geq \mu_{t} +  \sigma_{t} D_{1} \sqrt{\log m}} \left(  \|\bx\|_{\infty}^2 + \mu_{t}^2  \right) p_t(\bx) \d \bx
\\
&& \leq 2 C_{S,2}^2 \sigma_{t}^{-4}\int_{\|\bx\|_{\infty} \geq \mu_{t} +  \sigma_{t} D_{1} \sqrt{\log m}}\| \bx\|_2^2 p_t(\bx) \d \bx +  2 D C_{S,2}^2 \sigma_{t}^{-4} \mu_{t}^2 m^{- D_1^2 / 2}
\eean
for $t \geq 0$.
A simple calculation yields that
\bean
&& \int_{\|\bx\|_{\infty} \geq \mu_{t} +  \sigma_{t} D_{1} \sqrt{\log m}} \|\bx\|_{2}^2  p_t(\bx) \d \bx 
\\
&& = \sum_{i=1}^{D} \int_{\| \by \|_{\infty} \leq 1} p_0(\by)  \int_{\|\bx\|_{\infty} \geq \mu_{t} + \sigma_{t}  D_{1} \sqrt{\log m}}  x_i^2  \phi_{\sigma_t}(\bx- \mu_{t} \by) \d \bx \d \by 
\\
&& = \sum_{i=1}^{D} \int_{\| \by \|_{\infty} \leq 1} p_0(\by) \int_{\|\sigma_{t}\bz + \mu_{t} \by  \|_{\infty} \geq \mu_{t} + \sigma_{t}  D_{1} \sqrt{\log m}} (\sigma_{t} z_i + \mu_{t} y_i)^2 \phi_{1}(\bz) \d \bz \d \by
\\
&& \leq \sum_{i=1}^{D} \int_{\| \by \|_{\infty} \leq 1} p_0(\by) \int_{\|\sigma_{t}\bz + \mu_{t} \by  \|_{\infty} \geq \mu_{t} + \sigma_{t}  D_{1} \sqrt{\log m}} 2 ( \sigma_{t}^2 z_i^2 + \mu_{t}^2 y_i^2) \phi_{1}(\bz) \d \bz \d \by 
\\
&&\leq \sum_{i=1}^{D} \int_{\| \by \|_{\infty} \leq 1} p_0(\by) \sum_{j=1}^{D}  \int_{ \vert \sigma_{t} z_j + \mu_{t} y_j \vert \geq \mu_{t} + \sigma_{t}  D_{1} \sqrt{\log m}} 2 ( \sigma_{t}^2 z_i^2 + \mu_{t}^2 y_i^2) \phi_{1}(\bz) \d \bz \d \by 
\eean
for $t \geq 0$.
Since $\vert y_j \vert \leq 1$ in the last integral, the last display is bounded by
\bean
&& \sum_{i=1}^{D} \int_{\| \by \|_{\infty} \leq 1} p_0(\by) \sum_{j=1}^{D}  \int_{ \vert  z_j \vert \geq  D_{1} \sqrt{\log m}} 2 ( \sigma_{t}^2 z_i^2 + \mu_{t}^2 y_i^2 ) \phi_{1}(\bz) \d \bz \d \by 
\\
&& \leq 2 \sigma_{t}^2 \sum_{i=1}^{D} \sum_{j=1}^{D}  \int_{\vert z_j \vert \geq  D_{1} \sqrt{\log m}} z_i^2 \phi_{1}(\bz) \d \bz  + 4 D^2 \mu_{t}^2 m^{- D_1^2 /2},
\eean
where the last inequality holds by (\ref{eqthm4:tailbd1}).
Furthermore,
\bean
&& \sum_{i=1}^{D} \sum_{j=1}^{D}  \int_{\vert z_j \vert \geq  D_{1} \sqrt{\log m}} z_i^2 \phi_{1}(\bz) \d \bz
\\
&& = \sum_{i=1}^{D} \left\{ (D-1) \bbE[Z^2] \int_{\vert z \vert \geq  D_{1} \sqrt{\log m}} \phi(z) \d z +  \int_{\vert z \vert \geq  D_{1} \sqrt{\log m}} z^2 \phi(z) \d z   \right\}
\\
&& \leq 2 D \left\{ (D-1) m^{-D_{1}^2/2} + \sqrt{\bbE [Z^4]} m^{-D_{1}^2/4} \right\} \leq 2 D (D+ \sqrt{3} -1 ) m^{-D_1^2 /4 },
\eean
where the first inequality holds by the Cauchy-Schwarz inequality.
Hence, there exists a constant $D_{7} = D_{7}(D, \underline{\tau}, C_{S,2})$ such that
\bean
\int_{\|\bx\|_{\infty} \geq \mu_{t} +  \sigma_{t} D_{1} \sqrt{\log m}} \| \nabla \log p_t(\bx) \|_2^2 p_t(\bx) \d \bx  \leq D_{7} \sigma_{t}^{-2} m^{-2\beta/d}
\eean
for $t \geq m^{-\tau_{\rm min}}$ because $\mu_{t} \leq 1$ and $\sigma_{t} \geq \sqrt{\underline{\tau} m^{-\tau_{\rm min}}}$.
Combining with (\ref{eqthm4:tailbd1}), we have
\bean
&& \int_{\|\bx\|_{\infty} \geq \mu_{t} +  \sigma_{t} D_{1} \sqrt{\log m}} \left\| \nabla \log p_t(\bx) -  \bff(\bx,t) \right\|_2^2 p_t(\bx) \d \bx  
\\
&& \leq \int_{\|\bx\|_{\infty} \geq \mu_{t} +  \sigma_{t} D_{1} \sqrt{\log m}} 2 \| \nabla \log p_t(\bx) \|_2^2 p_t(\bx) \d \bx + \int_{\|\bx\|_{\infty} \geq \mu_{t} +  \sigma_{t} D_{1} \sqrt{\log m}} 2 \| \bff(\bx,t) \|_2^2 p_t(\bx) \d \bx 
\\
&& \leq 2 D_{7} \sigma_{t}^{-2} m^{-2 \beta/d}
+ 4 \sigma_{t}^{-2} D^2 D_{6}^2 m^{-D_{1}^2 /2 } \log m
\\
&& \leq \sigma_{t}^{-2} \left( 2 D_{7} m^{-2 \beta/d}
+ 4 D^2 D_{6}^2 m^{-2\beta/d} \log m \right)
\eean
for $t \geq m^{-\tau_{\rm min}}$, where the second inequality holds because $\| \bff(\bx,t)\|_{\infty} \leq \sigma_{t}^{-1} D_{6} \sqrt{\log m}$ for $\bx \in \bbR^{D}, t \in \bbR$.
Combining with (\ref{eqthm4:nnerrwhole}), we have 
\bean
&& \sigma_{t}^2 \int_{\bbR^{D}} \left\| \nabla \log p_t(\bx) -  \bff(\bx,t) \right\|_2^2 p_t(\bx) \d \bx  
\\
&& \leq  D 11^2 D_{5}^2 m^{-\frac{2\beta}{d}} (\log m)^{4D +4\beta } + 2 D_{7} m^{-\frac{2\beta}{d}} + 4 D^2 D_{6}^2 m^{-D_{1}^2 /2 } \log m
\\
&& \leq D_{8} m^{-\frac{2\beta}{d}} (\log m)^{4D + 4\beta} 
\eean
for $m^{-\tau_{\rm min}} \leq t \leq \tau_{\rm max} \log m$, where $D_{8} = D_{7}(D, D_{5}, D_{6}, D_{7})$.
Since $\sigma_{t}^2 \geq \underline{\tau} t $ for $t \geq 0$, we have
\bean
&& \int_{\underline{T}}^{\overline{T}} \int_{\bbR^{D}} \left\| \nabla \log p_t(\bx) -  \bff(\bx,t) \right\|_2^2 p_t(\bx) \d \bx  \d t
\\
&& \leq  D_{8} \underline{\tau}^{-1} m^{-\frac{2\beta}{d}} (\log m)^{4D + 4\beta}  \left( \log \overline{T} - \log \underline {T} \right)
\\
&& \leq D_{9} m^{-\frac{2\beta}{d}} (\log m)^{4D + 4\beta + 1},
\eean
where $D_{9} = D_{9}(\underline{\tau}, \tau_{\rm max}, \tau_{\rm min}, D_{8})$.
Lemma~\ref{secnn:comp}, Lemma~\ref{secnn:par}, Lemma~\ref{secnn:lin}, Lemma~\ref{secnn:id} and Lemma~\ref{secnn:sharing} imply that $\bff^{(\rm x,t)}  \in \cF_{\rm WSNN}(L,\bd,s,m,M,\cP)$ with
\bean
&& L \leq D_{10} ( \log m )^{6} \log \log m , \quad \| \bd\|_{\infty} \leq D_{10} m^{D+1},
\\
&& s \leq D_{10}  m (\log m)^{5} \log \log m, \quad M \leq \exp(D_{10} \{ \log m \}^6 ),
\eean
$\| \bfm \|_{\infty} \leq D_{10} m^{D}$ and the set of permutation matrices $\cP$,
where 
\bean
D_{10} = D_{10}(\beta, d, D, \tau_{\rm min}, \widetilde C_{3}, \widetilde C_{6}, \widetilde C_{9}, C_{N,1}, C_{N,4}, C_{N,5}, C_{S,1}, D_{1})
\eean
is a large enough constant.
The assertion follows by re-defining the constants.
\hfill\BlackBox\\[2mm]

\section{Proofs for the Convergence Rate}

In this section, we provide the proof of Theorem~\ref{secthm:2}.
We begin by outlining auxiliary lemmas and a proposition.

\begin{lemma}[Error bound for small $t$]
\label{secsc:p0pt}
Let $\beta, K > 0$ be given and suppose the true density $p_0$ belongs to $\cH^{\beta, K}([-1,1]^{D})$. 
Then, there exist positive constants $\widetilde C_{12} = \widetilde C_{12}(\beta, D, K ,\overline{\tau}, \underline{\tau})$ and $ \widetilde C_{13} = \widetilde C_{13}(\overline{\tau}, \underline{\tau})$ such that
\bean
\int_{\bbR^{D}} \left\vert p_0(\bx) - p_t(\bx) \right\vert \d \bx \leq \widetilde C_{12} \left\{ t \log (1/t) \right\}^{\frac{\beta \wedge 1}{2}} 
\eean
for $0 < t \leq \widetilde C_{13}$.
\end{lemma}

For any function $\bff : \bbR^{n_1} \to \bbR^{n_2}, n_1, n_2 \in \bbN$ and $C > 0$, denote $\| \cdot \|_{L^{\infty}( [-C,C]^{n_1} ) }$ as the sup-norm over $[-C,C]^{n_1}$, defined as 
\bean
\| \bff \|_{L^{\infty}( [-C,C]^{n_1} ) }
= \sup_{\bx \in [-C,C]^{n_1}} \| \bff(\bx) \|_{\infty}.
\eean
For a function space $\cF$, $N(\delta,\cF,d)$ denotes the covering numbers with respect to the (pseudo)-metric $d$; see \citet{van1996weak} for the detailed definition.
The following lemma provides a covering number of $\cF_{\rm WSNN}$.
Our main proof strategy follows the proof of Lemma 5 from \citet{schmidt2020nonparametric}, with modifications for weight-sharing networks.

\begin{lemma}[Covering number of $\cF_{\rm WSNN}$]
\label{sec:covering}
Let $C$ and $ 0 <  \delta < 1 $ be given.
For the class of weight-sharing neural networks $ \cF_{\rm WSNN} ( L, \bd, s,  M, \cP_{\bfm})$,
we have
\bean
&& \log N \left(\delta, \cF_{\rm WSNN}(L,\bd,s,M,\cP_{\bfm}), \| \cdot \|_{L^{\infty}([-C,C]^{d_{1}})} \right) 
\\
&& \leq (s+1) \log \left( \frac{4 L^2 \|\bd\|_{\infty}^2 \left\{  \| \bfm\|_{\infty} \|\bd\|_{\infty} (M \vee 1) \right\}^{L} \left(L + C + 3 \right)   }{\delta}  \right).
\eean
\end{lemma}

For $\underline{T}, \overline{T} > 0$ with $\underline{T} < \overline{T}$ and a function $\bff : \bbR^{D} \times \bbR \rightarrow \bbR^{D}$, recall the definition of loss function $\ell_{\bff} : [-1,1]^{D} \rightarrow \bbR$, given by
\bean
\ell_{\bff}(\bx) && = \int_{\underline{T}}^{\overline{T}} \bbE \Big[ \Big\| \bff(\bX_{t}, t) + \frac{\bX_t - \mu_t \bX_0}{\sigma_t^2} \Big\|_2^2 ~ \big| ~ \bX_0 = \bx \Big]\d t
\\
&& = \int_{\underline{T}}^{\overline{T}} \bbE \Big[ \Big\| \bff(\mu_t \bx+ \sigma_t \bZ, t) + \frac{\bZ}{\sigma_t} \Big\|_2^2 \Big]\d t,
\eean
where the last equality holds because $\bX_{t} = \mu_t \bX_{0} + \sigma_{t} \bZ$ with $D$-dimensional standard normal variable $\bZ$.
Define the pointwise excess risk by
\bean
\nu_{\bff}(\bx) = \ell_{\bff}(\bx) - \ell_{\bff_0}(\bx),
\eean
where $(\bx,t) \mapsto \bff_0(\bx,t) = \nabla \log p_t(\bx)$ is a score function.
Combining with (\ref{eq:smvincent}), we have
\bean
\bbE[ \nu_{\bff} (\bX_0) ] = \int_{\underline{T}}^{\overline{T}} \bbE \left[ \left\| \bff(\bX_t,t) - \bff_0(\bX_t,t) \right\|_2^2 \right] \d t.
\eean
The following lemma provides an upper bound on the sup-norm and the second moment (hence a variance-type bound) of the excess risk over the bounded function class.

\begin{lemma}[Upper bounds of excess risk]
\label{sec:exbd}
Let $K, \tau_1, F, \underline{T}, \overline{T} > 0$ be given and suppose the true density $p_0$ satisfies that $\tau_1 \leq p_0(\bx) \leq K$ for any $\bx \in [-1,1]^{D}$.
For any function $\bff : \bbR^{D} \times \bbR \rightarrow \bbR^{D}$ satisfying
$
\| \bff(\cdot, t) \|_{L^{\infty}(\bbR^{D})} \leq F \sigma_{t}^{-1}
$ for all $t \in [\underline{T}, \overline{T}]$, there exists a positive constant $\widetilde C_{14} = \widetilde C_{14} (D, K, \tau_1, \overline{\tau}, \underline{\tau})$ such that
\bean
\| \nu_{\bff} \|_{L^{\infty}([-1,1]^{D})} \leq \widetilde C_{14}  (F^2 \vee 1)  \left( \log \overline{T} - \log \underline{T} \right)
\eean
and
\bean
\bbE \left[ \left\{ \nu_{\bff}(\bX_0)\right\}^2 \right] \leq 2 \widetilde C_{14}  (F^2 \vee 1) \left( \log \overline{T} - \log \underline{T} \right) \bbE \left[ \nu_{\bff}(\bX_0) \right].
\eean
\end{lemma}

The following proposition provides an oracle inequality for the ERM estimator under the score-matching loss.
The statement is similar to Theorem C.4 from \citet{oko2023diffusion} while addressing the issue identified by \citet{yakovlev2025generalization}.
Our main proof strategy follows the proof of Proposition 12 from \cite{stephanovitch2025generalization} with modifications that control the excess risk variance.

\begin{proposition}[Oracle inequality for score matching]
\label{sec:oracle}
Let $K, \tau_1, F, \underline{T}, \overline{T} > 0$ be given and suppose the true density $p_0$ satisfies that $\tau_1 \leq p_0(\bx) \leq K$ for any $\bx \in [-1,1]^{D}$.
Let $\bX^1,\ldots,\bX^n$ be the \iid \ samples drawn from $p_0$. 
For the class $\cF$, consisting of continuous functions $\bff : \bbR^{D} \times \bbR \to \bbR^{D}$ with $\| \bff(\cdot,t)\|_{L^{\infty}(\bbR^{D})} \leq F \sigma_{t}^{-1}$ for all $t \in [\underline{T}, \overline{T}]$, let 
\bean
\widehat \bff \in \argmin_{\bff \in \cF} \frac{1}{n} \sum_{i=1}^{n} \ell_{\bff}(\bX^i)
\eean
be the ERM estimator over the class $\cF$.
Then, there exists a positive constant $\widetilde C_{14} = \widetilde C_{14} (D,K,\tau_1,\overline{\tau},\underline{\tau})$ such that
\bean
&& \int_{\underline{T}}^{\overline{T}} \bbE \left[ \left\| \widehat \bff(\bX_t,t) - \bff_0(\bX_t,t) \right\|_2^2 \right] \d t
\\
&& \leq 3 \inf_{\bff \in \cF} \int_{\underline{T}}^{\overline{T}} \bbE \left[ \left\|  \bff(\bX_t,t) - \bff_0(\bX_t,t) \right\|_2^2 \right] \d t
\\
&& \quad + \frac{\overline{C}_1}{n} \left\{ \log N \left( n^{-2}, \cF, \| \cdot\|_{L^{\infty} ([-\overline{C}_2,\overline{C}_2]^{D+1})} \right) + \log (2n) \right\},
\eean
where 
\bean
\overline{C}_1 = \widetilde C_{14} (F^2 \vee 1)  \left( \sqrt{\overline{T}} + \sqrt{\log \overline{T} - \log \underline{T}} \right) \left( \sqrt{\log \overline{T} - \log \underline{T}} \right)
\eean
and
\bean
\overline{C}_2 = \left( 1 + 2\sqrt{2 \log n}  \right) \vee \overline{T}.
\eean
\end{proposition}

\subsection{Proof of Lemma~\ref{secsc:p0pt}}
\begin{proof}
Note that $p_0$ is continuous and supported on $[-1,1]^{D}$.
Thus, for $t \geq 0$,
\bean
\int_{\bbR^{D}} \left\vert p_0(\bx) - p_t(\bx) \right\vert \d \bx
= \int_{ \| \bx \|_{\infty} \leq 1 } \left\vert p_0(\bx) - p_t(\bx) \right\vert \d \bx
+ \int_{ \| \bx \|_{\infty} \geq 1} p_t(\bx) \d \bx.
\eean
We will derive error bounds for each integral on the RHS.

Note that $p_0 \in \cH^{\beta, K}([-1,1]^{D})$.
If $\beta \leq 1$, $\vert p_0(\bx) - p_0(\by) \vert \leq K \| \bx - \by\|_{\infty}^{\beta}$  for any $\bx, \by \in [-1,1]^{D}$ by the definition of $\cH^{\beta, K}$.
If $\beta > 1$, Mean value theorem implies that 
$\vert p_0(\bx) - p_0(\by) \vert \leq K \| \bx - \by\|_{2}$ for any $\bx, \by \in [-1,1]^{D}$. 
Combining two cases, we have
\be
\vert p_0(\bx) - p_0(\by) \vert \leq K\sqrt{D} \| \bx - \by\|_{\infty}^{\beta \wedge 1}
, \quad \forall \bx,\by \in [-1,1]^{D}. \label{eqlemtv:lip}
\ee
A simple calculation yields that for any $\bx \in [-1,1]^{D}$,
\bean
&& p_0(\bx)-p_t(\bx)
\\
&& = \int_{\| \by \|_{\infty} \leq 1} \left\{ 
p_0(\bx) - p_0(\by) \right\} \phi_{\sigma_t}(\bx-\mu_t \by) \d \by + p_{0}(\bx) \left( 1 - \int_{\| \by \|_{\infty} \leq 1} \phi_{\sigma_t}(\bx-\mu_t \by) \d \by \right).
\eean
Then, for any $t \geq 0$,
\be \begin{split}
&  \int_{\| \bx \|_{\infty} \leq 1} \left\vert p_0(\bx) - p_t(\bx) \right\vert \d \bx 
\\
& \leq \int_{\| \bx \|_{\infty} \leq 1} \left\vert   \int_{\| \by \|_{\infty} \leq 1} \left\{ p_0(\bx) - p_0(\by) \right\} \phi_{\sigma _t}(\bx - \mu_t \by) \d \by \right\vert \d \bx 
\\
& \quad + \int_{\| \bx \|_{\infty} \leq 1} p_0(\bx) \left\vert  1 - \int_{\| \by \|_{\infty} \leq 1} \phi_{\sigma_t}(\bx-\mu_t \by) \d \by  \right\vert \d \bx
\label{eqlemtv:int}
\end{split}
\ee

We bound each term on the RHS.
For any $\delta > 0$, a simple calculation yields that
\be
\begin{split}
& \int_{\| \bx \|_{\infty} \leq 1} \left\vert   \int_{\| \by \|_{\infty} \leq 1} \left\{ p_0(\bx) - p_0(\by) \right\} \phi_{\sigma _t}(\bx - \mu_t \by) \d \by \right\vert \d \bx 
\\
& \leq K\sqrt{D} \int_{\| \bx \|_{\infty} \leq 1}    \int_{\| \by \|_{\infty} \leq 1}  \|\bx-\by\|_{\infty}^{\beta \wedge 1} \phi_{\sigma _t}(\bx - \mu_t \by) \d \by  \d \bx 
\\
& \leq  K\sqrt{D}\int_{\| \bx \|_{\infty} \leq 1} \left(  \delta^{\beta \wedge 1} \int_{\substack{ \| \by \|_{\infty} \leq 1 \\ \|\bx-\by\|_{\infty}\leq \delta  }}  \phi_{\sigma_{t}}(\bx-\mu_t \by) \d \by + 2^{\beta \wedge 1}  \int_{\substack{ \| \by \|_{\infty} \leq 1 \\ \|\bx-\by\|_{\infty} \geq \delta  }}  \phi_{\sigma_{t}}(\bx-\mu_t \by) \d \by \right) \d \bx
\\
& \leq K\sqrt{D} 2^{D} \mu_{t}^{-D} \delta^{\beta \wedge 1}
+ K\sqrt{D} 2^{\beta \wedge 1} \int_{\| \bx \|_{\infty} \leq 1} \int_{\substack{ \| \by \|_{\infty} \leq 1 \\ \|\bx-\by\|_{\infty}\geq \delta  }}  \phi_{\sigma_{t}}(\bx-\mu_t \by) \d \by \d \bx,
\label{eqlemtv:int1}
\end{split}
\ee
where the first inequality holds by (\ref{eqlemtv:lip}) and the last inequality holds because
\be
\int_{\bbR^{D}} \phi_{\sigma_t}(\bx-\mu_t \by) \d \by
= \mu_{t}^{-D}, \quad \forall \bx \in \bbR^{D}.
\label{eqlemtv:norm}
\ee

Since $1-\mu_t \geq 0 $, we have $\| \bx- \mu_{t} \by \|_{\infty} \geq \| \bx - \by\|_{\infty} -  ( 1 - \mu_t ) \|\by\|_{\infty} \geq \delta - ( 1 - \mu_t )$ for $\|\bx-\by\|_{\infty} \geq \delta$ and $\| \by\|_{\infty} \leq 1$.
Then, a simple calculation yields that
\bean
&& \int_{\| \bx \|_{\infty} \leq 1} \int_{\substack{ \| \by \|_{\infty} \leq 1 \\ \|\bx-\by\|_{\infty}\geq \delta  }}  \phi_{\sigma_{t}}(\bx-\mu_t \by) \d \by \d \bx \leq 
\int_{\| \bx \|_{\infty} \leq 1} \int_{ \|\bx-\mu_t\by\|_{\infty} \geq \delta - (1-\mu_t)  } \phi_{\sigma_t}(\bx-\mu_t \by) \d \by \d \bx
\\
&& = \mu_t^{-D} \int_{\| \bx \|_{\infty} \leq 1} \int_{\|\bz\|_{\infty} \geq \frac{\delta - (1-\mu_t)}{\sigma_t}} \phi_{1}(\bz) \d \bz \d \bx
= 2^{D} \mu_t^{-D} \int_{\|\bz\|_{\infty} \geq \frac{\delta - (1-\mu_t)}{\sigma_t}} \phi_{1}(\bz) \d \bz 
\\
&& \leq 2^{D} \mu_t^{-D} \sum_{i=1}^{D} \int_{\vert z_i \vert \geq \frac{\delta - (1-\mu_t)}{\sigma_t}} \phi_{1}(\bz) \d \bz
\leq D 2^{D+1} \mu_{t}^{-D} \exp \left( -\frac{ (\delta - (1-\mu_t))^2  }{2\sigma_t^2} \right)
\eean
for $\delta \geq 1-\mu_t$, where the last inequality holds by the tail probability of the standard normal distribution.
By (\ref{eqthm:mtstbd}), we have
\bean
 \frac{1}{2} \leq 1- \overline{\tau}t \leq \mu_t \leq 1- \frac{\underline{\tau}t}{2} \quad \text{and} \quad 
\sqrt{\underline{\tau}t } \leq \sigma_t \leq \sqrt{2\overline{\tau}t}
\eean
for any $0 \leq t \leq (2\overline{\tau})^{-1}$.
Let $\delta = 1-\mu_t +  2\sigma_{t} \sqrt{\log (1/\sigma_t)}$.
Combining last display with (\ref{eqlemtv:int1}), a simple calculation yields that
\be \begin{split}
 & \int_{\| \bx \|_{\infty} \leq 1} \left\vert   \int_{\| \by \|_{\infty} \leq 1} \left\{ p_0(\bx) - p_0(\by) \right\} \phi_{\sigma _t}(\bx - \mu_t \by) \d \by \right\vert \d \bx 
 \\
 & \leq K\sqrt{D} 2^{D} \mu_{t}^{-D} \delta^{\beta \wedge 1} + KD\sqrt{D} 2^{D + 1 + \beta \wedge 1} \mu_{t}^{-D} \exp \left( -\frac{ (\delta - (1-\mu_t))^2  }{2\sigma_t^2} \right)
 \\
 & \leq K\sqrt{D} 2^{D+1} \left( \overline{\tau} t + 2\sqrt{2 \overline{\tau} t \log (1/\sqrt{\underline{\tau} t}) } \right)^{\beta \wedge 1} + KD\sqrt{D} 2^{ 2D + 1 + \beta \wedge 1} \left( \sqrt{2\overline{\tau}t} \right)^{2}
 \\
 & \leq D_{1}  \left\{ t \log (1/t) \right\}^{\frac{\beta \wedge 1}{2}}
 \label{eqlemtv:intfin1}
 \end{split}
\ee
for any $0 \leq t \leq (2\overline{\tau})^{-1} \wedge 1$,
where $D_{1} = D_{1}(\beta,D,K,\overline{\tau},\underline{\tau})$.

For any $\bx \in \bbR^{D}$, we have
\bean
1 - \int_{\| \by \|_{\infty} \leq 1} \phi_{\sigma_t}(\bx-\mu_t \by) \d \by
&& = 1 - \mu_{t}^{-D} + \mu_{t}^{-D} - \int_{\| \by \|_{\infty} \leq 1} \phi_{\sigma_t}(\bx-\mu_t \by) \d \by
\\
&& = 1 - \mu_{t}^{-D} +  \int_{\| \by \|_{\infty} \geq 1} \phi_{\sigma_t}(\bx-\mu_t \by) \d \by,
\eean
where the last equality holds by (\ref{eqlemtv:norm}).
A simple calculation yields that 
\bean
\vert 1 - \mu_{t}^{-D} \vert
= \mu_{t}^{-D} \vert \mu_{t}^{D} - 1 \vert 
= \mu_{t}^{-D} \vert 1 -\mu_{t} \vert \sum_{k=0}^{D-1} \mu_{t}^{k}
\leq D 2^{D} \vert 1 - \mu_{t} \vert
\leq D 2^{D} \overline{\tau} t
\eean
for any $0 \leq t \leq (2\overline{\tau})^{-1}$.
For $\|\bx\|_{\infty} \leq \mu_{t} - 2 \sigma_{t} \sqrt{\log (1/\sigma_t)}$ and $\| \sigma_{t} \bz + \bx \|_{\infty} \geq \mu_t $, we have $\| \sigma_t \bz \|_{\infty} \geq \| \sigma_{t} \bz + \bx \|_{\infty} - \| \bx \|_{\infty} \geq 2 \sigma_{t} \sqrt{\log (1/\sigma_t)}$.
Then,
\bean
\int_{\| \by \|_{\infty} \geq 1} \phi_{\sigma_t}(\bx-\mu_t \by) \d \by
=
\mu_{t}^{-D} \int_{\| \sigma_t \bz + \bx \|_{\infty} \geq \mu_t} \phi_{1}(\bz) \d \bz
\leq \mu_{t}^{-D} \int_{\| \bz \|_{\infty} \geq 2 \sqrt{\log (1/\sigma_t)}} \phi_{1}(\bz) \d \bz.
\eean
By the tail probability of the standard normal distribution, the last display is bounded by
\bean
\mu_{t}^{-D} \sum_{i=1}^{D} \int_{\vert z_i \vert \geq 2 \sqrt{\log (1/\sigma_t) }} \phi_{1}(\bz) \d \bz
\leq 2 D \mu_{t}^{-D} \sigma_{t}^2
\leq  2^{D+2} D \overline{\tau} t
\eean
for any $0 \leq t \leq (2\overline{\tau})^{-1} \wedge 1$.
Then,
\bean
&& \int_{\| \bx \|_{\infty} \leq 1} p_0(\bx) \left\vert  1 - \int_{\| \by \|_{\infty} \leq 1} \phi_{\sigma_t}(\bx-\mu_t \by) \d \by  \right\vert \d \bx
\\
&& \leq D 2^{D} \overline{\tau} t + \int_{\| \bx \|_{\infty} \leq 1} p_0(\bx)  \int_{\| \by \|_{\infty} \geq 1} \phi_{\sigma_t}(\bx-\mu_t \by) \d \by \d \bx
\\
&& \leq D 2^{D} \overline{\tau} t + 
 2^{D+2} D \overline{\tau} t +
 \int_{\mu_{t} - 2 \sigma_{t} \sqrt{\log (1/\sigma_t)} \leq \| \bx \|_{\infty} \leq 1} p_0(\bx)  \int_{\| \by \|_{\infty} \geq 1} \phi_{\sigma_t}(\bx-\mu_t \by) \d \by \d \bx.
\eean
A simple telescoping sum implies that
\bean
&& 2^{D} - \left(2 \mu_{t} - 4 \sigma_{t} \sqrt{\log(1/\sigma_{t})} \right)^{D}
\\
&& = \left(2 - 2\mu_{t} +  4 \sigma_{t} \sqrt{\log(1/\sigma_{t})} \right) \sum_{k=0}^{D-1} 2^{D-1-k} \left(2 \mu_{t} - 4 \sigma_{t} \sqrt{\log(1/\sigma_{t})} \right)^{k} 
\\
&& \leq \left( 2 \overline{\tau} t + 4\sqrt{2 \overline{\tau}t \log (1/\sqrt{\underline{\tau} t} ) }    \right) 2^{D-1} D
\eean
for $0 \leq t \leq D_{2}$, where $D_{2} = D_{2} (\overline{\tau}, \underline{\tau})$ is a small enough constant so that $D_{2} \leq (2\overline{\tau})^{-1} \wedge 1$ and $2\mu_{t}-4\sigma_{t}\sqrt{\log(1/\sigma_{t})} \leq 2$ for $t \leq D_{2}$.
Combining (\ref{eqlemtv:norm}) with the last two displays, we have
\bean
 && \int_{\| \bx \|_{\infty} \leq 1} p_0(\bx) \left\vert  1 - \int_{\| \by \|_{\infty} \leq 1} \phi_{\sigma_t}(\bx-\mu_t \by) \d \by  \right\vert \d \bx
 \\
 && \leq  D 2^{D} \mu_{t}^{-D} \overline{\tau} t + 
 2^{D+2} D \overline{\tau} t +
 K D 2^{D} \left( \overline{\tau} t + 2\sqrt{2 \overline{\tau}t \log (1/\sqrt{\underline{\tau}t} }   \right)
 \\
 && \leq D_{3} \sqrt{t \log (1/t)}
\eean
for $0 \leq t \leq D_{2}$, where $D_{3} = D_{3}( D, K,  \overline{\tau}, \underline{\tau} ).$
Combinig (\ref{eqlemtv:int}) and (\ref{eqlemtv:intfin1}) with the last display, we have
\be
\int_{\| \bx \|_{\infty} \leq 1} \left\vert p_0(\bx) - p_t(\bx) \right\vert \d \bx 
\leq D_{1}  \left\{ t \log (1/t) \right\}^{\frac{\beta \wedge 1}{2}} + D_{3} \sqrt{t \log (1/t)}
\label{eqlemtv:intfinf}
\ee
for $0 \leq t \leq D_{2}$.

Note that 
\bean
\int_{\|\bx\|_{\infty} \geq 1} p_t(\bx) \d \bx
\leq \int_{\|\bx\|_{\infty} \geq \mu_t +  2 \sigma_{t} \sqrt{\log(1/\sigma_t)}} p_t(\bx) \d \bx
+ \int_{1 \leq \|\bx\|_{\infty} \leq \mu_t +  2 \sigma_{t} \sqrt{\log(1/\sigma_t)}} p_t(\bx) \d \bx
\eean
and we bound each term on the RHS.
For $\| \bx\|_{\infty} \geq  \mu_t + 2 \sigma_{t} \sqrt{\log(1/\sigma_t)}$ and $\| \by \|_{\infty} \leq 1$,
we have $\| \bx - \mu_t \by \|_{\infty} \geq \|\bx\|_{\infty}- \mu_t \|\by\|_{\infty} \geq 2 \sigma_{t} \sqrt{\log(1/\sigma_t)}$.

Then, 
\bean
&& \int_{\| \bx \|_{\infty} \geq \mu_t +  2 \sigma_{t} \sqrt{\log(1/\sigma_t)}} p_t(\bx) \d \bx = 
\int_{\| \bx \|_{\infty} \geq \mu_t +  2 \sigma_{t} \sqrt{\log(1/\sigma_t)}} \int_{\| \by \|_{\infty} \leq 1} p_0(\by) \phi_{\sigma_t}(\bx-\mu_t \by) \d \by \d \bx
\\
&&= \int_{\| \by \|_{\infty} \leq 1} p_0(\by) \int_{\| \bx \|_{\infty} \geq \mu_t +  2 \sigma_{t} \sqrt{\log(1/\sigma_t)}} \phi_{\sigma_t}(\bx-\mu_t \by)  \d \bx \d \by
\\
&& \leq \int_{\| \by \|_{\infty} \leq 1} p_0(\by) \int_{ \| \bx - \mu_t \by \|_{\infty} \geq  2 \sigma_{t} \sqrt{\log(1/\sigma_t)}} \phi_{\sigma_t}(\bx-\mu_t \by)  \d \bx \d \by
\\
&& = \int_{\| \by \|_{\infty} \leq 1} p_0(\by) \int_{ \| \bz \|_{\infty} \geq  2 \sqrt{\log(1/\sigma_t)}} \phi_{1}(\bz)  \d \bz \d \by
=  \int_{ \| \bz \|_{\infty} \geq  2 \sqrt{\log(1/\sigma_t)}} \phi_{1}(\bz)  \d \bz
\\
&& \leq  \sum_{i=1}^{D} \int_{\vert z_i \vert \geq 2 \sqrt{\log(1/\sigma_t)}} \phi_{1}(\bz) \d \bz
\leq 2 D  \sigma_t^{2} \leq 4 D  \overline{\tau} t 
,\quad \forall 0 \leq t \leq (2\overline{\tau})^{-1},
\eean
where the third inequality holds by the tail probability of the standard normal distribution.
A simple calculation yields that
\bean
&& \int_{1 \leq \| \bx \|_{\infty} \leq \mu_t + 2\sigma_t \sqrt{\log(1/\sigma_t)} } p_t(\bx) \d \bx 
\\
&& = \int_{1 \leq \| \bx \|_{\infty} \leq \mu_t + 2\sigma_t \sqrt{\log(1/\sigma_t)} } \int_{\|\by\|_{\infty} \leq 1} p_0(\by) \phi_{\sigma_t}(\bx-\mu_t\by) \d \by \d \bx 
\\
&& \leq \int_{1 \leq \| \bx \|_{\infty} \leq \mu_t + 2\sigma_t \sqrt{\log(1/\sigma_t)} } K \int_{\bbR^{D}} \phi_{\sigma_t}(\bx - \mu_t \by ) \d \by  \d \bx
= K \mu_{t}^{-D} \left\vert  \left( 2 \mu_t + 4\sigma_t \sqrt{\log(1/\sigma_t)} \right)^{D} - 2^{D}  \right\vert
\\
&& 
\leq K \mu_{t}^{-D} \left(\vert 2\mu_{t} - 2 \vert +  4 \sigma_{t} \sqrt{\log(1/\sigma_{t})} \right) \sum_{k=0}^{D-1} \left(2 \mu_{t} + 4 \sigma_{t} \sqrt{\log(1/\sigma_{t})} \right)^{{D-1-k}} 2^{k}
\\
&& \leq K 2^{D} \left( 2 \overline{\tau} t + 4 \sqrt{2 \overline{\tau} t \log (1/\sqrt{\underline{\tau}t}) } \right) D 4^{D}
\eean
for $0 \leq t \leq D_{4}$, where $D_{4} = D_{4}(\overline{\tau}, \underline{\tau})$ is a small enough constant so that $D_{4} \leq D_{2}$ and $\mu_{t} + 2 \sigma_{t} \sqrt{\log(1/\sigma_{t})} \leq 2$ for $t \leq D_{4}$.
Combining with the last two displays, we have
\bean
\int_{ \|\bx\|_{\infty} \geq 1 } p_t(\bx) \d \bx 
\leq D_{5} \sqrt{t \log (1/t) }
\eean
for $0 \leq t \leq D_{4}$, where $D_{5} = D_{5}(D,K,\overline{\tau})$.
Therefore, combining (\ref{eqlemtv:intfinf}) with the last display, 
 \bean
&& \int_{\bbR^{D}} \left\vert p_0(\bx) - p_t(\bx) \right\vert \d \bx = \int_{\| \bx \|_{\infty} \leq 1} \left\vert p_0(\bx) - p_t(\bx) \right\vert \d \bx
+ \int_{\| \bx \|_{\infty} \geq 1}  p_t(\bx) \d \bx
\\
&& \leq  D_{1}  \left\{ t \log (1/t) \right\}^{\frac{\beta \wedge 1}{2}} + D_{3} \sqrt{t \log (1/t)} + D_{5} \sqrt{t \log (1/t) } \leq D_{6} \left\{ t \log (1/t) \right\}^{\frac{\beta \wedge 1}{2}} 
\eean
where $D_{6} = D_1 + D_3 + D_5 $.
The assertion follows by redefining the constants.

\end{proof}

\subsection{Proof of Lemma~\ref{sec:covering}}
\begin{proof}
For neural networks $\bff,\widetilde \bff \in \cF_{\rm WSNN}(L,\bd,s,M,\cP_{\bfm})$, let $\{W_{l}, \bb_{l} \}_{l \in [L]}$ and $\{\widetilde W_{l}, \widetilde \bb_{l} \}_{l \in [L]}$ be the parameter matrices of $\bff$ and $\widetilde \bff$, repsectively.
For $l \in [L-1]$, let
\bean
\bff_{l}(\cdot) =  \rho \left( \sum_{j=1}^{m_{l}} R_{l}^{(j)} \left(  W_{l} Q_{l}^{(j)} \cdot + \bb_{l} \right) \right)
\text{\quad and \quad}
\widetilde \bff_{l}(\cdot) = \rho \left( \sum_{j=1}^{m_{l}} R_{l}^{(j)} \left(  \widetilde W_{l} Q_{l}^{(j)} \cdot + \widetilde \bb_{l} \right) \right).
\eean
Given $\epsilon > 0$, assume that all parameter values of $\bff$ and $\widetilde \bff$ are at most $\epsilon$ away from each other.
Then,
\bean
&& \left\| \bff_{l}(\bx) -  \bff_{l}(\widetilde \bx) \right\|_{\infty}
\leq 
\left\| \sum_{j=1}^{m_{l}}   R_{l}^{(j)}  W_{l} Q_{l}^{(j)} ( \bx - \widetilde \bx  ) \right\|_{\infty}
\\
&& \leq \sum_{j=1}^{m_{l}} \left\| W_{l} Q_{l}^{(j)} \left( \bx - \widetilde \bx \right) \right\|_{\infty}
\leq \sum_{j=1}^{m_l} \left\{ d_{l} \left\| W_{l} \right\|_{\infty} \cdot \left\| Q_{l}^{(j)} \left( \bx - \widetilde \bx \right) \right\|_{\infty}  \right\}
\\
&& \leq m_{l} d_{l} M \|\bx - \widetilde \bx\|_{\infty} , \quad l \in [L-1], \  \bx, \widetilde \bx \in \bbR^{d_{l}}.
\eean
because for  $\bz \in \bbR^{n_{2}}$, $n_{1},n_{2} \in \bbN$, we have $\| W \bz \|_{\infty} \leq n_{2} \| W\|_{\infty} \| \bz\|_{\infty}$ for any matrix $W \in \bbR^{n_{1} \times n_{2}}$ and $\| Q \bz \|_{\infty} = \| \bz\|_{\infty}$ for any $n_{1} \times n_{1}$ permutation matrix $Q$.
A simple calculation yields that
\be \begin{split}
& \left\|  \left( \bff_{L-1} \circ \cdots \circ \bff_{l}  \right) (\bx) -  \left( \bff_{L-1} \circ \cdots \circ \bff_{l}  \right) (\widetilde \bx) \right\|_{\infty}
\\
& \leq \left(\prod_{i=l}^{L-1}  m_{i} d_{i} M \right) \|\bx - \widetilde \bx\|_{\infty}, \quad l \in [L-1], \bx,\widetilde \bx \in \bbR^{d_{l}}. \label{eqcov:1}
\end{split} \ee
Similarly, we have
\bean
&& \left\| \bff_{l}(\bx) - \widetilde \bff_{l}(\bx) \right\|_{\infty}
\leq \left\| \sum_{j=1}^{m_{l}} R_{l}^{(j)} \left( \left(  W_{l} - \widetilde W_{l} \right) Q_{l}^{(j)} \bx + \bb_{l} - \widetilde \bb_{l} \right)  \right\|_{\infty}
\\
&& \leq \sum_{j=1}^{m_{l}} \left\| \left( W_{l} - \widetilde W_{l}  \right) Q_{l}^{(j)} \bx + \left( \bb_{l} - \widetilde \bb_{l} \right) \right\|_{\infty}
\leq \sum_{j=1}^{m_l} \left\{ \left\| W_{l} - \widetilde W_{l} \right\|_{\infty} \cdot \left\| Q_{l}^{(j)} \bx \right\|_{\infty} + \left\| \bb_{l} - \widetilde \bb_{l} \right\|_{\infty} \right\}
\\
&& \leq m_{l} \epsilon (1 + d_{l} \|\bx\|_{\infty}) 
\leq m_{l} d_{l} \epsilon (1 + \|\bx\|_{\infty}) , \quad l \in [L-1], \  \bx \in \bbR^{d_{l}}.
\eean
Note also that
\bean
\left\| \widetilde  \bff_{l}  (\bx)   \right\|_{\infty}
\leq  \sum_{i=1}^{m_l} \left\| R_{l}^{(i)} \left( \widetilde W_{l} Q_{l}^{(i)} \bx + \widetilde \bb_{l} \right) \right\|_{\infty}
\leq m_{l} d_{l} M \left( \| \bx \|_{\infty} + 1 \right)
\eean
and
\bean
\left\| \left( \widetilde  \bff_{l} \circ \cdots \circ \widetilde \bff_{1}  \right) (\bx)   \right\|_{\infty}
&& \leq \left(  \prod_{i=1}^{l} m_i d_i M \right) \| \bx \|_{\infty} + \sum_{k=0}^{l-1} \left( \prod_{i=k+1}^{l} m_i d_i M \right)
\\
&& \leq  \left( \| \bx \|_{\infty} + l \right) \prod_{i=1}^{l} \{ m_i d_i ( M \vee 1 ) \} .
\eean
It follows that
\be \begin{split}
& \left\| \left( \bff_{l+1} \circ \widetilde \bff_{l} \circ \cdots \circ \widetilde \bff_{1}  \right) (\bx) - \left( \widetilde \bff_{l+1} \circ \widetilde  \bff_{l} \circ \cdots \circ \widetilde \bff_{1}  \right) (\bx) \right\|_{\infty}
\\
& \leq  m_{l+1} d_{l+1} \epsilon \left\{ 1 + \left\| \left(  \widetilde \bff_{l} \circ \cdots \circ \widetilde \bff_{1} \right)(\bx) \right\|_{\infty} \right\}
\\
& \leq  m_{l+1} d_{l+1} \epsilon \left[ 1 + \left( \| \bx \|_{\infty} + l \right) \prod_{i=1}^{l} \{ m_i d_i ( M \vee 1 ) \}  \right]
\\
& \leq \epsilon  (M \vee 1)^{-1}   \left( \| \bx \|_{\infty} + l + 1 \right) \prod_{i=1}^{ l+1 } \left\{ m_i d_i (M \vee 1) \right\}, \quad l \in [L-2], \ \bx \in \bbR^{d_{1}}. \label{eqcov:2}
\end{split} \ee
Let
\bean
&& \bff^{(l)} (\cdot) =  \left( \bff_{L-1} \circ \cdots \circ \bff_{l+1} \circ \widetilde \bff_{l} \circ \cdots \widetilde \bff_{1}  \right) (\cdot), \quad l \in [L-2]
\\
&& \bff^{(0)}(\cdot) =  \left( \bff_{L-1} \circ \cdots \circ \bff_{1}  \right) (\cdot) \text{\quad and \quad} \bff^{(L-1)}(\cdot) =  \left( \widetilde \bff_{L-1} \circ \cdots \circ \widetilde \bff_{1} \right) (\cdot).
\eean
Combining with (\ref{eqcov:1}) and (\ref{eqcov:2}), we have
\bean
&& \left\| \bff^{(l)}(\bx) - \bff^{(l+1)}(\bx) \right\|_{\infty} 
\\
&& = \left\| \left( \bff_{L-1} \circ \cdots \circ \bff_{l+1} \circ \widetilde \bff_{l} \circ \cdots \widetilde \bff_{1}  \right) (\bx)
- \left( \bff_{L-1} \circ \cdots \circ \bff_{l+2} \circ \widetilde \bff_{l+1} \circ \cdots \widetilde \bff_{1}  \right) (\bx)
\right\|_{\infty}
\\
&&  \leq \epsilon \left( \prod_{i=l+2}^{L-1}  m_{i} d_{i} M \right)   (M \vee 1)^{-1}   \left( \| \bx \|_{\infty} + l + 1 \right) \prod_{i=1}^{ l+1 } \left\{ m_i d_i (M \vee 1) \right\}
\\
&& \leq \epsilon (M \vee 1)^{-1} \left( \| \bx\|_{\infty} + l + 1 \right)  \prod_{i=1}^{L-1} \left\{  m_{i} d_{i} ( M \vee 1) \right\} , \quad l \in [L-3], \bx \in \bbR^{d_{1}}
\eean
and
\bean
&& \left\| \bff^{(L-2)}(\bx) - \bff^{(L-1)}(\bx) \right\|_{\infty} 
\\
&& = \left\|  \left( \bff_{L-1} \circ \widetilde \bff_{L-2} \circ \cdots \circ \widetilde \bff_{1}  \right) (\bx) - \left( \widetilde \bff_{L-1} \circ \cdots \circ \widetilde \bff_{1} \right) (\bx)
\right\|_{\infty}
\\
&& \leq m_{L-1} d_{L-1} \epsilon \left[ 1 +  \left( \| \bx \|_{\infty} + L - 2 \right) \prod_{i=1}^{L - 2} \{ m_i d_i ( M \vee 1 ) \}   \right]
\\
&& \leq \epsilon (M \vee 1)^{-1} \left( \|\bx\|_{\infty}   + L - 1 \right)  \prod_{i=1}^{L-1}\left\{ m_i d_i (M \vee 1) \right\}  , \quad \bx \in \bbR^{d_{1}}.
\eean
A simple calculation yields that
\bean
&& \left\| \bff(\bx) - \widetilde \bff(\bx) \right\|_{\infty}
\leq \sum_{l=0}^{L-2} \left\| W_{L} \left( \bff^{(l)}(\bx) - \bff^{(l+1)}(\bx) \right) \right\|_{\infty} + \left\| \left(W_{L} - \widetilde W_{L} \right) \bff^{(L-1)}(\bx) - \left( \bb_{L} - \widetilde \bb_{L} \right) \right\|_{\infty}
\\
&& \leq d_{L} \left\| W_{L} \right\|_{\infty}  \sum_{l=0}^{L-2} \left\| \bff^{(l)}(\bx) - \bff^{(l+1)}(\bx) \right\|_{\infty} + d_{L} \left\| W_{L} - \widetilde W_{L} \right\|_{\infty} \cdot \left\| \bff^{(L-1)}(\bx) \right\|_{\infty} + \left\| \bb_{L} - \widetilde \bb_{L} \right\|_{\infty}
\\
&& \leq \epsilon d_{L} M (M \vee 1)^{-1} \left\{ \sum_{l = 0}^{L-2} ( l + 1 + \|\bx\|_{\infty}) \right\} \prod_{i=1}^{L-1} \left\{  m_i d_i (M \vee 1) \right\}  
\\
&& \quad + \epsilon d_{L} \left( \| \bx \|_{\infty} + L - 1 \right) \prod_{i=1}^{L-1} \{ m_i d_i ( M \vee 1 ) \}
+ \epsilon
\\
&& \leq  \epsilon d_{L} \left\{  \frac{L (L + 1)}{2} + L \| \bx \|_{\infty} + 2 \right\}
\prod_{i=1}^{L-1} \left\{ m_i d_i (M \vee 1) \right\} , \quad \bx \in \bbR^{d_{1}}.
\eean
For $\bx \in [-C,C]^{d_1}$, the last display is bounded by
\bean
\epsilon d_{L} L  \left( L + C + 3 \right)  \prod_{i=1}^{L-1} \left\{ m_i d_i (M \vee 1) \right\} \defeq \delta.
\eean
The total number of parameters in $\bff$ is $T \defeq \sum_{i=1}^{L} (d_{i}+1)d_{i+1}$ and there are $\binom{T}{s}$ combinations to pick $s$ non-zero paramters.
Since $s \leq T \leq 2 L \|\bd\|_{\infty}^2$ and $\binom{T}{s} \leq T{^s} \leq (2 L \|\bd\|_{\infty}^2 )^{s}$, we have
\bean
&& N \left(\delta, \cF_{\rm WSNN}(L,\bd,s,M,\cP_{\bfm}), \| \cdot \|_{L^{\infty}[-C,C]^{d_{1}}} \right)
\\
&& \leq \sum_{s_{0} = 1}^{s} \binom{T}{s_0}  N \left( \epsilon, [-M,M]^{s_{0}}, \| \cdot \|_{\infty} \right)
\\
&& \leq s \left(2 L \|\bd\|_{\infty}^2 \right)^{s} \left( \frac{2M}{\epsilon} \vee 1 \right)^{s}
\\
&& \leq \left(2 L \|\bd\|_{\infty}^2 \right)^{s+1}  \left( \frac{2 M L d_{L} \left(L + C + 3 \right) \prod_{i=1}^{L-1} \left\{  m_i d_i (M \vee 1) \right\}   }{\delta} \vee 1  \right)^{s}.
\eean
For $\delta < 1$, the last display is bounded by
\bean
\left( \frac{4 L^2 \|\bd\|_{\infty}^2 \left\{  \| \bfm\|_{\infty} \|\bd\|_{\infty} (M \vee 1) \right\}^{L} \left(L + C + 3 \right)   }{\delta}  \right)^{s+1}.
\eean
The assertion follows by taking the logarithm.
\end{proof}

\subsection{ Proof of Lemma~\ref{sec:exbd} }
\begin{proof}
For any $\bx \in [-1,1]^{D}$, we have
\bean
\ell_{\bff}(\bx) \leq \int_{\underline{T}}^{\overline{T}} 2 \bbE \left[ \left\| \bff(\mu_t \bx + \sigma_t \bZ, t) \right\|_2^2\right] + \frac{2 \bbE[ \|\bZ\|_2^2 ]}{\sigma_t^2} \d t
\leq \int_{\underline{T}}^{\overline{T}} \frac{2DF^2 + 2D}{\sigma_t^2} \d t,
\eean
where the last inequality holds because $\| \bff(\bx,t) \|_{\infty} \leq F \sigma_{t}^{-1}$.
Since $\sigma_t \geq \sqrt{\underline{\tau} t }$ for all $t \geq 0$ by (\ref{eqthm:mtstbd}), we have
\be
\vert \ell_{\bff}(\bx) \vert \leq \underline{\tau}^{-1} 2D(F^2 + 1) ( \log \overline{T} - \log \underline{T}).  \label{eqlemexbd:lf}
\ee
Lemma~\ref{secsc:scorebound} implies that there exists a constant $C_{S,2} = C_{S,2}(D, K,\tau_1, \overline{\tau}, \underline{\tau}) > 0 $ such that 
\bean
\| \bff_0(\mu_t \bx + \sigma_t \bz,t) \|_2 \leq \frac{C_{S,2}}{\sigma_t} \left( \|\bz\|_{\infty} \vee 1 \right)
\eean
for any $\bx \in [-1,1]^{D}$ and $\bz \in \bbR^{D}$. 
Then,
\bean
\ell_{\bff_0} (\bx) && \leq \int_{\underline{T}}^{\overline{T}} 2 \bbE \left[ \left\| \bff_0(\mu_t \bx + \sigma_t \bZ, t) \right\|_2^2\right] + \frac{2 \bbE[ \|\bZ\|_2^2 ]}{\sigma_t^2} \d t
\\
&& \leq \int_{\underline{T}}^{\overline{T}} \frac{2 C_{S,2}^2  \bbE[ \left( \| \bZ\|_{\infty} \vee 1 \right)^2  ] +  2D }{\sigma_t^2} \d t.
\eean
Since $(\| \bz \|_{\infty} \vee 1)^2 \leq (\| \bz \|_{\infty} + 1)^2 \leq (\| \bz \|_2 + 1)^2 \leq 2 \| \bz \|_2^2 + 2$, the last display is bounded by
\bean
\int_{\underline{T}}^{\overline{T}} \frac{2 C_{S,2}^2  (2D + 2)  +  2D }{\sigma_t^2} \d t \leq \underline{\tau}^{-1} \left\{ 2 C_{S,2}^2  (2D + 2)  +  2D \right\} ( \log \overline{T} - \log \underline{T}).
\eean
Combining the last display with (\ref{eqlemexbd:lf}), we have
\bean
\vert \nu_{\bff}(\bx) \vert = \vert \ell_{\bff}(\bx) - \ell_{\bff_0}(\bx) \vert \leq \ell_{\bff}(\bx) + \ell_{\bff_0}(\bx) \leq D_{1} (F^2 \vee 1) (\log \overline{T} - \log \underline{T}),
\eean
where $D_{1} = D_{1}(D, \underline{\tau}, C_{S,2})$. The first assertion follows by redefining the constant.

Simple calculation yields that
\bean
 && \vert \nu_{\bff}(\bx) \vert = \vert \ell_{\bff}(\bx) - \ell_{\bff_0}(\bx) \vert
= \left\vert \bbE \left[ \int_{\underline{T}}^{\overline{T}} \left\| \bff(\mu_t \bx + \sigma_t \bZ,t) + \frac{\bZ}{\sigma_t} \right\|_2^2 - \left\| \bff_0(\mu_t \bx + \sigma_t \bZ,t) + \frac{\bZ}{\sigma_t} \right\|_2^2 \d t   \right] \right\vert
\\
&& = \left\vert \bbE \left[ \int_{\underline{T}}^{\overline{T}} \left\{ (\bff - \bff_0) (\mu_t \bx + \sigma_t \bZ,t)  \right\}^{\top} \left\{ (\bff + \bff_0) (\mu_t \bx + \sigma_t \bZ,t)  + \frac{2\bZ}{\sigma_t} \right\} \d t
\right] \right\vert
\\
&& \leq \bbE\left[  \sqrt{ \int_{\underline{T}}^{\overline{T}} \left\| (\bff - \bff_0)(\mu_t \bx + \sigma_t \bZ,t)  \right\|_2^2 \d t }
\sqrt{ \int_{\underline{T}}^{\overline{T}} \left\| (\bff + \bff_0)(\mu_t \bx + \sigma_t \bZ,t) + \frac{2\bZ}{\sigma_t}  \right\|_2^2 \d t }
\right]
\\
&& \leq \sqrt{ \bbE\left[ \int_{\underline{T}}^{\overline{T}} \left\| (\bff - \bff_0)(\mu_t \bx + \sigma_t \bZ,t)  \right\|_2^2 \d t  \right] }
\sqrt{ \bbE\left[ 
 \int_{\underline{T}}^{\overline{T}} \left\| (\bff + \bff_0)(\mu_t \bx + \sigma_t \bZ,t) + \frac{2\bZ}{\sigma_t}  \right\|_2^2 \d t 
\right] }
\\
&& \leq \sqrt{ \bbE\left[ \int_{\underline{T}}^{\overline{T}} \left\| (\bff - \bff_0)(\mu_t \bx + \sigma_t \bZ,t)  \right\|_2^2 \d t  \right] }
\sqrt{ 2 \ell_{\bff}(\bx) + 2 \ell_{\bff_0}(\bx) },
\eean
where the first and second inequalities hold by the Cauchy-Schwarz inequality.
Combining with the last two displays, the second assertion follows by
\bean
\bbE \left[ \left\{ \nu_{\bff}(\bX_0) \right\}^2 \right]
&& \leq \bbE\left[  \left( 2 \ell_{\bff}(\bX_0) + 2 \ell_{\bff_0}(\bX_0) \right) \int_{\underline{T}}^{\overline{T}} \left\| (\bff - \bff_0)(\mu_t \bX_0 + \sigma_t \bZ,t)  \right\|_2^2 \d t 
\right]
\\
&& \leq 2 D_{1} (F^2 \vee 1) (\log \overline{T} - \log \underline{T}) \int_{\underline{T}}^{\overline{T}} \bbE \left[ \left\| \bff(\bX_t,t)  - \bff_0(\bX_t,t) \right\|_2^2 \right]  \d t .
\eean
\end{proof}

\subsection{  Proof of Proposition~\ref{sec:oracle} }

\begin{proof}
Let $\cF$ be the class of functions $\bff : \bbR^{D} \times \bbR \rightarrow \bbR^{D}$ satisfying $\| \bff(\cdot,t)\|_{L^{\infty}(\bbR^{D})} \leq F \sigma_{t}^{-1}$ for all $t \in [\underline{T}, \overline{T}]$.
Let $\cV = \{\nu_{\bff}(\cdot) : \bff \in \cF  \}$.
For $\nu \in \cV$, denote
\bean
R(\nu) = \int_{[-1,1]^{D}} \nu(\bx) p_0(\bx) \d \bx
\text{\quad and \quad}
\widehat R_{n}(\nu) = \frac{1}{n} \sum_{i=1}^{n} \nu(\bX^i).
\eean
Moreover, let
\bean
\widetilde{\nu} \in \argmin_{\nu \in \cV} R(\nu)
\text{\quad and \quad}
\widehat{\nu} \in \argmin_{\nu \in \cV} \widehat R_{n}(\nu).
\eean
Note 
\be
\begin{split}
R(\widehat \nu)
& = R(\widehat \nu) - \widehat R_{n}(\widehat \nu) + \widehat R_{n}(\widehat \nu)
\\
& \leq R(\widehat \nu) - \widehat R_{n}(\widehat \nu) + \widehat R_{n}(\widetilde \nu)
\\
& = R(\widetilde \nu) + \left\{ R(\widehat \nu) - \widehat R_{n}(\widehat \nu) \right\} + \left\{ \widehat R_{n}(\widetilde \nu) - R(\widetilde \nu) \right\},
\label{eq:exbd_totalbd}
\end{split}
\ee
where the inequality holds because $\widehat \nu$ is the ERM estimator.
We bound each bracket term on the RHS.

Let $\{\nu_1,\ldots,\nu_{N_0}\} \subseteq \cV$ be a minimal $\epsilon$-covering of $\cV$ in $\| \cdot \|_{L^{\infty}([-1,1]^{D})}$-norm.
Then, there exists $j_{*} \in \{1,\ldots,N_0\}$ such that $\| \widehat \nu - \nu_{j_*} \|_{L^{\infty}([-1,1]^{D})} \leq \epsilon$.
A simple calculation yields that
\be
 \vert R(\widehat \nu) - \widehat R_{n}(\widehat \nu) \vert
 \leq 2\epsilon +  \vert R(\nu_{j_*}) - \widehat R_{n}(\nu_{j_*}) \vert.
 \label{eq:exbdcov}
\ee
For any $\nu \in \cV$, Lemma~\ref{sec:exbd} implies that
\bean
\vert \nu(\bX^i) - R(\nu) \vert \leq D_1
\text{\quad and \quad}
\bbE \left[ \left\{ \nu(\bX^i) - R(\nu) \right\}^2 \right] \leq D_{1} R(\nu)
\quad \forall i \in [n],
\eean
where $D_{1} = 2 \widetilde C_{14} (F^2 \vee 1) (\log \overline{T} - \log \underline{T})$ and $\widetilde C_{14} = \widetilde C_{14} (D, K, \tau_1, \overline{\tau}, \underline{\tau} )$ is the contstant in Lemma~\ref{sec:exbd}.
Lemma 2.2.9 from \cite{van1996weak}, which is standard Bernstein's inequality for bounded random variables, implies that for any $r \geq 0$ and $\nu \in \cV$,
\bean
\bbP \left(  \left\vert R(\nu) - \widehat R_{n}(\nu) \right\vert \geq \frac{1}{n} \left( \frac{2D_{1}r}{3} + \sqrt{2 n D_{1} R(\nu) r} \right) \right) \leq 2 e^{-r}
\eean
because for any $M,v,t,x \geq 0$ with $x \geq 2Mt/3 + \sqrt{2vt}$, $x^2/(2v+2Mx/3) \geq t$.
For any $\delta > 0$,
$ \sqrt{ 2 D_{1} R(\nu) r / n } \leq \delta D_{1} R(\nu) + r/(n\delta)$
because $\sqrt{2 D_{1} R(\nu) r / n} = \sqrt{2  \{ \delta D_1 R(\nu)\} \{  r /(n\delta) \}   }$ and $\sqrt{2xy} \leq x + y$ for any $x,y \geq 0$.
Combining with the last two displays, we have
\bean
\bbP \left(  \left\vert R(\nu) - \widehat R_{n}(\nu) \right\vert \geq  \delta D_{1} R(\nu) + \frac{r}{n} \left( \frac{2D_{1}}{3} + \frac{1}{\delta} \right) \right) \leq 2 e^{-r}
\eean
for any $\delta, r > 0$ and $\nu \in \cV$.
The union bound implies that
\bean
\bbP \left(  \left\vert R(\nu_j) - \widehat R_{n}(\nu_j) \right\vert \geq  \delta D_{1} R(\nu_j) + \frac{r + \log N_0}{n} \left( \frac{2D_{1}}{3} + \frac{1}{\delta} \right)
\quad \text{for some } j \in [N_0]
\right) \leq 2 e^{-r}
\eean
and moreover,
\bean
\bbP \left(  \left\vert R(\nu_j) - \widehat R_{n}(\nu_j) \right\vert \leq  \delta D_{1} R(\nu_j) + \frac{r + \log N_0}{n} \left( \frac{2D_{1}}{3} + \frac{1}{\delta} \right)
\quad \text{for all } j \in [N_0]
\right) \geq 1 - 2 e^{-r}
\eean
for any $\delta>0$ and $r > 0$.
Then,
\bean
\bbE \left[ \vert R(\nu_{j_*}) - \widehat R_{n}(\nu_{j_*}) \vert \right]
&& \leq \delta D_{1} \bbE[ R(\nu_{j_*}) ] + \frac{r + \log N_0}{n} \left( \frac{2D_{1}}{3} + \frac{1}{\delta} \right) + 2 D_{1} e^{-r}
\\
&& \leq \delta D_{1} \epsilon + \delta D_{1} \bbE[ R(\widehat \nu) ]  + \frac{r + \log N_0}{n} \left( \frac{2D_{1}}{3} + \frac{1}{\delta} \right) + 2 D_{1} e^{-r}
\eean
for any $\delta>0$ and $r > 0$.
With $\delta = 1/(2D_1)$ and $r = 2 \log (2n)$, simple calculation yields that 
\bean
\bbE \left[ \vert R(\nu_{j_*}) - \widehat R_{n}(\nu_{j_*}) \vert \right]
\leq \frac{\bbE[ R(\widehat \nu) ]}{2}  + \frac{8D_1 (2 \log 2 + 2 \log n + \log N_0)}{3n} + \frac{D_{1}}{2 n^2} + \frac{\epsilon}{2}.
\eean
Combining with (\ref{eq:exbdcov}), we have
\bean
\bbE \left[ \vert R(\widehat \nu) - \widehat R_{n}(\widehat \nu) \vert \right]
\leq \frac {\bbE[ R(\widehat \nu) ]}{2}  + \frac{8D_1 (2 \log (2n) + \log N_0)}{3n} + \frac{D_{1}}{2 n^2} + \frac{5\epsilon}{2}.
\eean
Using the same computation as above for $\bbE \left[ \vert R(\widetilde \nu) - \widehat R_{n}(\widetilde \nu) \vert \right]$, we have
\bean
\bbE \left[ \vert R(\widetilde \nu) - \widehat R_{n}(\widetilde \nu) \vert \right]
\leq \frac{ R(\widetilde \nu)}{2}  + \frac{ 8D_1 (2 \log (2n) + \log N_0 ) }{3n} + \frac{ D_{1}}{2 n^2} + \frac{5\epsilon}{2}.
\eean
Combining (\ref{eq:exbd_totalbd}) with the last two displays, a simple calculation yields that
\bean
\frac{\bbE [R (\widehat \nu) ]}{2} \leq \frac{3 R(\widetilde \nu) }{2} + \frac{16 D_1 (2 \log (2n) + \log N_0)}{3n} + \frac{D_{1}}{ n^2} + 5\epsilon
\eean
and moreover,
\be
\begin{split}
& \int_{\underline{T}}^{\overline{T}} \bbE \left[ \left\| \widehat \bff(\bX_t,t) - \bff_0(\bX_t,t) \right\|_2^2 \right] \d t
\\
& \leq 3 \inf_{\bff \in \cF} \int_{\underline{T}}^{\overline{T}} \bbE \left[ \left\|  \bff(\bX_t,t) - \bff_0(\bX_t,t) \right\|_2^2 \right] \d t
+ \frac{ 32 D_1 (2 \log (2n) + \log N_0)}{3n} + \frac{2D_{1}}{ n^2} + 10\epsilon, \label{eq:exbdfin}
\end{split}
\ee
where $\widehat \nu = \nu_{\widehat \bff}$ for some $\widehat \bff \in \cF$.

Let $\widetilde \epsilon > 0$, which will be specified later and $C = (1 + 2 \sqrt{\log (1 / \widetilde \epsilon)}) \vee \overline{T}$.
Consider functions $\bff_1, \bff_2 \in \cF$ such that $\| \bff_1(\cdot) - \bff_2(\cdot) \|_{L^{\infty}([-C,C]^{D+1})} \leq \widetilde \epsilon$.
Then,
\be
\begin{split}
& \left\vert \nu_{\bff_1}(\bx) - \nu_{\bff_2}(\bx)  \right\vert
 = \left\vert \ell_{\bff_1}(\bx) - \ell_{\bff_2}(\bx)  \right\vert
\\
& = \left\vert \bbE \left[ \int_{\underline{T}}^{\overline{T}} \left\| \bff_1(\mu_t \bx + \sigma_t \bZ,t) + \frac{\bZ}{\sigma_t} \right\|_2^2 - \left\| \bff_2(\mu_t \bx + \sigma_t \bz,t) + \frac{\bZ}{\sigma_t} \right\|_2^2 \d t   \right] \right\vert
\\
& = \left\vert \bbE \left[ \int_{\underline{T}}^{\overline{T}} \left\{ (\bff_1 - \bff_2) (\mu_t \bx + \sigma_t \bZ,t)  \right\}^{\top} \left\{ (\bff_1 + \bff_2) (\mu_t \bx + \sigma_t \bZ,t)  + \frac{2\bZ}{\sigma_t} \right\} \d t
\right] \right\vert
\\
& \leq \bbE\left[  \sqrt{ \int_{\underline{T}}^{\overline{T}} \left\| (\bff_1 - \bff_2)(\mu_t \bx + \sigma_t \bZ,t)  \right\|_2^2 \d t }
\sqrt{ \int_{\underline{T}}^{\overline{T}} \left\| (\bff_1 + \bff_2)(\mu_t \bx + \sigma_t \bZ,t) + \frac{2\bZ}{\sigma_t}  \right\|_2^2 \d t }
\right]
\\
& \leq \sqrt{ \bbE\left[ \int_{\underline{T}}^{\overline{T}} \left\| (\bff_1 - \bff_2)(\mu_t \bx + \sigma_t \bZ,t)  \right\|_2^2 \d t  \right] }
\sqrt{ \bbE\left[ 
 \int_{\underline{T}}^{\overline{T}} \left\| (\bff_1 + \bff_2)(\mu_t \bx + \sigma_t \bZ,t) + \frac{2\bZ}{\sigma_t}  \right\|_2^2 \d t 
\right] }
\\
& \leq \sqrt{ \bbE\left[ \int_{\underline{T}}^{\overline{T}} \left\| (\bff_1 - \bff_2)(\mu_t \bx + \sigma_t \bZ,t)  \right\|_2^2 \d t  \right] }
\sqrt{ 2 \ell_{\bff_1}(\bx) + 2 \ell_{\bff_2}(\bx) }, 
\label{eq:exbddiff}
\end{split}
\ee
where the first and second inequalities hold by the Cauchy-Schwarz inequality.
Since $\| \bff_1 (\bx,t) - \bff_2 (\bx,t) \|_{2}^2 \leq 2 \| \bff_1(\bx,t) \|_2^2 + 2 \| \bff_2 (\bx,t)\|_2^2 \leq 4 D F^2 \sigma_{t}^{-2}$, we have
\bean
 \bbE\left[ \left\| (\bff_1 - \bff_2)(\mu_t \bx + \sigma_t \bZ,t)  \right\|_2^2 \right] \leq D \widetilde \epsilon^2 + 4 D F^2 \sigma_{t}^{-2} \bbP \left( \| \mu_t\bx + \sigma_{t}\bZ \|_{\infty} \geq C \right)
\eean
for any $t \in [\underline{T}, \overline{T}]$.
Let $\bZ = (Z_1,\ldots,Z_{D})$ and $\bx = (x_1,\ldots,x_{D})$.
Simple calculation yields that
\bean
&& \bbP \left( \| \mu_t\bx + \sigma_{t}\bZ \|_{\infty} \geq C \right)
\leq \sum_{i=1}^{D} \bbP \left( \vert \mu_t x_i + \sigma_{t}Z_i \vert \geq C \right)
\\
&& = \sum_{i=1}^{D} \left\{ \bbP \left(Z_i \geq \sigma_t^{-1} (C - \mu_t x_i) \right) +  \bbP \left( Z_i \leq -\sigma_{t}^{-1}(C+\mu_t x_i) \right) \right\}
\\
&& \leq  2  \sum_{i=1}^{D}\bbP \left(Z_i \geq 2\sqrt{\log (1/\widetilde \epsilon) } \right) ,
\eean
where the last inequality holds because $0 < \sigma_t \leq 1$ and $\mu_t \in [-1,1]$.
Combining with the tail probability of the standard normal distribution, we have
\bean
\bbP \left( \| \mu_t\bx + \sigma_{t}\bZ \|_{\infty} \geq C \right)
\leq 2 D \widetilde \epsilon^2.
\eean
Since $\sigma_{t} \geq \sqrt{\underline{\tau} t }$ for any $t \geq 0$, we have
\be
\begin{split}
 & \int_{\underline{T}}^{\overline{T}} \bbE\left[ \left\| (\bff_1 - \bff_2)(\mu_t \bx + \sigma_t \bZ,t)  \right\|_2^2 \right] \d t
\leq \int_{\underline{T}}^{\overline{T}} D \widetilde \epsilon^2 + 8 D^2 F^2 \sigma_{t}^{-2} \widetilde \epsilon^2 \d t
\\
 & \leq D \overline{T} \widetilde \epsilon^2 + 8D^2F^2 \underline{\tau}^{-1} \widetilde \epsilon^2 (\log \overline{T} - \log \underline{T})
 \leq D_{2} \left\{ \overline{T} + F^2 (\log \overline{T} - \log \underline{T} ) \right\} \widetilde \epsilon^2,
 \label{eq:exdiffbd}
 \end{split}
\ee
where $D_{2} = D_{2}(D, \underline{\tau})$.
For any $\bx \in [-1,1]^{D}$ and $\bff \in \cF$, we have
\bean
\ell_{\bff}(\bx) \leq \int_{\underline{T}}^{\overline{T}} 2 \bbE \left[ \left\| \bff(\mu_t \bx + \sigma_t \bZ, t) \right\|_2^2\right] + \frac{2 \bbE[ \|\bZ\|_2^2 ]}{\sigma_t^2} \d t
\leq \int_{\underline{T}}^{\overline{T}} \frac{2DF^2 + 2D}{\sigma_t^2} \d t,
\eean
where the last inequality holds because $\| \bff(\bx,t) \|_{\infty} \leq F \sigma_{t}^{-1}$.
Moreover,
\bean
\| \ell_{\bff}(\cdot) \|_{L^{\infty}([-1,1]^{D})} \leq \underline{\tau}^{-1} 2D(F^2 + 1) ( \log \overline{T} - \log \underline{T}).
\eean
Combining (\ref{eq:exbddiff}) and (\ref{eq:exdiffbd}) with the last display, we have
\bean
&& \left\vert \nu_{\bff_1}(\bx) - \nu_{\bff_2}(\bx)  \right\vert
\leq \sqrt{4 D_{2} D \underline{\tau}^{-1}  \left\{ \overline{T} + F^2 (\log \overline{T} - \log \underline{T} ) \right\} (F^2 + 1) ( \log \overline{T} - \log \underline{T}) \widetilde \epsilon^2  }
\\
&& \leq D_{3} (F + 1)^2 \left( \sqrt{\overline{T}} + \sqrt{\log \overline{T} - \log \underline{T}} \right) \left( \sqrt{\log \overline{T} - \log \underline{T}} \right) \widetilde \epsilon,
\eean
where $D_{3} = D_{3}(D_2, D, \underline{\tau})$.
Let $\epsilon =  D_{3} (F + 1)^2 ( \sqrt{\overline{T}} + \sqrt{\log \overline{T} - \log \underline{T}})  ( \sqrt{\log \overline{T} - \log \underline{T}} ) n^{-2}$.
Then, 
\bean
N \left( \epsilon, \cV, \| \cdot \|_{L^\infty([-1,1]^{D})}   \right)
\leq N \left( n^{-2}, \cF, \| \cdot \|_{L^\infty([-C,C]^{D+1})}   \right).
\eean
Combining with (\ref{eq:exbdfin}), we have
\bean
&& \int_{\underline{T}}^{\overline{T}} \bbE \left[ \left\| \widehat \bff(\bX_t,t) - \bff_0(\bX_t,t) \right\|_2^2 \right] \d t
\\
&& \leq 3 \inf_{\bff \in \cF} \int_{\underline{T}}^{\overline{T}} \bbE \left[ \left\|  \bff(\bX_t,t) - \bff_0(\bX_t,t) \right\|_2^2 \right] \d t
\\
&& \quad + \frac{ 32 D_1 \left\{ 2 \log (2n) + \log N \left( n^{-2}, \cF, \| \cdot \|_{L^\infty([-C,C]^{D+1})}   \right) \right\}}{3n}
\\
&& \quad + \frac{ D_{1} + 10 D_{3} (F + 1)^2 \left( \sqrt{\overline{T}} + \sqrt{\log \overline{T} - \log \underline{T}} \right) \left( \sqrt{\log \overline{T} - \log \underline{T}} \right) }{n^2}.
\eean
Since $D_{1} = 2 \widetilde C_{14} (F^2 \vee 1) (\log \overline{T} - \log \underline{T})$, there exists a constant $D_{4} = D_{4}(\widetilde C_{14}, D_{3})$ such that
\bean
&& \int_{\underline{T}}^{\overline{T}} \bbE \left[ \left\| \widehat \bff(\bX_t,t) - \bff_0(\bX_t,t) \right\|_2^2 \right] \d t
\\
&& \leq 3 \inf_{\bff \in \cF} \int_{\underline{T}}^{\overline{T}} \bbE \left[ \left\|  \bff(\bX_t,t) - \bff_0(\bX_t,t) \right\|_2^2 \right] \d t
\\
&& \quad + \frac{ D_{5}  \left\{ \log (2n) + \log N \left( n^{-2}, \cF, \| \cdot \|_{L^\infty([-C,C]^{D+1})}   \right)  \right\}  }{n},
\eean
where
\bean
D_{5} = D_{4}  (F^2 \vee 1)  \left( \sqrt{\overline{T}} + \sqrt{\log \overline{T} - \log \underline{T}} \right) \left( \sqrt{\log \overline{T} - \log \underline{T}} \right).
\eean
The assertion follows by redefining the constants.
\end{proof}

\subsection{ Proof of Theorem~\ref{secthm:2} }

In this subsection, we provide the proof of Theorem~\ref{secthm:2} with auxiliary lemmas.
For two probability measures $P$ and $Q$ on $\cX \subseteq \bbR^{D}$, the Kullback-Leibler (KL) divergence is defined as
\bean
\text{KL}(P,Q) = 
\begin{cases}
\int_{\cX} \log \frac{\d P}{\d Q} \d P, & \text{if } P \ll Q
\\
\infty, & \text{else}.
\end{cases}
\eean
We often denote $\text{KL}(P,Q)$ as $\text{KL}(p,q)$, where $p$ and $q$ are densities of $P$ and $Q$, respectively.
Hereafter, $C = C({\rm all})$ means that $C$ is a constant depending on $(\beta,d,D,K,\tau_{\rm min},\tau_{\rm max},\tau_{1},\tau_{2},\overline{\tau},\underline{\tau})$.

\medskip

\textit{Proof of Theorem~\ref{secthm:2}.}

Let $\widetilde \tau_{\rm min} = \tau_{\rm min} (2\beta + d) / d$ and $\widetilde \tau_{\rm max} = \tau_{\rm max}  (2\beta + d) / d$.
Let $C_{3}, C_{4} $ be the constants in Theorem~\ref{secthm:1} depending on $(\beta,d,D,K,\widetilde \tau_{\rm min},\widetilde \tau_{\rm max}, \tau_{1}, \tau_{2}, \overline{\tau}, \underline{\tau})$.
Then, by replacing $m$ in Theorem~\ref{secthm:1} with $n^{d / (2\beta+d)}$, $\underline{T} $ with $m^{- \widetilde \tau_{\rm min}}$, and $\overline{T}$ with $\widetilde \tau_{\rm max} \log m$, there exists a class of permutation matrices $\cP_{\bfm} $ and a class of weight-sharing neural networks $ \cF_{\rm WSNN} =  \cF_{\rm WSNN} (L,\bd,s,M,\cP_{\bfm})$ with
\bean
&& L \leq D_{1} ( \log n )^{6} \log \log n , \quad \| \bd\|_{\infty} \leq D_{1}  n^{\frac{d (D+1) }{2\beta + d}},
\\
&& s \leq D_{1}  n^{\frac{d}{2\beta + d}}  (\log n)^{5} \log \log n, \quad M \leq \exp(D_{1}  \{ \log n \}^6 ),
\\
&& \| \bfm \|_{\infty} \leq D_{1}  n^{\frac{dD}{2\beta + d}}
\eean
satisfying
\be
\inf_{  \bff \in \cF_{\rm WSNN} \cap \cF_{\infty} } \int_{\underline{T}}^{\overline{T}} \int_{\bbR^{D}} \left\| \bff(\bx,t) - \nabla \log p_{t}(\bx)  \right\|_2^2 p_t(\bx) \ \d \bx \d t
\leq D_{1}  n^{-\frac{2\beta}{2\beta + d}} (\log n )^{ 4D+ 4\beta +1 }, \label{eqthm2:approx}
\ee
for every $n \geq C_{4}^{(2\beta + d) / d}$, where $\underline{T} = n^{-\tau_{\rm min}}, \overline{T} = \tau_{\rm max} \log n, D_{1} = D_{1}(\beta, d, D, C_{3})$ and
\bean
\cF_{\infty}
=
\left\{ \| \bff(\cdot, t) \|_{L^\infty(\bbR^{D})} \leq D_{1}  \sigma_{t}^{-1} \sqrt{ \log n} \quad \forall t \in [\underline{T}, \overline{T}]  \right\}.
\eean
Let $\widehat \bff$ be an empirical risk minimizer over the class $\cF_{\rm WSNN} \cap \cF_{\infty}$, defined as in (\ref{eqthm2:erm}).
By the Triangle inequality, we have
\be
\bbE \left[ d_{\rm TV}(P_{0}, \widehat P_{\underline{T}} ) \right]
\leq d_{\rm TV}(P_{0}, P_{\underline{T}})  
+ \bbE \left[ d_{\rm TV}( P_{\underline{T}},  \widehat P_{\underline{T}}) \right]. \label{eqthm2:pftriangle}
\ee
We proceed to control each term on the RHS separately.

Let $\widetilde C_{12} = \widetilde C_{12}(\beta, D, K, \overline{\tau}, \underline{\tau})$ and $\widetilde C_{13}(\overline{\tau}, \underline{\tau})$ be the constants in Lemma~\ref{secsc:p0pt}. 
Let $ D_{2} = D_{2} ( \beta, d, \widetilde C_{13}, C_{4} )  $ be a positive constant such that $ D_{2} \geq C_{4}^{(2\beta + d) / d} $ and $ \underline{T} = n^{-\tau_{\rm min}} \leq \widetilde C_{13} $ for every $n \geq D_{2}$.
Then, Lemma~\ref{secsc:p0pt} implies that
\bean
&& d_{\rm TV}(P_{0}, P_{\underline{T}}) = \frac{1}{2} \int_{\bbR^{D}} \left\vert p_0(\bx) - p_{\underline{T}}(\bx) \right\vert \d \bx
\\
&& \leq 2^{-1} \widetilde  C_{12} \left\{ \underline{T} \log (1/ \underline{T}) \right\}^{\frac{ \beta \wedge 1 }{2}} 
= 2^{-1} \widetilde C_{12} \tau_{\rm min}^{(\beta \wedge 1)/2}   n^{- \frac{ \tau_{\rm min} (\beta \wedge 1) }{2}} ( \log n)^{\frac{\beta \wedge 1}{2}},
\eean
for every $n \geq D_{2}$.
Since $\tau_{\rm min} \geq \frac{2\beta}{(2\beta + d)(\beta \wedge 1)}$, the last display is bounded by
\be
2^{-1} \widetilde C_{12} \tau_{\rm min}^{(\beta \wedge 1)/2}   n^{-\frac{\beta }{2\beta + d}} (\log n)^{\frac{\beta \wedge 1}{2}}. \label{eqthm2:firbd}
\ee


Consider functions $q_1, q_2 : \bbR^{D} \times [0,\overline{T}-\underline{T}] \rightarrow \bbR$ such that $q_1(\bx,t) = p_{\overline{T}-t}(\bx)$ and $q_2(\bx,t) = \widehat p_{\overline{T}-t} (\bx)$.
Then, each $q_1$ and $q_2$ satisfy the corresponding well-known Fokker-Planck equation \citep{lebris2008existence, bogachev2022fokker, pavliotis2014stochastic} :
\bean
&& \frac{\partial }{\partial t} q_1(\bx,t) = - \sum_{i=1}^{D} \frac{\partial}{\partial x_i} \left[ \bb_1 (\bx,t) q_1(\bx,t) \right] + \sum_{i=1}^{D} \sum_{j=1}^{D} \frac{\partial^2}{\partial x_i \partial x_j} \left[ a_1(t) \delta_{ij} q_1(\bx,t) \right]
\\
&& \frac{\partial }{\partial t} q_2(\bx,t) = - \sum_{i=1}^{D} \frac{\partial}{\partial x_i} \left[ \bb_2 (\bx,t) q_2(\bx,t) \right] + \sum_{i=1}^{D} \sum_{j=1}^{D} \frac{\partial^2}{\partial x_i \partial x_j} \left[ a_2(t) \delta_{ij} q_2(\bx,t) \right],
\eean
where $\delta_{ij}$ denotes the Kronecker delta and
\bean
&& \bb_1 (\bx,t) = \alpha_{\overline{T}-t} \bx + 2 \alpha_{\overline{T}-t} \nabla \log p_{\overline{T}-t} (\bx), \quad a_1(t) = \alpha_{\overline{T}-t},
\\
&& \bb_2 (\bx,t) = \alpha_{\overline{T}-t} \bx + 2 \alpha_{\overline{T}-t} \widehat \bff( \bx, \overline{T}-t), \quad a_2(t) =  \alpha_{\overline{T}-t}.
\eean
Following the Remark 2.3 of \citet{bogachev2016distances}, we have
\be \begin{split}
& d_{\rm TV} (P_{\underline{T}},  \widehat P_{\underline{T}}) = d_{\rm TV} \left( q_{1}(\cdot, \overline{T} - \underline{T}),  q_{2}(\cdot, \overline{T} - \underline{T})  \right) 
\\
& \leq d_{\rm TV} \left( q_{1}(\cdot, 0),  q_{2}(\cdot, 0 )  \right)  + 
\left( \int_{\underline{T}}^{\overline{T}} \int_{\bbR^{D}} 4 \alpha_{t} \left\| \widehat \bff (\bx,t) - \nabla \log p_t( \bx) \right\|_2^2 p_t(\bx) \d \bx \d t \right)^{1/2}.
\label{eqthm2:secineq}
\end{split} \ee

Recall that $\widehat p_{\overline{T}} = \phi_1$. By Pinsker's inequality, we have
\be
d_{\rm TV} \left( q_{1}(\cdot, 0),  q_{2}(\cdot, 0)  \right) 
= d_{\rm TV} \left( p_{\overline{T}}, \phi_1 \right) 
\leq \sqrt{  \text{KL} \left( p_{\overline{T}}, \phi_{1} \right) / 2 }.
\label{eqthm2:thiineq}
\ee
Since KL divergence is convex in its first argument, Jensen's inequality implies that
\bean
 \text{KL} \left( p_{\overline{T}}, \phi_{1} \right) \leq  \int_{\bbR^{D}}  \text{KL} \left( \cN(\mu_{\overline{T}} \by, \sigma_{\overline{T}} \bbI_{D}) , \cN(\mathbf{0}_{D}, \bbI_{D}) \right) \d P_0(\by).
\eean
because $p_{\overline{T}}(\cdot) = \int_{\bbR^{D}} \phi_{\sigma_{\overline{T}}}(\cdot - \mu_{\overline{T}} \by) \d P_{0}(\by)$.
KL divergence between two $D$-dimensional Gaussian random variables is known as
\bean
&& \text{KL} \left( \cN(\mu_1, \Sigma_1)   , \cN(\mu_2, \Sigma_2) \right)
 \\
&& = \frac{1}{2} \left[ \log \left( \frac{ \det(\Sigma_{2}) }{\det(\Sigma_{1})} \right) - D + (\mu_1 - \mu_2)^{\top} \Sigma_{2}^{-1}  (\mu_1-\mu_2) + {\rm tr}\left( \Sigma_2^{-1} \Sigma_1 \right) \right].
\eean
Using that, we have
\bean
&&  \text{KL} \left( p_{\overline{T}}, \phi_{1} \right)
\leq \int_{\bbR^{D}} \frac{1}{2} \left( \mu_{\overline{T}}^2 \| \by \|_2^2 + D \sigma_{\overline{T}} - D - D \log \sigma_{\overline{T}}  \right) \d P_{0}(\by)
\\
&& = \frac{1}{2} \left( \mu_{\overline{T}}^2 \bbE_{P_0} \left[ \| \bX_{0} \|_2^2 \right]   + D \sigma_{\overline{T}} - D - D \log \sigma_{\overline{T}}  \right)
\\
&& \leq \frac{D}{2} \left\{ \mu_{\overline{T}}^2 + \left\vert \sqrt{1 - \mu_{\overline{T}}^2} -1 \right\vert + \left\vert \log \left(1 - \mu_{\overline{T}}^2 \right) / 2 \right\vert  \right\},
\eean
where the last inequality holds because $\bbE [ \| \bX_{0} \|_2^2 ] \leq D$ and $\mu_{\overline{T}}^2 + \sigma_{\overline{T}}^2 = 1$.
Since $\mu_{\overline{T}} \leq \exp(-\underline{\tau} \overline{T})$ and $ \vert \log (1-x) \vert \leq x/(1-x)$ for $ 0 < x < 1$, the last display is bounded by
\bean
\frac{D}{2} \left\{ \mu_{\overline{T}}^2 + \frac{\mu_{\overline{T}}^2}{\sqrt{1-\mu_{\overline{T}}^2} + 1} + \frac{ \mu_{\overline{T}}^2}{2 ( 1 - \mu_{\overline{T}}^2  ) }    \right\}.
\eean
Note also that $\mu_{\overline{T}}^2 \geq \exp ( - 2 \overline{\tau} \overline{T}  ) =  n^{- 2 \overline{\tau} \tau_{\rm max} }$ and $ \mu_{\overline{T}}^2 \leq \exp ( - 2 \underline{\tau} \overline{T}  ) = n^{ - 2 \underline{\tau} \tau_{\rm max}} $ by the definition. 
Let $D_{3} = D_{3} ( \tau_{\rm max}, \overline{\tau}, D_{2} ) > 0$ be a constant such that $D_{3} \geq D_{2} $ and $n^{- 2 \overline{\tau} \tau_{\rm max} } \leq 1/2 $ for every $n \geq D_{3}$.
Then, the last display is further bounded by
\bean
\frac{5 D \mu_{\overline{T}}^2 }{4}
\leq \left( \frac{5D}{4} \right)  n^{ - 2 \underline{\tau} \tau_{\rm max}}
\leq \left( \frac{5 D}{4} \right) n^{-\frac{2\beta}{2\beta + d}}, 
\eean
where the last inequality holds because $\tau_{\rm max} \geq \frac{\beta}{\underline{\tau} (2\beta + d)}$.
Combining with (\ref{eqthm2:thiineq}), we have
\be
d_{\rm TV} \left( q_{1}(\cdot, 0),  q_{2}(\cdot, 0)  \right) 
= d_{\rm TV} \left( p_{\overline{T}}, \phi_1 \right)   \leq \sqrt{5D/8} n^{-\frac{\beta}{2\beta + d}},
\quad \forall n \geq D_{3}.
\label{eqthm2:thibd}
\ee


Lemma~\ref{sec:covering} implies that there exists a constant $D_{4} = D_{4} (\beta, d, D, D_{1})$ such that the metric entropy bound for $\cF$ follows by
\be
\log N \left(n^{-2}, \cF, \| \cdot \|_{L^{\infty}([-D_{5},D_{5}]^{D+1})} \right) \leq D_{4} n^{\frac{d}{2\beta +d}} (\log n)^{17} (  \log \log n )^2, \label{eqthm2:covering}
\ee
where $D_{5} = (1+2\sqrt{2 \log n}) \vee (\tau_{\rm max} \log n)$.
Let $\widetilde C_{14} = \widetilde C_{14} ( D, K, \tau_{1}, \overline{\tau}, \underline{\tau})$ be the constant in Proposition~\ref{sec:oracle}.
Then, by replacing $F$ in Proposition~\ref{sec:oracle} with $D_{1} \sqrt{ \log n}$, it follows that
\bean
&& \int_{\underline{T}}^{\overline{T}} \bbE \left[  \left\| \widehat \bff (\bX_t,t) - \nabla \log p_t(X_t) \right\|_2^2 \right] \d t
\\
&& \leq 3 \inf_{\bff \in \cF \cap \cF_{\infty}} \int_{\underline{T}}^{\overline{T}} \bbE \left[  \left\| \bff (\bX_t,t) - \nabla \log p_t( \bX_t) \right\|_2^2 \right] \d t
\\
&& \quad + \frac{D_{6} (\log n)^2}{n} \left\{\log N \left(n^{-2}, \cF \cap \cF_{\infty}, \| \cdot \|_{L^{\infty}([-D_{5},D_{5}]^{D+1})} \right) + \log(2n) \right\},
\eean
where $D_{6} = D_{6}( \tau_{\rm max}, \tau_{\rm min}, D_{1}, \widetilde C_{14} )$.
Combining (\ref{eqthm2:approx}) and (\ref{eqthm2:covering}) with the last display, Jensen's inequality, we have
\bean
 && \bbE \left[ \left\{ \int_{\underline{T}}^{\overline{T}}   \left\| \widehat \bff (\bX_t,t) - \nabla \log p_t( \bX_t) \right\|_2^2 \d t \right\}^{1/2}  \right]
 \\
 && \leq \left\{ \int_{\underline{T}}^{\overline{T}} \bbE \left[    \left\| \widehat \bff (\bX_t,t) - \nabla \log p_t( \bX_t) \right\|_2^2    \right] \d t \right\}^{1/2} 
 \\
 && \leq
D_{7} n^{-\frac{\beta}{2\beta + d}} \left\{ (\log n )^{ 2D+ 2\beta +1 } + (\log n)^{10}   \right\},
\quad \forall n \geq D_{3},
\eean
where the first inequality holds by Jensen's inequality and $D_{7} = D_{7} (D_{1}, D_{4}, D_{6} )$.
Since $\alpha_{t} \leq \overline{\tau} $ for all $t \in [\underline{T}, \overline{T}]$, combining (\ref{eqthm2:secineq}) and (\ref{eqthm2:thibd}) with the last display implies that
\bean
\bbE \left[ d_{\rm TV}  \left (  P_{ \underline{T}}, \widehat P_{\underline{T}} \right) \right]
\leq D_{8} n^{-\frac{\beta}{2\beta + d}} \left\{ (\log n )^{ 2D+ 2\beta +1 } +  (\log n)^{10} \right\},
\quad \forall n \geq D_{3}
\eean
where $D_{8} = D_{8} ( D, \overline{\tau}, D_{6}) $
Combining (\ref{eqthm2:pftriangle}) and (\ref{eqthm2:firbd}) with the last display, we have
\bean
\bbE \left[ d_{\rm TV} \left( P_{0},  \widehat P_{\underline{T}} \right) \right]
\leq D_{9} n^{-\frac{\beta}{2\beta+d}} \left\{ (\log n )^{ 2D+ 2\beta +1 } +  (\log n)^{10} \right\},
\quad \forall n \geq D_{3},
\eean
where $D_{9} = D_{9}( \beta, d, \widetilde C_{12}, D_{7})$.
The assertion follows by re-defining the constants. \hfill\BlackBox\\[2mm]

\end{document}